\documentclass[a4paper]{amsart}
\usepackage{amssymb,amsthm}
\usepackage[abbrev,nobysame]{amsrefs}

\usepackage{appendix}

\theoremstyle{plain}
\newtheorem{theorem}{Theorem}[section]
\newtheorem{lemma}[theorem]{Lemma}

\theoremstyle{definition}
\newtheorem{definition}[theorem]{Definition}
\newtheorem{assumption}[theorem]{Assumption}

\theoremstyle{remark}
\newtheorem{remark}[theorem]{Remark}

\numberwithin{equation}{section}

\begin{document}

\title[Navier--Stokes equations in a curved thin domain, Part III]{Navier--Stokes equations in a curved thin domain, Part III: thin-film limit}

\author[T.-H. Miura]{Tatsu-Hiko Miura}
\address{Department of Mathematics, Kyoto University, Kitashirakawa Oiwake-cho, Sakyo-ku, Kyoto 606-8502, Japan}
\email{t.miura@math.kyoto-u.ac.jp}

\subjclass[2010]{Primary: 35B25, 35Q30, 76D05; Secondary: 35R01, 76A20}

\keywords{Navier--Stokes equations, curved thin domain, slip boundary conditions, thin-film limit, surface fluids}

\begin{abstract}
  We consider the Navier--Stokes equations with Navier's slip boundary conditions in a three-dimensional curved thin domain around a given closed surface.
  Under suitable assumptions we show that the average in the thin direction of a strong solution to the bulk Navier--Stokes equations converges weakly in appropriate function spaces on the limit surface as the thickness of the thin domain tends to zero.
  Moreover, we characterize the limit as a weak solution to limit equations, which are the damped and weighted Navier--Stokes equations on the limit surface.
  We also prove the strong convergence of the average of a strong solution to the bulk equations towards a weak solution to the limit equations by showing estimates for the difference between them.
  In some special case our limit equations agree with the Navier--Stokes equations on a Riemannian manifold in which the viscous term contains the Ricci curvature.
  This is the first result on a rigorous derivation of the surface Navier--Stokes equations on a general closed surface by the thin-film limit.
\end{abstract}

\maketitle

\section{Introduction} \label{S:Intro}

\subsection{Problem settings} \label{SS:In_Prob}
In this paper, as a continuation of \cites{Miu_NSCTD_01,Miu_NSCTD_02}, we consider the Navier--Stokes equations with Navier's slip boundary conditions
\begin{align} \label{E:NS_CTD}
  \left\{
  \begin{alignedat}{3}
    \partial_tu^\varepsilon+(u^\varepsilon\cdot\nabla)u^\varepsilon-\nu\Delta u^\varepsilon+\nabla p^\varepsilon &= f^\varepsilon &\quad &\text{in} &\quad &\Omega_\varepsilon\times(0,\infty), \\
    \mathrm{div}\,u^\varepsilon &= 0 &\quad &\text{in} &\quad &\Omega_\varepsilon\times(0,\infty), \\
    u^\varepsilon \cdot n_\varepsilon &= 0 &\quad &\text{on} &\quad &\Gamma_\varepsilon\times(0,\infty), \\
    [\sigma(u^\varepsilon,p^\varepsilon)n_\varepsilon]_{\mathrm{tan}}+\gamma_\varepsilon u^\varepsilon &= 0 &\quad &\text{on} &\quad &\Gamma_\varepsilon\times(0,\infty), \\
    u^\varepsilon|_{t=0} &= u_0^\varepsilon &\quad &\text{in} &\quad &\Omega_\varepsilon.
  \end{alignedat}
  \right.
\end{align}
Here $\Omega_\varepsilon$ is a curved thin domain in $\mathbb{R}^3$ with small thickness of order $\varepsilon>0$ given by
\begin{align} \label{E:Def_CTD}
  \Omega_\varepsilon := \{y+rn(y) \mid y\in\Gamma,\, \varepsilon g_0(y) < r < \varepsilon g_1(y)\},
\end{align}
where $\Gamma$ is a given closed surface in $\mathbb{R}^3$ with unit outward normal vector field $n$ and $g_0$ and $g_1$ are functions on $\Gamma$ satisfying
\begin{align*}
  g := g_1-g_0 \geq c \quad\text{on}\quad \Gamma
\end{align*}
with some constant $c>0$.
Also, $\nu>0$ is the viscosity coefficient independent of $\varepsilon$.
We define the inner and outer boundaries $\Gamma_\varepsilon^0$ and $\Gamma_\varepsilon^1$ of $\Omega_\varepsilon$ by
\begin{align*}
  \Gamma_\varepsilon^i := \{y+\varepsilon g_i(y)n(y) \mid y\in\Gamma\}, \quad i=0,1
\end{align*}
and write $\Gamma_\varepsilon:=\Gamma_\varepsilon^0\cup\Gamma_\varepsilon^1$ and $n_\varepsilon$ for the whole boundary of $\Omega_\varepsilon$ and its unit outward normal vector field.
Moreover, $\gamma_\varepsilon\geq 0$ is the friction coefficient given by
\begin{align} \label{E:Def_Fric}
  \gamma_\varepsilon := \gamma_\varepsilon^i \quad\text{on}\quad \Gamma_\varepsilon^i,\, i=0,1
\end{align}
with nonnegative constants $\gamma_\varepsilon^0$ and $\gamma_\varepsilon^1$ depending on $\varepsilon$ and
\begin{align*}
  \sigma(u^\varepsilon,p^\varepsilon) := 2\nu D(u^\varepsilon)-p^\varepsilon I_3, \quad [\sigma(u^\varepsilon,p^\varepsilon)n_\varepsilon]_{\mathrm{tan}} := P_\varepsilon[\sigma(u^\varepsilon,p^\varepsilon)n_\varepsilon]
\end{align*}
are the stress tensor and the tangential component of the stress vector on $\Gamma_\varepsilon$, where $I_3$ is the $3\times3$ identity matrix,
\begin{align*}
  D(u^\varepsilon) := \frac{\nabla u^\varepsilon+(\nabla u^\varepsilon)^T}{2}, \quad P_\varepsilon:=I_3-n_\varepsilon\otimes n_\varepsilon
\end{align*}
are the strain rate tensor and the orthogonal projection onto the tangent plane of $\Gamma_\varepsilon$, and $n_\varepsilon\otimes n_\varepsilon$ is the tensor product of $n_\varepsilon$ with itself.
Note that
\begin{align*}
  [\sigma(u^\varepsilon,p^\varepsilon)n_\varepsilon]_{\mathrm{tan}} = 2\nu P_\varepsilon D(u^\varepsilon)n_\varepsilon
\end{align*}
is independent of the pressure $p^\varepsilon$ and thus the slip boundary conditions reduce to
\begin{align} \label{E:Slip_Intro}
  u^\varepsilon\cdot n_\varepsilon = 0, \quad 2\nu P_\varepsilon D(u^\varepsilon)n_\varepsilon+\gamma_\varepsilon u^\varepsilon = 0 \quad\text{on}\quad \Gamma_\varepsilon.
\end{align}
In what follows, we mainly refer to \eqref{E:Slip_Intro} as the slip boundary conditions.

The aim of this paper is to study the behavior of a solution $u^\varepsilon$ to \eqref{E:NS_CTD} as $\varepsilon\to0$.
Our goal is to derive limit equations on $\Gamma$ for \eqref{E:NS_CTD} by showing the convergence of $u^\varepsilon$ as $\varepsilon\to0$ and characterizing its limit as a solution to the limit equations.
Such a problem is a kind of singular limit problem since $\Omega_\varepsilon$ degenerates into the lower dimensional set $\Gamma$ as $\varepsilon\to0$.
The study of a singular limit problem for partial differential equations (PDEs) in thin domains is important for a good understanding of the effects of the thin and other directions on those PDEs.
There are many works on a singular limit problem for reaction-diffusion equations in thin domains around lower dimensional sets (see e.g. \cites{HaRa92b,PrRiRy02,PrRy03,Yan90}).
Also, a few researchers have studied a singular limit problem for the Navier--Stokes equations in flat thin domains in $\mathbb{R}^3$ around two-dimensional domains \cites{Hu07,IfRaSe07,TeZi96} and in a thin spherical shell between two concentric spheres \cite{TeZi97}.
However, there has been no result on a curved thin domain around a general closed surface in the study of the Navier--Stokes equations.

\subsection{Outline of main results} \label{SS:In_Main}
Let us briefly state the main results of the series of our study (see Section \ref{S:Main} for details of the main results of this paper).

In the first part \cite{Miu_NSCTD_01} we studied the Stokes operator $A_\varepsilon$ for $\Omega_\varepsilon$ under the slip boundary conditions \eqref{E:Slip_Intro} and derived the uniform norm equivalence
\begin{align} \label{E:NoEq_In}
  c^{-1}\|u\|_{H^k(\Omega_\varepsilon)} \leq \|A_\varepsilon^{k/2}u\|_{L^2(\Omega_\varepsilon)} \leq c\|u\|_{H^k(\Omega_\varepsilon)}, \quad u\in D(A_\varepsilon^{k/2})
\end{align}
for $k=1,2$ and the uniform difference estimate for $A_\varepsilon$ and $-\nu\Delta$ of the form
\begin{align} \label{E:UnDi_In}
  \|A_\varepsilon u+\nu\Delta u\|_{L^2(\Omega_\varepsilon)} \leq c\|u\|_{H^1(\Omega_\varepsilon)}, \quad u\in D(A_\varepsilon)
\end{align}
with a constant $c>0$ independent of $\varepsilon$.
Using them and average operators in the thin direction we established in the second part \cite{Miu_NSCTD_02} the global-in-time existence of a strong solution $u^\varepsilon$ to \eqref{E:NS_CTD} for large data $u_0^\varepsilon$ and $f^\varepsilon$ in the sense that
\begin{align*}
  \|u_0^\varepsilon\|_{H^1(\Omega_\varepsilon)},\, \|f^\varepsilon\|_{L^\infty(0,\infty;L^2(\Omega_\varepsilon))} = O(\varepsilon^{-1/2})
\end{align*}
when $\varepsilon$ is sufficiently small.
We also derived estimates for $u^\varepsilon$ of the form
\begin{align} \label{E:UE_In}
  \begin{aligned}
    \|u^\varepsilon(t)\|_{H^k(\Omega_\varepsilon)}^2 &\leq c(\varepsilon^{-2k+1+\alpha}+\varepsilon^{-k+\beta}), \\
    \int_0^t\|u^\varepsilon(s)\|_{H^{k+1}(\Omega_\varepsilon)}^2\,ds &\leq c(\varepsilon^{-2k+1+\alpha}+\varepsilon^{-k+\beta})(1+t)
  \end{aligned}
\end{align}
for $t\geq0$ and $k=0,1$ with constants $c,\alpha,\beta>0$ independent of $\varepsilon$ under additional conditions on $u_0^\varepsilon$ and $f^\varepsilon$ (see Theorem \ref{T:UE}).
In this paper, based on the results of \cites{Miu_NSCTD_01,Miu_NSCTD_02}, we study the behavior as $\varepsilon\to0$ of the average in the thin direction of $u^\varepsilon$ given by
\begin{align*}
  Mu^\varepsilon(y,t) := \frac{1}{\varepsilon g(y)}\int_{\varepsilon g_0(y)}^{\varepsilon g_1(y)}u^\varepsilon(y+rn(y),t)\,dr, \quad (y,t)\in\Gamma\times(0,\infty).
\end{align*}
Under suitable assumptions we show that the normal and tangential components
\begin{align*}
  Mu^\varepsilon\cdot n, \quad M_\tau u^\varepsilon := Mu^\varepsilon-(Mu^\varepsilon\cdot n)n
\end{align*}
of the average converge strongly to zero on $\Gamma$ and weakly to a tangential vector field $v$ on $\Gamma$, respectively, and characterize $v$ as a unique weak solution to limit equations on $\Gamma$ (see Theorem \ref{T:SL_Weak}).
We also establish the strong convergence of $M_\tau u^\varepsilon$ to $v$ by estimating their difference (see Theorem \ref{T:SL_Strong}).

\subsection{Limit equations in the simplest case} \label{SS:In_Lim}
When $g\equiv1$ and $\gamma_\varepsilon=0$, i.e. the thickness of $\Omega_\varepsilon$ is $\varepsilon$ and we impose the perfect slip boundary conditions
\begin{align} \label{E:PeSl_In}
  u^\varepsilon\cdot n_\varepsilon = 0, \quad 2\nu P_\varepsilon D(u^\varepsilon)n_\varepsilon = 0 \quad\text{on}\quad \Gamma_\varepsilon,
\end{align}
our limit equations given in Theorem \ref{T:SL_Weak} are of the form
\begin{align} \label{E:Lim_In_Ex}
  \left\{
  \begin{alignedat}{3}
    \partial_tv+\overline{\nabla}_vv-2\nu P\mathrm{div}_\Gamma[D_\Gamma(v)]+\nabla_\Gamma q &= f &\quad &\text{on} &\quad &\Gamma\times(0,\infty), \\
    \mathrm{div}_\Gamma v &= 0 &\quad &\text{on} &\quad &\Gamma\times(0,\infty)
  \end{alignedat}
  \right.
\end{align}
for a tangential vector field $v$ on $\Gamma$ with initial condition.
Here $\nabla_\Gamma$ and $\mathrm{div}_\Gamma$ denote the tangential gradient and the surface divergence on $\Gamma$.
Also, $\overline{\nabla}_vv$ is the covariant derivative of $v$ along itself and
\begin{align*}
  P := I_3-n\otimes n, \quad D_\Gamma(v) := P\left(\frac{\nabla_\Gamma v+(\nabla_\Gamma v)^T}{2}\right)P
\end{align*}
are the orthogonal projection onto the tangent plane of $\Gamma$ and the surface strain rate tensor (see Section \ref{SS:Pre_Surf} for details).
When $g\not\equiv1$ and $\gamma_\varepsilon\neq0$, we have a damped and weighted version of \eqref{E:Lim_In_Ex} (see \eqref{E:NS_Limit} for the precise form of the limit equations).
Formally, if $\Gamma=\mathbb{T}^2$ is the flat torus, then the limit equations \eqref{E:Lim_In_Ex} reduce to the usual two-dimensional Navier--Stokes equations on $\mathbb{T}^2$.

The limit equations \eqref{E:Lim_In_Ex} are a priori extrinsic in the sense that they are written in terms of a fixed coordinate system of the ambient space $\mathbb{R}^3$, since we use the $3\times3$ matrices $P$ and $D_\Gamma(v)$ to describe the viscous term of \eqref{E:Lim_In_Ex}.
However, it turns out (see Remark \ref{R:Lim_Reduce}) that the limit equations \eqref{E:Lim_In_Ex} are equivalent to
\begin{align} \label{E:Lim_In_In}
  \left\{
  \begin{alignedat}{3}
    \partial_tv+\overline{\nabla}_vv-\nu\{\Delta_Bv+\mathrm{Ric}(v)\}+\nabla_\Gamma q &= f &\quad &\text{on} &\quad &\Gamma\times(0,\infty), \\
    \mathrm{div}_\Gamma v &= 0 &\quad &\text{on} &\quad &\Gamma\times(0,\infty),
  \end{alignedat}
  \right.
\end{align}
where $\Delta_B$ and $\mathrm{Ric}$ are the Bochner Laplacian on $\Gamma$ and the Ricci curvature of $\Gamma$ (see Appendix \ref{S:Ap_Visc}), and all terms in \eqref{E:Lim_In_In} are intrinsic, i.e. independent of an embedding of $\Gamma$ into $\mathbb{R}^3$.
Thus our limit equations \eqref{E:Lim_In_Ex} are in fact intrinsic.

The limit equations \eqref{E:Lim_In_Ex} and their equivalent form \eqref{E:Lim_In_In} are closely related to the Navier--Stokes equations on surfaces and manifolds.
Recently, the Navier--Stokes equations on an evolving surface in $\mathbb{R}^3$ were derived in \cites{JaOlRe18,KoLiGi17,Miu18}.
Those equations reduce to \eqref{E:Lim_In_Ex} if a surface is stationary.
Moreover, the equations \eqref{E:Lim_In_In} agree with the Navier--Stokes equations on a Riemannian manifold introduced in \cites{EbMa70,MitYa02,Ta92}.
We emphasize that this paper presents the first result on a rigorous derivation of \eqref{E:Lim_In_In} by the thin-film limit when a manifold is a closed surface in $\mathbb{R}^3$.
Our results also justify the result of \cite{Miu18} which derived the surface Navier--Stokes equations by formal calculations of the thin-film limit.

Let us also compare our limit equations \eqref{E:Lim_In_In} with limit equations derived in \cite{TeZi97} under different boundary conditions when $\Gamma=S^2$ is the unit sphere in $\mathbb{R}^3$.
Temam and Ziane \cite{TeZi97} studied the Navier--Stokes equations in a thin spherical shell
\begin{align*}
  \Omega_\varepsilon = \{x\in\mathbb{R}^3 \mid 1<|x|<1+\varepsilon\} \quad (\Gamma=S^2, \, g_0\equiv0, \, g_1\equiv1)
\end{align*}
around the unit sphere under the Hodge (or de Rham) boundary conditions
\begin{align} \label{E:Hod_In}
  u^\varepsilon\cdot n_\varepsilon = 0, \quad \mathrm{curl}\,u^\varepsilon\times n_\varepsilon = 0 \quad\text{on}\quad \Gamma_\varepsilon,
\end{align}
which were called the free boundary conditions in \cite{TeZi97}, and derived limit equations on $S^2$ which were described in spherical coordinates.
In terms of our notations the limit equations derived in \cite{TeZi97} are of the form (see Remark \ref{R:Limit_TZ})
\begin{align} \label{E:Lim_In_TZ}
  \left\{
  \begin{alignedat}{3}
    \partial_tv+\overline{\nabla}_vv-\nu(\Delta_Bv-v)+\nabla_\Gamma q &= f &\quad &\text{on} &\quad &S^2\times(0,\infty), \\
    \mathrm{div}_\Gamma v &= 0 &\quad &\text{on} &\quad &S^2\times(0,\infty).
  \end{alignedat}
  \right.
\end{align}
On the other hand, since the Ricci curvature of $S^2$ is equal to the Gaussian curvature $K=1$ of $S^2$, i.e. $\mathrm{Ric}(X)=X$ on $S^2$ for a tangential vector field $X$ on $S^2$ (see Lemma \ref{L:Cur_Ten}), our limit equations \eqref{E:Lim_In_In} read
\begin{align} \label{E:Lim_In_S2}
  \left\{
  \begin{alignedat}{3}
    \partial_tv+\overline{\nabla}_vv-\nu(\Delta_Bv+v)+\nabla_\Gamma q &= f &\quad &\text{on} &\quad &S^2\times(0,\infty), \\
    \mathrm{div}_\Gamma v &= 0 &\quad &\text{on} &\quad &S^2\times(0,\infty).
  \end{alignedat}
  \right.
\end{align}
Here the sign of the zeroth order term $v$ in our limit equations \eqref{E:Lim_In_S2} is opposite to that of the same term in the limit equations \eqref{E:Lim_In_TZ} derived in \cite{TeZi97}.
This is due to the fact that the boundary conditions \eqref{E:PeSl_In} and \eqref{E:Hod_In} for the original problem are different by the nonzero curvatures of the limit surface $S^2$.
Indeed, if a vector field $u$ on $\overline{\Omega}_\varepsilon$ satisfies $u\cdot n_\varepsilon=0$ on $\Gamma_\varepsilon$, then (see \cite{MitMon09}*{Section 2} and \cite{Miu_NSCTD_01}*{Lemma B.10})
\begin{align*}
  2P_\varepsilon D(u)n_\varepsilon-\mathrm{curl}\,u\times n_\varepsilon = 2W_\varepsilon u \quad\text{on}\quad \Gamma_\varepsilon,
\end{align*}
where $W_\varepsilon$ is the Weingarten map (or the shape operator) of $\Gamma_\varepsilon$ that represents the curvatures of $\Gamma_\varepsilon$ (see Section \ref{SS:Pre_CTD}).
Moreover, since the unit outward normal vector field of the boundary $\Gamma_\varepsilon=\Gamma_\varepsilon^0\cup\Gamma_\varepsilon^1$ of the thin spherical shell is of the form
\begin{align*}
  n_\varepsilon(x) =
  \begin{cases}
    -x, &x\in\Gamma_\varepsilon^0 = \{z\in\mathbb{R}^3 \mid |z| = 1\}, \\
    \displaystyle\frac{x}{1+\varepsilon}, &x\in\Gamma_\varepsilon^1 = \{z\in\mathbb{R}^3 \mid |z|=1+\varepsilon\},
  \end{cases}
\end{align*}
we observe that, under the condition $u\cdot n_\varepsilon=0$ on $\Gamma_\varepsilon$,
\begin{gather*}
  W_\varepsilon u = u \quad\text{on}\quad \Gamma_\varepsilon^0, \quad W_\varepsilon u = -\frac{u}{1+\varepsilon} \approx -u \quad\text{on}\quad \Gamma_\varepsilon^1, \\
  2P_\varepsilon D(u)n_\varepsilon-\mathrm{curl}\,u\times n_\varepsilon = 2W_\varepsilon u \approx \pm 2u \quad\text{on}\quad \Gamma_\varepsilon.
\end{gather*}
Therefore, roughly speaking, the perfect slip boundary conditions \eqref{E:PeSl_In} differ from the Hodge boundary conditions \eqref{E:Hod_In} by $2u$ and this difference results in the difference $2v$ of our limit equations \eqref{E:Lim_In_S2} from the limit equations \eqref{E:Lim_In_TZ} derived in \cite{TeZi97}.
For further discussions on \eqref{E:Lim_In_TZ} and \eqref{E:Lim_In_S2}, see Remark \ref{R:Limit_TZ}.

\subsection{Main ideas and outline of the proofs} \label{SS:In_Idea}
To prove the main results of this paper (Theorems \ref{T:SL_Weak} and \ref{T:SL_Strong}) we use the results of the first and second parts \cites{Miu_NSCTD_01,Miu_NSCTD_02} of our study.
So let us first explain main ideas used in \cites{Miu_NSCTD_01,Miu_NSCTD_02} and then give the outline of the proofs of Theorems \ref{T:SL_Weak} and \ref{T:SL_Strong}.

In \cite{Miu_NSCTD_01} we derived \eqref{E:NoEq_In}, \eqref{E:UnDi_In}, and other basic inequalities on $\Omega_\varepsilon$ such as Poincar\'{e} and trace type inequalities.
The important point of \cite{Miu_NSCTD_01} was to verify that constants in those inequalities are independent of $\varepsilon$ or depend explicitly on $\varepsilon$, since our aim is to study properties of a solution to \eqref{E:NS_CTD} related to the smallness of $\varepsilon$.
To achieve that we carefully analyzed surface quantities of $\Gamma_\varepsilon$ such as curvatures and carried out calculations by using the change of variables formula (here $J$ is the Jacobian)
\begin{align} \label{E:CoV_In}
  \int_{\Omega_\varepsilon}\varphi(x)\,dx = \int_\Gamma\int_{\varepsilon g_0(y)}^{\varepsilon g_1(y)}\varphi(y+rn(y))J(y,r)\,dr\,d\mathcal{H}^2(y)
\end{align}
for a function $\varphi$ on $\Omega_\varepsilon$ and a change of variables formula for an integral over $\Gamma_\varepsilon$.
We also analyzed the behavior on $\Gamma_\varepsilon$ of a vector field on $\overline{\Omega}_\varepsilon$ in order to prove \eqref{E:NoEq_In} with $k=2$.
For that purpose we usually take a local coordinate system of $\Gamma_\varepsilon$ or transform a part of $\Gamma_\varepsilon$ into the boundary of the half space, but in our case these methods would cause too complicated calculations since we were required to deal with the second order derivatives of a vector field.
Instead, we employed the Gauss formula for tangential vector fields $X$ and $Y$ on $\Gamma_\varepsilon$ of the form
\begin{align*}
  (Y\cdot\nabla)X = \overline{\nabla}_Y^\varepsilon X+(W_\varepsilon X\cdot Y)n_\varepsilon \quad\text{on}\quad \Gamma_\varepsilon
\end{align*}
which expresses the directional derivative $(Y\cdot\nabla)X$ in $\mathbb{R}^3$ by the covariant derivative $\overline{\nabla}_Y^\varepsilon X$ on $\Gamma_\varepsilon$ and the second fundamental form $(W_\varepsilon X\cdot Y)n_\varepsilon$ of $\Gamma_\varepsilon$.
Using this formula and other formulas for the covariant derivative we carried out calculations on $\Gamma_\varepsilon$ in a fixed coordinate system of the ambient space $\mathbb{R}^3$.

In \cite{Miu_NSCTD_02} we proved the global existence of a strong solution $u^\varepsilon$ to \eqref{E:NS_CTD} and obtained the estimates \eqref{E:UE_In} for $u^\varepsilon$ by a standard energy method.
The main ingredient for the proof was the following estimate for the trilinear term:
\begin{multline} \label{E:TrEs_In}
  \left|\bigl((u\cdot\nabla)u,A_\varepsilon u\bigr)_{L^2(\Omega_\varepsilon)}\right| \leq \left(\frac{1}{4}+d_1\varepsilon^{1/2}\|A_\varepsilon^{1/2}u\|_{L^2(\Omega_\varepsilon)}\right)\|A_\varepsilon u\|_{L^2(\Omega_\varepsilon)}^2 \\
  +d_2\left(\|u\|_{L^2(\Omega_\varepsilon)}^2\|A_\varepsilon^{1/2}u\|_{L^2(\Omega_\varepsilon)}^4+\varepsilon^{-1}\|u\|_{L^2(\Omega_\varepsilon)}^2\|A_\varepsilon^{1/2}u\|_{L^2(\Omega_\varepsilon)}^2\right)
\end{multline}
for $u\in D(A_\varepsilon)$ with positive constants $d_1$ and $d_2$ independent of $\varepsilon$.
To derive \eqref{E:TrEs_In} we decomposed $u$ into the almost two-dimensional average part $u^a$ and the residual part $u^r$.
Based on a detailed study of the average operators $M$ and $M_\tau$ we proved a good product estimate for $u^a$ of the form
\begin{align*}
  \bigl\|\,|u^a|\,\varphi\bigr\|_{L^2(\Omega_\varepsilon)} \leq c\varepsilon^{-1/2}\|\varphi\|_{L^2(\Omega_\varepsilon)}^{1/2}\|\varphi\|_{H^1(\Omega_\varepsilon)}^{1/2}\|u\|_{L^2(\Omega_\varepsilon)}^{1/2}\|u\|_{H^1(\Omega_\varepsilon)}^{1/2}, \quad \varphi\in H^1(\Omega_\varepsilon)
\end{align*}
with a similar one for $\nabla u^a$ and a good $L^\infty(\Omega_\varepsilon)$-estimate for $u^r$ of the form
\begin{align*}
  \|u^r\|_{L^\infty(\Omega_\varepsilon)} \leq c\left(\varepsilon^{1/2}\|u\|_{H^2(\Omega_\varepsilon)}+\|u\|_{L^2(\Omega_\varepsilon)}^{1/2}\|u\|_{H^2(\Omega_\varepsilon)}^{1/2}\right).
\end{align*}
We applied these estimates, the uniform estimates \eqref{E:NoEq_In} and \eqref{E:UnDi_In} for $A_\varepsilon$, and other inequalities for $M$ and $M_\tau$ to obtain \eqref{E:TrEs_In}, but the actual proof involved long and careful calculations of vector fields on $\Omega_\varepsilon$ and $\Gamma$.

Now let us explain the outline of the proofs of our main results (Theorems \ref{T:SL_Weak} and \ref{T:SL_Strong}).
For the global strong solution $u^\varepsilon$ to \eqref{E:NS_CTD} satisfying \eqref{E:UE_In}, we consider the normal and tangential components $Mu^\varepsilon\cdot n$ and $M_\tau u^\varepsilon$ of the average separately.
First we note that, since $u^\varepsilon$ satisfies the impermeable boundary condition
\begin{align*}
  u^\varepsilon\cdot n_\varepsilon = 0 \quad\text{on}\quad \Gamma_\varepsilon,
\end{align*}
the strong convergence of $Mu^\varepsilon\cdot n$ to zero on $\Gamma$ as $\varepsilon\to0$ easily follows from \eqref{E:UE_In} and a property of $M$ (see Lemma \ref{L:Ave_N_Lp}).
Thus our main task is to analyze the behavior of $M_\tau u^\varepsilon$ as $\varepsilon\to0$.
We first show that $M_\tau u^\varepsilon$ satisfies a weak formulation of the limit equations with a residual term that converges to zero as $\varepsilon\to0$ (see Lemma \ref{L:Mu_Weak}).
Using that weak formulation we next derive an energy estimate for $M_\tau u^\varepsilon$ and an estimate for its time derivative with bounds independent of $\varepsilon$ (see Lemmas \ref{L:Mu_Energy} and \ref{L:Mu_Dt}).
These estimates imply the weak convergence of a subsequence of $M_\tau u^\varepsilon$ in appropriate function spaces on $\Gamma$, and we show that the limit $v$ is a weak solution to the limit equations by sending $\varepsilon\to0$ in the weak formulation for $M_\tau u^\varepsilon$.
We can also prove the uniqueness of a weak solution to the limit equations as in the case of the two-dimensional Navier--Stokes equations.
By this uniqueness result we obtain the weak convergence of the full sequence of $M_\tau u^\varepsilon$ to $v$ as $\varepsilon\to0$, which completes the proof of Theorem \ref{T:SL_Weak} (see Section \ref{SS:SL_WeCh}).
To show the strong convergence of $M_\tau u^\varepsilon$ towards $v$ (Theorem \ref{T:SL_Strong}), we derive an energy estimate for the difference between $M_\tau u^\varepsilon$ and $v$ by using the weak formulations for them (see Section \ref{SS:SL_ErST}).

The proof of Theorem \ref{T:SL_Weak} outlined above seems similar to that of the existence of a weak solution to the Navier--Stokes equations by standard approximation methods such as the Galerkin and Yosida methods (see e.g. \cites{BoFa13,CoFo88,So01,Te79}), but in the actual proof we encounter some new difficulties arising from the complicated geometry of the curved thin domain $\Omega_\varepsilon$ and its limit surface $\Gamma$.
In the first step of the proof of Theorem \ref{T:SL_Weak} (Section \ref{SS:SL_Ave}) we substitute a test function $\eta\in H^1(\Gamma)^3$ for the limit equations satisfying
\begin{align} \label{E:SoSu_In}
  \eta\cdot n = 0, \quad \mathrm{div}_\Gamma(g\eta) = 0 \quad\text{on}\quad \Gamma
\end{align}
in a weak formulation of \eqref{E:NS_CTD} satisfied by $u^\varepsilon$ in order to derive the weak formulation for $M_\tau u^\varepsilon$.
To this end, we need to extend $\eta$ to an appropriate test function for \eqref{E:NS_CTD}, i.e. a vector field $\eta_\varepsilon\in H^1(\Omega_\varepsilon)^3$ close to $\eta$ in an appropriate sense and satisfying
\begin{align} \label{E:SoDo_In}
  \mathrm{div}\,\eta_\varepsilon = 0 \quad\text{in}\quad \Omega_\varepsilon, \quad \eta_\varepsilon\cdot n_\varepsilon = 0 \quad\text{on}\quad \Gamma_\varepsilon.
\end{align}
If $\Gamma$ is flat and $g\equiv1$, we have such a test function just by extending $\eta$ constantly in the normal direction of $\Gamma$.
In our case, however, the constant extension of $\eta$ does not necessarily satisfy \eqref{E:SoDo_In} since $\Gamma$ has nonzero curvatures and $g\not\equiv1$.
To get an appropriate test function for \eqref{E:NS_CTD}, we first employ an impermeable extension $E_\varepsilon\eta$ of $\eta$ to $\Omega_\varepsilon$ satisfying
\begin{align} \label{E:IE_In}
  E_\varepsilon\eta\cdot n_\varepsilon = 0 \quad\text{on}\quad \Gamma_\varepsilon
\end{align}
given in Section \ref{SS:Fund_IE} and apply the Helmholtz--Leray decomposition
\begin{align} \label{E:HLDo_In}
  E_\varepsilon\eta = \eta_\varepsilon+\nabla\varphi_\varepsilon \quad\text{in}\quad \Omega_\varepsilon, \quad (\eta_\varepsilon,\nabla\varphi_\varepsilon)_{L^2(\Omega_\varepsilon)} = 0
\end{align}
in which the solenoidal part $\eta_\varepsilon$ satisfies \eqref{E:SoDo_In}.
Then we can show (see Lemma \ref{L:Const_Test}) that $\eta_\varepsilon$ is close to $\eta$ by \eqref{E:SoSu_In} and good estimates for $E_\varepsilon\eta$ and the difference $E_\varepsilon\eta-\eta_\varepsilon$ shown in Sections \ref{SS:Fund_IE} and \ref{SS:Fund_HL}.
Hence we can take $\eta_\varepsilon$ in \eqref{E:HLDo_In} as an appropriate test function for \eqref{E:NS_CTD}.
Here we note that the impermeable boundary condition \eqref{E:IE_In} for $E_\varepsilon\eta$ is essential for the estimate of $E_\varepsilon\eta-\eta_\varepsilon$ (see Lemma \ref{L:HP_Dom}).

After substituting a test function for the limit equations in the weak formulation of \eqref{E:NS_CTD}, we transform the resulting equality into the weak formulation for $M_\tau u^\varepsilon$ by applying the change of variables formula \eqref{E:CoV_In}.
In this step we need to show that the residual term in the weak formulation for $M_\tau u^\varepsilon$ converges to zero as $\varepsilon\to0$, since we intend to obtain a weak formulation of the limit equations as a limit of the weak formulation for $M_\tau u^\varepsilon$ as $\varepsilon\to0$.
For this purpose, we give good approximations of bilinear and trilinear forms in the weak formulation of \eqref{E:NS_CTD} by those for the weak formulation of the limit equations (see Lemmas \ref{L:Appr_Bili} and \ref{L:Appr_Tri}) and use the estimates \eqref{E:UE_In} for $u^\varepsilon$ shown in \cite{Miu_NSCTD_02}.
Here we note that the $H^2(\Omega_\varepsilon)$-estimate for $u^\varepsilon$ is required for the good approximation of not only the trilinear form but also the bilinear form.
This is essentially due to the fact that the limit equations are described only in terms of intrinsic quantities of the embedded surface $\Gamma$ in $\mathbb{R}^3$ (see Remark \ref{R:Lim_Reduce}), while the bulk equations \eqref{E:NS_CTD} contain extrinsic quantities of $\Gamma$.
The intrinsic part of \eqref{E:NS_CTD} with respect to $\Gamma$ is approximated by the limit equations, but the extrinsic part is just estimated by the $H^2(\Omega_\varepsilon)$-norm of $u^\varepsilon$.
In other words, the $H^2$-regularity of a solution to \eqref{E:NS_CTD} compensates for lack of the extrinsic information of $\Gamma$ in the limit equations.

We also have difficulties in the second step of the proof of Theorem \ref{T:SL_Weak}.
In that step we derive the energy estimate for $M_\tau u^\varepsilon$ with a uniform bound in $\varepsilon$.
We might easily get it from the weak formulation for $M_\tau u^\varepsilon$ if we could take $M_\tau u^\varepsilon$ itself as a test function, but we cannot do that since $M_\tau u^\varepsilon$ does not necessarily satisfy \eqref{E:SoSu_In}.
To overcome this difficulty, we apply the weighted Helmholtz--Leray decomposition
\begin{align*}
  M_\tau u^\varepsilon = v^\varepsilon+g\nabla_\Gamma q^\varepsilon \quad\text{on}\quad \Gamma, \quad (v^\varepsilon,g\nabla_\Gamma q^\varepsilon)_{L^2(\Gamma)} = 0
\end{align*}
in which the weighted solenoidal part $v^\varepsilon$ satisfies \eqref{E:SoSu_In} (see Theorem \ref{T:HL_L2T}).
Using various estimates for the difference $M_\tau u^\varepsilon-v^\varepsilon$ shown in Section \ref{SS:WS_HL} we derive a weak formulation for $v^\varepsilon$ from that for $M_\tau u^\varepsilon$ (see Lemma \ref{L:PMu_Weak}).
Then since $v^\varepsilon$ satisfies \eqref{E:SoSu_In}, we can take $v^\varepsilon$ as a test function for its weak formulation to get an energy estimate for $v^\varepsilon$ (see Lemma \ref{L:PMu_Energy}), which implies the energy estimate for $M_\tau u^\varepsilon$ when combined with estimates for $M_\tau u^\varepsilon-v^\varepsilon$ (see Lemma \ref{L:Mu_Energy}).
We also use $v^\varepsilon$ to show the energy estimate for the difference between $M_\tau u^\varepsilon$ and the weak solution $v$ to the limit equations (see Theorem \ref{T:Diff_Mu_V} and Lemma \ref{L:Diff_Ve_V}).

\subsection{Literature overview} \label{SS:In_Lit}
PDEs in thin domains appear in many applications in solid mechanics (thin rods, plates, shells), fluid mechanics (lubrication, meteorology, ocean dynamics), etc.
Many researchers have studied PDEs in thin domains, mainly reaction-diffusion and the Navier--Stokes equations, since the pioneering works \cites{HaRa92a,HaRa92b} by Hale and Raugel on damped wave and reaction-diffusion equations.

The study of the Navier--Stokes equations in a three-dimensional thin domain was initiated by Raugel and Sell \cites{RaSe93_I,RaSe94_II,RaSe93_III}, who considered a thin product domain $Q\times(0,\varepsilon)$ in $\mathbb{R}^3$ with a rectangle $Q$ and a sufficiently small $\varepsilon>0$.
Under the purely periodic or mixed Dirichlet-periodic boundary conditions they established the global existence of a strong solution for large data depending on the smallness of $\varepsilon$ by dilating the thin domain $Q\times(0,\varepsilon)$ and analyzing scaled equations in $Q\times(0,1)$ as a perturbation of the two-dimensional Navier--Stokes equations.
Temam and Ziane \cite{TeZi96} generalized the results of \cites{RaSe93_I,RaSe94_II,RaSe93_III} to a thin product domain $\omega\times(0,\varepsilon)$ in $\mathbb{R}^3$ around a bounded domain $\omega$ in $\mathbb{R}^2$ under combinations of the Dirichlet, periodic, and Hodge boundary conditions.
By analyzing the average of a strong solution in the thin direction and the residual term separately, they proved the global existence of the strong solution without dilating the thin domain.
Moreover, they observed that the average of the strong solution under suitable boundary conditions converges to a solution of the two-dimensional Navier--Stokes equations in $\omega$ as $\varepsilon\to0$.
We refer to \cites{Hu07,If99,IfRa01,KuZi06,KuZi07,MoTeZi97,Mo99} and the references cited therein for further studies of the Navier--Stokes equations in thin product domains in $\mathbb{R}^3$.

Thin product domains studied in the above cited papers are flat in the sense that they have flat top and bottom boundaries and their limit sets are domains in $\mathbb{R}^2$.
In physical problems, however, there are many kinds of nonflat thin domains (see \cite{Ra95} for examples of nonflat thin domains), so it is not only mathematically challenging but also important for applications to generalize the shape of a thin domain.
There are a few works on the Navier--Stokes equations in nonflat thin domains.
The first study was done by Temam and Ziane \cite{TeZi97}, who dealt with a thin spherical shell
\begin{align*}
  \{x\in\mathbb{R}^3 \mid a<|x|<a+\varepsilon a\}, \quad a>0
\end{align*}
in order to give a mathematical justification of derivation of the primitive equations for the atmosphere and ocean dynamics (see \cites{LiTeWa92a,LiTeWa92b,LiTeWa95}).
As in their previous work \cite{TeZi96} on the thin product domain, they used an average operator in the thin direction to prove the global existence of a strong solution and the convergence of its average towards a solution of limit equations on a sphere as $\varepsilon\to0$.
A flat thin domain with nonflat top and bottom boundaries
\begin{align*}
  \{(x',x_3)\in\mathbb{R}^3 \mid x'\in(0,1)^2,\, \varepsilon g_0(x') < x_3 < \varepsilon g_1(x')\}, \quad g_0,g_1\colon(0,1)^2\to\mathbb{R}
\end{align*}
was also considered by Iftimie, Raugel, and Sell \cite{IfRaSe07} (with $g_0\equiv0$), Hoang \cites{Ho10}, and Hoang and Sell \cite{HoSe10}.
They obtained the global existence of a strong solution under the laterally periodic and vertically slip boundary conditions by analyzing average operators in detail.
Moreover, Iftimie, Raugel, and Sell \cite{IfRaSe07} compared the strong solution with a solution to limit equations on $(0,1)^2$.
Two-phase flows in a flat thin domain with nonflat top and bottom boundaries were also studied by Hoang \cite{Ho13}.

Let us also mention the slip boundary conditions \eqref{E:Slip_Intro}.
The slip boundary conditions were introduced by Navier \cite{Na1823}.
Unlike the usual no-slip boundary condition, the fluid subject to \eqref{E:Slip_Intro} slips on the boundary with velocity proportional to the tangential component of the stress vector.
Such conditions are considered as an appropriate model for flows with free boundaries and for flows past chemically reacting walls (see \cite{Ve87}).
They also appear in the study of the atmosphere and ocean dynamics \cites{LiTeWa92a,LiTeWa92b,LiTeWa95} and in the homogenization of the no-slip boundary condition on a rough boundary \cites{Hi16,JaMi01}.
We refer to \cites{SoSc73,Be04,AmRe14} for the study of the Stokes problem under the slip boundary conditions for a general bounded domain in $\mathbb{R}^3$.

In the series of \cites{Miu_NSCTD_01,Miu_NSCTD_02} and this paper we deal with the curved thin domain $\Omega_\varepsilon$ around the closed surface $\Gamma$ of the form \eqref{E:Def_CTD}.
PDEs in curved thin domains have been studied in various contexts.
In elasticity, the theory of thin shells has been developed over the years (see \cites{Ci97,Ci00} and the references cited therein).
There are several works on the asymptotic behavior of eigenvalues of the Laplace operator on a curved thin domain around a hypersurface (see e.g. \cites{JiKu16,Kr14,Sch96,Yac18}).
In \cites{PrRiRy02,PrRy03,Yan90} the authors studied reaction-diffusion equations in curved thin domains around lower dimensional manifolds.
Curved thin domains around evolving surfaces were also considered in the study of the heat equation \cite{Miu17} and of an advection-diffusion equation \cite{ElSt09}.
However, there has been no literature on the Navier--Stokes equations in a curved thin domain in $\mathbb{R}^3$ around a general closed surface due to difficulties in analyzing vector fields arising from the complicated geometry of the curved thin domain and its boundary.
In the series of our study we present mathematical tools for dealing with such difficulties and investigate the effect of the geometry of the curved thin domain and its limit surface on the original and limit equations.

The aim of our study is not just to further generalize the shape of a thin domain in the study of the Navier--Stokes equations, but to give a rigorous derivation of the surface Navier--Stokes equations by the thin-film limit.
As mentioned in Section \ref{SS:In_Lim} (see also Remark \ref{R:Lim_Reduce}), our limit equations in the spacial case $g\equiv1$ and $\gamma_\varepsilon=0$ agree with the Navier--Stokes equations on surfaces and manifolds.
Fluid flows on surfaces and manifolds have attracted interest of many researchers in various fields.
The Navier--Stokes equations on a Riemannian manifold were introduced in \cites{EbMa70,MitYa02,Ta92} and have been studied over the years (see e.g. \cites{ChaCzu13,DinMit04,KheMis12,KohWen18,MitTa01,Nag99,Pi17,Prie94,PrSiWi20pre,SaTu20}).
Also, the Navier--Stokes equations on an evolving surface in $\mathbb{R}^3$ were derived recently by the local conservation laws of mass and momentum \cite{JaOlRe18}, a global energetic variational approach \cite{KoLiGi17}, and the formal thin-film limit \cite{Miu18}.
The fluid model given in \cites{JaOlRe18,KoLiGi17,Miu18} agrees with the Boussinesq--Scriven surface fluid model introduced by Boussinesq \cite{Bo1913} and generalized by Scriven \cite{Sc60} to arbitrary curved surfaces (see also \cites{Ar89,SlSaOh07}).
It is used to formulate equations for the interface of a two-phase flow \cites{BaGaNu15,BoPr10} and for a fluid membrane \cite{ArrDe09}, although the acceleration term given in \cite{ArrDe09} should be corrected as pointed out in \cite{YaOzSa16}.
We also refer to \cites{Fr18,LeLeSc20,NiVoWe12,NiReVo17,OlQuReYu18,OlYu19,Reus20,ReVo15,ReVo18,SaOmSaMa20,ToSaAr20} for the recent developments of numerical methods for the fluid equations on stationary and evolving surfaces.
In this paper we give the first result on a rigorous derivation of the Navier--Stokes equations on a general closed surface in $\mathbb{R}^3$ by the thin-film limit, which justifies the result of \cite{Miu18} when a limit surface is stationary in time.

Finally, let us also mention our new result on the Helmholtz--Leray decomposition for a vector field on the closed surface $\Gamma$.
To prove the main results of this paper we show the weighted Helmholtz--Leray decomposition for a tangential vector field on $\Gamma$ and related results in Section \ref{S:WSol}.
As an easy consequence of those results, we derive in Section \ref{SS:WS_NTS} the Helmholtz--Leray decomposition
\begin{align} \label{E:NTHL_In}
  v = v_\sigma+\nabla_\Gamma q+qHn \quad\text{on}\quad \Gamma, \quad (v_\sigma,\nabla_\Gamma q+qHn)_{L^2(\Gamma)} = 0
\end{align}
for a not necessarily tangential vector field $v\in L^2(\Gamma)^3$, where the solenoidal part $v_\sigma$ satisfies $\mathrm{div}_\Gamma v_\sigma=0$ on $\Gamma$ and $H$ is the mean curvature of $\Gamma$ (see Theorem \ref{T:HL_L2Ge}).
Moreover, we obtain the uniqueness of the scalar potential $q$ in \eqref{E:NTHL_In}, not up to a constant, by the fact that $\Gamma$ is closed (see Lemma \ref{L:TGrHn_Hin_Con}).
The decomposition \eqref{E:NTHL_In} was already found in \cite{KoLiGi17} (see also \cite{KoLiGi18Er}), but the uniqueness of $q$ is a new result of this paper.
Although we do not apply \eqref{E:NTHL_In} and the other results of Section \ref{SS:WS_NTS} in the proof of our main results, we expect that they will play an important role in the future study of the Navier--Stokes equations on an evolving surface (see \cites{JaOlRe18,KoLiGi17,Miu18}).

\subsection{Organization of this paper} \label{SS:In_Out}
The rest of this paper is organized as follows.
In Section \ref{S:Main} we present the main results of this paper.
Section \ref{S:Pre} fixes notations on a closed surface and a curved thin domain and gives their basic properties.
We show fundamental inequalities for functions on the closed surface and the curved thin domain in Section \ref{S:Fund}.
Section \ref{S:Ave} deals with average operators in the thin direction.
We derive useful estimates for the average operators and use them to approximate bilinear and trilinear forms for functions on the curved thin domain by those for functions on the closed surface.
In Section \ref{S:WSol} we consider weighted solenoidal spaces on the closed surface.
The purpose of Section \ref{S:WSol} is to show the weighted Helmholtz--Leray decomposition for a tangential vector field on the closed surface and related results.
In Section \ref{S:SL} we study a singular limit problem for \eqref{E:NS_CTD} as the thickness of the curved thin domain tends to zero and establish the main results.
Appendix \ref{S:Ap_Vec} fixes notations on vectors and matrices.
In Appendix \ref{S:Ap_Aux} we give auxiliary results on the closed surface and prove Lemmas \ref{L:WTGr_Van}--\ref{L:Poin_Surf_Lp}.
Appendix \ref{S:Ap_Visc} provides formulas for differential operators on the closed surface related to the viscous term in the surface Navier--Stokes equations.
In Appendix \ref{S:Ap_CWL} we explain the outline of construction of a weak solution to limit equations on the closed surface by the Galerkin method.

Most results of this paper were obtained in the doctoral thesis of the author \cite{Miu_DT}.
However, we add the new condition (A3) in Assumption \ref{Assump_2} to consider some curved thin domains excluded in \cite{Miu_DT} by showing new results on a uniform Korn inequality and the axial symmetry of a curved thin domain in the first part \cite{Miu_NSCTD_01} of our study.
The most important example of a curved thin domain excluded in \cite{Miu_DT} but included in this paper is the thin spherical shell
\begin{align*}
  \Omega_\varepsilon = \{x\in\mathbb{R}^3 \mid 1 < |x| < 1+\varepsilon\}
\end{align*}
under the perfect slip boundary conditions \eqref{E:PeSl_In}.
As mentioned in Section \ref{SS:In_Lim}, this kind of curved thin domain was studied by Temam and Ziane \cite{TeZi97} under different boundary conditions (see also Remarks \ref{R:Ex_As} and \ref{R:Limit_TZ}).
We also add Appendix \ref{S:Ap_Visc} to derive some new formulas for differential operators on a closed surface.
Using them we compare our limit equations with the Navier--Stokes equations on a Riemannian manifold and limit equations derived in \cite{TeZi97} (see Remarks \ref{R:Lim_Reduce} and \ref{R:Limit_TZ}).

\section{Main results} \label{S:Main}
In this section we state the main results of this paper after we fix some notations and make assumptions (see also Section \ref{S:Pre} for notations).
The proofs of the main results (Theorems \ref{T:SL_Weak} and \ref{T:SL_Strong}) are given in Section \ref{S:SL}.

Let $\Gamma$ be a closed (i.e. compact and without boundary), connected, and oriented surface in $\mathbb{R}^3$ with unit outward normal vector field $n$.
Also, let $g_0,g_1\in C^4(\Gamma)$.
We assume that $\Gamma$ is of class $C^5$ and there exists a constant $c>0$ such that
\begin{align} \label{E:G_Inf}
  g := g_1-g_0 \geq c \quad\text{on}\quad \Gamma.
\end{align}
Note that we do not assume $g_0\leq0$ or $g_1\geq0$ on $\Gamma$.
For a sufficiently small $\varepsilon\in(0,1]$ let $\Omega_\varepsilon$ be the curved thin domain in $\mathbb{R}^3$ of the form \eqref{E:Def_CTD} and $L_\sigma^2(\Omega_\varepsilon)$ the standard $L^2$-solenoidal space on $\Omega_\varepsilon$ given by
\begin{align*}
  L_\sigma^2(\Omega_\varepsilon) := \{u\in L^2(\Omega_\varepsilon)^3 \mid \text{$\mathrm{div}\,u=0$ in $\Omega_\varepsilon$, $u\cdot n_\varepsilon=0$ on $\Gamma_\varepsilon$}\}.
\end{align*}
In the second part \cite{Miu_NSCTD_02} of our study we proved the global existence and estimates of a strong solution to \eqref{E:NS_CTD}, which are fundamental for the study of a singular limit problem for \eqref{E:NS_CTD} as $\varepsilon\to0$ carried out in this paper.
To state the results of \cite{Miu_NSCTD_02} we define function spaces and make assumptions as follows.

We denote by
\begin{align} \label{E:Def_R}
  \mathcal{R} := \{w(x)=a\times x+b,\,x\in\mathbb{R}^3 \mid a,b\in\mathbb{R}^3,\,\text{$w|_\Gamma\cdot n=0$ on $\Gamma$}\}
\end{align}
the space of all infinitesimal rigid displacements of $\mathbb{R}^3$ whose restrictions on $\Gamma$ are tangential.
It is of finite dimension and describes the axial symmetry of the closed surface $\Gamma$, i.e. $\mathcal{R}\neq\{0\}$ if and only if $\Gamma$ is invariant under a rotation by any angle around some line (see \cite{Miu_NSCTD_01}*{Lemma E.1}).
We define subspaces of $\mathcal{R}$ by
\begin{align} \label{E:Def_Rg}
  \begin{aligned}
    \mathcal{R}_i &:= \{w\in\mathcal{R} \mid \text{$w|_\Gamma\cdot\nabla_\Gamma g_i=0$ on $\Gamma$}\}, \quad i=0,1, \\
    \mathcal{R}_g &:= \{w\in\mathcal{R} \mid \text{$w|_\Gamma\cdot\nabla_\Gamma g=0$ on $\Gamma$}\} \quad (g=g_1-g_0),
  \end{aligned}
\end{align}
where $\nabla_\Gamma$ is the tangential gradient on $\Gamma$ (see Section \ref{SS:Pre_Surf}).
By the above definitions we immediately get $\mathcal{R}_0\cap\mathcal{R}_1\subset\mathcal{R}_g$.
It is also shown in \cite{Miu_NSCTD_01}*{Lemmas E.6 and E.7} that the curved thin domain $\Omega_\varepsilon$ is axially symmetric around the same line for all $\varepsilon\in(0,1]$ if $\mathcal{R}_0\cap\mathcal{R}_1\neq\{0\}$, while $\Omega_\varepsilon$ is not axially symmetric around any line for all $\varepsilon>0$ sufficiently small if $\mathcal{R}_g=\{0\}$.

Next we denote by
\begin{align*}
  P := I_3-n\otimes n, \quad (\nabla_\Gamma v)_S := \frac{\nabla_\Gamma v+(\nabla_\Gamma v)^T}{2} \quad\text{on}\quad \Gamma
\end{align*}
the orthogonal projection onto the tangent plane of $\Gamma$ and the symmetric part of the tangential gradient matrix of a (not necessarily tangential) vector field $v$ on $\Gamma$ (see Section \ref{SS:Pre_Surf}).
We define the surface strain rate tensor by
\begin{align*}
  D_\Gamma(v) := P(\nabla_\Gamma v)_SP \quad\text{on}\quad \Gamma
\end{align*}
and function spaces of tangential vector fields on $\Gamma$ by
\begin{align} \label{E:Def_Kil}
  \begin{aligned}
    \mathcal{K}(\Gamma) &:= \{v \in H^1(\Gamma)^3 \mid \text{$v\cdot n=0$, $D_\Gamma(v)=0$ on $\Gamma$}\}, \\
    \mathcal{K}_g(\Gamma) &:= \{v\in\mathcal{K}(\Gamma) \mid \text{$v\cdot\nabla_\Gamma g=0$ on $\Gamma$}\}.
  \end{aligned}
\end{align}
Then $v\in\mathcal{K}(\Gamma)$ satisfies
\begin{align*}
  \overline{\nabla}_Xv\cdot Y+X\cdot\overline{\nabla}_Yv = 0 \quad\text{on}\quad \Gamma
\end{align*}
for all tangential vector fields $X$ and $Y$ on $\Gamma$, where $\overline{\nabla}_Xv$ is the covariant derivative of $v$ along $X$ (see Section \ref{SS:Pre_Surf}).
A tangential vector field on $\Gamma$ satisfying this property generates a one-parameter group of isometries of $\Gamma$ and is called a Killing vector field on $\Gamma$ (see \cites{Jo11,Pe06} for details).
Direct calculations show that
\begin{align*}
  \mathcal{R}|_\Gamma := \{w|_\Gamma \mid w\in\mathcal{R}\} \subset \mathcal{K}(\Gamma).
\end{align*}
The sets $\mathcal{K}(\Gamma)$ and $\mathcal{R}|_\Gamma$ describe the intrinsic and extrinsic infinitesimal symmetry of $\Gamma$, respectively.
It is known (see \cite{Miu_NSCTD_01}*{Lemma E.3}) that $\mathcal{R}|_\Gamma=\mathcal{K}(\Gamma)$ when $\Gamma$ is axially symmetric.
The same relation holds for a closed and convex surface by the Cohn-Vossen rigidity theorem (see \cite{Sp79}).
However, it is not known whether this relation is valid for closed but nonconvex and not axially symmetric surfaces in $\mathbb{R}^3$.

We make the following assumptions on the closed surface $\Gamma$, the functions $g_0$ and $g_1$, and the friction coefficients $\gamma_\varepsilon^0$ and $\gamma_\varepsilon^1$ appearing in \eqref{E:Def_Fric}.

\begin{assumption} \label{Assump_1}
  There exists a constant $c>0$ such that
  \begin{align} \label{E:Fric_Upper}
    \gamma_\varepsilon^0 \leq c\varepsilon, \quad \gamma_\varepsilon^1 \leq c\varepsilon
  \end{align}
  for all $\varepsilon\in(0,1]$.
\end{assumption}

\begin{assumption} \label{Assump_2}
  Either of the following conditions is satisfied:
  \begin{itemize}
    \item[(A1)] There exists a constant $c>0$ such that
    \begin{align*}
      \gamma_\varepsilon^0 \geq c\varepsilon \quad\text{for all}\quad \varepsilon\in(0,1] \quad\text{or}\quad \gamma_\varepsilon^1 \geq c\varepsilon \quad\text{for all}\quad \varepsilon\in(0,1].
    \end{align*}
    \item[(A2)] The space $\mathcal{K}_g(\Gamma)$ contains only a trivial vector field, i.e. $\mathcal{K}_g(\Gamma)=\{0\}$.
    \item[(A3)] The relations
    \begin{align*}
      \mathcal{R}_g=\mathcal{R}_0\cap\mathcal{R}_1, \quad \mathcal{R}_g|_\Gamma:=\{w|_\Gamma\mid w\in\mathcal{R}_g\}=\mathcal{K}_g(\Gamma)
    \end{align*}
    hold and $\gamma_\varepsilon^0=\gamma_\varepsilon^1=0$ for all $\varepsilon\in(0,1]$.
  \end{itemize}
\end{assumption}

These assumptions are imposed only in this section, Section \ref{SS:Fund_St}, and Section \ref{S:SL}.
In Remark \ref{R:Ex_As} below we give a few examples for which Assumption \ref{Assump_2} is satisfied.
For further discussions on Assumption \ref{Assump_2} we refer to \cite{Miu_NSCTD_01}*{Remarks 2.9 and 2.10}.

Under Assumptions \ref{Assump_1} and \ref{Assump_2} we define
\begin{align} \label{E:Def_Heps}
  \begin{aligned}
    \mathcal{H}_\varepsilon &:=
    \begin{cases}
      L_\sigma^2(\Omega_\varepsilon) &\text{if the condition (A1) or (A2) is satisfied}, \\
      L_\sigma^2(\Omega_\varepsilon)\cap\mathcal{R}_g^\perp &\text{if the condition (A3) is satisfied},
  \end{cases} \\
  \mathcal{V}_\varepsilon &:= \mathcal{H}_\varepsilon\cap H^1(\Omega_\varepsilon)^3,
    \end{aligned}
\end{align}
where $\mathcal{R}_g^\perp$ is the orthogonal complement of $\mathcal{R}_g$ in $L^2(\Omega_\varepsilon)^3$.
Here we consider vector fields in $\mathcal{R}_g$ defined on $\mathbb{R}^3$ as elements of $L^2(\Omega_\varepsilon)^3$ just by restricting them on $\overline{\Omega}_\varepsilon$.
Note that $\mathcal{R}_0\cap\mathcal{R}_1\subset L_\sigma^2(\Omega_\varepsilon)$ (see \cite{Miu_NSCTD_01}*{Lemma E.8}).
Thus $\mathcal{R}_g$ is a finite dimensional subspace of $L_\sigma^2(\Omega_\varepsilon)$ under the condition (A3).
Also, $\mathcal{H}_\varepsilon$ and $\mathcal{V}_\varepsilon$ are closed subspaces of $L^2(\Omega_\varepsilon)^3$ and $H^1(\Omega_\varepsilon)^3$, respectively.
We denote by $\mathbb{P}_\varepsilon$ the orthogonal projection from $L^2(\Omega_\varepsilon)^3$ onto $\mathcal{H}_\varepsilon$.

The function spaces $\mathcal{H}_\varepsilon$ and $\mathcal{V}_\varepsilon$ play a fundamental role in the study of \eqref{E:NS_CTD}.
By integration by parts we see that the bilinear form for the Stokes problem in $\Omega_\varepsilon$ under the slip boundary conditions is of the form
\begin{align*}
  a_\varepsilon(u_1,u_2) := 2\nu\bigl(D(u_1),D(u_2)\bigr)_{L^2(\Omega_\varepsilon)}+\sum_{i=0,1}\gamma_\varepsilon^i(u_1,u_2)_{L^2(\Gamma_\varepsilon^i)}
\end{align*}
for $u_1,u_2\in H^1(\Omega_\varepsilon)^3$ (see Section \ref{SS:Fund_St}).
In the first paper \cite{Miu_NSCTD_01} we showed that $a_\varepsilon$ is bounded and coercive on $\mathcal{V}_\varepsilon$ uniformly in $\varepsilon$.

\begin{lemma}[{\cite{Miu_NSCTD_01}*{Theorem 2.4}}] \label{L:Uni_aeps}
  Under Assumptions \ref{Assump_1} and \ref{Assump_2}, there exist constants $\varepsilon_0\in(0,1]$ and $c>0$ such that
  \begin{align} \label{E:Uni_aeps}
    c^{-1}\|u\|_{H^1(\Omega_\varepsilon)}^2 \leq a_\varepsilon(u,u) \leq c\|u\|_{H^1(\Omega_\varepsilon)}^2
  \end{align}
  for all $\varepsilon\in(0,\varepsilon_0]$ and $u\in\mathcal{V}_\varepsilon$.
\end{lemma}

Here Assumption \ref{Assump_1} is used to get the right-hand inequality of \eqref{E:Uni_aeps}.
Also, the left-hand inequality follows from Assumption \ref{Assump_2} and a uniform Korn inequality on $\Omega_\varepsilon$, for which the function spaces $\mathcal{K}_g(\Gamma)$ and $\mathcal{R}_g$ play an important role (see \cite{Miu_NSCTD_01} and also \cite{LeMu11} for a uniform Korn inequality on a curved thin domain).

By Lemma \ref{L:Uni_aeps} and the Lax--Milgram theorem we see that $a_\varepsilon$ induces a bounded linear operator $A_\varepsilon$ from $\mathcal{V}_\varepsilon$ into its dual space.
We consider $A_\varepsilon$ as an unbounded operator on $\mathcal{H}_\varepsilon$ with domain $D(A_\varepsilon)$ and call it the Stokes operator on $\mathcal{H}_\varepsilon$.
In \cite{Miu_NSCTD_01} we derived several estimates for $A_\varepsilon$ which was essential mainly for the second paper \cite{Miu_NSCTD_02}.
We briefly review basic properties of $A_\varepsilon$ in Section \ref{SS:Fund_St}.

\begin{remark} \label{R:Assump}
  Assumptions \ref{Assump_1} and \ref{Assump_2} and the assumptions on the regularity of $\Gamma$, $g_0$, and $g_1$ are required mainly for the study of the Stokes operator $A_\varepsilon$ carried out in \cite{Miu_NSCTD_01}.
  We employ the conditions of Assumptions \ref{Assump_1} and \ref{Assump_2} in Section \ref{S:SL}, but do not use the $C^5$-regularity of $\Gamma$ and the $C^4$-regularity of $g_0$ and $g_1$ on $\Gamma$ explicitly in this paper.
\end{remark}

Based on the results of \cite{Miu_NSCTD_01} we studied the abstract evolution equation
\begin{align} \label{E:NS_Abst}
  \partial_tu^\varepsilon+A_\varepsilon u^\varepsilon+\mathbb{P}_\varepsilon(u^\varepsilon\cdot\nabla)u^\varepsilon = \mathbb{P}_\varepsilon f^\varepsilon \quad\text{on}\quad (0,\infty), \quad u^\varepsilon|_{t=0} = u_0^\varepsilon
\end{align}
in $\mathcal{H}_\varepsilon$ and established the global existence of a strong solution to \eqref{E:NS_CTD} and estimates for it with constants explicitly depending on $\varepsilon$ in \cite{Miu_NSCTD_02}.
For a vector field $u$ on $\Omega_\varepsilon$ let
\begin{align*}
  Mu(y) := \frac{1}{\varepsilon g(y)}\int_{\varepsilon g_0(y)}^{\varepsilon g_1(y)} u(y+rn(y))\,dr, \quad M_\tau u(y) := P(y)Mu(y), \quad y\in\Gamma
\end{align*}
be the average of $u$ in the thin direction and the averaged tangential component of $u$ (see Section \ref{S:Ave}).
Also, we set
\begin{align*}
  H^1(\Gamma,T\Gamma) := \{v\in H^1(\Gamma)^3 \mid \text{$v\cdot n=0$ on $\Gamma$}\},
\end{align*}
which is a Hilbert space equipped with inner product of $H^1(\Gamma)^3$, and denote by $H^{-1}(\Gamma,T\Gamma)$ the dual space of $H^1(\Gamma,T\Gamma)$ (see Section \ref{SS:Pre_Surf}).

\begin{theorem}[{\cite{Miu_NSCTD_02}*{Theorem 2.7}}] \label{T:UE}
  Under Assumptions \ref{Assump_1} and \ref{Assump_2}, let $\varepsilon_0$ be the constant given in Lemma \ref{L:Uni_aeps}.
  Also, let $c_1$, $c_2$, $\alpha$, and $\beta$ be positive constants.
  Then there exists a constant $\varepsilon_1\in(0,\varepsilon_0]$ such that the following statement holds: for $\varepsilon\in(0,\varepsilon_1]$ suppose that the given data
  \begin{align*}
    u_0^\varepsilon\in \mathcal{V}_\varepsilon, \quad f^\varepsilon\in L^\infty(0,\infty;L^2(\Omega_\varepsilon)^3)
  \end{align*}
  satisfy
  \begin{align} \label{E:UE_Data}
    \begin{aligned}
      \|u_0^\varepsilon\|_{H^1(\Omega_\varepsilon)}^2+\|\mathbb{P}_\varepsilon f^\varepsilon\|_{L^\infty(0,\infty;L^2(\Omega_\varepsilon))}^2 &\leq c_1\varepsilon^{-1+\alpha}, \\
      \|M_\tau u_0^\varepsilon\|_{L^2(\Gamma)}^2+\|M_\tau\mathbb{P}_\varepsilon f^\varepsilon\|_{L^\infty(0,\infty;H^{-1}(\Gamma,T\Gamma))}^2 &\leq c_2\varepsilon^{-1+\beta}.
    \end{aligned}
  \end{align}
  If the condition (A3) of Assumption \ref{Assump_2} is imposed, suppose further that $f^\varepsilon(t)\in\mathcal{R}_g^\perp$ for a.a. $t\in(0,\infty)$.
  Then there exists a global-in-time strong solution
  \begin{align*}
    u^\varepsilon \in C([0,\infty);\mathcal{V}_\varepsilon)\cap L_{loc}^2([0,\infty);D(A_\varepsilon))\cap H_{loc}^1([0,\infty);\mathcal{H}_\varepsilon)
  \end{align*}
  to \eqref{E:NS_CTD}.
  Moreover, there exists a constant $c>0$ independent of $\varepsilon$ and $u^\varepsilon$ such that
  \begin{align} \label{E:UE_L2}
    \begin{aligned}
      \|u^\varepsilon(t)\|_{L^2(\Omega_\varepsilon)}^2 &\leq c(\varepsilon^{1+\alpha}+\varepsilon^\beta), \\
      \int_0^t\|u^\varepsilon(s)\|_{H^1(\Omega_\varepsilon)}^2\,ds &\leq c(\varepsilon^{1+\alpha}+\varepsilon^\beta)(1+t)
    \end{aligned}
  \end{align}
  for all $t\geq0$ and
  \begin{align} \label{E:UE_H1}
    \begin{aligned}
      \|u^\varepsilon(t)\|_{H^1(\Omega_\varepsilon)}^2 &\leq c(\varepsilon^{-1+\alpha}+\varepsilon^{-1+\beta}), \\
      \int_0^t\|u^\varepsilon(s)\|_{H^2(\Omega_\varepsilon)}^2\,ds &\leq c(\varepsilon^{-1+\alpha}+\varepsilon^{-1+\beta})(1+t)
    \end{aligned}
  \end{align}
  for all $t\geq 0$.
\end{theorem}

The estimates \eqref{E:UE_L2} and \eqref{E:UE_H1} for the strong solution $u^\varepsilon$ to \eqref{E:NS_CTD} are important for the study of a singular limit problem for \eqref{E:NS_CTD} as $\varepsilon\to0$ carried out in Section \ref{S:SL}.
Note that the assumption $f^\varepsilon(t)\in\mathcal{R}_g^\perp$ for a.a. $t\in(0,\infty)$ under the condition (A3) is required to recover the momentum equations of \eqref{E:NS_CTD} from the abstract evolution equation \eqref{E:NS_Abst} in $\mathcal{H}_\varepsilon$ properly.
For details, we refer to \cite{Miu_NSCTD_02}*{Remark 2.8}.

Now let us present the main results of this paper.
We define function spaces of tangential vector fields on $\Gamma$ by
\begin{align*}
  L^2(\Gamma,T\Gamma) &:= \{v\in L^2(\Gamma)^3 \mid \text{$v\cdot n=0$ on $\Gamma$}\}, \\
  \mathcal{H}_g &:= \{v\in L^2(\Gamma,T\Gamma) \mid \text{$\mathrm{div}_\Gamma(gv)=0$ on $\Gamma$}\}, \\
  \mathcal{V}_g &:= \mathcal{H}_g\cap H^1(\Gamma,T\Gamma),
\end{align*}
where $\mathrm{div}_\Gamma$ is the surface divergence on $\Gamma$ (see Sections \ref{S:Pre} and \ref{S:WSol}).

\begin{theorem} \label{T:SL_Weak}
  Under Assumptions \ref{Assump_1} and \ref{Assump_2}, let $\varepsilon_0$ be the constant given in Lemma \ref{L:Uni_aeps} and
  \begin{align*}
    u_0^\varepsilon\in \mathcal{V}_\varepsilon, \quad f^\varepsilon\in L^\infty(0,\infty;L^2(\Omega_\varepsilon)^3)
  \end{align*}
  for $\varepsilon\in(0,\varepsilon_0]$.
  Suppose that the following conditions hold:
  \begin{itemize}
    \item[(a)] There exist constants $c>0$, $\varepsilon_2\in(0,\varepsilon_0]$, and $\alpha\in(0,1]$ such that
    \begin{align*}
      \|u_0^\varepsilon\|_{H^1(\Omega_\varepsilon)}^2+\|\mathbb{P}_\varepsilon f^\varepsilon\|_{L^\infty(0,\infty;L^2(\Omega_\varepsilon))}^2 \leq c\varepsilon^{-1+\alpha}
    \end{align*}
    for all $\varepsilon\in(0,\varepsilon_2]$.
    \item[(b)] There exist $v_0\in L^2(\Gamma,T\Gamma)$ and $f\in L^\infty(0,\infty;H^{-1}(\Gamma,T\Gamma))$ such that
    \begin{alignat*}{3}
      \lim_{\varepsilon\to0}M_\tau u_0^\varepsilon &= v_0 &\quad &\text{weakly in} &\quad &L^2(\Gamma,T\Gamma), \\
      \lim_{\varepsilon\to0}M_\tau\mathbb{P}_\varepsilon f^\varepsilon &= f &\quad &\text{weakly-$\star$ in} &\quad &L^\infty(0,\infty;H^{-1}(\Gamma,T\Gamma)).
    \end{alignat*}
    \item[(c)] For $i=0,1$ there exists a constant $\gamma^i\geq0$ such that
    \begin{align*}
      \lim_{\varepsilon\to0}\frac{\gamma_\varepsilon^i}{\varepsilon} = \gamma^i.
    \end{align*}
  \end{itemize}
  If the condition (A3) of Assumption \ref{Assump_2} is imposed, suppose further that $f^\varepsilon(t)\in\mathcal{R}_g^\perp$ for all $\varepsilon\in(0,\varepsilon_2]$ and a.a. $t\in(0,\infty)$.
  Then there exists a constant $\varepsilon_3\in(0,\varepsilon_2]$ such that a global-in-time strong solution
  \begin{align*}
    u^\varepsilon \in C([0,\infty);\mathcal{V}_\varepsilon)\cap L_{loc}^2([0,\infty);D(A_\varepsilon))\cap H_{loc}^1([0,\infty);\mathcal{H}_\varepsilon)
  \end{align*}
  to \eqref{E:NS_CTD} exists for each $\varepsilon\in(0,\varepsilon_3]$ and
  \begin{align*}
    \lim_{\varepsilon\to0}Mu^\varepsilon\cdot n = 0 \quad\text{strongly in}\quad C([0,\infty);L^2(\Gamma)).
  \end{align*}
  Moreover, there exists a vector field
  \begin{align*}
    v \in C([0,\infty);\mathcal{H}_g)\cap L_{loc}^2([0,\infty);\mathcal{V}_g)\cap H_{loc}^1([0,\infty);H^{-1}(\Gamma,T\Gamma))
  \end{align*}
  such that
  \begin{align*}
    \lim_{\varepsilon\to0}M_\tau u^\varepsilon = v \quad\text{weakly in}\quad L^2(0,T;H^1(\Gamma,T\Gamma))
  \end{align*}
  for each $T>0$ and $v$ is a unique weak solution to
  \begin{align} \label{E:NS_Limit}
    \left\{
    \begin{aligned}
      &g\Bigl(\partial_tv+\overline{\nabla}_vv\Bigr)-2\nu\left\{P\mathrm{div}_\Gamma[gD_\Gamma(v)]-\frac{1}{g}(v\cdot\nabla_\Gamma g)\nabla_\Gamma g\right\} \\
      &\mspace{200mu}
      \begin{alignedat}{3}
        +(\gamma^0+\gamma^1)v+g\nabla_\Gamma q &= gf &\quad &\text{on} &\quad &\Gamma\times(0,\infty), \\
        \mathrm{div}_\Gamma(gv) &= 0 &\quad &\text{on} &\quad &\Gamma\times(0,\infty), \\
        v|_{t=0} &= v_0 &\quad &\text{on} &\quad &\Gamma
      \end{alignedat}
    \end{aligned}
    \right.
  \end{align}
  with an associated pressure $q$.
\end{theorem}

Here $\overline{\nabla}_vv$ is the covariant derivative of the tangential vector field $v$ on $\Gamma$ along itself and $D_\Gamma(v)$ is the surface strain rate tensor (see Section \ref{SS:Pre_Surf}).

We give the definition of a weak solution to \eqref{E:NS_Limit} and the proof of Theorem \ref{T:SL_Weak} in Section \ref{SS:SL_WeCh}.
Note that the weak limit $v_0$ of $M_\tau u_0^\varepsilon$ in the condition (b) actually belongs to $\mathcal{H}_g$, while $M_\tau u_0^\varepsilon$ does not so in general (see Lemma \ref{L:WC_Sole}).
Also, we do not divide the first equations of \eqref{E:NS_Limit} by $g$ since we derive the term $g\nabla_\Gamma q$ in those equations from a weak formulation of \eqref{E:NS_Limit} in Lemma \ref{L:LW_Pres} by using the weighted de Rham theorem related to the weighted Helmholtz--Leray decomposition
\begin{align} \label{E:HLT_Intro}
  v = v_g+g\nabla_\Gamma q \quad\text{in}\quad L^2(\Gamma,T\Gamma), \quad v_g\in\mathcal{H}_g,\, g\nabla_\Gamma q\in\mathcal{H}_g^\perp
\end{align}
for a tangential vector field $v\in L^2(\Gamma,T\Gamma)$ (see Theorems \ref{T:DeRham_T} and \ref{T:HL_L2T}).
We further note that the Helmholtz--Leray decomposition for a not necessarily tangential vector field on $\Gamma$ is given in this paper (see Remark \ref{R:HL_NTSV}).

If the weak and weak-$\star$ convergence of $M_\tau u_0^\varepsilon$ and $M_\tau\mathbb{P}_\varepsilon f^\varepsilon$ are replaced by the strong convergence, then we get the strong convergence of $M_\tau u^\varepsilon$.

\begin{theorem} \label{T:SL_Strong}
  Under Assumptions \ref{Assump_1} and \ref{Assump_2}, let $\varepsilon_0$ be the constant given in Lemma \ref{L:Uni_aeps} and
  \begin{align*}
    u_0^\varepsilon\in \mathcal{V}_\varepsilon, \quad f^\varepsilon\in L^\infty(0,\infty;L^2(\Omega_\varepsilon)^3)
  \end{align*}
  for $\varepsilon\in(0,\varepsilon_0]$.
  Suppose that the assumptions of Theorem \ref{T:SL_Weak} are satisfied with the condition (b) replaced by the following condition:
  \begin{itemize}
    \item [(b')] There exist $v_0\in L^2(\Gamma,T\Gamma)$ and $f\in L^\infty(0,\infty;H^{-1}(\Gamma,T\Gamma))$ such that
    \begin{alignat*}{3}
      \lim_{\varepsilon\to0}M_\tau u_0^\varepsilon &= v_0 &\quad &\text{strongly in} &\quad &L^2(\Gamma,T\Gamma), \\
      \lim_{\varepsilon\to0}M_\tau\mathbb{P}_\varepsilon f^\varepsilon &= f &\quad &\text{strongly in} &\quad &L^\infty(0,\infty;H^{-1}(\Gamma,T\Gamma)).
  \end{alignat*}
  \end{itemize}
  Then the statements of Theorem \ref{T:SL_Weak} hold.
  Moreover, for each $T>0$ we have
  \begin{align*}
    \lim_{\varepsilon\to0}M_\tau u^\varepsilon = v \quad\text{strongly in}\quad C([0,T];L^2(\Gamma,T\Gamma))\cap L^2(0,T;H^1(\Gamma,T\Gamma)).
  \end{align*}
\end{theorem}

In Section \ref{SS:SL_ErST} we establish Theorem \ref{T:SL_Strong} by showing an energy estimate for the difference between $M_\tau u^\varepsilon$ and $v$ (see Theorem \ref{T:Diff_Mu_V}).
We also derive estimates in $\Omega_\varepsilon$ for the difference between $u^\varepsilon$ and the constant extension of $v$ in the normal direction of $\Gamma$ (see Theorem \ref{T:Diff_Ue_Cv}).
It is worth noting that, if we define the derivative of $u^\varepsilon$ in the normal direction of $\Gamma$ by
\begin{align*}
  \partial_nu^\varepsilon(x,t) := \frac{d}{d\tilde{r}}\bigl(u^\varepsilon(y+\tilde{r}n(y),t)\bigr)\Big|_{\tilde{r}=r}, \quad x = y+rn(y) \in \Omega_\varepsilon,
\end{align*}
then $\partial_nu^\varepsilon$ is close to a surface vector field of the form (see Theorem \ref{T:Diff_DnUe_Cv})
\begin{align*}
  -W(y)v(y,t)+\frac{1}{g(y)}\{v(y,t)\cdot\nabla_\Gamma g(y)\}n(y), \quad (y,t)\in\Gamma\times(0,\infty).
\end{align*}
Here $W$ is the Weingarten map (or the shape operator) of the surface $\Gamma$ representing the curvatures of $\Gamma$ (see Section \ref{SS:Pre_Surf}).
Thus, when $\Gamma$ is not flat, the velocity $u^\varepsilon$ of the bulk fluid is not constant in the normal direction of $\Gamma$ even though it is approximated by the constant extension of the velocity $v$ of the surface fluid.

Finally, we give remarks on Assumption \ref{Assump_2}, the limit equations \eqref{E:NS_Limit}, and the Helmholtz--Leray decomposition for not necessarily tangential vector fields on $\Gamma$.

\begin{remark} \label{R:Ex_As}
  The following examples satisfy Assumption \ref{Assump_2}.
  \begin{itemize}
    \item[(A1)] We can take any closed surface $\Gamma$ when at least one of the friction coefficients $\gamma_\varepsilon^0$ and $\gamma_\varepsilon^1$ in \eqref{E:Def_Fric} is bounded from below by $\varepsilon$.
    In this case, however, either of the constants $\gamma^0$ and $\gamma^1$ in the limit equations \eqref{E:NS_Limit} must be positive.
    \item[(A2)] It is known (see e.g. \cite{Sh_18pre}*{Proposition 2.2}) that $\mathcal{K}(\Gamma)=\{0\}$, i.e. $\Gamma$ does not admit any nontrivial Killing vector field if its genus is greater than one.
    In this case $\mathcal{K}_g(\Gamma)=\{0\}$ for any $g=g_1-g_0$ and we assume only that $\gamma_\varepsilon^0$ and $\gamma_\varepsilon^1$ are nonnegative (and bounded above by $\varepsilon$).
    \item[(A3)] A typical but an important example satisfying the condition (A3) is a thin spherical shell around the unit sphere $S^2$ in $\mathbb{R}^3$ of the form
    \begin{align*}
      \Omega_\varepsilon = \{x\in\mathbb{R}^3 \mid 1<|x|<1+\varepsilon\} \quad (\Gamma = S^2,\, g_0 \equiv 0,\, g_1 \equiv 1).
    \end{align*}
    In this case we consider only the perfect slip boundary conditions
    \begin{align} \label{E:Per_Slip}
      u\cdot n_\varepsilon = 0, \quad 2\nu P_\varepsilon D(u)n_\varepsilon = 0 \quad\text{on}\quad \Gamma_\varepsilon.
    \end{align}
    The Navier--Stokes equations in the thin spherical shell were studied in \cite{TeZi97} under boundary conditions different from \eqref{E:Per_Slip} (see Remark \ref{R:Limit_TZ} below).
  \end{itemize}
\end{remark}

\begin{remark} \label{R:Lim_Reduce}
  If $g\equiv1$ and $\gamma^0=\gamma^1=0$ in \eqref{E:NS_Limit}, then we have
  \begin{align} \label{E:Lim_Reduce}
    \left\{
    \begin{alignedat}{3}
      \partial_tv+\overline{\nabla}_vv-2\nu P\mathrm{div}_\Gamma[D_\Gamma(v)]+\nabla_\Gamma q &= f &\quad &\text{on} &\quad &\Gamma\times(0,\infty), \\
      \mathrm{div}_\Gamma v &= 0 &\quad &\text{on} &\quad &\Gamma\times(0,\infty).
    \end{alignedat}
    \right.
  \end{align}
  The equations \eqref{E:Lim_Reduce} agree with the Navier--Stokes equations on an evolving surface in $\mathbb{R}^3$ derived in \cites{JaOlRe18,KoLiGi17,Miu18} when a surface is stationary in time.
  In particular, the paper \cite{Miu18} derived the equations by the thin-film limit based on formal calculations.
  Hence Theorems \ref{T:SL_Weak} and \ref{T:SL_Strong} justify the result of \cite{Miu18} for a stationary surface.
  Also, we observe in Lemma \ref{L:DiSR_BL} that
  \begin{align} \label{E:PD_Main}
    2P\mathrm{div}_\Gamma[D_\Gamma(X)] = \Delta_BX+\nabla_\Gamma(\mathrm{div}_\Gamma X)+\mathrm{Ric}(X) \quad\text{on}\quad \Gamma
  \end{align}
  for a tangential vector field $X$ on $\Gamma$, where $\Delta_B$ and $\mathrm{Ric}$ are the Bochner Laplacian on $\Gamma$ and the Ricci curvature of $\Gamma$ (see Appendix \ref{S:Ap_Visc}).
  Note that the Ricci curvature agrees with the Gaussian curvature $K$ of $\Gamma$ (see Section \ref{SS:Pre_Surf}), i.e.
  \begin{align} \label{E:RG_Main}
    \mathrm{Ric}(X) = KX \quad\text{on}\quad \Gamma
  \end{align}
  for a tangential vector field $X$ on $\Gamma$ since $\Gamma$ is two-dimensional (see Lemma \ref{L:Cur_Ten}).
  By \eqref{E:PD_Main} we observe that the equations \eqref{E:Lim_Reduce} are equivalent to
  \begin{align} \label{E:NS_Riem}
    \left\{
    \begin{alignedat}{3}
      \partial_tv+\overline{\nabla}_vv-\nu\{\Delta_Bv+\mathrm{Ric}(v)\}+\nabla_\Gamma q &= f &\quad &\text{on} &\quad &\Gamma\times(0,\infty), \\
      \mathrm{div}_\Gamma v &= 0 &\quad &\text{on} &\quad &\Gamma\times(0,\infty).
    \end{alignedat}
    \right.
  \end{align}
  Note that these equations are described only by intrinsic quantities of the embedded surface $\Gamma$.
  Also, the equations \eqref{E:NS_Riem} can be expressed as
  \begin{align} \label{E:NS_HD}
    \left\{
    \begin{alignedat}{3}
      \partial_tv+\overline{\nabla}_vv-\nu\{\Delta_Hv+2\mathrm{Ric}(v)\}+\nabla_\Gamma q &= f &\quad &\text{on} &\quad &\Gamma\times(0,\infty), \\
      \mathrm{div}_\Gamma v &= 0 &\quad &\text{on} &\quad &\Gamma\times(0,\infty)
    \end{alignedat}
    \right.
  \end{align}
  by the Weitzenb\"{o}ck formula
  \begin{align*}
    \Delta_BX = \Delta_HX+\mathrm{Ric}(X) \quad\text{on}\quad \Gamma
  \end{align*}
  for a tangential vector field $X$ on $\Gamma$ (see Lemma \ref{L:Weitzen}).
  Here $\Delta_H$ is the Hodge Laplacian on $\Gamma$ (see Appendix \ref{S:Ap_Visc}).
  The equations \eqref{E:NS_Riem} are the Navier--Stokes equations on a Riemannian manifold introduced by Ebin and Marsden \cite{EbMa70}, Mitsumatsu and Yano \cite{MitYa02}, and Taylor \cite{Ta92} and studied by many authors (see e.g. \cites{ChaCzu13,DinMit04,KheMis12,KohWen18,MitTa01,Nag99,Pi17,Prie94,PrSiWi20pre,SaTu20}).
  Theorems \ref{T:SL_Weak} and \ref{T:SL_Strong} provide the first result on a rigorous derivation of \eqref{E:NS_Riem} by the thin-film limit when a manifold is a general two-dimensional closed surface in $\mathbb{R}^3$.
  We also show in Lemma \ref{L:Vis_Defo} that
  \begin{align*}
    P\mathrm{div}_\Gamma[gD_\Gamma(X)] = -\mathrm{Def}^\ast(g\,\mathrm{Def}\,X) \quad\text{on}\quad \Gamma
  \end{align*}
  for a function $g$ and a tangential vector field $X$ on $\Gamma$, where $\mathrm{Def}\,X$ is the deformation tensor for $X$ and $\mathrm{Def}^\ast$ is the formal adjoint of $\mathrm{Def}$ (see Appendix \ref{S:Ap_Visc}).
  Here the right-hand side is intrinsically defined and thus the limit equations \eqref{E:NS_Limit} are described only by intrinsic quantities of $\Gamma$.
  Also, since the equations \eqref{E:Lim_Reduce} are equivalent to \eqref{E:NS_Riem}, we can consider \eqref{E:NS_Limit} as the damped and weighted Navier--Stokes equations on a Riemannian manifold.
  Note that the damping term $(\gamma^0+\gamma^1)v$ and the weight function $g$ in \eqref{E:NS_Limit} come from the friction term $\gamma_\varepsilon u^\varepsilon$ in the slip boundary conditions \eqref{E:Slip_Intro} and the thickness $\varepsilon g$ of the curved thin domain $\Omega_\varepsilon$, respectively.
\end{remark}

\begin{remark} \label{R:Limit_TZ}
  As mentioned in Remark \ref{R:Ex_As}, Temam and Ziane \cite{TeZi97} considered the Navier--Stokes equations in the thin spherical shell
  \begin{align*}
    \Omega_\varepsilon = \{x \in \mathbb{R}^3 \mid 1<|x|<1+\varepsilon\} \quad (\Gamma = S^2,\, g_0 \equiv 0,\, g_1 \equiv 1)
  \end{align*}
  around the unit sphere under the Hodge (or de Rham) boundary conditions
  \begin{align} \label{E:Hodge_BC}
    u\cdot n_\varepsilon = 0, \quad \mathrm{curl}\,u\times n_\varepsilon = 0 \quad\text{on}\quad \Gamma_\varepsilon.
  \end{align}
  The authors of \cite{TeZi97} called \eqref{E:Hodge_BC} the free boundary conditions and mentioned that these conditions were equivalent to the perfect slip boundary conditions \eqref{E:Per_Slip}, but this is valid only when the boundary $\Gamma_\varepsilon$ is flat.
  Indeed, it is shown in \cite{MitMon09}*{Section 2} and \cite{Miu_NSCTD_01}*{Lemma B.10} that under the condition $u\cdot n_\varepsilon=0$ on $\Gamma_\varepsilon$ we have
  \begin{align*}
    2P_\varepsilon D(u)n_\varepsilon-\mathrm{curl}\,u\times n_\varepsilon = 2W_\varepsilon u \quad\text{on}\quad \Gamma_\varepsilon,
  \end{align*}
  where $W_\varepsilon$ is the Weingarten map (or the shape operator) of $\Gamma_\varepsilon$ that represents the curvatures of $\Gamma_\varepsilon$ (see Section \ref{SS:Pre_CTD}).
  Moreover, this difference affects the form of limit equations on $S^2$ derived from the Navier--Stokes equations in $\Omega_\varepsilon$.
  When
  \begin{align*}
    \Gamma = S^2, \quad g_0 \equiv 0, \quad g_1 \equiv 1, \quad \gamma_\varepsilon^0=\gamma_\varepsilon^1=0,
  \end{align*}
  the condition (A3) of Assumption \ref{Assump_2} is satisfied and our limit equations \eqref{E:NS_Limit} are equivalent to \eqref{E:NS_Riem} by $g\equiv1$ and $\gamma^0=\gamma^1=0$ (see Remark \ref{R:Lim_Reduce}).
  Moreover, since the Gaussian curvature of $S^2$ is $K=1$ and the relation \eqref{E:RG_Main} holds, the equations \eqref{E:NS_Riem} read
  \begin{align} \label{E:Limit_S2}
    \left\{
    \begin{alignedat}{3}
      \partial_tv+\overline{\nabla}_vv-\nu(\Delta_Bv+v)+\nabla_\Gamma q &= f &\quad &\text{on} &\quad &S^2\times(0,\infty), \\
      \mathrm{div}_\Gamma v &= 0 &\quad &\text{on} &\quad &S^2\times(0,\infty).
    \end{alignedat}
    \right.
  \end{align}
  On the other hand, the limit equations derived in \cite{TeZi97} are of the form
  \begin{align} \label{E:Lim_TZ_TV}
    \left\{
    \begin{alignedat}{3}
      \partial_tv+\overline{\nabla}_vv-\nu\Delta_2v+\nabla_\Gamma q &= f &\quad &\text{on} &\quad &S^2\times(0,\infty), \\
      \mathrm{div}_\Gamma v &= 0 &\quad &\text{on} &\quad &S^2\times(0,\infty).
    \end{alignedat}
    \right.
  \end{align}
  Here $\Delta_2v$ is the tangential vector Laplacian of $v$ on $S^2$ expressed in the spherical coordinate system (see \cite{TeZi97}*{Appendix}).
  In Lemma \ref{L:Vec_Sph} we derive
  \begin{align} \label{E:VeL_Main}
    \Delta_2X = \Delta_HX = \Delta_BX-X \quad\text{on}\quad S^2
  \end{align}
  for a tangential vector field $X$ on $S^2$.
  Thus the equations \eqref{E:Lim_TZ_TV} read
  \begin{align} \label{E:Lim_TZ_Bo}
    \left\{
    \begin{alignedat}{3}
      \partial_tv+\overline{\nabla}_vv-\nu(\Delta_Bv-v)+\nabla_\Gamma q &= f &\quad &\text{on} &\quad &S^2\times(0,\infty), \\
      \mathrm{div}_\Gamma v &= 0 &\quad &\text{on} &\quad &S^2\times(0,\infty).
    \end{alignedat}
    \right.
  \end{align}
  Here the sign of the zeroth order term $v$ in our limit equations \eqref{E:Limit_S2} is opposite to that of the same term in the limit equations \eqref{E:Lim_TZ_Bo} derived in \cite{TeZi97}.
  Due to this fact the structure of \eqref{E:Limit_S2} such as the stability of a solution is quite different from that of \eqref{E:Lim_TZ_Bo}.
  In particular, for an arbitrary nonzero vector $a\in\mathbb{R}^3$ if we set
  \begin{align*}
    v(y) := a\times y, \quad q(y) := \frac{1}{2}|v(y)|^2, \quad y\in S^2,
  \end{align*}
  where $\times$ denotes the vector product in $\mathbb{R}^3$, then direct calculations show
  \begin{gather*}
    D_\Gamma(v) = 0, \quad \Delta_Bv+v = 2P\mathrm{div}_\Gamma[D_\Gamma(v)] = 0, \\
    \overline{\nabla}_vv+\nabla_\Gamma q = 2D_\Gamma(v)v = 0, \quad \mathrm{div}_\Gamma v = 0 \quad\text{on}\quad S^2
  \end{gather*}
  and $v$ is not the tangential gradient of a function on $S^2$ by the last relation.
  Hence $v$ is a stationary solution to \eqref{E:Limit_S2} with $f=0$ but does not satisfy \eqref{E:Lim_TZ_Bo} (note that the above $v$ is a Killing vector field on $S^2$).
  Also, by \eqref{E:VeL_Main} the limit equations \eqref{E:Lim_TZ_TV} derived in \cite{TeZi97} can be rewritten as
  \begin{align} \label{E:NS_Hodge}
    \left\{
    \begin{alignedat}{3}
      \partial_tv+\overline{\nabla}_vv-\nu\Delta_Hv+\nabla_\Gamma q &= f &\quad &\text{on} &\quad &\Gamma\times(0,\infty), \\
      \mathrm{div}_\Gamma v &= 0 &\quad &\text{on} &\quad &\Gamma\times(0,\infty)
    \end{alignedat}
    \right.
  \end{align}
  with $\Gamma=S^2$.
  Note that we can consider \eqref{E:NS_Hodge} on a general Riemannian manifold $\Gamma$.
  There are several works on the equations \eqref{E:NS_Hodge} (see e.g. \cites{Ili90,IliFil88,Lic16,TemWan93}), but they were called the ``wrong'' Navier--Stokes equations in \cite{Ta92} since the viscous term in \eqref{E:NS_Hodge} is not equal to the divergence of the deformation tensor which gives the viscous term in \eqref{E:NS_Riem}.
  We refer to \cite{ChaCzuDis17} for discussions on the choice of the viscous term in the Navier--Stokes equations on a Riemannian manifold.
\end{remark}

\begin{remark} \label{R:HL_NTSV}
  In Section \ref{SS:WS_NTS} we show the Helmholtz--Leray decomposition for a not necessarily tangential vector field $v\in L^2(\Gamma)^3$ of the form
  \begin{align} \label{E:NTHL_Main}
    \begin{gathered}
      v = v_\sigma+\nabla_\Gamma q+qHn \quad\text{in}\quad L^2(\Gamma)^3, \\
      v_\sigma \in L_\sigma^2(\Gamma),\,\nabla_\Gamma q+qHn \in L_\sigma^2(\Gamma)^\perp
    \end{gathered}
  \end{align}
  as an easy consequence of the results in Section \ref{S:WSol} used for the proof of the tangential Helmholtz--Leray decomposition \eqref{E:HLT_Intro} (see Theorem \ref{T:HL_L2Ge}).
  Here
  \begin{align*}
    L_\sigma^2(\Gamma):=\{v\in L^2(\Gamma)^3 \mid \text{$\mathrm{div}_\Gamma v=0$ on $\Gamma$}\}
  \end{align*}
  and $H$ is the mean curvature of $\Gamma$ (see Section \ref{SS:Pre_Surf}).
  The decomposition \eqref{E:NTHL_Main} was already found in \cite{KoLiGi17} (see also \cite{KoLiGi18Er}), but here we further prove the uniqueness of the scalar potential $q$ in \eqref{E:NTHL_Main} by using the fact that $\Gamma$ is closed (see Lemma \ref{L:TGrHn_Hin_Con}).
  Note that this result does not hold for \eqref{E:HLT_Intro} in which the scalar potential is determined only up to a constant (see Theorem \ref{T:HL_L2T}).
  In Remark \ref{R:HL_L2Ge} we also observe that the decomposition \eqref{E:NTHL_Main} for a tangential vector field on $\Gamma$ may be different from \eqref{E:HLT_Intro} with $g\equiv1$.
  Although we do not use \eqref{E:NTHL_Main} in the proofs of Theorems \ref{T:SL_Weak} and \ref{T:SL_Strong}, it is fundamental for the future study of the Navier--Stokes equations on an evolving surface in which the pressure term is of the form $\nabla_\Gamma q+qHn$ (see \cites{JaOlRe18,KoLiGi17}).
\end{remark}

\section{Preliminaries} \label{S:Pre}
We fix notations on a closed surface and a curved thin domain and provide some basic results used in the sequel.

In what follows, we fix a coordinate system of $\mathbb{R}^3$ and write $x_i$, $i=1,2,3$ for the $i$-th component of a point $x\in\mathbb{R}^3$ in this coordinate system.
Moreover, we denote by $c$ a general positive constant independent of the parameter $\varepsilon$.
Other notations on vectors and matrices are given in Appendix \ref{S:Ap_Vec}.

Throughout this paper we omit the proofs of the results established in the first and second parts \cites{Miu_NSCTD_01,Miu_NSCTD_02} of our study unless otherwise stated.

\subsection{Closed surface} \label{SS:Pre_Surf}
Let $\Gamma$ be a closed, connected, and oriented surface in $\mathbb{R}^3$ of class $C^\ell$ with $\ell\geq2$.
We denote by $n$, $d$, and $\kappa_1$ and $\kappa_2$ the unit outward normal vector field of $\Gamma$, the signed distance function from $\Gamma$ increasing in the direction of $n$, and the principal curvatures of $\Gamma$, respectively.
The $C^\ell$-regularity of $\Gamma$ implies
\begin{align*}
  n \in C^{\ell-1}(\Gamma)^3, \quad \kappa_1,\kappa_2 \in C^{\ell-2}(\Gamma)
\end{align*}
and thus $\kappa_1$ and $\kappa_2$ are bounded on the compact set $\Gamma$.
By this fact there exists a tubular neighborhood of $\Gamma$ of the form
\begin{align*}
  N := \{x\in\mathbb{R}^3 \mid \mathrm{dist}(x,\Gamma) < \delta\}, \quad \delta > 0
\end{align*}
such that for each $x\in N$ we have a unique point $\pi(x)\in\Gamma$ satisfying
\begin{align} \label{E:Nor_Coord}
  x = \pi(x)+d(x)n(\pi(x)), \quad \nabla d(x) = n(\pi(x))
\end{align}
and $d$ and $\pi$ are of class $C^\ell$ and $C^{\ell-1}$ on $\overline{N}$ (see \cite{GiTr01}*{Section 14.6}).
We may also assume that
\begin{align} \label{E:Curv_Bound}
  c^{-1} \leq 1-r\kappa_i(y) \leq c \quad\text{for all}\quad y\in\Gamma,\,r\in(-\delta,\delta),\,i=1,2
\end{align}
by taking $\delta>0$ sufficiently small.

Let us define differential operators on $\Gamma$ and surface quantities of $\Gamma$.
We write
\begin{align*}
  P(y) := I_3-n(y)\otimes n(y), \quad Q(y) := n(y)\otimes n(y), \quad y\in\Gamma
\end{align*}
for the orthogonal projections onto the tangent plane of $\Gamma$ and the normal direction of $\Gamma$.
They are of class $C^{\ell-1}$ on $\Gamma$ and satisfy $|P|=2$, $|Q|=1$, and
\begin{gather*}
  I_3 = P+Q, \quad PQ = QP = 0, \quad P^T = P^2 = P, \quad Q^T = Q^2 = Q, \\
  |a|^2 = |Pa|^2+|Qa|^2, \quad |Pa| \leq |a|, \quad Pa\cdot n = 0, \quad a\in\mathbb{R}^3
\end{gather*}
on $\Gamma$.
These relations are used frequently in the sequel.
For $\eta\in C^1(\Gamma)$ we define the tangential gradient and the tangential derivatives of $\eta$ by
\begin{align} \label{E:Def_TGr}
  \nabla_\Gamma\eta(y) := P(y)\nabla\tilde{\eta}(y), \quad \underline{D}_i\eta(y) := \sum_{j=1}^3P_{ij}(y)\partial_j\tilde{\eta}(y), \quad y\in\Gamma,\, i=1,2,3
\end{align}
so that $\nabla_\Gamma\eta=(\underline{D}_1\eta,\underline{D}_2\eta,\underline{D}_3\eta)^T$.
Here $\tilde{\eta}$ is a $C^1$-extension of $\eta$ to $N$ with $\tilde{\eta}|_\Gamma=\eta$.
From the definition of $\nabla_\Gamma\eta$ it follows that
\begin{align} \label{E:P_TGr}
  P\nabla_\Gamma\eta = \nabla_\Gamma\eta, \quad n\cdot\nabla_\Gamma\eta = 0 \quad\text{on}\quad \Gamma.
\end{align}
Note that $\nabla_\Gamma\eta$ defined by \eqref{E:Def_TGr} agrees with the gradient on a Riemannian manifold expressed under a local coordinate system (see Lemma \ref{L:TGr_DG}).
Hence the values of $\nabla_\Gamma\eta$ and $\underline{D}_i\eta$ do not depend on the choice of an extension $\tilde{\eta}$.
In particular, for the constant extension $\bar{\eta}:=\eta\circ\pi$ of $\eta$ in the normal direction of $\Gamma$ we have
\begin{align} \label{E:ConDer_Surf}
  \nabla\bar{\eta}(y) = \nabla_\Gamma\eta(y), \quad \partial_i\bar{\eta}(y) = \underline{D}_i\eta(y), \quad y\in\Gamma,\,i=1,2,3
\end{align}
since $\nabla\pi(y)=P(y)$ for $y\in\Gamma$ by \eqref{E:Nor_Coord} and $d(y)=0$.
From now on, we always use the notation $\bar{\eta}$ with an overline for the constant extension of a function $\eta$ on $\Gamma$ in the normal direction of $\Gamma$.
For $\eta\in C^2(\Gamma)$ we define the tangential Hessian matrix of $\eta$ and the Laplace--Beltrami operator on $\Gamma$ by
\begin{align} \label{E:Def_THess}
  \nabla_\Gamma^2\eta := (\underline{D}_i\underline{D}_j\eta)_{i,j}, \quad \Delta_\Gamma\eta := \mathrm{tr}[\nabla_\Gamma^2\eta] = \sum_{i=1}^3\underline{D}_i^2\eta \quad\text{on}\quad \Gamma.
\end{align}
Let $v=(v_1,v_2,v_3)^T\in C^1(\Gamma)^3$ be a (not necessarily tangential) vector field on $\Gamma$.
We define the tangential gradient matrix and the surface divergence of $v$ by
\begin{align} \label{E:Def_DivG}
  \nabla_\Gamma v :=
  \begin{pmatrix}
    \underline{D}_1v_1 & \underline{D}_1v_2 & \underline{D}_1v_3 \\
    \underline{D}_2v_1 & \underline{D}_2v_2 & \underline{D}_2v_3 \\
    \underline{D}_3v_1 & \underline{D}_3v_2 & \underline{D}_3v_3
  \end{pmatrix}, \quad
  \mathrm{div}_\Gamma v := \mathrm{tr}[\nabla_\Gamma v] = \sum_{i=1}^3\underline{D}_iv_i
\end{align}
on $\Gamma$ and the surface strain rate tensor for $v$ by
\begin{align} \label{E:Def_SSR}
  D_\Gamma(v) := P(\nabla_\Gamma v)_SP \quad\text{on}\quad \Gamma, \quad (\nabla_\Gamma v)_S = \frac{\nabla_\Gamma v+(\nabla_\Gamma v)^T}{2}.
\end{align}
For $v\in C^1(\Gamma)^3$ and $\eta\in C(\Gamma)^3$ we set
\begin{align*}
  (\eta\cdot\nabla_\Gamma)v :=
  \begin{pmatrix}
    \eta\cdot\nabla_\Gamma v_1 \\
    \eta\cdot\nabla_\Gamma v_2 \\
    \eta\cdot\nabla_\Gamma v_3
  \end{pmatrix}
   = (\nabla_\Gamma v)^T\eta \quad\text{on}\quad \Gamma.
\end{align*}
Note that for any $C^1$-extension $\tilde{v}$ of $v$ to $N$ with $\tilde{v}|_\Gamma=v$ we have
\begin{align} \label{E:TGrM_Surf}
  \nabla_\Gamma v = P\nabla\tilde{v}, \quad (\eta\cdot\nabla_\Gamma)v = [(P\eta)\cdot\nabla_\Gamma]v = [(P\eta)\cdot\nabla]\tilde{v} \quad\text{on}\quad \Gamma.
\end{align}
When $v\in C^1(\Gamma)^3$ and $\eta\in C(\Gamma)^3$ are tangential on $\Gamma$ (i.e. $v\cdot n=\eta\cdot n=0$ on $\Gamma$), we define the covariant derivative of $v$ along $\eta$ by
\begin{align*}
  \overline{\nabla}_\eta v := P(\eta\cdot\nabla_\Gamma)v = P(\nabla_\Gamma v)^T\eta \quad\text{on}\quad \Gamma.
\end{align*}
If $v\in C^2(\Gamma)^3$, then we write
\begin{align*}
  |\nabla_\Gamma^2v|^2 := \sum_{i,j,k=1}^3|\underline{D}_i\underline{D}_jv_k|^2, \quad \Delta_\Gamma v :=
  \begin{pmatrix}
    \Delta_\Gamma v_1 \\
    \Delta_\Gamma v_2 \\
    \Delta_\Gamma v_3
  \end{pmatrix}
  \quad\text{on}\quad \Gamma.
\end{align*}
For a matrix-valued function $A\in C^1(\Gamma)^{3\times 3}$ of the form
\begin{align*}
  A = (A_{ij})_{i,j} =
  \begin{pmatrix}
    A_{11} & A_{12} & A_{13} \\
    A_{21} & A_{22} & A_{23} \\
    A_{31} & A_{32} & A_{33}
  \end{pmatrix}
\end{align*}
we define the surface divergence of $A$ as a vector field
\begin{align*}
  \mathrm{div}_\Gamma A :=
  \begin{pmatrix}
    [\mathrm{div}_\Gamma A]_1 \\
    [\mathrm{div}_\Gamma A]_2 \\
    [\mathrm{div}_\Gamma A]_3
  \end{pmatrix}, \quad
  [\mathrm{div}_\Gamma A]_j := \sum_{i=1}^3\underline{D}_iA_{ij} \quad\text{on}\quad \Gamma,\,j=1,2,3.
\end{align*}
Note that $\mathrm{div}_\Gamma[\nabla_\Gamma v]=\Delta_\Gamma v$ on $\Gamma$ for $v\in C^2(\Gamma)^3$ in the above notations.
We further define the Weingarten map $W$, (twice) the mean curvature $H$, and the Gaussian curvature $K$ of $\Gamma$ by
\begin{align} \label{E:Def_Wein}
  W := -\nabla_\Gamma n, \quad H := \mathrm{tr}[W] = -\mathrm{div}_\Gamma n, \quad K := \kappa_1\kappa_2 \quad\text{on}\quad \Gamma.
\end{align}
They are of class $C^{\ell-2}$ and thus bounded on $\Gamma$ by the $C^\ell$-regularity of $\Gamma$.

\begin{lemma} \label{L:Form_W}
  The Weingarten map $W$ is symmetric and
  \begin{align} \label{E:Form_W}
    Wn = 0, \quad PW = WP = W \quad\text{on}\quad \Gamma.
  \end{align}
  If $v\in C^1(\Gamma)^3$ is tangential, i.e. $v\cdot n=0$ on $\Gamma$, then
  \begin{align} \label{E:Grad_W}
    (\nabla_\Gamma v)n = Wv, \quad \nabla_\Gamma v = P(\nabla_\Gamma v)P+(Wv)\otimes n \quad\text{on}\quad \Gamma.
  \end{align}
  Also, the surface divergence of $P$ is of the form
  \begin{align} \label{E:Div_P}
    \mathrm{div}_\Gamma P = Hn \quad\text{on}\quad \Gamma.
  \end{align}
\end{lemma}

\begin{proof}
  We deduce from \eqref{E:Nor_Coord}, \eqref{E:ConDer_Surf}, and $|n|^2=1$ on $\Gamma$ that
  \begin{align*}
    W = -\nabla\bar{n} = -\nabla^2d, \quad Wn = -(\nabla_\Gamma n)n = -\frac{1}{2}\nabla_\Gamma(|n|^2) = 0 \quad\text{on}\quad \Gamma.
  \end{align*}
  Thus we have $W^T=W$ on $\Gamma$ and \eqref{E:Form_W}.

  If $v\in C^1(\Gamma)^3$ satisfies $v\cdot n=0$ on $\Gamma$, then
  \begin{align*}
    0 = \nabla_\Gamma(v\cdot n) = (\nabla_\Gamma v)n+(\nabla_\Gamma n)v = (\nabla_\Gamma v)n-Wn \quad\text{on}\quad \Gamma,
  \end{align*}
  which shows the first equality of \eqref{E:Grad_W}.
  Also, by $I_3=P+Q$ on $\Gamma$ and \eqref{E:P_TGr},
  \begin{align*}
    \nabla_\Gamma v = (\nabla_\Gamma v)P+(\nabla_\Gamma v)Q = P(\nabla_\Gamma v)P+\{(\nabla_\Gamma v)n\}\otimes n \quad\text{on}\quad \Gamma.
  \end{align*}
  Thus the second equality of \eqref{E:Grad_W} follows from the first one.

  Let us show \eqref{E:Div_P}.
  Denoting the Kronecker delta by $\delta_{ij}$ we have
  \begin{align*}
    [\mathrm{div}_\Gamma P]_j = \sum_{i=1}^3\underline{D}_i(\delta_{ij}-n_in_j) = \sum_{i=1}^3(W_{ii}n_j+W_{ij}n_i) = \mathrm{tr}[W]n_j+[W^Tn]_j
  \end{align*}
  on $\Gamma$ for $j=1,2,3$ and thus $\mathrm{div}_\Gamma P=\mathrm{tr}[W]n+W^Tn$ on $\Gamma$.
  To the right-hand side we apply \eqref{E:Def_Wein} and $W^Tn=Wn=0$ on $\Gamma$ to obtain \eqref{E:Div_P}.
\end{proof}

The Weingarten map appears in the exchange of the tangential derivatives.

\begin{lemma}[{\cite{Miu_NSCTD_01}*{Lemma 3.2}}] \label{L:TD_Exc}
  For $\eta\in C^2(\Gamma)$ we have
  \begin{align} \label{E:TD_Exc}
    \underline{D}_i\underline{D}_j\eta-\underline{D}_j\underline{D}_i\eta = [W\nabla_\Gamma\eta]_in_j-[W\nabla_\Gamma\eta]_jn_i \quad\text{on}\quad \Gamma,\,i,j=1,2,3.
  \end{align}
  Here $[W\nabla_\Gamma\eta]_i$ is the $i$-th component of the vector field $W\nabla_\Gamma\eta$ for $i=1,2,3$.
\end{lemma}

By \eqref{E:Form_W} we see that $W$ has the eigenvalue zero associated with the eigenvector $n$.
Its other eigenvalues are the principal curvatures $\kappa_1$ and $\kappa_2$ and thus $H=\kappa_1+\kappa_2$ on $\Gamma$ (see e.g. \cites{GiTr01,Lee18}).
From this fact, \eqref{E:Curv_Bound}, and \eqref{E:Form_W} the next lemma follows.

\begin{lemma}[{\cite{Miu_NSCTD_01}*{Lemma 3.3}}] \label{L:Wein}
  The matrix
  \begin{align*}
    I_3-d(x)\overline{W}(x) = I_3-rW(y)
  \end{align*}
  is invertible for all $x=y+rn(y)\in N$ with $y\in\Gamma$ and $r\in(-\delta,\delta)$.
  Moreover,
  \begin{align} \label{E:WReso_P}
    \{I_3-rW(y)\}^{-1}P(y) = P(y)\{I_3-rW(y)\}^{-1}
  \end{align}
  for all $y\in\Gamma$ and $r\in(-\delta,\delta)$ and there exists a constant $c>0$ such that
  \begin{gather}
    c^{-1}|a| \leq \bigl|\{I_3-rW(y)\}^ka\bigr| \leq c|a|, \quad k=\pm1, \label{E:Wein_Bound} \\
    \bigl|I_3-\{I_3-rW(y)\}^{-1}\bigr| \leq c|r| \label{E:Wein_Diff}
  \end{gather}
  for all $y\in\Gamma$, $r\in(-\delta,\delta)$, and $a\in\mathbb{R}^3$.
\end{lemma}

We also get the following relations by \eqref{E:Nor_Coord}, \eqref{E:P_TGr}, \eqref{E:ConDer_Surf}, and \eqref{E:WReso_P}--\eqref{E:Wein_Diff}.

\begin{lemma}[{\cite{Miu_NSCTD_01}*{Lemma 3.4}}] \label{L:Pi_Der}
  Let $\eta\in C^1(\Gamma)$ and $\bar{\eta}=\eta\circ\pi$.
  Then
  \begin{align} \label{E:ConDer_Dom}
    \nabla\bar{\eta}(x) = \left\{I_3-d(x)\overline{W}(x)\right\}^{-1}\overline{\nabla_\Gamma\eta}(x), \quad x\in N
  \end{align}
  and there exists a constant $c>0$ independent of $\eta$ such that
  \begin{gather}
    c^{-1}\left|\overline{\nabla_\Gamma\eta}(x)\right| \leq |\nabla\bar{\eta}(x)| \leq c\left|\overline{\nabla_\Gamma\eta}(x)\right|, \label{E:ConDer_Bound} \\
    \left|\nabla\bar{\eta}(x)-\overline{\nabla_\Gamma\eta}(x)\right| \leq c\left|d(x)\overline{\nabla_\Gamma\eta}(x)\right| \label{E:ConDer_Diff}
  \end{gather}
  for all $x\in N$.
\end{lemma}

Let us define the Sobolev spaces on $\Gamma$.
For a $C^1$ tangential vector field $X$ on $\Gamma$, by a standard localization argument and a local expression of $\mathrm{div}_\Gamma X$ (see Lemma \ref{L:DivG_DG}) we observe that the surface divergence theorem
\begin{align} \label{E:SD_Thm}
  \int_\Gamma\mathrm{div}_\Gamma X\,d\mathcal{H}^2 = 0
\end{align}
holds (here $\mathcal{H}^2$ is the two-dimensional Hausdorff measure).
By this formula we have the following integration by parts formulas.

\begin{lemma}[{\cite{Miu_NSCTD_01}*{Lemma 3.5}}] \label{L:IbP_TD}
  For $v\in C^1(\Gamma)^3$ we have
  \begin{align} \label{E:IbP_DivG}
    \int_\Gamma\mathrm{div}_\Gamma v\,d\mathcal{H}^2 = -\int_\Gamma(v\cdot n)H\,d\mathcal{H}^2.
  \end{align}
  Also, for $\eta,\xi\in C^1(\Gamma)$ and $i=1,2,3$,
  \begin{align} \label{E:IbP_TD}
    \int_\Gamma(\eta\underline{D}_i\xi+\xi\underline{D}_i\eta)\,d\mathcal{H}^2 = -\int_\Gamma \eta\xi Hn_i\,d\mathcal{H}^2.
  \end{align}
\end{lemma}

Based on \eqref{E:IbP_TD}, for $p\in[1,\infty]$ and $i=1,2,3$ we say that $\eta\in L^p(\Gamma)$ has the $i$-th weak tangential derivative if there exists $\eta_i\in L^p(\Gamma)$ such that
\begin{align} \label{E:Def_WTD}
  \int_\Gamma \eta_i\xi\,d\mathcal{H}^2 = -\int_\Gamma \eta(\underline{D}_i\xi+\xi Hn_i)\,d\mathcal{H}^2
\end{align}
for all $\xi\in C^1(\Gamma)$.
We write $\underline{D}_i\eta$ for this $\eta_i$ and define the Sobolev space
\begin{align*}
  W^{1,p}(\Gamma) &:= \{\eta \in L^p(\Gamma) \mid \text{$\underline{D}_i\eta\in L^p(\Gamma)$ for all $i=1,2,3$}\}, \\
  \|\eta\|_{W^{1,p}(\Gamma)} &:=
  \begin{cases}
    \left(\|\eta\|_{L^p(\Gamma)}^p+\|\nabla_\Gamma\eta\|_{L^p(\Gamma)}^p\right)^{1/p} &\text{if}\quad p\in[1,\infty), \\
    \|\eta\|_{L^\infty(\Gamma)}+\|\nabla_\Gamma\eta\|_{L^\infty(\Gamma)} &\text{if}\quad p=\infty.
  \end{cases}
\end{align*}
In the above, $\nabla_\Gamma\eta:=(\underline{D}_1\eta,\underline{D}_2\eta,\underline{D}_3\eta)^T$ is the weak tangential gradient of $\eta\in W^{1,p}(\Gamma)$, which is consistent with \eqref{E:Def_TGr} for a $C^1$ function on $\Gamma$.
Let us give basic properties of functions in $W^{1,p}(\Gamma)$.

\begin{lemma} \label{L:WTGr_Van}
  Let $p\in[1,\infty]$ and $\eta\in W^{1,p}(\Gamma)$.
  Then
  \begin{align*}
    \nabla_\Gamma \eta = 0 \quad\text{in}\quad L^p(\Gamma)^3, \quad\text{i.e}\quad \underline{D}_i\eta = 0 \quad\text{in}\quad L^p(\Gamma), \, i=1,2,3
  \end{align*}
  if and only if $\eta$ is constant on $\Gamma$.
\end{lemma}

\begin{lemma} \label{L:RK_Surf}
  For $p\in[1,\infty]$ the embedding $W^{1,p}(\Gamma)\hookrightarrow L^p(\Gamma)$ is compact.
\end{lemma}

We also have Poincar\'{e}'s inequality on the closed surface $\Gamma$.

\begin{lemma} \label{L:Poin_Surf_Lp}
  Let $p\in[1,\infty]$.
  There exists a constant $c>0$ such that
  \begin{align} \label{E:Poin_Surf_Lp}
    \|\eta\|_{L^p(\Gamma)} \leq c\|\nabla_\Gamma\eta\|_{L^p(\Gamma)}
  \end{align}
  for all $\eta\in W^{1,p}(\Gamma)$ satisfying $\int_\Gamma\eta\,d\mathcal{H}^2=0$.
\end{lemma}

Lemmas \ref{L:WTGr_Van} and \ref{L:RK_Surf} are proved by a standard localization argument.
Also, the proof of Lemma \ref{L:Poin_Surf_Lp} is the same as in the case of a flat domain.
In Appendix \ref{S:Ap_Aux} we give the proofs of Lemmas \ref{L:WTGr_Van}--\ref{L:Poin_Surf_Lp} for the completeness.

For $p\in[1,\infty]$ we define the second order Sobolev space
\begin{align*}
  W^{2,p}(\Gamma) &:= \{\eta \in W^{1,p}(\Gamma) \mid \text{$\underline{D}_i\underline{D}_j\eta\in L^p(\Gamma)$ for all $i,j=1,2,3$}\}, \\
  \|\eta\|_{W^{2,p}(\Gamma)} &:=
  \begin{cases}
    \left(\|\eta\|_{W^{1,p}(\Gamma)}^p+\|\nabla_\Gamma^2\eta\|_{L^p(\Gamma)}^p\right)^{1/p} &\text{if}\quad p\in[1,\infty), \\
    \|\eta\|_{W^{1,\infty}(\Gamma)}+\|\nabla_\Gamma^2\eta\|_{L^\infty(\Gamma)} &\text{if}\quad p=\infty
  \end{cases}
\end{align*}
and the higher order Sobolev space $W^{m,p}(\Gamma)$ with $m\geq 2$ similarly, and write
\begin{align*}
  W^{0,p}(\Gamma) := L^p(\Gamma), \quad H^m(\Gamma) := W^{m,2}(\Gamma), \quad p\in[1,\infty], \, m\geq 0.
\end{align*}
Here $\nabla_\Gamma^2\eta:=(\underline{D}_i\underline{D}_j\eta)_{i,j}$ for $\eta\in W^{2,p}(\Gamma)$.
Note that $W^{m,p}(\Gamma)$ is a Banach space.
In particular, $H^1(\Gamma)$ is a Hilbert space equipped with inner product
\begin{align*}
  (\eta,\xi)_{H^1(\Gamma)} := (v,\eta)_{L^2(\Gamma)}+(\nabla_\Gamma\eta,\nabla_\Gamma\xi)_{L^2(\Gamma)}, \quad \eta,\xi\in H^1(\Gamma).
\end{align*}
Moreover, a density result on $W^{m,p}(\Gamma)$ holds as in the case of a flat domain.

\begin{lemma}[{\cite{Miu_NSCTD_01}*{Lemma 3.6}}] \label{L:Wmp_Appr}
  Let $m=0,1,\dots,\ell$ and $p\in[1,\infty)$.
  Then $C^\ell(\Gamma)$ is dense in $W^{m,p}(\Gamma)$.
\end{lemma}

For a function space $\mathcal{X}(\Gamma)$ such as $C^m(\Gamma)$ and $W^{m,p}(\Gamma)$ we denote by
\begin{align*}
  \mathcal{X}(\Gamma,T\Gamma) := \{v\in\mathcal{X}(\Gamma)^3 \mid \text{$v\cdot n=0$ on $\Gamma$}\}
\end{align*}
the space of all tangential vector fields on $\Gamma$ whose components belong to $\mathcal{X}(\Gamma)$.
Then $W^{m,p}(\Gamma,T\Gamma)$ is a closed subspace of $W^{m,p}(\Gamma)^3$ for $m\geq0$ and $p\in[1,\infty]$.
Moreover, for $v\in W^{1,p}(\Gamma,T\Gamma)$ with $p\in[1,\infty]$ we see that
\begin{align} \label{E:IbP_WDivG_T}
  \int_\Gamma\mathrm{div}_\Gamma v\,d\mathcal{H}^2 = -\int_\Gamma(v\cdot n)H\,d\mathcal{H}^2 = 0
\end{align}
by \eqref{E:Def_WTD} with $\xi\equiv1$ (note that $\nabla_\Gamma\xi=0$ on $\Gamma$ when $\xi$ is constant).
We also have the following density result on $W^{m,p}(\Gamma,T\Gamma)$.

\begin{lemma}[{\cite{Miu_NSCTD_01}*{Lemma 3.7}}] \label{L:Wmp_Tan_Appr}
  Let $m=0,1,\dots,\ell-1$ and $p\in[1,\infty)$.
  Then $C^{\ell-1}(\Gamma,T\Gamma)$ is dense in $W^{m,p}(\Gamma,T\Gamma)$.
\end{lemma}

Next we fix notations on the dual spaces of the Sobolev spaces on $\Gamma$.
Let $H^{-1}(\Gamma)$ and $\langle\cdot,\cdot\rangle_\Gamma$ be the dual space of $H^1(\Gamma)$ and the duality product between $H^{-1}(\Gamma)$ and $H^1(\Gamma)$.
We consider $\eta\in L^2(\Gamma)$ as an element of $H^{-1}(\Gamma)$ by setting
\begin{align} \label{E:L2_Hin}
  \langle \eta,\xi\rangle_\Gamma:=(\eta,\xi)_{L^2(\Gamma)}, \quad \xi\in H^1(\Gamma).
\end{align}
Then by Lemma \ref{L:RK_Surf} we have the compact embeddings
\begin{align*}
  H^1(\Gamma) \hookrightarrow L^2(\Gamma) \hookrightarrow H^{-1}(\Gamma).
\end{align*}
Let $H^{-1}(\Gamma)^3$ be the dual space of $H^1(\Gamma)^3$.
We use the same notation $\langle\cdot,\cdot\rangle_\Gamma$ for the duality product between $H^{-1}(\Gamma)^3$ and $H^1(\Gamma)^3$.
For $f\in H^{-1}(\Gamma)^3$ and $i=1,2,3$ we define $f_i\in H^{-1}(\Gamma)$ by
\begin{align*}
  \langle f_i,\eta\rangle_\Gamma := \langle f,\eta e_i\rangle_\Gamma, \quad \eta\in H^1(\Gamma),
\end{align*}
where $\{e_1,e_2,e_3\}$ is the standard basis of $\mathbb{R}^3$.
Then we can consider $f\in H^{-1}(\Gamma)^3$ as a vector field on $\Gamma$ with components in $H^{-1}(\Gamma)$ and write
\begin{align*}
  \langle f,v\rangle_\Gamma = \sum_{i=1}^3\langle f_i,v_i\rangle_\Gamma, \quad f =
  \begin{pmatrix}
    f_1 \\ f_2 \\ f_3
  \end{pmatrix}
  \in H^{-1}(\Gamma)^3, \, v =
  \begin{pmatrix}
    v_1 \\ v_2 \\ v_3
  \end{pmatrix}
  \in H^1(\Gamma)^3.
\end{align*}
Let $\eta\in W^{1,\infty}(\Gamma)$ and $\xi\in H^{-1}(\Gamma)$.
Since
\begin{align*}
  |\langle \xi,\eta\varphi\rangle_\Gamma| &\leq \|\xi\|_{H^{-1}(\Gamma)}\|\eta\varphi\|_{H^1(\Gamma)} \leq c\|\eta\|_{W^{1,\infty}(\Gamma)}\|\xi\|_{H^{-1}(\Gamma)}\|\varphi\|_{H^1(\Gamma)}
\end{align*}
for $\varphi\in H^1(\Gamma)$, we can define $\eta\xi\in H^{-1}(\Gamma)$ by
\begin{align} \label{E:Def_Mul_Hin}
  \langle\eta \xi, \varphi\rangle_\Gamma := \langle \xi,\eta\varphi\rangle_\Gamma, \quad \varphi\in H^1(\Gamma).
\end{align}
Let $A\in W^{1,\infty}(\Gamma)^{3\times 3}$.
When $v,w\in L^2(\Gamma)^3$ we have
\begin{align*}
  (Av,w)_{L^2(\Gamma)} = \sum_{i,j=1}^3(A_{ij}v_j,w_i)_{L^2(\Gamma)} = (v,A^Tw)_{L^2(\Gamma)}.
\end{align*}
Based on this equality, for $f\in H^{-1}(\Gamma)^3$ we define $Af\in H^{-1}(\Gamma)^3$ by
\begin{align} \label{E:Def_Av_Hin}
  \langle Af,v\rangle_\Gamma := \langle f,A^Tv\rangle_\Gamma, \quad v\in H^1(\Gamma)^3.
\end{align}
Note that for $i=1,2,3$ the $i$-th component of $Af$ is
\begin{align*}
  [Af]_i = \sum_{j=1}^3 A_{ij}f_j \quad\text{in}\quad H^{-1}(\Gamma).
\end{align*}
Let $\eta\in L^2(\Gamma)$.
Based on \eqref{E:Def_WTD} we define $\underline{D}_i\eta\in H^{-1}(\Gamma)$, $i=1,2,3$ by
\begin{align} \label{E:Def_TD_Hin}
  \langle\underline{D}_i\eta,\xi\rangle_\Gamma := -(\eta,\underline{D}_i\xi+\xi Hn_i)_{L^2(\Gamma)}, \quad \xi \in H^1(\Gamma).
\end{align}
This definition makes sense since $n$ and $H$ are bounded on $\Gamma$.
We consider the weak tangential gradient $\nabla_\Gamma\eta$ as an element of $H^{-1}(\Gamma)^3$ satisfying
\begin{align} \label{E:TGr_Hin}
  \langle\nabla_\Gamma\eta,v\rangle_\Gamma = -(\eta,\mathrm{div}_\Gamma v+(v\cdot n)H)_{L^2(\Gamma)}, \quad v\in H^1(\Gamma)^3.
\end{align}
Also, the surface divergence of $v\in L^2(\Gamma)^3$ is given by
\begin{align} \label{E:Sdiv_Hin}
  \langle\mathrm{div}_\Gamma v,\eta\rangle_\Gamma = -(v,\nabla_\Gamma\eta+\eta Hn)_{L^2(\Gamma)}, \quad \eta\in H^1(\Gamma).
\end{align}
Let $H^{-1}(\Gamma,T\Gamma)$ be the dual space of $H^1(\Gamma,T\Gamma)$ and $[\cdot,\cdot]_{T\Gamma}$ the duality product between $H^{-1}(\Gamma,T\Gamma)$ and $H^1(\Gamma,T\Gamma)$.
As in \eqref{E:L2_Hin}, we set
\begin{align*}
  [f,v]_{T\Gamma} := (f,v)_{L^2(\Gamma)}, \quad f\in L^2(\Gamma,T\Gamma),\, v\in H^1(\Gamma,T\Gamma)
\end{align*}
to consider $L^2(\Gamma,T\Gamma)$ as a subspace of $H^{-1}(\Gamma,T\Gamma)$.
Let us show that $H^{-1}(\Gamma,T\Gamma)$ is homeomorphic to a quotient space of $H^{-1}(\Gamma)^3$.

\begin{lemma} \label{L:HinT_Homeo}
  For $f\in H^{-1}(\Gamma)^3$ we define an equivalence class
  \begin{align*}
    [f] := \{\tilde{f} \in H^{-1}(\Gamma)^3 \mid \text{$Pf=P\tilde{f}$ in $H^{-1}(\Gamma)^3$}\}.
  \end{align*}
  Then the quotient space $\mathcal{Q}:=\{[f] \mid f\in H^{-1}(\Gamma)^3\}$ is homeomorphic to $H^{-1}(\Gamma,T\Gamma)$.
\end{lemma}

Note that $\mathcal{Q}$ is a Banach space equipped with norm
\begin{align*}
  \|[f]\|_{\mathcal{Q}}:=\inf_{\tilde{f}\in[f]}\|\tilde{f}\|_{H^{-1}(\Gamma)}.
\end{align*}
For details, we refer to \cite{Ru91}.

\begin{proof}
  Let $f_1,f_2\in H^{-1}(\Gamma)^3$.
  If $Pf_1=Pf_2$ in $H^{-1}(\Gamma)^3$, then
  \begin{align*}
    \langle f_1,v\rangle_\Gamma = \langle f_1,Pv\rangle_\Gamma = \langle Pf_1,v\rangle_\Gamma = \langle Pf_2,v\rangle_\Gamma = \langle f_2,Pv\rangle_\Gamma = \langle f_2,v\rangle_\Gamma
  \end{align*}
  for all $v\in H^1(\Gamma,T\Gamma)$ by $Pv=v$ and $P^T=P$ on $\Gamma$ and \eqref{E:Def_Av_Hin}.
  Hence we can define a linear operator $L$ from $\mathcal{Q}$ to $H^{-1}(\Gamma,T\Gamma)$ by
  \begin{align*}
    [L[f],v]_{T\Gamma} := \langle\tilde{f},v\rangle_\Gamma, \quad [f]\in\mathcal{Q},\, v\in H^1(\Gamma,T\Gamma),
  \end{align*}
  where $\tilde{f}$ is an arbitrary element of $[f]$.
  By this definition we also have
  \begin{align*}
    \|L[f]\|_{H^{-1}(\Gamma,T\Gamma)} \leq \inf_{\tilde{f}\in[f]}\|\tilde{f}\|_{H^{-1}(\Gamma)} = \|[f]\|_{\mathcal{Q}}.
  \end{align*}
  Hence $L$ is bounded.
  Moreover, if $L[f_1]=L[f_2]$ in $H^{-1}(\Gamma,T\Gamma)$, then
  \begin{align*}
    \langle Pf_1,v\rangle_\Gamma &= \langle f_1,Pv\rangle_\Gamma = [L[f_1],Pv]_{T\Gamma} \\
    &= [L[f_2],Pv]_{T\Gamma} = \langle f_2,Pv\rangle_\Gamma = \langle Pf_2,v\rangle_\Gamma
  \end{align*}
  for all $v\in H^1(\Gamma)^3$ and thus $Pf_1=Pf_2$ in $H^{-1}(\Gamma)^3$, which means that $[f_1]=[f_2]$ and $L$ is injective.
  To show its surjectivity, let $F\in H^{-1}(\Gamma,T\Gamma)$.
  Since $H^1(\Gamma,T\Gamma)$ is a Hilbert space equipped with inner product of $H^1(\Gamma)^3$, there exists a tangential vector field $v_F\in H^1(\Gamma,T\Gamma)$ such that
  \begin{align*}
    [F,v]_{T\Gamma} = (v_F,v)_{H^1(\Gamma)} = \sum_{i,j=1}^3\{(v_F^j,v^j)_{L^2(\Gamma)}+(\underline{D}_iv_F^j,\underline{D}_iv^j)_{L^2(\Gamma)}\}
  \end{align*}
  for all $v\in H^1(\Gamma,T\Gamma)$ by the Riesz representation theorem, where $a^j$ stands for the $j$-th component of a vector $a\in\mathbb{R}^3$.
  Then setting
  \begin{align*}
    f_F := v_F-\sum_{i=1}^3\underline{D}_i^2v_F\in H^{-1}(\Gamma)^3
  \end{align*}
  we observe by \eqref{E:L2_Hin}, \eqref{E:Def_TD_Hin}, and $n\cdot\nabla_\Gamma v_F^j=0$ on $\Gamma$ for $j=1,2,3$ that
  \begin{align*}
    [F,v]_{T\Gamma} &= \sum_{i,j=1}^3\{(v_F^j,v^j)_{L^2(\Gamma)}+(\underline{D}_iv_F^j,\underline{D}_iv^j)_{L^2(\Gamma)}\} \\
    &= \sum_{i,j=1}^3\langle v_F^j-\underline{D}_i^2v_F^j,v^j\rangle_\Gamma-\sum_{i,j=1}^3(\underline{D}_iv_F^j,v^jHn_i)_{L^2(\Gamma)}  \\
    &= \langle f_F,v\rangle_\Gamma-\sum_{j=1}^3(n\cdot\nabla_\Gamma v_F^j,v_jH)_{L^2(\Gamma)} = [L[f_F],v]_{T\Gamma}
  \end{align*}
  for all $v\in H^1(\Gamma,T\Gamma)$.
  Thus $F=L[f_F]$ in $H^{-1}(\Gamma,T\Gamma)$ and $L\colon\mathcal{Q}\to H^{-1}(\Gamma,T\Gamma)$ is a bounded, injective, and surjective linear operator.
  Since its inverse is also bounded by the open mapping theorem, $\mathcal{Q}$ is homeomorphic to $H^{-1}(\Gamma,T\Gamma)$.
\end{proof}

In what follows, we identify the equivalence class $[f]$ for $f\in H^{-1}(\Gamma)^3$ with the functional $L[f]\in H^{-1}(\Gamma,T\Gamma)$ given in the proof of Lemma \ref{L:HinT_Homeo}.
We further identify $[f]$ with its representative $Pf$ and write
\begin{align*}
  [Pf,v]_{T\Gamma} = \langle f,v\rangle_\Gamma, \quad v\in H^1(\Gamma,T\Gamma)
\end{align*}
to consider $Pf\in H^{-1}(\Gamma)^3$ as an element of $H^{-1}(\Gamma,T\Gamma)$.
When $Pf=f$ in $H^{-1}(\Gamma)^3$, we take $f$ as a representative of $[f]$ instead of $Pf$.
For example, if $\eta\in L^2(\Gamma)$, then
\begin{align*}
  \langle\nabla_\Gamma\eta,v\rangle_\Gamma &= -(\eta,\mathrm{div}_\Gamma v+(v\cdot n)H)_{L^2(\Gamma)} = -\bigl(\eta,\mathrm{div}_\Gamma(Pv)\bigr)_{L^2(\Gamma)} \\
  &= \langle P\nabla_\Gamma\eta,v\rangle_\Gamma
\end{align*}
for all $v\in H^1(\Gamma)^3$ and thus $P\nabla_\Gamma\eta=\nabla_\Gamma\eta$ in $H^{-1}(\Gamma)^3$.
In this case we have
\begin{align} \label{E:TGr_HinT}
  [\nabla_\Gamma\eta,v]_{T\Gamma} = -(\eta,\mathrm{div}_\Gamma v)_{L^2(\Gamma)}, \quad \eta\in L^2(\Gamma),\,v\in H^1(\Gamma,T\Gamma).
\end{align}
For $\eta\in W^{1,\infty}(\Gamma)$ and $f\in H^{-1}(\Gamma,T\Gamma)$ we can define $\eta f\in H^{-1}(\Gamma,T\Gamma)$ by
\begin{align} \label{E:Def_Mul_HinT}
  [\eta f,v]_{T\Gamma} := [f,\eta v]_{T\Gamma}, \quad v\in H^1(\Gamma,T\Gamma)
\end{align}
since $\eta v\in H^1(\Gamma,T\Gamma)$.
In Section \ref{S:WSol} we give the characterization of the annihilators in $H^{-1}(\Gamma)^3$ and $H^{-1}(\Gamma,T\Gamma)$ of solenoidal spaces on $\Gamma$.

Since $\Gamma$ is not of class $C^\infty$, the space $C^\infty(\Gamma)$ does not make sense and we cannot consider distributions on $\Gamma$.
To consider the time derivative of functions with values in function spaces on $\Gamma$ we introduce the notion of distributions with values in a Banach space (see \cites{LiMa72,So01,Te79} for details).
For $T>0$ let $C_c^\infty(0,T)$ be the space of all smooth and compactly supported functions on $(0,T)$.
We define $\mathcal{D}'(0,T;\mathcal{X})$ as the space of all continuous linear operators from $C_c^\infty(0,T)$ (equipped with locally convex topology described in \cite{Ru91}*{Definition 6.3}) into a Banach space $\mathcal{X}$.
For $p\in[1,\infty]$ we consider
\begin{align*}
  L^p(0,T;\mathcal{X}) \subset \mathcal{D}'(0,T;\mathcal{X})
\end{align*}
by identifying $f\in L^p(0,T;\mathcal{X})$ with a continuous linear operator
\begin{align*}
  \hat{f}(\varphi) := \int_0^T\varphi(t)f(t)\,dt \in \mathcal{X}, \quad \varphi\in C_c^\infty(0,T).
\end{align*}
Let $f\in\mathcal{D}'(0,T;\mathcal{X})$.
We define the time derivative $\partial_tf\in\mathcal{D}'(0,T;\mathcal{X})$ of $f$ by
\begin{align*}
  \partial_tf(\varphi) := -f(\partial_t\varphi) \in \mathcal{X}, \quad \varphi \in C_c^\infty(0,T).
\end{align*}
Note that, if $f\in L^p(0,T;\mathcal{X})$ with $p\in[1,\infty]$, then
\begin{align} \label{E:Dt_Dist_L2}
  \partial_tf(\varphi) = -f(\partial_t\varphi) = -\int_0^T\partial_t\varphi(t)f(t)\,dt \in \mathcal{X}, \quad \varphi\in C_c^\infty(0,T).
\end{align}
Let $p\in[1,\infty]$.
For $f\in L^p(0,T;\mathcal{X})$ if there exists $\xi\in L^p(0,T;\mathcal{X})$ such that
\begin{align*}
  \partial_tf(\varphi) = \xi(\varphi), \quad\text{i.e.}\quad -\int_0^T\partial_t\varphi(t)f(t)\,dt = \int_0^T\varphi(t)\xi(t)\,dt \quad\text{in $\mathcal{X}$}
\end{align*}
for all $\varphi\in C_c^\infty(0,T)$, then we write $\partial_tf=\xi\in L^p(0,T;\mathcal{X})$ and define
\begin{align*}
  W^{1,p}(0,T;\mathcal{X}) := \{f\in L^p(0,T;\mathcal{X}) \mid \partial_tf \in L^p(0,T;\mathcal{X})\}.
\end{align*}
We also write $H^1(0,T;\mathcal{X}):=W^{1,2}(0,T;\mathcal{X})$.
Note that
\begin{align*}
  W^{1,p}(0,T;\mathcal{X}) \subset W^{1,1}(0,T;\mathcal{X}) \subset C([0,T];\mathcal{X}), \quad p\in[1,\infty].
\end{align*}
When $q\in L^2(0,T;L^2(\Gamma))$ we have
\begin{align*}
  \nabla_\Gamma q \in L^2(0,T;H^{-1}(\Gamma,T\Gamma)), \quad \partial_t(\nabla_\Gamma q) \in \mathcal{D}'(0,T;H^{-1}(\Gamma,T\Gamma)).
\end{align*}
Also, for $\varphi\in C_c^\infty(0,T)$ we can consider the weak tangential gradient of
\begin{align*}
  \partial_tq(\varphi) = -\int_0^T\partial_t\varphi(t)q(t)\,dt \in L^2(\Gamma)
\end{align*}
in $H^{-1}(\Gamma,T\Gamma)$ since $\partial_tq\in\mathcal{D}'(0,T;L^2(\Gamma))$.
Let us show that the time derivative commutes with the weak tangential gradient in an appropriate sense.

\begin{lemma} \label{L:Dt_TGr_Com}
  Let $q\in L^2(0,T;L^2(\Gamma))$.
  Then
  \begin{align*}
    \nabla_\Gamma[\partial_tq(\varphi)] = [\partial_t(\nabla_\Gamma q)](\varphi) \quad\text{in}\quad H^{-1}(\Gamma,T\Gamma)
  \end{align*}
  for all $\varphi\in C_c^\infty(0,T)$.
\end{lemma}

\begin{proof}
  For all $v\in H^1(\Gamma,T\Gamma)$ we have
  \begin{align*}
    [\nabla_\Gamma[\partial_tq(\varphi)],v]_{T\Gamma} &= (q(\partial_t\varphi),\mathrm{div}_\Gamma v)_{L^2(\Gamma)} = \int_0^T\partial_t\varphi(t)(q(t),\mathrm{div}_\Gamma v)_{L^2(\Gamma)}\,dt \\
    &= -\int_0^T\partial_t\varphi(t)[\nabla_\Gamma q(t),v]_{T\Gamma}\,dt = \Bigl[\,[\partial_t(\nabla_\Gamma q)](\varphi),v\Bigr]_{T\Gamma}
  \end{align*}
  by \eqref{E:TGr_HinT} and \eqref{E:Dt_Dist_L2}.
  Hence the claim is valid.
\end{proof}

Let $q\in L^2(0,T;L^2(\Gamma))$.
Based on Lemma \ref{L:Dt_TGr_Com}, we set
\begin{align*}
  [\nabla_\Gamma(\partial_tq)](\varphi) := \nabla_\Gamma[(\partial_tq)(\varphi)] = [\partial_t(\nabla_\Gamma q)](\varphi)\in H^{-1}(\Gamma,T\Gamma)
\end{align*}
for $\varphi\in C_c^\infty(0,T)$ to define $\nabla_\Gamma(\partial_tq)\in \mathcal{D}'(0,T;H^{-1}(\Gamma,T\Gamma))$ and consider
\begin{align} \label{E:TGrDt_DHin}
  \nabla_\Gamma(\partial_tq) = \partial_t(\nabla_\Gamma q) \quad\text{in}\quad \mathcal{D}'(0,T;H^{-1}(\Gamma,T\Gamma)).
\end{align}
We use this relation in construction of an associated pressure in the limit equations (see Lemma \ref{L:LW_Pres}).

\subsection{Curved thin domain} \label{SS:Pre_CTD}
From now on, except for Section \ref{S:WSol}, we assume that $\Gamma$ is of class $C^5$ and $g_0,g_1\in C^4(\Gamma)$, and $g:=g_1-g_0$ satisfies \eqref{E:G_Inf}.
For $\varepsilon\in(0,1]$ we define a curved thin domain $\Omega_\varepsilon$ in $\mathbb{R}^3$ by \eqref{E:Def_CTD}, i.e.
\begin{align*}
  \Omega_\varepsilon := \{y+rn(y) \mid y\in\Gamma,\,\varepsilon g_0(y) < r < \varepsilon g_1(y)\}.
\end{align*}
The boundary of $\Omega_\varepsilon$ is denoted by $\Gamma_\varepsilon:=\Gamma_\varepsilon^0\cup\Gamma_\varepsilon^1$, where $\Gamma_\varepsilon^0$ and $\Gamma_\varepsilon^1$ are the inner and outer boundaries given by
\begin{align*}
  \Gamma_\varepsilon^i := \{y+\varepsilon g_i(y)n(y) \mid y\in\Gamma\}, \quad i=0,1.
\end{align*}
Note that $\Gamma_\varepsilon$ is of class $C^4$ by the regularity of $\Gamma$, $g_0$, and $g_1$.
Since $g_0$ and $g_1$ are bounded on $\Gamma$, there exists $\tilde{\varepsilon}\in(0,1]$ such that $\tilde{\varepsilon}|g_i|<\delta$ on $\Gamma$ for $i=0,1$, where $\delta>0$ is the radius of the tubular neighborhood $N$ of $\Gamma$ given in Section \ref{SS:Pre_Surf}.
We assume $\tilde{\varepsilon}=1$ by replacing $g_i$ with $\tilde{\varepsilon}g_i$.
Then $\overline{\Omega}_\varepsilon\subset N$ and we can apply the lemmas given in Section \ref{SS:Pre_Surf} in $\overline{\Omega}_\varepsilon$ for all $\varepsilon\in(0,1]$.

Let $\tau_\varepsilon^i$ and $n_\varepsilon^i$ be vector fields on $\Gamma$ given by
\begin{align}
  \tau_\varepsilon^i(y) &:= \{I_3-\varepsilon g_i(y)W(y)\}^{-1}\nabla_\Gamma g_i(y), \label{E:Def_NB_Aux}\\
  n_\varepsilon^i(y) &:= (-1)^{i+1}\frac{n(y)-\varepsilon\tau_\varepsilon^i(y)}{\sqrt{1+\varepsilon^2|\tau_\varepsilon^i(y)|^2}} \label{E:Def_NB}
\end{align}
for $y\in\Gamma$ and $i=0,1$.
Then $\tau_\varepsilon^i$ is tangential on $\Gamma$ by \eqref{E:P_TGr} and \eqref{E:WReso_P}, and
\begin{align} \label{E:Tau_Bound}
  |\tau_\varepsilon^i(y)| \leq c \quad\text{for all}\quad y\in\Gamma
\end{align}
with a constant $c>0$ independent of $\varepsilon$ by \eqref{E:Wein_Bound} and the $C^4$-regularity of $g_i$ on $\Gamma$.
Let $n_\varepsilon$ be the unit outward normal vector field of $\Gamma_\varepsilon$ and
\begin{align*}
  P_\varepsilon(x) := I_3-n_\varepsilon(x)\otimes n_\varepsilon(x), \quad x\in\Gamma_\varepsilon.
\end{align*}
Note that $|P_\varepsilon|=2$ on $\Gamma_\varepsilon$.
It is shown in \cite{Miu_NSCTD_01}*{Lemma 3.9} that
\begin{align} \label{E:Nor_Bo}
  n_\varepsilon(x) = \bar{n}_\varepsilon^i(x), \quad x\in\Gamma_\varepsilon^i,\,i=0,1,
\end{align}
where $\bar{n}_\varepsilon^i=n_\varepsilon^i\circ\pi$ is the constant extension of $n_\varepsilon^i$.
By this equality we easily observe that $n_\varepsilon$ and $P_\varepsilon$ are close to the constant extensions of $n$ and $P$.

\begin{lemma}[{\cite{Miu_NSCTD_01}*{Lemma 3.10}}] \label{L:Comp_Nor}
  For $x\in\Gamma_\varepsilon^i$, $i=0,1$ we have
  \begin{align}
    \left|n_\varepsilon(x)-(-1)^{i+1}\left\{\bar{n}(x)-\varepsilon\overline{\nabla_\Gamma g_i}(x)\right\}\right| &\leq c\varepsilon^2, \label{E:Comp_N} \\
    \left|P_\varepsilon(x)-\overline{P}(x)\right| &\leq c\varepsilon \label{E:Comp_P}
  \end{align}
  with a constant $c>0$ independent of $\varepsilon$.
\end{lemma}

As in Section \ref{SS:Pre_Surf}, for $\varphi\in C^1(\Gamma_\varepsilon)$ we denote by
\begin{align*}
  \nabla_{\Gamma_\varepsilon}\varphi := P_\varepsilon\nabla\tilde{\varphi}, \quad \underline{D}_i^\varepsilon\varphi := \sum_{j=1}^3[P_\varepsilon]_{ij}\partial_j\tilde{\varphi} \quad\text{on}\quad \Gamma_\varepsilon,\, i=1,2,3
\end{align*}
the tangential gradient and the tangential derivatives of $\varphi$, where $\tilde{\varphi}$ is an arbitrary $C^1$-extension of $\varphi$ to an open neighborhood of $\Gamma_\varepsilon$ with $\tilde{\varphi}|_{\Gamma_\varepsilon}=\varphi$.
Also, let
\begin{align*}
  \nabla_{\Gamma_\varepsilon}u :=
  \begin{pmatrix}
    \underline{D}_1^\varepsilon u_1 & \underline{D}_1^\varepsilon u_2 & \underline{D}_1^\varepsilon u_3 \\
    \underline{D}_2^\varepsilon u_1 & \underline{D}_2^\varepsilon u_2 & \underline{D}_2^\varepsilon u_3 \\
    \underline{D}_3^\varepsilon u_1 & \underline{D}_3^\varepsilon u_2 & \underline{D}_3^\varepsilon u_3
  \end{pmatrix} \quad\text{on}\quad \Gamma_\varepsilon
\end{align*}
for $u=(u_1,u_2,u_3)^T\in C^1(\Gamma_\varepsilon)^3$.
Note that
\begin{align*}
  \nabla_{\Gamma_\varepsilon}u = P_\varepsilon\nabla\tilde{u} \quad\text{on}\quad \Gamma_\varepsilon
\end{align*}
for any $C^1$-extension $\tilde{u}$ of $u$ to an open neighborhood of $\Gamma_\varepsilon$ with $\tilde{u}|_{\Gamma_\varepsilon}=u$.
We set
\begin{align*}
  W_\varepsilon := -\nabla_{\Gamma_\varepsilon}n_\varepsilon \quad\text{on}\quad \Gamma_\varepsilon
\end{align*}
and call $W_\varepsilon$ the Weingarten map of $\Gamma_\varepsilon$.

Let us give a change of variables formula for an integral over $\Omega_\varepsilon$.
For functions $\varphi$ on $\Omega_\varepsilon$ and $\eta$ on $\Gamma_\varepsilon^i$, $i=0,1$ we use the notations
\begin{alignat}{2}
  \varphi^\sharp(y,r) &:= \varphi(y+rn(y)), &\quad &y\in\Gamma,\,r\in(\varepsilon g_0(y),\varepsilon g_1(y)), \label{E:Pull_Dom} \\
  \eta_i^\sharp(y) &:= \eta(y+\varepsilon g_i(y)n(y)), &\quad &y\in\Gamma. \label{E:Pull_Bo}
\end{alignat}
We define a function $J=J(y,r)$ for $y\in\Gamma$ and $r\in(-\delta,\delta)$ by
\begin{align} \label{E:Def_Jac}
  J(y,r) := \mathrm{det}[I_3-rW(y)] = \{1-r\kappa_1(y)\}\{1-r\kappa_2(y)\}.
\end{align}
Then it follows from \eqref{E:Curv_Bound} and $\kappa_1,\kappa_2\in C^3(\Gamma)$ that
\begin{align} \label{E:Jac_Bound_03}
  c^{-1} \leq J(y,r) \leq c, \quad |J(y,r)-1| \leq c|r|
\end{align}
for all $y\in\Gamma$ and $r\in(-\delta,\delta)$.
In particular,
\begin{align} \label{E:Jac_Diff_03}
  |J(y,r)-1| \leq c\varepsilon \quad\text{for all}\quad y\in\Gamma,\,r\in[\varepsilon g_0(y),\varepsilon g_1(y)].
\end{align}
For a function $\varphi$ on $\Omega_\varepsilon$ the change of variables formula
\begin{align} \label{E:CoV_Dom}
  \int_{\Omega_\varepsilon}\varphi(x)\,dx = \int_\Gamma\int_{\varepsilon g_0(y)}^{\varepsilon g_1(y)}\varphi(y+rn(y))J(y,r)\,dr\,d\mathcal{H}^2(y)
\end{align}
holds (see e.g. \cite{GiTr01}*{Section 14.6}).
By \eqref{E:Jac_Bound_03} and \eqref{E:CoV_Dom} we observe that
\begin{align} \label{E:CoV_Equiv}
  c^{-1}\|\varphi\|_{L^p(\Omega_\varepsilon)}^p \leq \int_\Gamma\int_{\varepsilon g_0(y)}^{\varepsilon g_1(y)}|\varphi^\sharp(y,r)|^p\,dr\,d\mathcal{H}^2(y) \leq c\|\varphi\|_{L^p(\Omega_\varepsilon)}^p
\end{align}
for $\varphi\in L^p(\Omega_\varepsilon)$, $p\in[1,\infty)$.
By this inequality, \eqref{E:G_Inf}, and \eqref{E:ConDer_Bound} we easily get the following lemma for the constant extension $\bar{\eta}=\eta\circ\pi$ of a function $\eta$ on $\Gamma$.

\begin{lemma}[{\cite{Miu_NSCTD_01}*{Lemma 3.12}}] \label{L:Con_Lp_W1p}
  For $p\in[1,\infty)$ we have $\eta\in L^p(\Gamma)$ if and only if $\bar{\eta}\in L^p(\Omega_\varepsilon)$, and there exists a constant $c>0$ independent of $\varepsilon$ and $\eta$ such that
  \begin{align} \label{E:Con_Lp}
    c^{-1}\varepsilon^{1/p}\|\eta\|_{L^p(\Gamma)} \leq \|\bar{\eta}\|_{L^p(\Omega_\varepsilon)} \leq c\varepsilon^{1/p}\|\eta\|_{L^p(\Gamma)}.
  \end{align}
  Moreover, $\eta\in W^{1,p}(\Gamma)$ if and only if $\bar{\eta}\in W^{1,p}(\Omega_\varepsilon)$.
\end{lemma}

We also have a change of variables formula for an integral over $\Gamma_\varepsilon^i$, $i=0,1$.

\begin{lemma}[{\cite{Miu_NSCTD_01}*{Lemma 3.13}}] \label{L:CoV_Surf}
  For $\varphi\in L^1(\Gamma_\varepsilon^i)$, $i=0,1$ we have
  \begin{align} \label{E:CoV_Surf}
    \int_{\Gamma_\varepsilon^i}\varphi(x)\,d\mathcal{H}^2(x) = \int_\Gamma \varphi_i^\sharp(y)J(y,\varepsilon g_i(y))\sqrt{1+\varepsilon^2|\tau_\varepsilon^i(y)|^2}\,d\mathcal{H}^2(y),
  \end{align}
  where $\tau_\varepsilon^i$ and $\varphi_i^\sharp$ are given by \eqref{E:Def_NB_Aux} and \eqref{E:Pull_Bo}.
  Moreover, if $\varphi\in L^p(\Gamma_\varepsilon^i)$, $p\in[1,\infty)$, then $\varphi_i^\sharp\in L^p(\Gamma)$ and there exists a constant $c>0$ independent of $\varepsilon$ and $\varphi$ such that
  \begin{align} \label{E:Lp_CoV_Surf}
    c^{-1}\|\varphi\|_{L^p(\Gamma_\varepsilon^i)} \leq \|\varphi_i^\sharp\|_{L^p(\Gamma)} \leq c\|\varphi\|_{L^p(\Gamma_\varepsilon^i)}.
  \end{align}
\end{lemma}

\section{Fundamental inequalities} \label{S:Fund}
This section provides fundamental inequalities for functions on $\Gamma$ and $\Omega_\varepsilon$.

\subsection{Inequalities on a closed surface} \label{SS:Ineq_Surf}
Let us give two basic inequalities on $\Gamma$.

\begin{lemma}[{\cite{Miu_NSCTD_02}*{Lemma 4.1}}] \label{L:La_Surf}
  There exists a constant $c>0$ such that
  \begin{align} \label{E:La_Surf}
    \|\eta\|_{L^4(\Gamma)} \leq c\|\eta\|_{L^2(\Gamma)}^{1/2}\|\nabla_\Gamma\eta\|_{L^2(\Gamma)}^{1/2}
  \end{align}
  for all $\eta\in H^1(\Gamma)$.
\end{lemma}

The inequality \eqref{E:La_Surf} is Ladyzhenskaya's inequality on the two-dimensional closed surface $\Gamma$.
Next we show Korn's inequality for a tangential vector field on $\Gamma$.

\begin{lemma} \label{L:Korn_STG}
  There exists a constant $c>0$ such that
  \begin{align} \label{E:Korn_STG}
    \|\nabla_\Gamma v\|_{L^2(\Gamma)}^2 \leq c\left(\|D_\Gamma(v)\|_{L^2(\Gamma)}^2+\|v\|_{L^2(\Gamma)}^2\right)
  \end{align}
  for all $v\in H^1(\Gamma,T\Gamma)$.
  Here $D_\Gamma(v)$ is the surface strain rate tensor given by \eqref{E:Def_SSR}.
\end{lemma}

\begin{proof}
  Since $C^2(\Gamma,T\Gamma)$ is dense in $H^1(\Gamma,T\Gamma)$ by Lemma \ref{L:Wmp_Tan_Appr}, it is sufficient to show \eqref{E:Korn_STG} for $v\in C^2(\Gamma,T\Gamma)$ by a density argument.
  First we derive
  \begin{align} \label{Pf_KoST:First}
    \|\nabla_\Gamma v\|_{L^2(\Gamma)}^2 \leq 2\|(\nabla_\Gamma v)_S\|_{L^2(\Gamma)}^2+c\|v\|_{L^2(\Gamma)}^2.
  \end{align}
  Since $2|(\nabla_\Gamma v)_S|^2=|\nabla_\Gamma v|^2+\nabla_\Gamma v:(\nabla_\Gamma v)^T$ on $\Gamma$, we have
  \begin{align} \label{Pf_KoST:TGr_Eq}
    \|\nabla_\Gamma v\|_{L^2(\Gamma)}^2 = 2\|(\nabla_\Gamma v)_S\|_{L^2(\Gamma)}^2-\int_\Gamma\nabla_\Gamma v:(\nabla_\Gamma v)^T\,d\mathcal{H}^2.
  \end{align}
  To the last term we apply \eqref{E:InTr_IbP} with $X=Y=v\in C^2(\Gamma,T\Gamma)$ to get
  \begin{align*}
    -\int_\Gamma\nabla_\Gamma v:(\nabla_\Gamma v)^T\,d\mathcal{H}^2 = \int_\Gamma v\cdot\{\nabla_\Gamma(\mathrm{div}_\Gamma v)+(HW-W^2)v\}\,d\mathcal{H}^2.
  \end{align*}
  Moreover, it follows from \eqref{E:IbP_TD} and $v\cdot n=0$ on $\Gamma$ that
  \begin{align*}
    \int_\Gamma v\cdot\nabla_\Gamma(\mathrm{div}_\Gamma v)\,d\mathcal{H}^2 &= -\int_\Gamma\{\mathrm{div}_\Gamma v+(v\cdot n)H\}(\mathrm{div}_\Gamma v)\,d\mathcal{H}^2 \\
    &= -\|\mathrm{div}_\Gamma v\|_{L^2(\Gamma)}^2 \leq 0.
  \end{align*}
  By the above relations and the boundedness of $W$ and $H$ on $\Gamma$ we have
  \begin{align*}
    -\int_\Gamma\nabla_\Gamma v:(\nabla_\Gamma v)^T\,d\mathcal{H}^2 \leq \int_\Gamma|v|(|HWv|+|W^2v|)\,d\mathcal{H}^2 \leq c\|v\|_{L^2(\Gamma)}^2.
  \end{align*}
  Applying this inequality to \eqref{Pf_KoST:TGr_Eq} we obtain \eqref{Pf_KoST:First}.

  Next we observe by \eqref{E:Grad_W} and $P^T=P$ on $\Gamma$ that
  \begin{align*}
    (\nabla_\Gamma v)_S &= P(\nabla_\Gamma v)_SP+\frac{1}{2}\{(Wv)\otimes n+n\otimes(Wv)\} \\
    &= D_\Gamma(v)+\frac{1}{2}\{(Wv)\otimes n+n\otimes(Wv)\}
  \end{align*}
  on $\Gamma$.
  Hence by the boundedness of $W$ and $|n|=1$ on $\Gamma$ we get
  \begin{align*}
    |(\nabla_\Gamma v)_S| \leq |D_\Gamma(v)|+|Wv||n| \leq |D_\Gamma(v)|+c|v| \quad\text{on}\quad \Gamma
  \end{align*}
  and thus
  \begin{align*}
    \|(\nabla_\Gamma v)_S\|_{L^2(\Gamma)}^2 \leq c\left(\|D_\Gamma(v)\|_{L^2(\Gamma)}^2+\|v\|_{L^2(\Gamma)}^2\right).
  \end{align*}
  We apply this inequality to \eqref{Pf_KoST:First} to conclude that \eqref{E:Korn_STG} is valid.
\end{proof}

\subsection{Consequences of the boundary conditions} \label{SS:Fund_Bd}
In this subsection we present estimates for a vector field $u$ on $\Omega_\varepsilon$ satisfying the impermeable boundary condition
\begin{align} \label{E:Bo_Imp}
  u\cdot n_\varepsilon = 0 \quad\text{on}\quad \Gamma_\varepsilon
\end{align}
or the slip boundary conditions
\begin{align} \label{E:Bo_Slip}
  u\cdot n_\varepsilon = 0, \quad 2\nu P_\varepsilon D(u)n_\varepsilon+\gamma_\varepsilon u = 0 \quad\text{on}\quad \Gamma_\varepsilon.
\end{align}
Here $\nu>0$ is the viscosity coefficient independent of $\varepsilon$ and $\gamma_\varepsilon\geq0$ is the friction coefficient on $\Gamma_\varepsilon$ given by \eqref{E:Def_Fric}.
Moreover, we write
\begin{align*}
  D(u) := (\nabla u)_S = \frac{\nabla u+(\nabla u)^T}{2}
\end{align*}
for the strain rate tensor.
First we give an inequality related to \eqref{E:Bo_Imp}.

\begin{lemma}[{\cite{Miu_NSCTD_02}*{Lemma 4.5}}] \label{L:Poin_Nor}
  There exists a constant $c>0$ independent of $\varepsilon$ such that
  \begin{align} \label{E:Poin_Nor}
    \|u\cdot\bar{n}\|_{L^p(\Omega_\varepsilon)} \leq c\varepsilon\|u\|_{W^{1,p}(\Omega_\varepsilon)}
  \end{align}
  for all $u\in W^{1,p}(\Omega_\varepsilon)^3$ with $p\in[1,\infty)$ satisfying \eqref{E:Bo_Imp} on $\Gamma_\varepsilon^0$ or on $\Gamma_\varepsilon^1$.
\end{lemma}

Next we provide two estimates for a vector field satisfying \eqref{E:Bo_Slip}.
For a function $\varphi$ on $\Omega_\varepsilon$ and $x\in\Omega_\varepsilon$ we define the derivative of $\varphi$ in the normal direction of $\Gamma$ by
\begin{align*}
  \partial_n\varphi(x) := (\bar{n}(x)\cdot\nabla)\varphi(x) = \frac{d}{dr}\bigl(\varphi(y+rn(y))\bigr)\Big|_{r=d(x)} \quad (y=\pi(x)\in\Gamma).
\end{align*}
Note that the constant extension $\bar{\eta}=\eta\circ\pi$ of $\eta\in C^1(\Gamma)$ satisfies
\begin{align} \label{E:NorDer_Con}
  \partial_n\bar{\eta}(x) = (\bar{n}(x)\cdot\nabla)\bar{\eta}(x) = 0, \quad x\in \Omega_\varepsilon.
\end{align}

\begin{lemma}[{\cite{Miu_NSCTD_02}*{Lemma 4.7}}] \label{L:PDnU_WU}
  Let $p\in[1,\infty)$.
  Suppose that the inequalities \eqref{E:Fric_Upper} are valid.
  Then there exists a constant $c>0$ independent of $\varepsilon$ such that
  \begin{align} \label{E:PDnU_WU}
    \left\|\overline{P}\partial_nu+\overline{W}u\right\|_{L^p(\Omega_\varepsilon)} \leq c\varepsilon\|u\|_{W^{2,p}(\Omega_\varepsilon)}
  \end{align}
  for all $u\in W^{2,p}(\Omega_\varepsilon)^3$ satisfying \eqref{E:Bo_Slip} on $\Gamma_\varepsilon^0$ or on $\Gamma_\varepsilon^1$.
\end{lemma}

\begin{lemma} \label{L:Poin_Str}
  Let $p\in[1,\infty)$.
  Suppose that the inequalities \eqref{E:Fric_Upper} are valid.
  Then there exists a constant $c>0$ independent of $\varepsilon$ such that
  \begin{align} \label{E:Poin_Str}
    \left\|\overline{P}D(u)\bar{n}\right\|_{L^p(\Omega_\varepsilon)} \leq c\varepsilon\|u\|_{W^{2,p}(\Omega_\varepsilon)}
  \end{align}
  for all $u\in W^{2,p}(\Omega_\varepsilon)^3$ satisfying \eqref{E:Bo_Slip} on $\Gamma_\varepsilon^0$ or on $\Gamma_\varepsilon^1$.
\end{lemma}

To prove \eqref{E:Poin_Str} we use the following Poincar\'{e} and trace type inequalities.

\begin{lemma}[{\cite{Miu_NSCTD_01}*{Lemma 4.1}}] \label{L:Poincare}
  There exists a constant $c>0$ independent of $\varepsilon$ such that
  \begin{align}
    \|\varphi\|_{L^p(\Omega_\varepsilon)} &\leq c\left(\varepsilon^{1/p}\|\varphi\|_{L^p(\Gamma_\varepsilon^i)}+\varepsilon\|\partial_n\varphi\|_{L^p(\Omega_\varepsilon)}\right), \label{E:Poin_Dom} \\
    \|\varphi\|_{L^p(\Gamma_\varepsilon^i)} &\leq c\left(\varepsilon^{-1/p}\|\varphi\|_{L^p(\Omega_\varepsilon)}+\varepsilon^{1-1/p}\|\partial_n\varphi\|_{L^p(\Omega_\varepsilon)}\right) \label{E:Poin_Bo}
  \end{align}
  for all $\varphi\in W^{1,p}(\Omega_\varepsilon)$ with $p\in[1,\infty)$ and $i=0,1$.
\end{lemma}

\begin{proof}[Proof of Lemma \ref{L:Poin_Str}]
  Let $u\in W^{2,p}(\Omega_\varepsilon)^3$ satisfy \eqref{E:Bo_Slip} on $\Gamma_\varepsilon^i$ with $i=0$ or $i=1$.
  By \eqref{E:NorDer_Con} with $\eta=n,P$ and $|n|=1$, $|P|=2$ on $\Gamma$ we have
  \begin{align*}
    \left|\partial_n\Bigl[\overline{P}D(u)\bar{n}\Bigr]\right| \leq c|\nabla^2u| \quad\text{in}\quad \Omega_\varepsilon.
  \end{align*}
  We apply \eqref{E:Poin_Dom} to $\overline{P}D(u)\bar{n}$ and use the above inequality to get
  \begin{align} \label{Pf_PSt:L2_Dom}
    \left\|\overline{P}D(u)\bar{n}\right\|_{L^p(\Omega_\varepsilon)} \leq c\left(\varepsilon^{1/p}\left\|\overline{P}D(u)\bar{n}\right\|_{L^p(\Gamma_\varepsilon^i)}+\varepsilon\|u\|_{W^{2,p}(\Omega_\varepsilon)}\right).
  \end{align}
  Moreover, since $u$ satisfies \eqref{E:Bo_Slip} on $\Gamma_\varepsilon^i$, we have
  \begin{align*}
    \overline{P}D(u)\bar{n} &= (-1)^{i+1}P_\varepsilon D(u)n_\varepsilon+P_\varepsilon D(u)\{\bar{n}-(-1)^{i+1}n_\varepsilon\}+\Bigl(\overline{P}-P_\varepsilon\Bigr)D(u)\bar{n} \\
    &= (-1)^i\frac{\gamma_\varepsilon}{2\nu}u+P_\varepsilon D(u)\{\bar{n}-(-1)^{i+1}n_\varepsilon\}+\Bigl(\overline{P}-P_\varepsilon\Bigr)D(u)\bar{n}
  \end{align*}
  on $\Gamma_\varepsilon^i$.
  Using \eqref{E:Fric_Upper}, \eqref{E:Comp_N}, \eqref{E:Comp_P}, and $|P_\varepsilon|=2$ on $\Gamma_\varepsilon^i$ to the last line we get
  \begin{align*}
    \left|\overline{P}D(u)\bar{n}\right| \leq c\varepsilon(|u|+|\nabla u|) \quad\text{on}\quad \Gamma_\varepsilon^i.
  \end{align*}
  From this inequality and \eqref{E:Poin_Bo} we deduce that
  \begin{align*}
    \left\|\overline{P}D(u)\bar{n}\right\|_{L^p(\Gamma_\varepsilon^i)} \leq c\varepsilon\left(\|u\|_{L^p(\Gamma_\varepsilon^i)}+\|\nabla u\|_{L^p(\Gamma_\varepsilon^i)}\right) \leq c\varepsilon^{1-1/p}\|u\|_{W^{2,p}(\Omega_\varepsilon)}.
  \end{align*}
  We apply this inequality to \eqref{Pf_PSt:L2_Dom} to obtain \eqref{E:Poin_Str}.
\end{proof}

\subsection{Impermeable extension of surface vector fields} \label{SS:Fund_IE}
When we deal with a vector field on $\Gamma$ in the analysis of \eqref{E:NS_CTD}, it is convenient to extend it to a vector field on $\Omega_\varepsilon$ satisfying the impermeable boundary condition \eqref{E:Bo_Imp}.
Let $\tau_\varepsilon^0$ and $\tau_\varepsilon^1$ be the vector fields on $\Gamma$ given by \eqref{E:Def_NB_Aux}.
For $x\in N$ we set
\begin{align} \label{E:Def_ExAux}
  \Psi_\varepsilon(x) := \frac{1}{\bar{g}(x)}\bigl\{\bigl(d(x)-\varepsilon\bar{g}_0(x)\bigr)\bar{\tau}_\varepsilon^1(x)+\bigl(\varepsilon\bar{g}_1(x)-d(x)\bigr)\bar{\tau}_\varepsilon^0(x)\bigr\},
\end{align}
where $\bar{\eta}=\eta\circ\pi$ is the constant extension of a function $\eta$ on $\Gamma$, and define
\begin{align} \label{E:Def_ExImp}
  E_\varepsilon v(x) := \bar{v}(x)+\{\bar{v}(x)\cdot\Psi_\varepsilon(x)\}\bar{n}(x), \quad x\in N
\end{align}
for a tangential vector field $v$ on $\Gamma$.
Then $E_\varepsilon v$ satisfies \eqref{E:Bo_Imp}.

\begin{lemma} \label{L:ExImp_Bo}
  For all $v\in C(\Gamma,T\Gamma)$ we have $E_\varepsilon v\cdot n_\varepsilon=0$ on $\Gamma_\varepsilon$.
\end{lemma}

\begin{proof}
  For $i=0,1$ it follows from \eqref{E:Def_NB}, \eqref{E:Nor_Bo}, and $v\cdot n=0$ on $\Gamma$ that
  \begin{align*}
    \bar{v}\cdot n_\varepsilon = (-1)^i\frac{\varepsilon\bar{v}\cdot\bar{\tau}_\varepsilon^i}{\sqrt{1+\varepsilon^2|\bar{\tau}_\varepsilon^i|^2}}, \quad \bar{n}\cdot n_\varepsilon = \frac{(-1)^{i+1}}{\sqrt{1+\varepsilon^2|\bar{\tau}_\varepsilon^i|^2}} \quad\text{on}\quad \Gamma_\varepsilon^i.
  \end{align*}
  By these equalities and $\Psi_\varepsilon=\varepsilon\bar{\tau}_\varepsilon^i$ on $\Gamma_\varepsilon^i$ we obtain $E_\varepsilon v\cdot n_\varepsilon=0$ on $\Gamma_\varepsilon^i$.
\end{proof}

Moreover, $E_\varepsilon v$ belongs to $W^{m,p}(\Omega_\varepsilon)^3$ if $v\in W^{m,p}(\Gamma,T\Gamma)$.

\begin{lemma}[{\cite{Miu_NSCTD_02}*{Lemma 4.10}}] \label{L:ExImp_Wmp}
  There exists a constant $c>0$ independent of $\varepsilon$ such that
  \begin{align} \label{E:ExImp_Wmp}
    \|E_\varepsilon v\|_{W^{m,p}(\Omega_\varepsilon)} \leq c\varepsilon^{1/p}\|v\|_{W^{m,p}(\Gamma)}
  \end{align}
  for all $v\in W^{m,p}(\Gamma,T\Gamma)$ with $m=0,1,2$ and $p\in[1,\infty)$.
\end{lemma}

Let us estimate the difference between $E_\varepsilon v$ and $\bar{v}$.

\begin{lemma} \label{L:ExImp_DiLp}
  There exists a constant $c>0$ independent of $\varepsilon$ such that
  \begin{align} \label{E:ExImp_DiLp}
    \|E_\varepsilon v-\bar{v}\|_{L^p(\Omega_\varepsilon)} \leq c\varepsilon^{1+1/p}\|v\|_{L^p(\Gamma)}
  \end{align}
  for all $v\in L^p(\Gamma,T\Gamma)$ with $p\in[1,\infty)$.
\end{lemma}

\begin{proof}
  From \eqref{E:Tau_Bound} and
  \begin{align*}
    0 \leq d(x)-\varepsilon\bar{g}_0(x) \leq \varepsilon\bar{g}(x), \quad 0 \leq \varepsilon\bar{g}_1(x)-d(x) \leq \varepsilon\bar{g}(x), \quad x\in\Omega_\varepsilon
  \end{align*}
  we deduce that $|\Psi_\varepsilon|\leq c\varepsilon$ in $\Omega_\varepsilon$ and thus
  \begin{align*}
    |E_\varepsilon v-\bar{v}| = |(\bar{v}\cdot\Psi_\varepsilon)\bar{n}| \leq c\varepsilon|\bar{v}| \quad\text{in}\quad \Omega_\varepsilon.
  \end{align*}
  By this inequality and \eqref{E:Con_Lp} we obtain
  \begin{align*}
    \|E_\varepsilon v-\bar{v}\|_{L^p(\Omega_\varepsilon)} \leq c\varepsilon\|\bar{v}\|_{L^p(\Omega_\varepsilon)} \leq c\varepsilon^{1+1/p}\|v\|_{L^p(\Gamma)}.
  \end{align*}
  Thus \eqref{E:ExImp_DiLp} is valid.
\end{proof}

We also derive estimates for the gradient and divergence of $E_\varepsilon v$.

\begin{lemma} \label{L:ExImp_DiGr}
  There exists a constant $c>0$ independent of $\varepsilon$ such that
  \begin{align} \label{E:ExImp_DiGr}
    \left\|\nabla E_\varepsilon v-\overline{F(v)}\right\|_{L^p(\Omega_\varepsilon)} \leq c\varepsilon^{1+1/p}\|v\|_{W^{1,p}(\Gamma)}
  \end{align}
  for all $v\in W^{1,p}(\Gamma,T\Gamma)$ with $p\in[1,\infty)$, where
  \begin{align} \label{E:Def_Fv}
    F(v) := \nabla_\Gamma v+\frac{1}{g}(v\cdot\nabla_\Gamma g)Q \quad\text{on}\quad \Gamma.
  \end{align}
\end{lemma}

\begin{proof}
  Since $p\in[1,\infty)$, $C^1(\Gamma,T\Gamma)$ is dense in $W^{1,p}(\Gamma,T\Gamma)$ by Lemma \ref{L:Wmp_Tan_Appr}.
  Hence by a density argument it is sufficient to show \eqref{E:ExImp_DiGr} for $v\in C^1(\Gamma,T\Gamma)$.
  For such a $v$, we proved in our second paper \cite{Miu_NSCTD_02}*{Lemma 4.11} that
  \begin{align*}
    \left|\nabla E_\varepsilon v-\overline{F(v)}\right| \leq c\varepsilon\left(|\bar{v}|+\left|\overline{\nabla_\Gamma v}\right|\right) \quad\text{in}\quad \Omega_\varepsilon.
  \end{align*}
  By this inequality and \eqref{E:Con_Lp} we get
  \begin{align*}
    \left\|\nabla E_\varepsilon v-\overline{F(v)}\right\|_{L^p(\Omega_\varepsilon)} &\leq c\varepsilon\left(\|\bar{v}\|_{L^p(\Omega_\varepsilon)}+\left\|\overline{\nabla_\Gamma v}\right\|_{L^p(\Omega_\varepsilon)}\right) \\
    &\leq c\varepsilon^{1+1/p}\left(\|v\|_{L^p(\Gamma)}+\|\nabla_\Gamma v\|_{L^p(\Gamma)}\right) \\
    &\leq c\varepsilon^{1+1/p}\|v\|_{W^{1,p}(\Gamma)}.
  \end{align*}
  Hence \eqref{E:ExImp_DiGr} holds.
\end{proof}

\begin{lemma} \label{L:ExImp_LpDiv}
  Let $p\in[1,\infty)$.
  There exists $c>0$ independent of $\varepsilon$ such that
  \begin{align} \label{E:ExImp_LpDiv}
    \left\|\mathrm{div}(E_\varepsilon v)-\frac{1}{\bar{g}}\overline{\mathrm{div}_\Gamma(gv)}\right\|_{L^p(\Omega_\varepsilon)} \leq c\varepsilon^{1+1/p}\|v\|_{W^{1,p}(\Gamma)}
  \end{align}
  for all $v\in W^{1,p}(\Gamma,T\Gamma)$.
  In particular, if $v$ satisfies $\mathrm{div}_\Gamma(gv)=0$ on $\Gamma$, then
  \begin{align} \label{E:ExImp_LpSol}
    \|\mathrm{div}(E_\varepsilon v)\|_{L^p(\Omega_\varepsilon)} \leq c\varepsilon^{1+1/p}\|v\|_{W^{1,p}(\Gamma)}.
  \end{align}
\end{lemma}

\begin{proof}
  Since $\mathrm{tr}[Q]=n\cdot n=1$ on $\Gamma$,
  \begin{align*}
    \mathrm{tr}[F(v)] = \mathrm{tr}[\nabla_\Gamma v]+\frac{1}{g}(v\cdot\nabla_\Gamma g)\mathrm{tr}[Q] = \mathrm{div}_\Gamma v+\frac{1}{g}(v\cdot\nabla_\Gamma g) = \frac{1}{g}\mathrm{div}_\Gamma(gv)
  \end{align*}
  on $\Gamma$.
  From this equality and \eqref{E:ExImp_DiGr} we deduce that
  \begin{align*}
    \left\|\mathrm{div}(E_\varepsilon v)-\frac{1}{\bar{g}}\overline{\mathrm{div}_\Gamma(gv)}\right\|_{L^p(\Omega_\varepsilon)} &= \left\|\mathrm{tr}\Bigl[\nabla E_\varepsilon v-\overline{F(v)}\Bigr]\right\|_{L^p(\Omega_\varepsilon)} \\
    &\leq c\varepsilon^{1+1/p}\|v\|_{W^{1,p}(\Gamma)}.
  \end{align*}
  Hence \eqref{E:ExImp_LpDiv} is valid.
\end{proof}

\subsection{Helmholtz--Leray decomposition on a curved thin domain} \label{SS:Fund_HL}
Let
\begin{align*}
  C_{c,\sigma}^\infty(\Omega_\varepsilon) := \{u\in C_c^\infty(\Omega_\varepsilon)^3 \mid \text{$\mathrm{div}\, u=0$ in $\Omega_\varepsilon$}\}
\end{align*}
and $L_\sigma^2(\Omega_\varepsilon)$ be the norm closure of $C_{c,\sigma}^\infty(\Omega_\varepsilon)$ in $L^2(\Omega_\varepsilon)^3$.
It is known (see \cites{BoFa13,Ga11,Te79}) that $L_\sigma^2(\Omega_\varepsilon)$ is characterized by
\begin{align*}
  L_\sigma^2(\Omega_\varepsilon) = \{u \in L^2(\Omega_\varepsilon)^3 \mid \text{$\mathrm{div}\,u=0$ in $\Omega_\varepsilon$, $u\cdot n_\varepsilon=0$ on $\Gamma_\varepsilon$}\}
\end{align*}
and the Helmholtz--Leray decomposition $L^2(\Omega_\varepsilon)^3=L_\sigma^2(\Omega_\varepsilon)\oplus G^2(\Omega_\varepsilon)$ holds with
\begin{align*}
  G^2(\Omega_\varepsilon) = L_\sigma^2(\Omega_\varepsilon)^\perp = \{\nabla p\in L^2(\Omega_\varepsilon)^3 \mid p\in H^1(\Omega_\varepsilon)\}.
\end{align*}
By $\mathbb{L}_\varepsilon$ we denote the Helmholtz--Leray projection from $L^2(\Omega_\varepsilon)^3$ onto $L_\sigma^2(\Omega_\varepsilon)$.
Here we use the nonstandard notation $\mathbb{L}_\varepsilon$ in order to avoid confusion of the Helmholtz--Leray projection with the orthogonal projection $\mathbb{P}_\varepsilon$ from $L^2(\Omega_\varepsilon)^3$ onto $\mathcal{H}_\varepsilon$ given by \eqref{E:Def_Heps}.
For $u\in L^2(\Omega_\varepsilon)^3$ its solenoidal part is given by $\mathbb{L}_\varepsilon u=u-\nabla\varphi$, where $\varphi\in H^1(\Omega_\varepsilon)$ is a weak solution to the Neumann problem of Poisson's equation
\begin{align*}
  \Delta\varphi = \mathrm{div}\,u \quad\text{in}\quad \Omega_\varepsilon, \quad \frac{\partial\varphi}{\partial n_\varepsilon} = u\cdot n_\varepsilon \quad\text{on}\quad \Gamma_\varepsilon.
\end{align*}
Moreover, if $u\in H^1(\Omega_\varepsilon)^3$, then the elliptic regularity theorem (see \cites{Ev10,GiTr01}) yields
\begin{align*}
  \varphi \in H^2(\Omega_\varepsilon), \quad \mathbb{L}_\varepsilon u = u-\nabla\varphi \in H^1(\Omega_\varepsilon)^3.
\end{align*}
The goal of this subsection is to establish a uniform $H^1(\Omega_\varepsilon)$-estimate of $u-\mathbb{L}_\varepsilon u$ for $u\in H^1(\Omega_\varepsilon)^3$ satisfying the impermeable boundary condition \eqref{E:Bo_Imp}.
First we derive the uniform Poincar\'{e} inequality on $\Omega_\varepsilon$.

\begin{lemma} \label{L:Uni_Poin_Dom}
  There exist constants $\varepsilon_\sigma\in(0,1]$ and $c>0$ such that
  \begin{align} \label{E:Uni_Poin_Dom}
    \|\varphi\|_{L^2(\Omega_\varepsilon)} \leq c\|\nabla\varphi\|_{L^2(\Omega_\varepsilon)}
  \end{align}
  for all $\varepsilon\in(0,\varepsilon_\sigma]$ and $\varphi\in H^1(\Omega_\varepsilon)$ satisfying $\int_{\Omega_\varepsilon}\varphi\,dx=0$.
\end{lemma}

We prove Lemma \ref{L:Uni_Poin_Dom} by contradiction.
To this end, we transform integrals over $\Omega_\varepsilon$ into those over $\Omega_1$ with fixed thickness by using the next lemma (note that we assume $\overline{\Omega}_1\subset N$ by scaling $g_0$ and $g_1$, see Section \ref{SS:Pre_CTD}).

\begin{lemma}[{\cite{Miu_NSCTD_01}*{Lemma 5.4}}] \label{L:CoV_Fixed}
  For $\varepsilon\in(0,1]$ let
  \begin{align} \label{E:Def_Bij}
    \Phi_\varepsilon(X) := \pi(X)+\varepsilon d(X)\bar{n}(X), \quad X \in \Omega_1.
  \end{align}
  Then $\Phi_\varepsilon$ is a bijection from $\Omega_1$ onto $\Omega_\varepsilon$ and for a function $\varphi$ on $\Omega_\varepsilon$ we have
  \begin{align} \label{E:CoV_Fixed}
    \int_{\Omega_\varepsilon}\varphi(x)\,dx = \varepsilon\int_{\Omega_1}\xi(X)J(\pi(X),d(X))^{-1}J(\pi(X),\varepsilon d(X))\,dX,
  \end{align}
  where $\xi:=\varphi\circ\Phi_\varepsilon$ on $\Omega_1$ and $J$ is given by \eqref{E:Def_Jac}.
  Moreover, if $\varphi\in L^2(\Omega_\varepsilon)$, then $\xi\in L^2(\Omega_1)$ and there exist constants $c_1,c_2>0$ independent of $\varepsilon$ and $\varphi$ such that
  \begin{align} \label{E:L2_Fixed}
    c_1\varepsilon^{-1}\|\varphi\|_{L^2(\Omega_\varepsilon)}^2 \leq \|\xi\|_{L^2(\Omega_1)}^2 \leq c_2\varepsilon^{-1}\|\varphi\|_{L^2(\Omega_\varepsilon)}^2.
  \end{align}
  If in addition $\varphi\in H^1(\Omega_\varepsilon)$, then $\xi\in H^1(\Omega_1)$ and
  \begin{align} \label{E:H1_Fixed}
    \varepsilon^{-1}\|\nabla\varphi\|_{L^2(\Omega_\varepsilon)}^2 \geq c\left(\left\|\overline{P}\nabla\xi\right\|_{L^2(\Omega_1)}^2+\varepsilon^{-2}\|\partial_n\xi\|_{L^2(\Omega_1)}^2\right),
  \end{align}
  where $\partial_n\xi=(\bar{n}\cdot\nabla)\xi$ on $\Omega_1$ and $c>0$ is a constant independent of $\varepsilon$ and $\varphi$.
\end{lemma}

\begin{proof}[Proof of Lemma \ref{L:Uni_Poin_Dom}]
  Assume to the contrary that there exist a sequence $\{\varepsilon_k\}_{k=1}^\infty$ of positive numbers convergent to zero and $\varphi_k\in H^1(\Omega_{\varepsilon_k})$ such that
  \begin{align} \label{Pf_UPD:Contra}
    \|\varphi_k\|_{L^2(\Omega_{\varepsilon_k})}^2 > k\|\nabla\varphi_k\|_{L^2(\Omega_{\varepsilon_k})}^2, \quad \int_{\Omega_{\varepsilon_k}}\varphi_k\,dx = 0, \quad k\in\mathbb{N}.
  \end{align}
  For $k\in\mathbb{N}$ let $\Phi_{\varepsilon_k}$ be the bijection from $\Omega_1$ onto $\Omega_{\varepsilon_k}$ given by \eqref{E:Def_Bij} and
  \begin{align*}
    \xi_k := \varphi_k\circ\Phi_{\varepsilon_k}\in H^1(\Omega_1).
  \end{align*}
  We divide both sides of the first inequality of \eqref{Pf_UPD:Contra} by $\varepsilon_k$ and use \eqref{E:L2_Fixed} and \eqref{E:H1_Fixed} to deduce that
  \begin{align*}
    \|\xi_k\|_{L^2(\Omega_1)}^2 > ck\left(\left\|\overline{P}\nabla\xi_k\right\|_{L^2(\Omega_1)}^2+\varepsilon_k^{-2}\|\partial_n\xi_k\|_{L^2(\Omega_1)}^2\right).
  \end{align*}
  Since $\|\xi_k\|_{L^2(\Omega_1)}>0$, we may assume that
  \begin{align} \label{Pf_UPD:Xik_L2}
    \|\xi_k\|_{L^2(\Omega_1)}=1
  \end{align}
  by replacing $\xi_k$ with $\xi_k/\|\xi_k\|_{L^2(\Omega_1)}$ and thus
  \begin{align} \label{Pf_UPD:Nabla_k}
    \left\|\overline{P}\nabla\xi_k\right\|_{L^2(\Omega_1)}^2 < ck^{-1}, \quad \|\partial_n\xi_k\|_{L^2(\Omega_1)}^2 < c\varepsilon_k^2k^{-1}.
  \end{align}
  Then $\{\xi_k\}_{k=1}^\infty$ is bounded in $H^1(\Omega_1)$ by \eqref{Pf_UPD:Xik_L2}, \eqref{Pf_UPD:Nabla_k}, and
  \begin{align*}
    |\nabla\xi_k|^2 = \left|\overline{P}\nabla\xi_k\right|^2+\left|\overline{Q}\nabla\xi_k\right|^2, \quad \left|\overline{Q}\nabla\xi_k\right| = |\bar{n}\otimes\partial_n\xi_k| = |\partial_n\xi_k| \quad\text{in}\quad \Omega_1.
  \end{align*}
  By this fact and the compact embedding $H^1(\Omega_1)\hookrightarrow L^2(\Omega_1)$ we see that $\{\xi_k\}_{k=1}^\infty$ converges (up to a subsequence) to some $\xi\in H^1(\Omega_1)$ weakly in $H^1(\Omega_1)$ and strongly in $L^2(\Omega_1)$.
  Hence by \eqref{Pf_UPD:Xik_L2} we get
  \begin{align} \label{Pf_UPD:Limit_L2}
    \|\xi\|_{L^2(\Omega_1)} = \lim_{k\to\infty}\|\xi_k\|_{L^2(\Omega_1)} = 1.
  \end{align}
  Moreover, the weak convergence of $\{\xi_k\}_{k=1}^\infty$ to $\xi$ in $H^1(\Omega_1)$ and \eqref{Pf_UPD:Nabla_k} imply
  \begin{align} \label{Pf_UPD:Xi_Grad}
    \overline{P}\nabla\xi = 0, \quad \partial_n\xi = 0 \quad\text{in}\quad \Omega_1.
  \end{align}
  Let $\eta(y):=\xi(y+g_0(y)n(y))$ for $y\in\Gamma$.
  Then $\xi=\bar{\eta}$ is the constant extension of $\eta$ by the second equality of \eqref{Pf_UPD:Xi_Grad} and thus $\eta\in H^1(\Gamma)$ by $\xi\in H^1(\Omega_1)$ and Lemma \ref{L:Con_Lp_W1p}.
  Also, by \eqref{E:P_TGr}, \eqref{E:WReso_P}, \eqref{E:ConDer_Dom}, and the first equality of \eqref{Pf_UPD:Xi_Grad},
  \begin{align*}
    0 = \overline{P}\nabla\xi = \overline{P}\Bigl(I_3-d\overline{W}\Bigr)^{-1}\overline{\nabla_\Gamma\eta} = \Bigl(I_3-d\overline{W}\Bigr)^{-1}\overline{\nabla_\Gamma\eta} \quad\text{in}\quad \Omega_1,
  \end{align*}
  which yields $\overline{\nabla_\Gamma\eta}=0$ in $\Omega_1$, i.e. $\nabla_\Gamma\eta=0$ on $\Gamma$.
  Thus $\eta$ is constant on $\Gamma$ by Lemma \ref{L:WTGr_Van}.
  Now we apply \eqref{E:CoV_Fixed} to the second equality of \eqref{Pf_UPD:Contra} to have
  \begin{align} \label{Pf_UPD:Int_CTD}
    \int_{\Omega_1}\xi_k(X)J(\pi(X),d(X))^{-1}J(\pi(X),\varepsilon_kd(X))\,dX = 0.
  \end{align}
  Moreover, since $\{\xi_k\}_{k=1}^\infty$ converges to $\xi=\bar{\eta}$ strongly in $L^2(\Omega_1)$ and
  \begin{align*}
    |J(\pi(X),\varepsilon_kd(X))-1| \leq c\varepsilon_k|d(X)| \leq c\varepsilon_k \to 0 \quad\text{as}\quad k \to \infty
  \end{align*}
  uniformly in $X\in\Omega_1$ by \eqref{E:Jac_Bound_03} and $|d|\leq c$ in $\Omega_1$, we send $k\to\infty$ in \eqref{Pf_UPD:Int_CTD} to get
  \begin{align*}
    \int_{\Omega_1}\eta(\pi(X))J(\pi(X),d(X))^{-1}\,dX = 0.
  \end{align*}
  Noting that $\eta$ is constant on $\Gamma$, we apply \eqref{E:CoV_Dom} to this equality to find that
  \begin{align*}
    \eta\int_\Gamma\int_{g_0(y)}^{g_1(y)} J(y,r)^{-1}J(y,r)\,dr\,d\mathcal{H}^2(y) = \eta\int_\Gamma g(y)\,d\mathcal{H}^2(y) = 0.
  \end{align*}
  From this equality and \eqref{E:G_Inf} we deduce that $\eta=0$ on $\Gamma$ and thus $\xi=\bar{\eta}=0$ in $\Omega_1$, which contradicts with \eqref{Pf_UPD:Limit_L2}.
  Therefore, the uniform inequality \eqref{E:Uni_Poin_Dom} is valid.
\end{proof}

Next we consider the Neumann problem of Poisson's equation
\begin{align} \label{E:Poisson_Dom}
  -\Delta\varphi = \xi \quad\text{in}\quad \Omega_\varepsilon, \quad \frac{\partial\varphi}{\partial n_\varepsilon} = 0 \quad\text{on}\quad \Gamma_\varepsilon
\end{align}
for $\xi\in H^{-1}(\Omega_\varepsilon)$ satisfying $\langle\xi,1\rangle_{\Omega_\varepsilon}=0$, where $\langle\cdot,\cdot\rangle_{\Omega_\varepsilon}$ denotes the duality product between $H^{-1}(\Omega_\varepsilon)$ and $H^1(\Omega_\varepsilon)$.
By the Lax--Milgram theory there exists a unique weak solution $\varphi\in H^1(\Omega_\varepsilon)$ satisfying
\begin{align} \label{E:PD_Weak}
  (\nabla\varphi,\nabla\zeta)_{L^2(\Omega_\varepsilon)} = \langle\xi,\zeta\rangle_{\Omega_\varepsilon} \quad\text{for all}\quad \zeta\in H^1(\Omega_\varepsilon),\quad \int_{\Omega_\varepsilon}\varphi\,dx = 0.
\end{align}
Moreover, if $\xi\in L^2(\Omega_\varepsilon)$, then $\varphi\in H^2(\Omega_\varepsilon)$ and
\begin{align} \label{E:PD_H2_Ep}
  \|\varphi\|_{H^2(\Omega_\varepsilon)} \leq c_\varepsilon\|\xi\|_{L^2(\Omega_\varepsilon)}
\end{align}
with a constant $c_\varepsilon>0$ depending on $\varepsilon$ by the elliptic regularity theorem.
In this case, the equation \eqref{E:Poisson_Dom} is satisfied in the strong sense.
Let us show that $c_\varepsilon$ in \eqref{E:PD_H2_Ep} can be taken independently of $\varepsilon$.

\begin{lemma} \label{L:PD_H2}
  Let $\varepsilon_\sigma$ be the constant given in Lemma \ref{L:Uni_Poin_Dom}.
  For $\varepsilon\in(0,\varepsilon_\sigma]$ suppose that $\xi\in L^2(\Omega_\varepsilon)$ satisfies
  \begin{align} \label{E:PD_H2_So}
    \langle\xi,1\rangle_{\Omega_\varepsilon} = \int_{\Omega_\varepsilon}\xi\,dx = 0.
  \end{align}
  Then there exists a constant $c>0$ independent of $\varepsilon$ and $\xi$ such that
  \begin{align} \label{E:PD_H2}
    \|\varphi\|_{H^2(\Omega_\varepsilon)} \leq c\|\xi\|_{L^2(\Omega_\varepsilon)},
  \end{align}
  where $\varphi\in H^2(\Omega_\varepsilon)$ is a unique solution to \eqref{E:Poisson_Dom} with source term $\xi$.
\end{lemma}

A key tool for the proof of Lemma \ref{L:PD_H2} is the following estimate shown in our first paper \cite{Miu_NSCTD_01} based on a careful analysis of surface quantities of $\Gamma_\varepsilon$.

\begin{lemma}[{\cite{Miu_NSCTD_01}*{Lemma 4.3}}] \label{L:Int_UgUn}
  There exists a constant $c>0$ such that
  \begin{align} \label{E:Int_UgUn}
    \left|\int_{\Gamma_\varepsilon}(u\cdot\nabla)u\cdot n_\varepsilon\,d\mathcal{H}^2\right| \leq c\left(\|u\|_{L^2(\Omega_\varepsilon)}^2+\|u\|_{L^2(\Omega_\varepsilon)}\|\nabla u\|_{L^2(\Omega_\varepsilon)}\right)
  \end{align}
  for all $\varepsilon\in(0,1]$ and $u\in C^1(\overline{\Omega}_\varepsilon)^3\cup H^2(\Omega_\varepsilon)^3$ satisfying \eqref{E:Bo_Imp}.
\end{lemma}

\begin{proof}[Proof of Lemma \ref{L:PD_H2}]
  Let $\varphi\in H^2(\Omega_\varepsilon)$ be a unique solution to \eqref{E:Poisson_Dom} with source term $\xi\in L^2(\Omega_\varepsilon)$ satisfying \eqref{E:PD_H2_So}.
  Noting that $\varphi$ satisfies the second equality of \eqref{E:PD_Weak}, we set $\zeta=\varphi$ in the first equality of \eqref{E:PD_Weak} and apply \eqref{E:Uni_Poin_Dom} to $\varphi$ to get
  \begin{align*}
    \|\nabla\varphi\|_{L^2(\Omega_\varepsilon)}^2 = \langle\xi,\varphi\rangle_{\Omega_\varepsilon} = (\xi,\varphi)_{L^2(\Omega_\varepsilon)} &\leq \|\xi\|_{L^2(\Omega_\varepsilon)}\|\varphi\|_{L^2(\Omega_\varepsilon)} \\
    &\leq c\|\xi\|_{L^2(\Omega_\varepsilon)}\|\nabla\varphi\|_{L^2(\Omega_\varepsilon)}
  \end{align*}
  with a constant $c>0$ independent of $\varepsilon$.
  By this inequality and \eqref{E:Uni_Poin_Dom},
  \begin{align*}
    \|\varphi\|_{H^1(\Omega_\varepsilon)} \leq c\|\nabla\varphi\|_{L^2(\Omega_\varepsilon)} \leq c\|\xi\|_{L^2(\Omega_\varepsilon)}.
  \end{align*}
  Thus it is sufficient for \eqref{E:PD_H2} to show that
  \begin{align} \label{Pf_PDH:Goal}
    \begin{aligned}
      \|\nabla^2\varphi\|_{L^2(\Omega_\varepsilon)} &\leq c\left(\|\Delta\varphi\|_{L^2(\Omega_\varepsilon)}+\|\varphi\|_{H^1(\Omega_\varepsilon)}\right) \\
      &= c\left(\|\xi\|_{L^2(\Omega_\varepsilon)}+\|\varphi\|_{H^1(\Omega_\varepsilon)}\right)
    \end{aligned}
  \end{align}
  with some constant $c>0$ independent of $\varepsilon$ (note that $-\Delta\varphi=\xi$ a.e. in $\Omega_\varepsilon$).

  Since $C_c^\infty(\Omega_\varepsilon)$ is dense in $L^2(\Omega_\varepsilon)$, we can take a sequence $\{\xi_k\}_{k=1}^\infty$ of functions in $C_c^\infty(\Omega_\varepsilon)$ that converges strongly to $\xi$ in $L^2(\Omega_\varepsilon)$.
  For each $k\in\mathbb{N}$ let $\varphi_k\in H^1(\Omega_\varepsilon)$ be a unique weak solution to \eqref{E:Poisson_Dom} with source term
  \begin{align*}
    \tilde{\xi}_k(x) := \xi_k(x)-\frac{1}{|\Omega_\varepsilon|}\int_{\Omega_\varepsilon}\xi_k(z)\,dz, \quad x\in\Omega_\varepsilon \quad \left(\langle\tilde{\xi}_k,1\rangle_{\Omega_\varepsilon} = \int_{\Omega_\varepsilon}\tilde{\xi}_k\,dx = 0\right).
  \end{align*}
  Here $|\Omega_\varepsilon|$ is the volume of $\Omega_\varepsilon$.
  Since $\tilde{\xi}_k\in C^\infty(\overline{\Omega}_\varepsilon)$ and $\Gamma_\varepsilon$ is of class $C^4$, the elliptic regularity theorem yields $\varphi_k\in H^3(\Omega_\varepsilon)$ (in fact the $C^3$-regularity of $\Gamma_\varepsilon$ is sufficient for our purpose here).
  Moreover,
  \begin{align*}
    \lim_{k\to\infty}\int_{\Omega_\varepsilon}\xi_k\,dx = \lim_{k\to\infty}(\xi_k,1)_{L^2(\Omega_\varepsilon)} = (\xi,1)_{L^2(\Omega_\varepsilon)} = \int_{\Omega_\varepsilon}\xi\,dx = 0
  \end{align*}
  by the strong convergence of $\{\xi_k\}_{k=1}^\infty$ to $\xi$ in $L^2(\Omega_\varepsilon)$ and thus
  \begin{align} \label{Pf_PDH:Xi_Conv}
    \|\xi-\tilde{\xi}_k\|_{L^2(\Omega_\varepsilon)} \leq \|\xi-\xi_k\|_{L^2(\Omega_\varepsilon)}+\frac{1}{|\Omega_\varepsilon|^{1/2}}\left|\int_{\Omega_\varepsilon}\xi_k\,dx\right| \to 0 \quad\text{as}\quad k\to\infty.
  \end{align}
  Since $\varphi-\varphi_k$ is a unique solution to \eqref{E:Poisson_Dom} for the source term $\xi-\tilde{\xi}_k$,
  \begin{align*}
    \|\varphi-\varphi_k\|_{H^2(\Omega_\varepsilon)} \leq c_\varepsilon\|\xi-\tilde{\xi}_k\|_{L^2(\Omega_\varepsilon)} \to 0 \quad\text{as}\quad k\to\infty
  \end{align*}
  by \eqref{E:PD_H2_Ep} and \eqref{Pf_PDH:Xi_Conv} (note that the constant $c_\varepsilon$ does not depend on $k$).
  Hence we can get \eqref{Pf_PDH:Goal} by showing the same inequality for $\varphi_k$ and sending $k\to\infty$.

  From now on, we fix and suppress the subscript $k$.
  Hence we suppose that $\varphi$ is in $H^3(\Omega_\varepsilon)$ and satisfies \eqref{E:Poisson_Dom} in the strong sense.
  In particular, we have
  \begin{align} \label{Pf_PDH:Bo}
    \nabla\varphi\cdot n_\varepsilon = \frac{\partial\varphi}{\partial n_\varepsilon} = 0 \quad\text{on}\quad \Gamma_\varepsilon.
  \end{align}
  By the regularity of $\varphi$ we can carry out integration by parts twice to get
  \begin{align*}
    \|\nabla^2\varphi\|_{L^2(\Omega_\varepsilon)}^2 &= \|\Delta\varphi\|_{L^2(\Omega_\varepsilon)}^2+\int_{\Gamma_\varepsilon}[\{(\nabla\varphi\cdot\nabla)\nabla\varphi\}\cdot n_\varepsilon-(\nabla\varphi\cdot n_\varepsilon)\Delta\varphi]\,d\mathcal{H}^2 \\
    &= \|\Delta\varphi\|_{L^2(\Omega_\varepsilon)}^2+\int_{\Gamma_\varepsilon}\{(\nabla\varphi\cdot\nabla)\nabla\varphi\}\cdot n_\varepsilon\,d\mathcal{H}^2.
  \end{align*}
  Here the second equality is due to \eqref{Pf_PDH:Bo}.
  Moreover, since $\nabla\varphi\in H^2(\Omega_\varepsilon)^3$ satisfies \eqref{Pf_PDH:Bo}, we can apply \eqref{E:Int_UgUn} with $u=\nabla\varphi$ to the last term to obtain
  \begin{align*}
    \|\nabla^2\varphi\|_{L^2(\Omega_\varepsilon)}^2 &\leq \|\Delta\varphi\|_{L^2(\Omega_\varepsilon)}^2+c\left(\|\nabla\varphi\|_{L^2(\Omega_\varepsilon)}^2+\|\nabla\varphi\|_{L^2(\Omega_\varepsilon)}\|\nabla^2\varphi\|_{L^2(\Omega_\varepsilon)}\right) \\
    &\leq \|\Delta\varphi\|_{L^2(\Omega_\varepsilon)}^2+c\|\nabla\varphi\|_{L^2(\Omega_\varepsilon)}^2+\frac{1}{2}\|\nabla^2\varphi\|_{L^2(\Omega_\varepsilon)}^2,
  \end{align*}
  where we also used Young's inequality in the second inequality.
  Hence \eqref{Pf_PDH:Goal} follows and we conclude that \eqref{E:PD_H2} is valid.
\end{proof}

Finally, we obtain a uniform $H^1(\Omega_\varepsilon)$-estimate for $u-\mathbb{L}_\varepsilon u$ by Lemma \ref{L:PD_H2}.

\begin{lemma} \label{L:HP_Dom}
  Let $\varepsilon_\sigma$ be the constant given in Lemma \ref{L:Uni_Poin_Dom}.
  For $\varepsilon\in(0,\varepsilon_\sigma]$ let $\mathbb{L}_\varepsilon$ be the Helmholtz--Leray projection from $L^2(\Omega_\varepsilon)^3$ onto $L_\sigma^2(\Omega_\varepsilon)$.
  Then there exists a constant $c>0$ independent of $\varepsilon$ such that
  \begin{align} \label{E:HP_Dom}
    \|u-\mathbb{L}_\varepsilon u\|_{H^1(\Omega_\varepsilon)} \leq c\|\mathrm{div}\,u\|_{L^2(\Omega_\varepsilon)}
  \end{align}
  for all $u\in H^1(\Omega_\varepsilon)^3$ satisfying \eqref{E:Bo_Imp}.
\end{lemma}

\begin{proof}
  Let $\xi:=-\mathrm{div}\,u\in L^2(\Omega_\varepsilon)$.
  Since $u$ satisfies \eqref{E:Bo_Imp}, we see that
  \begin{align*}
    \langle\xi,1\rangle_{\Omega_\varepsilon} = \int_{\Omega_\varepsilon}\xi\,dx = -\int_{\Gamma_\varepsilon}u\cdot n_\varepsilon\,d\mathcal{H}^2 = 0
  \end{align*}
  by the divergence theorem.
  Hence the uniform estimate \eqref{E:PD_H2} holds for the unique solution $\varphi\in H^2(\Omega_\varepsilon)$ to \eqref{E:Poisson_Dom} with source term $\xi=-\mathrm{div}\,u$ by Lemma \ref{L:PD_H2}.
  From this fact and $\mathbb{L}_\varepsilon u=u-\nabla\varphi$ we deduce that
  \begin{align*}
    \|u-\mathbb{L}_\varepsilon u\|_{H^1(\Omega_\varepsilon)} = \|\nabla\varphi\|_{H^1(\Omega_\varepsilon)} \leq c\|\xi\|_{L^2(\Omega_\varepsilon)} = c\|\mathrm{div}\,u\|_{L^2(\Omega_\varepsilon)}
  \end{align*}
  with a constant $c>0$ independent of $\varepsilon$.
  Hence \eqref{E:HP_Dom} is valid.
\end{proof}

\subsection{Stokes operator under the slip boundary conditions} \label{SS:Fund_St}
Let us give basic properties of the Stokes operator for $\Omega_\varepsilon$ under the slip boundary conditions used in Section \ref{S:SL}.
Throughout this subsection we impose Assumptions \ref{Assump_1} and \ref{Assump_2} and fix the constant $\varepsilon_0$ given in Lemma \ref{L:Uni_aeps}.

For $u_1\in H^2(\Omega_\varepsilon)^3$ and $u_2\in H^1(\Omega_\varepsilon)^3$ we have
\begin{multline*}
  \int_{\Omega_\varepsilon}\{\Delta u_1+\nabla(\mathrm{div}\,u_1)\}\cdot u_2\,dx \\
  = -2\int_{\Omega_\varepsilon}D(u_1):D(u_2)\,dx+2\int_{\Gamma_\varepsilon}[D(u_1)n_\varepsilon]\cdot u_2\,d\mathcal{H}^2
\end{multline*}
by integration by parts (see \cite{Miu_NSCTD_01}*{Lemma 7.1}).
Hence if $u_1$ satisfies $\mathrm{div}\,u_1=0$ in $\Omega_\varepsilon$ and the slip boundary conditions \eqref{E:Bo_Slip} and $u_2$ satisfies the impermeable boundary condition \eqref{E:Bo_Imp}, then it follows that
\begin{align*}
  \nu\int_{\Omega_\varepsilon}\Delta u_1\cdot u_2\,dx = -2\nu\int_{\Omega_\varepsilon}D(u_1):D(u_2)\,dx-\sum_{i=0,1}\gamma_\varepsilon^i\int_{\Gamma_\varepsilon^i}u_1\cdot u_2\,d\mathcal{H}^2.
\end{align*}
By this equality we see that the bilinear form for the Stokes problem
\begin{align*}
  \left\{
  \begin{alignedat}{3}
    -\nu\Delta u+\nabla p &= f &\quad &\text{in} &\quad &\Omega_\varepsilon, \\
    \mathrm{div}\,u &= 0 &\quad &\text{in} &\quad &\Omega_\varepsilon, \\
    u\cdot n_\varepsilon &= 0 &\quad &\text{on} &\quad &\Gamma_\varepsilon, \\
    2\nu P_\varepsilon D(u)n_\varepsilon+\gamma_\varepsilon u &= 0 &\quad &\text{on} &\quad &\Gamma_\varepsilon
  \end{alignedat}
  \right.
\end{align*}
is of the form
\begin{align} \label{E:Def_Bi_Dom}
  a_\varepsilon(u_1,u_2) := 2\nu\bigl(D(u_1),D(u_2)\bigr)_{L^2(\Omega_\varepsilon)}+\sum_{i=0,1}\gamma_\varepsilon^i(u_1,u_2)_{L^2(\Gamma_\varepsilon^i)}
\end{align}
for $u_1,u_2\in H^1(\Omega_\varepsilon)^3$.
Let $\mathcal{H}_\varepsilon$ and $\mathcal{V}_\varepsilon$ be the function spaces given by \eqref{E:Def_Heps} and $\mathbb{P}_\varepsilon$ the orthogonal projection from $L^2(\Omega_\varepsilon)^3$ onto $\mathcal{H}_\varepsilon$.
For each $\varepsilon\in(0,\varepsilon_0]$ the bilinear form $a_\varepsilon$ is uniformly bounded and coercive on $\mathcal{V}_\varepsilon$ by Lemma \ref{L:Uni_aeps}.
Hence it induces a bounded linear operator $A_\varepsilon$ from $\mathcal{V}_\varepsilon$ into its dual space $\mathcal{V}_\varepsilon'$ such that
\begin{align*}
  {}_{\mathcal{V}_\varepsilon'}\langle A_\varepsilon u_1,u_2\rangle_{\mathcal{V}_\varepsilon} = a_\varepsilon(u_1,u_2), \quad u_1,u_2\in \mathcal{V}_\varepsilon
\end{align*}
by the Lax--Milgram theorem.
Here ${}_{\mathcal{V}_\varepsilon'}\langle\cdot,\cdot\rangle_{\mathcal{V}_\varepsilon}$ is the duality product between $\mathcal{V}_\varepsilon'$ and $\mathcal{V}_\varepsilon$.
Moreover, if we consider $A_\varepsilon$ as an unbounded operator on $\mathcal{H}_\varepsilon$ with domain
\begin{align*}
  D(A_\varepsilon) = \{u\in\mathcal{V}_\varepsilon\mid A_\varepsilon u\in\mathcal{H}_\varepsilon\},
\end{align*}
then the Lax--Milgram theory implies that $A_\varepsilon$ is positive and self-adjoint on $\mathcal{H}_\varepsilon$ and thus its square root $A_\varepsilon^{1/2}$ is well-defined on $D(A_\varepsilon^{1/2})=\mathcal{V}_\varepsilon$ and
\begin{align} \label{E:L2in_Ahalf}
  a_\varepsilon(u_1,u_2) = (A_\varepsilon u_1,u_2)_{L^2(\Omega_\varepsilon)} = (A_\varepsilon^{1/2}u_1,A_\varepsilon^{1/2} u_2)_{L^2(\Omega_\varepsilon)}
\end{align}
for all $u_1\in D(A_\varepsilon)$ and $u_2\in \mathcal{V}_\varepsilon$ (see \cites{BoFa13,So01}).
By a regularity result for a solution to the Stokes problem (see \cites{AmRe14,Be04,SoSc73}) we also observe that
\begin{align*}
  D(A_\varepsilon) = \{u\in \mathcal{V}_\varepsilon\cap H^2(\Omega_\varepsilon)^3 \mid \text{$2\nu P_\varepsilon D(u)n_\varepsilon+\gamma_\varepsilon u=0$ on $\Gamma_\varepsilon$}\}
\end{align*}
and $A_\varepsilon u=-\nu\mathbb{P}_\varepsilon\Delta u$ for $u\in D(A_\varepsilon)$.
We call $A_\varepsilon$ the Stokes operator for $\Omega_\varepsilon$ under the slip boundary conditions or simply the Stokes operator on $\mathcal{H}_\varepsilon$.

In our first paper \cite{Miu_NSCTD_01} we derived several kinds of uniform estimates for $A_\varepsilon$ based on a careful analysis of surface quantities of $\Gamma_\varepsilon$.
They were essential mainly for the proof of the global existence of a strong solution to \eqref{E:NS_CTD} carried out in our second paper \cite{Miu_NSCTD_02}.
In this paper we just use the following estimates.

\begin{lemma}[{\cite{Miu_NSCTD_01}*{Lemma 2.5}}] \label{L:Stokes_H1}
  There exists a constant $c>0$ such that
  \begin{align} \label{E:Stokes_H1}
    c^{-1}\|u\|_{H^1(\Omega_\varepsilon)} \leq \|A_\varepsilon^{1/2}u\|_{L^2(\Omega_\varepsilon)} \leq c\|u\|_{H^1(\Omega_\varepsilon)}
  \end{align}
  for all $\varepsilon\in(0,\varepsilon_0]$ and $u\in \mathcal{V}_\varepsilon$.
\end{lemma}

\begin{lemma}[{\cite{Miu_NSCTD_01}*{Corollary 2.8}}] \label{L:St_Inter}
  There exists a constant $c>0$ such that
  \begin{align} \label{E:St_Inter}
    \|u\|_{H^1(\Omega_\varepsilon)} \leq c\|u\|_{L^2(\Omega_\varepsilon)}^{1/2}\|u\|_{H^2(\Omega_\varepsilon)}^{1/2}
  \end{align}
  for all $\varepsilon\in(0,\varepsilon_0]$ and $u\in D(A_\varepsilon)$.
\end{lemma}

\section{Average operators in the thin direction} \label{S:Ave}
We introduce average operators in the thin direction and study their properties which are essential for the study of a singular limit problem for \eqref{E:NS_CTD}.
In our second paper \cite{Miu_NSCTD_02} we derived various kinds of estimates for the average operators to show the global existence of a strong solution to \eqref{E:NS_CTD}.
We give some of them and related results in Sections \ref{SS:Ave_Def}--\ref{SS:Ave_Decom}.
Also, we consider the average of bilinear and trilinear forms for functions on $\Omega_\varepsilon$ in Section \ref{SS:Ave_BT}.

Throughout this section we assume $\varepsilon\in(0,1]$ and write $\bar{\eta}=\eta\circ\pi$ for the constant extension of a function $\eta$ on $\Gamma$ in the normal direction of $\Gamma$.
We also denote by $\partial_n\varphi=(\bar{n}\cdot\nabla)\varphi$ the derivative of a function $\varphi$ on $\Omega_\varepsilon$ in the normal direction of $\Gamma$.

\subsection{Definition and basic properties of the average operators} \label{SS:Ave_Def}

\begin{definition} \label{D:Average}
  We define the average operator $M$ as
  \begin{align} \label{E:Def_Ave}
    M\varphi(y) := \frac{1}{\varepsilon g(y)}\int_{\varepsilon g_0(y)}^{\varepsilon g_1(y)}\varphi(y+rn(y))\,dr, \quad y\in\Gamma
  \end{align}
  for a function $\varphi$ on $\Omega_\varepsilon$.
  The operator $M$ is also applied to a vector field $u\colon\Omega_\varepsilon\to\mathbb{R}^3$ and we define the averaged tangential component $M_\tau u$ of $u$ by
  \begin{align} \label{E:Def_Tan_Ave}
    M_\tau u(y) := P(y)Mu(y) = \frac{1}{\varepsilon g(y)}\int_{\varepsilon g_0(y)}^{\varepsilon g_1(y)}P(y)u(y+rn(y))\,dr, \quad y\in\Gamma.
  \end{align}
\end{definition}

Let us give basic properties of $M$ and $M_\tau$.

\begin{lemma}[{\cite{Miu_NSCTD_02}*{Lemma 6.2}}] \label{L:Ave_Lp}
  There exists a constant $c>0$ independent of $\varepsilon$ such that
  \begin{align} \label{E:Ave_Lp_Surf}
    \|M\varphi\|_{L^p(\Gamma)} \leq c\varepsilon^{-1/p}\|\varphi\|_{L^p(\Omega_\varepsilon)}
  \end{align}
  for all $\varphi\in L^p(\Omega_\varepsilon)$ with $p\in[1,\infty)$.
\end{lemma}

\begin{lemma}[{\cite{Miu_NSCTD_02}*{Lemma 6.3}}] \label{L:AveT_Lp}
  There exists a constant $c>0$ independent of $\varepsilon$ such that
  \begin{align} \label{E:AveT_Lp_Surf}
    \|M_\tau u\|_{L^p(\Gamma)} &\leq c\varepsilon^{-1/p}\|u\|_{L^p(\Omega_\varepsilon)}
  \end{align}
  for all $u\in L^p(\Omega_\varepsilon)^3$ with $p\in[1,\infty)$.
\end{lemma}

\begin{lemma}[{\cite{Miu_NSCTD_02}*{Lemma 6.4}}] \label{L:Ave_Diff}
  There exists a constant $c>0$ independent of $\varepsilon$ such that
  \begin{align}
    \left\|\varphi-\overline{M\varphi}\right\|_{L^p(\Omega_\varepsilon)} &\leq c\varepsilon\|\partial_n\varphi\|_{L^p(\Omega_\varepsilon)}, \label{E:Ave_Diff_Dom}\\
    \left\|\varphi-\overline{M\varphi}\right\|_{L^p(\Gamma_\varepsilon^i)} &\leq c\varepsilon^{1-1/p}\|\partial_n\varphi\|_{L^p(\Omega_\varepsilon)}, \quad i=0,1 \label{E:Ave_Diff_Bo}
  \end{align}
  for all $\varphi\in W^{1,p}(\Omega_\varepsilon)$ with $p\in[1,\infty)$.
\end{lemma}

\begin{lemma}[{\cite{Miu_NSCTD_02}*{Lemma 6.5}}] \label{L:Ave_N_Lp}
  There exists a constant $c>0$ independent of $\varepsilon$ such that
  \begin{align} \label{E:Ave_N_Lp}
    \|Mu\cdot n\|_{L^p(\Gamma)} &\leq c\varepsilon^{1-1/p}\|u\|_{W^{1,p}(\Omega_\varepsilon)}
  \end{align}
  for all $u\in W^{1,p}(\Omega_\varepsilon)^3$ with $p\in[1,\infty)$ satisfying \eqref{E:Bo_Imp} on $\Gamma_\varepsilon^0$ or on $\Gamma_\varepsilon^1$.
\end{lemma}

\begin{lemma}[{\cite{Miu_NSCTD_02}*{Lemma 6.6}}] \label{L:AveT_Diff}
  There exists a constant $c>0$ independent of $\varepsilon$ such that
  \begin{align} \label{E:AveT_Diff_Dom}
    \left\|u-\overline{M_\tau u}\right\|_{L^p(\Omega_\varepsilon)} \leq c\varepsilon\|u\|_{W^{1,p}(\Omega_\varepsilon)}
  \end{align}
  for all $u\in W^{1,p}(\Omega_\varepsilon)^3$ with $p\in[1,\infty)$ satisfying \eqref{E:Bo_Imp} on $\Gamma_\varepsilon^0$ or on $\Gamma_\varepsilon^1$.
\end{lemma}

\subsection{Derivatives of averaged functions} \label{SS:Ave_Grad}
Next we show formulas and estimates for the tangential and time derivatives of the average of a function on $\Omega_\varepsilon$.

\begin{lemma}[{\cite{Miu_NSCTD_02}*{Lemma 6.8}}] \label{L:Ave_Der}
  For $\varphi\in C^1(\Omega_\varepsilon)$ we have
  \begin{align} \label{E:Ave_Der}
    \nabla_\Gamma M\varphi = M(B\nabla\varphi)+M\bigl((\partial_n\varphi)\psi_\varepsilon\bigr) \quad\text{on}\quad \Gamma,
  \end{align}
  where the matrix-valued function $B$ and the vector field $\psi_\varepsilon$ are given by
  \begin{align} \label{E:Ave_Der_Aux}
    \begin{aligned}
      B(x) &:= \left\{I_3-d(x)\overline{W}(x)\right\}\overline{P}(x), \\
      \psi_\varepsilon(x) &:= \frac{1}{\bar{g}(x)}\left\{\bigl(d(x)-\varepsilon\bar{g}_0(x)\bigr)\overline{\nabla_\Gamma g_1}(x)+\bigl(\varepsilon\bar{g}_1(x)-d(x)\bigr)\overline{\nabla_\Gamma g_0}(x)\right\}
    \end{aligned}
  \end{align}
  for $x\in N$.
\end{lemma}

\begin{remark} \label{R:Ave_Der}
  Since $\nabla_\Gamma g_0$ and $\nabla_\Gamma g_1$ are bounded on $\Gamma$ and
  \begin{align*}
    0 \leq d(x)-\varepsilon\bar{g}_0(x) \leq \varepsilon\bar{g}(x), \quad 0 \leq \varepsilon\bar{g}_1(x)-d(x) \leq \varepsilon\bar{g}(x), \quad x\in\Omega_\varepsilon,
  \end{align*}
  there exists a constant $c>0$ independent of $\varepsilon$ such that
  \begin{align} \label{E:ADA_Bound}
    |\psi_\varepsilon| \leq c\varepsilon \quad\text{in}\quad \Omega_\varepsilon.
  \end{align}
\end{remark}

\begin{lemma}[{\cite{Miu_NSCTD_02}*{Lemma 6.10}}] \label{L:Ave_Wmp}
  There exists a constant $c>0$ independent of $\varepsilon$ such that
  \begin{align} \label{E:Ave_Wmp_Surf}
    \|M\varphi\|_{W^{m,p}(\Gamma)} &\leq c\varepsilon^{-1/p}\|\varphi\|_{W^{m,p}(\Omega_\varepsilon)}
  \end{align}
  for all $\varphi\in W^{m,p}(\Omega_\varepsilon)$ with $m=1,2$ and $p\in[1,\infty)$.
\end{lemma}

\begin{lemma}[{\cite{Miu_NSCTD_02}*{Lemma 6.11}}] \label{L:AveT_Wmp}
 There exists a constant $c>0$ independent of $\varepsilon$ such that
  \begin{align} \label{E:AveT_Wmp_Surf}
    \|M_\tau u\|_{W^{m,p}(\Gamma)} &\leq c\varepsilon^{-1/p}\|u\|_{W^{m,p}(\Omega_\varepsilon)}
  \end{align}
  for all $u\in W^{m,p}(\Omega_\varepsilon)^3$ with $m=1,2$ and $p\in[1,\infty)$.
\end{lemma}

\begin{lemma}[{\cite{Miu_NSCTD_02}*{Lemma 6.12}}] \label{L:Ave_Der_Diff}
  There exists a constant $c>0$ independent of $\varepsilon$ such that
  \begin{align} \label{E:ADD_Dom}
    \left\|\overline{P}\nabla\varphi-\overline{\nabla_\Gamma M\varphi}\right\|_{L^p(\Omega_\varepsilon)} &\leq c\varepsilon\|\varphi\|_{W^{2,p}(\Omega_\varepsilon)}
  \end{align}
  for all $\varphi\in W^{2,p}(\Omega_\varepsilon)$ with $p\in[1,\infty)$.
\end{lemma}

\begin{lemma}[{\cite{Miu_NSCTD_02}*{Lemma 6.13}}] \label{L:Ave_N_W1p}
  There exists a constant $c>0$ independent of $\varepsilon$ such that
  \begin{align} \label{E:Ave_N_W1p}
    \|Mu\cdot n\|_{W^{1,p}(\Gamma)} \leq c\varepsilon^{1-1/p}\|u\|_{W^{2,p}(\Omega_\varepsilon)}
  \end{align}
  for all $u\in W^{2,p}(\Omega_\varepsilon)^3$with $p\in[1,\infty)$ satisfying \eqref{E:Bo_Imp} on $\Gamma_\varepsilon^0$ or on $\Gamma_\varepsilon^1$.
\end{lemma}

\begin{lemma}[{\cite{Miu_NSCTD_02}*{Lemma 6.15}}] \label{L:ADiv_Tan}
  There exists a constant $c>0$ independent of $\varepsilon$ such that
  \begin{align} \label{E:ADiv_Tan_Lp}
    \|\mathrm{div}_\Gamma(gM_\tau u)\|_{L^p(\Gamma)} \leq c\varepsilon^{1-1/p}\|u\|_{W^{1,p}(\Omega_\varepsilon)}
  \end{align}
  for all $u\in W^{1,p}(\Omega_\varepsilon)^3$ with $p\in[1,\infty)$ satisfying $\mathrm{div}\,u=0$ in $\Omega_\varepsilon$ and \eqref{E:Bo_Imp}.
\end{lemma}

By \eqref{E:ADD_Dom} and an $L^p(\Gamma)$-estimate for $\mathrm{div}_\Gamma(gMu)$ similar to \eqref{E:ADiv_Tan_Lp} shown in \cite{Miu_NSCTD_02} we also have the next estimate for the normal component of $\partial_nu$.

\begin{lemma}[{\cite{Miu_NSCTD_02}*{Lemma 6.16}}] \label{L:DnU_N_Ave}
  There exists a constant $c>0$ independent of $\varepsilon$ such that
  \begin{align} \label{E:DnU_N_Ave}
    \left\|\partial_nu\cdot\bar{n}-\frac{1}{\bar{g}}\overline{M_\tau u}\cdot\overline{\nabla_\Gamma g}\right\|_{L^p(\Omega_\varepsilon)} \leq c\varepsilon\|u\|_{W^{2,p}(\Omega_\varepsilon)}
  \end{align}
  for all $u\in W^{2,p}(\Omega_\varepsilon)^3$ with $p\in[1,\infty)$ satisfying $\mathrm{div}\,u=0$ in $\Omega_\varepsilon$ and \eqref{E:Bo_Imp}.
\end{lemma}

For $u\in L^2(\Omega_\varepsilon)^3$ we can consider $\mathrm{div}_\Gamma(gM_\tau u)$ as an element of $H^{-1}(\Gamma)$ by \eqref{E:Sdiv_Hin}.
Moreover, we have an $H^{-1}(\Gamma)$-estimate for it similar to \eqref{E:ADiv_Tan_Lp} when $u\in L_\sigma^2(\Omega_\varepsilon)$.

\begin{lemma} \label{L:ADiv_Tan_Hin}
  There exists a constant $c>0$ independent of $\varepsilon$ such that
  \begin{align} \label{E:ADiv_Tan_Hin}
    \|\mathrm{div}_\Gamma(gM_\tau u)\|_{H^{-1}(\Gamma)} \leq c\varepsilon^{1/2}\|u\|_{L^2(\Omega_\varepsilon)}
  \end{align}
  for all $u\in L_\sigma^2(\Omega_\varepsilon)$.
\end{lemma}

\begin{proof}
  We use the notation \eqref{E:Pull_Dom} and suppress the arguments of functions.
  Let $\eta$ be an arbitrary function in $H^1(\Gamma)$.
  By \eqref{E:Sdiv_Hin} and $M_\tau u\cdot n=0$ on $\Gamma$,
  \begin{align} \label{Pf_ADTH:Equ}
    \langle\mathrm{div}_\Gamma(gM_\tau u),\eta\rangle_\Gamma = -\int_\Gamma gM_\tau u\cdot\nabla_\Gamma\eta\,d\mathcal{H}^2.
  \end{align}
  For the right-hand side we see by $M_\tau u\cdot\nabla_\Gamma\eta=Mu\cdot\nabla_\Gamma\eta$ on $\Gamma$ and \eqref{E:Def_Ave} that
  \begin{align*}
    \int_\Gamma gM_\tau u\cdot\nabla_\Gamma\eta\,d\mathcal{H}^2 = \varepsilon^{-1}\int_\Gamma\int_{\varepsilon g_0}^{\varepsilon g_1}u^\sharp\cdot\nabla_\Gamma\eta\,dr\,d\mathcal{H}^2.
  \end{align*}
  From this equality and the change of variables formula \eqref{E:CoV_Dom} it follows that
  \begin{align*}
    \left|\varepsilon^{-1}\int_{\Omega_\varepsilon}u\cdot\nabla\bar{\eta}\,dx-\int_\Gamma gM_\tau u\cdot\nabla_\Gamma\eta\,d\mathcal{H}^2\right| \leq \varepsilon^{-1}(J_1+J_2),
  \end{align*}
  where $\bar{\eta}=\eta\circ\pi$ is the constant extension of $\eta$ and
  \begin{align*}
    J_1 := \left|\int_{\Omega_\varepsilon}u\cdot\Bigl(\nabla\bar{\eta}-\overline{\nabla_\Gamma\eta}\Bigr)dx\right|, \quad J_2 := \left|\int_\Gamma\int_{\varepsilon g_0}^{\varepsilon g_1}(u^\sharp\cdot\nabla_\Gamma\eta)(J-1)\,dr\,d\mathcal{H}^2\right|.
  \end{align*}
  By \eqref{E:ConDer_Diff} with $|d|\leq c\varepsilon$ in $\Omega_\varepsilon$, H\"{o}lder's inequality, and \eqref{E:Con_Lp} we have
  \begin{align*}
    J_1 \leq c\varepsilon\|u\|_{L^2(\Omega_\varepsilon)}\left\|\overline{\nabla_\Gamma\eta}\right\|_{L^2(\Omega_\varepsilon)} \leq c\varepsilon^{3/2}\|u\|_{L^2(\Omega_\varepsilon)}\|\nabla_\Gamma\eta\|_{L^2(\Gamma)}.
  \end{align*}
  We also have the same inequality for $J_2$ by \eqref{E:Jac_Diff_03}, \eqref{E:CoV_Equiv}, and \eqref{E:Con_Lp}.
  Hence
  \begin{multline} \label{Pf_ADTH:Ine}
    \left|\varepsilon^{-1}\int_{\Omega_\varepsilon}u\cdot\nabla\bar{\eta}\,dx-\int_\Gamma gM_\tau u\cdot\nabla_\Gamma\eta\,d\mathcal{H}^2\right| \\
    \leq \varepsilon^{-1}(J_1+J_2) \leq c\varepsilon^{1/2}\|u\|_{L^2(\Omega_\varepsilon)}\|\nabla_\Gamma\eta\|_{L^2(\Gamma)}.
  \end{multline}
  Moreover, noting that $\bar{\eta}\in H^1(\Omega_\varepsilon)$ by $\eta\in H^1(\Gamma)$ and Lemma \ref{L:Con_Lp_W1p}, we have
  \begin{align*}
    \int_{\Omega_\varepsilon}u\cdot\nabla\bar{\eta}\,dx = 0
  \end{align*}
  by $u\in L_\sigma^2(\Omega_\varepsilon)$ and $\nabla\bar{\eta}\in L_\sigma^2(\Omega_\varepsilon)^\perp$.
  Thus by \eqref{Pf_ADTH:Equ} and \eqref{Pf_ADTH:Ine} we get
  \begin{align*}
    |\langle\mathrm{div}_\Gamma(gM_\tau u),\eta\rangle_\Gamma| = \left|\int_\Gamma gM_\tau u\cdot\nabla_\Gamma\eta\,d\mathcal{H}^2\right| \leq c\varepsilon^{1/2}\|u\|_{L^2(\Omega_\varepsilon)}\|\nabla_\Gamma\eta\|_{L^2(\Gamma)}.
  \end{align*}
  Since this inequality holds for all $\eta\in H^1(\Gamma)$, we obtain \eqref{E:ADiv_Tan_Hin}.
\end{proof}

We also observe that $M$ and $M_\tau$ commute with the time derivative.

\begin{lemma} \label{L:Ave_Dt}
  For $T>0$ let $\varphi\in H^1(0,T;L^2(\Omega_\varepsilon))$.
  Then
  \begin{align} \label{E:ADt_Com}
    M\varphi \in H^1(0,T;L^2(\Gamma)), \quad \partial_tM\varphi = M(\partial_t\varphi) \quad\text{in}\quad L^2(0,T;L^2(\Gamma))
  \end{align}
  and there exists a constant $c>0$ independent of $\varepsilon$ and $\varphi$ such that
  \begin{align} \label{E:Ave_Dt}
    \|\partial_tM\varphi\|_{L^2(0,T;L^2(\Gamma))} \leq c\varepsilon^{-1/2}\|\partial_t\varphi\|_{L^2(0,T;L^2(\Omega_\varepsilon))}.
  \end{align}
  Also, if $u\in H^1(0,T;L^2(\Omega_\varepsilon)^3)$, then
  \begin{align*}
    M_\tau u\in H^1(0,T;L^2(\Gamma,T\Gamma)), \quad \partial_tM_\tau u = M_\tau(\partial_tu) \quad\text{in}\quad L^2(0,T;L^2(\Gamma,T\Gamma)),
  \end{align*}
  and \eqref{E:Ave_Dt} holds with $\partial_t\varphi$ and $\partial_tM\varphi$ replaced by $\partial_tu$ and $\partial_tM_\tau u$.
\end{lemma}

\begin{proof}
  First note that $M(\partial_t\varphi) \in L^2(0,T;L^2(\Gamma))$ and
  \begin{align} \label{Pf_DtA:Ineq}
    \|M(\partial_t\varphi)\|_{L^2(0,T;L^2(\Gamma))} \leq c\varepsilon^{-1/2}\|\partial_t\varphi\|_{L^2(0,T;L^2(\Omega_\varepsilon))}
  \end{align}
  by $\partial_t\varphi\in L^2(0,T;L^2(\Omega_\varepsilon))$ and \eqref{E:Ave_Lp_Surf}.
  The relations \eqref{E:ADt_Com} are formally trivial since the definition \eqref{E:Def_Ave} of $M$ is independent of time.
  To verify them rigorously we prove
  \begin{align*}
    \int_0^T\partial_t\xi(t)M\varphi(t)\,dt = -\int_0^T\xi(t)[M(\partial_t\varphi)](t)\,dt \quad\text{in}\quad L^2(\Gamma)
  \end{align*}
  for all $\xi\in C_c^\infty(0,T)$.
  Since $L^2(\Gamma)$ is a Hilbert space, this is equivalent to
  \begin{align} \label{Pf_DtA:Goal}
    \int_0^T\partial_t\xi(t)(M\varphi(t),\eta)_{L^2(\Gamma)}\,dt = -\int_0^T\xi(t)([M(\partial_t\varphi)](t),\eta)_{L^2(\Gamma)}\,dt
  \end{align}
  for all $\xi\in C_c^\infty(0,T)$ and $\eta\in L^2(\Gamma)$.
  We define a function $\check{\eta}$ on $\Omega_\varepsilon$ by
  \begin{align*}
    \check{\eta}(x) := \frac{\eta(\pi(x))}{\varepsilon g(\pi(x))J(\pi(x),d(x))}, \quad x\in\Omega_\varepsilon.
  \end{align*}
  Then $\check{\eta}\in L^2(\Omega_\varepsilon)$ by \eqref{E:G_Inf}, \eqref{E:Jac_Bound_03}, and \eqref{E:Con_Lp}.
  Also, we observe by \eqref{E:CoV_Dom} and \eqref{E:Def_Ave} that (here we use the notation \eqref{E:Pull_Dom})
  \begin{align} \label{Pf_DtA:CoV}
    (\varphi(t),\check{\eta})_{L^2(\Omega_\varepsilon)} = \int_\Gamma\left(\frac{1}{\varepsilon g}\int_{\varepsilon g_0}^{\varepsilon g_1}\varphi^\sharp(t)\,dr\right)\eta\,d\mathcal{H}^2 = (M\varphi(t),\eta)_{L^2(\Gamma)}
  \end{align}
  for a.a. $t\in(0,T)$.
  Noting that $\varphi\in H^1(0,T;L^2(\Omega_\varepsilon))$ and $\check{\eta}\in L^2(\Omega_\varepsilon)$ is independent of time, we deduce from \eqref{Pf_DtA:CoV} that
  \begin{align*}
    \int_0^T\partial_t\xi(t)(M\varphi(t),\eta)_{L^2(\Gamma)}\,dt &= \int_0^T\partial_t\xi(t)(\varphi(t),\check{\eta})_{L^2(\Omega_\varepsilon)}\,dt \\
    &= -\int_0^T\xi(t)(\partial_t\varphi(t),\check{\eta})_{L^2(\Omega_\varepsilon)}\,dt.
  \end{align*}
  Moreover, the last term is equal to the right-hand side of \eqref{Pf_DtA:Goal} by \eqref{Pf_DtA:CoV} with $\varphi(t)$ replaced by $\partial_t\varphi(t)$.
  Hence \eqref{Pf_DtA:Goal} is valid and we obtain \eqref{E:ADt_Com}.
  Also, \eqref{E:Ave_Dt} follows from \eqref{Pf_DtA:Ineq}.
  The statements for $u\in H^1(0,T;L^2(\Omega_\varepsilon)^3)$ are shown by \eqref{E:Def_Tan_Ave} and \eqref{E:AveT_Lp_Surf} in the same way.
\end{proof}

\subsection{Decomposition of a vector field into the average and residual parts} \label{SS:Ave_Decom}
In the second part \cite{Miu_NSCTD_02} of our study we established the global existence of a strong solution to \eqref{E:NS_CTD} by showing a good estimate for the trilinear term, i.e. the $L^2(\Omega_\varepsilon)$-inner product of the convection and viscous terms.
A key tool for the proof of that estimate was a good decomposition of a vector field on $\Omega_\varepsilon$.
Using the operators $E_\varepsilon$ and $M_\tau$ given by \eqref{E:Def_ExImp} and \eqref{E:Def_Tan_Ave} we decomposed a vector field on $\Omega_\varepsilon$ into the almost two-dimensional average part and the residual part, and derived good estimates for them separately.
Let us recall the definitions of the average and residual parts.

\begin{definition} \label{D:Def_ExAve}
  For a vector field $u$ on $\Omega_\varepsilon$ we define its average part by
  \begin{align} \label{E:Def_ExAve}
    u^a(x) := E_\varepsilon M_\tau u(x) = \overline{M_\tau u}(x)+\left\{\overline{M_\tau u}(x)\cdot\Psi_\varepsilon(x)\right\}\bar{n}(x), \quad x\in N,
  \end{align}
  where $\Psi_\varepsilon$ is the vector field given by \eqref{E:Def_ExAux} and $M_\tau u$ is the averaged tangential component of $u$ given by \eqref{E:Def_Tan_Ave}, and set the residual part $u^r:=u-u^a$.
\end{definition}

By Lemmas \ref{L:ExImp_Wmp}, \ref{L:AveT_Lp}, and \ref{L:AveT_Wmp} the average and residual parts belong to the same Sobolev space that an original vector field belongs to.

\begin{lemma} \label{L:Wmp_UaUr}
  There exists a constant $c>0$ independent of $\varepsilon$ such that
  \begin{align*}
    \|u^a\|_{W^{m,p}(\Omega_\varepsilon)} \leq c\|u\|_{W^{m,p}(\Omega_\varepsilon)}, \quad \|u^r\|_{W^{m,p}(\Omega_\varepsilon)} \leq c\|u\|_{W^{m,p}(\Omega_\varepsilon)}
  \end{align*}
  for all $u\in W^{m,p}(\Omega_\varepsilon)^3$ with $m=0,1,2$ and $p\in[1,\infty)$.
\end{lemma}

Since $u^a$ is close to a vector field on the two-dimensional surface $\Gamma$, we can show a good $L^2(\Omega_\varepsilon)$-estimate for the product of $u^a$ and a function on $\Omega_\varepsilon$ by applying the next product estimate on $\Omega_\varepsilon$ similar to a two-dimensional one.

\begin{lemma}[{\cite{Miu_NSCTD_02}*{Lemma 6.19}}] \label{L:Prod}
  There exists a constant $c>0$ independent of $\varepsilon$ such that
  \begin{align} \label{E:Prod_Surf}
    \|\bar{\eta}\varphi\|_{L^2(\Omega_\varepsilon)} &\leq c\|\eta\|_{L^2(\Gamma)}^{1/2}\|\eta\|_{H^1(\Gamma)}^{1/2}\|\varphi\|_{L^2(\Omega_\varepsilon)}^{1/2}\|\varphi\|_{H^1(\Omega_\varepsilon)}^{1/2}
  \end{align}
  for all $\eta\in H^1(\Gamma)$ and $\varphi\in H^1(\Omega_\varepsilon)$.
\end{lemma}

\begin{lemma}[{\cite{Miu_NSCTD_02}*{Lemma 6.20}}] \label{L:Prod_Ua}
  For $\varphi\in H^1(\Omega_\varepsilon)$, $u\in H^1(\Omega_\varepsilon)^3$, and $u^a$ given by \eqref{E:Def_ExAve} we have
  \begin{align} \label{E:Prod_Ua}
    \bigl\|\,|u^a|\,\varphi\bigr\|_{L^2(\Omega_\varepsilon)} &\leq c\varepsilon^{-1/2}\|\varphi\|_{L^2(\Omega_\varepsilon)}^{1/2}\|\varphi\|_{H^1(\Omega_\varepsilon)}^{1/2}\|u\|_{L^2(\Omega_\varepsilon)}^{1/2}\|u\|_{H^1(\Omega_\varepsilon)}^{1/2}
  \end{align}
  with a constant $c>0$ independent of $\varepsilon$, $\varphi$, and $u$.
\end{lemma}

For the residual part $u^r$ the following $L^\infty(\Omega_\varepsilon)$-estimate holds as a consequence of an anisotropic Agmon inequality on $\Omega_\varepsilon$ (see \cite{Miu_NSCTD_02}*{Lemma 4.3}) and Poincar\'{e} type inequalities for $u^r$ and $\nabla u^r$ (see \cite{Miu_NSCTD_02}*{Lemmas 6.21 and 6.22}).

\begin{lemma}[{\cite{Miu_NSCTD_02}*{Lemma 6.23}}] \label{L:Linf_Ur}
  Suppose that the inequalities \eqref{E:Fric_Upper} are valid and $u\in H^2(\Omega_\varepsilon)^3$ satisfies $\mathrm{div}\,u=0$ in $\Omega_\varepsilon$ and \eqref{E:Bo_Slip}.
  Then
  \begin{align} \label{E:Linf_Ur}
    \|u^r\|_{L^\infty(\Omega_\varepsilon)} \leq c\left(\varepsilon^{1/2}\|u\|_{H^2(\Omega_\varepsilon)}+\|u\|_{L^2(\Omega_\varepsilon)}^{1/2}\|u\|_{H^2(\Omega_\varepsilon)}^{1/2}\right)
  \end{align}
  for $u^r=u-u^a$, where $c>0$ is a constant independent of $\varepsilon$ and $u$.
\end{lemma}

Using \eqref{E:Prod_Ua} and \eqref{E:Linf_Ur} we estimate the $L^2(\Omega_\varepsilon)$-norms of $u\otimes u$ and $(u\cdot\nabla)u$.

\begin{lemma} \label{L:Est_UU}
  Suppose that the inequalities \eqref{E:Fric_Upper} are valid and $u\in H^2(\Omega_\varepsilon)^3$ satisfies $\mathrm{div}\,u=0$ in $\Omega_\varepsilon$ and \eqref{E:Bo_Slip}.
  Then
  \begin{multline} \label{E:Est_UU}
    \|u\otimes u\|_{L^2(\Omega_\varepsilon)} \leq c\left(\varepsilon^{-1/2}\|u\|_{L^2(\Omega_\varepsilon)}\|u\|_{H^1(\Omega_\varepsilon)}\right. \\
    \left.+\varepsilon^{1/2}\|u\|_{L^2(\Omega_\varepsilon)}\|u\|_{H^2(\Omega_\varepsilon)}+\|u\|_{L^2(\Omega_\varepsilon)}^{3/2}\|u\|_{H^2(\Omega_\varepsilon)}^{1/2}\right)
  \end{multline}
  and
  \begin{multline} \label{E:Est_UGU}
    \|(u\cdot\nabla)u\|_{L^2(\Omega_\varepsilon)} \leq c\left(\varepsilon^{-1/2}\|u\|_{L^2(\Omega_\varepsilon)}^{1/2}\|u\|_{H^1(\Omega_\varepsilon)}\|u\|_{H^2(\Omega_\varepsilon)}^{1/2}\right. \\
    \left.+\varepsilon^{1/2}\|u\|_{H^1(\Omega_\varepsilon)}\|u\|_{H^2(\Omega_\varepsilon)}\right)
  \end{multline}
  with a constant $c>0$ independent of $\varepsilon$ and $u$.
\end{lemma}

\begin{proof}
  By the assumptions of the lemma we can use \eqref{E:Prod_Ua} and \eqref{E:Linf_Ur} to get
  \begin{align*}
    \|u^a\otimes u\|_{L^2(\Omega_\varepsilon)} &\leq c\varepsilon^{-1/2}\|u\|_{L^2(\Omega_\varepsilon)}\|u\|_{H^1(\Omega_\varepsilon)}, \\
    \|u^r\otimes u\|_{L^2(\Omega_\varepsilon)} &\leq \|u^r\|_{L^\infty(\Omega_\varepsilon)}\|u\|_{L^2(\Omega_\varepsilon)} \\
    &\leq c\left(\varepsilon^{1/2}\|u\|_{H^2(\Omega_\varepsilon)}+\|u\|_{L^2(\Omega_\varepsilon)}^{1/2}\|u\|_{H^2(\Omega_\varepsilon)}^{1/2}\right)\|u\|_{L^2(\Omega_\varepsilon)}.
  \end{align*}
  Applying these inequalities to the right-hand side of
  \begin{align*}
    \|u\otimes u\|_{L^2(\Omega_\varepsilon)} \leq \|u^a\otimes u\|_{L^2(\Omega_\varepsilon)}+\|u^r\otimes u\|_{L^2(\Omega_\varepsilon)}
  \end{align*}
  we obtain \eqref{E:Est_UU}.
  From \eqref{E:Prod_Ua} and \eqref{E:Linf_Ur} we also deduce that
  \begin{align*}
    \|(u^a\cdot\nabla)u\|_{L^2(\Omega_\varepsilon)} &\leq c\varepsilon^{-1/2}\|u\|_{L^2(\Omega_\varepsilon)}^{1/2}\|u\|_{H^1(\Omega_\varepsilon)}\|u\|_{H^2(\Omega_\varepsilon)}^{1/2}, \\
    \|(u^r\cdot\nabla)u\|_{L^2(\Omega_\varepsilon)} & \leq \|u^r\|_{L^\infty(\Omega_\varepsilon)}\|u\|_{H^1(\Omega_\varepsilon)} \\
    &\leq c\left(\|u\|_{L^2(\Omega_\varepsilon)}^{1/2}\|u\|_{H^2(\Omega_\varepsilon)}^{1/2}+\varepsilon^{1/2}\|u\|_{H^2(\Omega_\varepsilon)}\right)\|u\|_{H^1(\Omega_\varepsilon)}.
  \end{align*}
  Noting that $1\leq\varepsilon^{-1/2}$ by $\varepsilon\leq1$, we apply these estimates to
  \begin{align*}
    \|(u\cdot\nabla)u\|_{L^2(\Omega_\varepsilon)} \leq \|(u^a\cdot\nabla)u\|_{L^2(\Omega_\varepsilon)}+\|(u^r\cdot\nabla)u\|_{L^2(\Omega_\varepsilon)}
  \end{align*}
  to get \eqref{E:Est_UGU}.
\end{proof}

\subsection{Average of bilinear and trilinear forms} \label{SS:Ave_BT}
In Section \ref{S:SL} we transform a weak formulation of the bulk equations \eqref{E:NS_CTD} into that of the limit equations \eqref{E:NS_Limit} with a residual term satisfied by the averaged tangential component of a strong solution to \eqref{E:NS_CTD}.
To this end, we consider approximation of bilinear and trilinear forms for functions on $\Omega_\varepsilon$ by those for functions on $\Gamma$ and the average operators.

In what follows, we use the notations \eqref{E:Pull_Dom} and \eqref{E:Pull_Bo} for functions on $\Omega_\varepsilon$ and $\Gamma_\varepsilon^i$, $i=0,1$ and suppress the arguments of functions.

First we deal with the $L^2$-inner products on $\Omega_\varepsilon$ and $\Gamma_\varepsilon^i$, $i=0,1$.

\begin{lemma} \label{L:Ave_BiL2_Dom}
  There exists a constant $c>0$ independent of $\varepsilon$ such that
  \begin{align} \label{E:Ave_BiL2_Dom}
    \left|\int_{\Omega_\varepsilon}\varphi\bar{\eta}\,dx-\varepsilon\int_\Gamma g(M\varphi)\eta\,d\mathcal{H}^2\right| \leq c\varepsilon^{3/2}\|\varphi\|_{L^2(\Omega_\varepsilon)}\|\eta\|_{L^2(\Gamma)}
  \end{align}
  for all $\varphi\in L^2(\Omega_\varepsilon)$ and $\eta\in L^2(\Gamma)$.
\end{lemma}

\begin{proof}
  By the formula \eqref{E:CoV_Dom} and the definition \eqref{E:Def_Ave} of $M$,
  \begin{align*}
    \int_{\Omega_\varepsilon}\varphi\bar{\eta}\,dx-\varepsilon\int_\Gamma g(M\varphi)\eta\,d\mathcal{H}^2 = \int_\Gamma\int_{\varepsilon g_0}^{\varepsilon g_1}\varphi^\sharp\eta(J-1)\,dr\,d\mathcal{H}^2.
  \end{align*}
  Moreover, by \eqref{E:Jac_Diff_03}, \eqref{E:CoV_Equiv}, H\"{o}lder's inequality, and \eqref{E:Con_Lp} we have
  \begin{align*}
    \left|\int_\Gamma\int_{\varepsilon g_0}^{\varepsilon g_1}\varphi^\sharp\eta(J-1)\,dr\,d\mathcal{H}^2\right| &\leq c\varepsilon \int_\Gamma\int_{\varepsilon g_0}^{\varepsilon g_1}|\varphi^\sharp||\eta|\,dr\,d\mathcal{H}^2 \leq c\varepsilon\int_{\Omega_\varepsilon}|\varphi||\bar{\eta}|\,dx \\
    &\leq c\varepsilon\|\varphi\|_{L^2(\Omega_\varepsilon)}\|\bar{\eta}\|_{L^2(\Omega_\varepsilon)} \\
    &\leq c\varepsilon^{3/2}\|\varphi\|_{L^2(\Omega_\varepsilon)}\|\eta\|_{L^2(\Gamma)}.
  \end{align*}
  Thus \eqref{E:Ave_BiL2_Dom} follows.
\end{proof}

\begin{lemma} \label{L:Ave_BiL2_Bo}
  There exists a constant $c>0$ independent of $\varepsilon$ such that
  \begin{align} \label{E:Ave_BiL2_Bo}
    \left|\int_{\Gamma_\varepsilon^i}\varphi\bar{\eta}\,d\mathcal{H}^2-\int_\Gamma(M\varphi)\eta\,d\mathcal{H}^2\right| \leq \varepsilon^{1/2}\|\varphi\|_{H^1(\Omega_\varepsilon)}\|\eta\|_{L^2(\Gamma)}, \quad i=0,1
  \end{align}
  for all $\varphi\in H^1(\Omega_\varepsilon)$ and $\eta\in L^2(\Gamma)$.
\end{lemma}

\begin{proof}
  Let $J_i^\sharp(y):=J(y,\varepsilon g_i(y))$ for $y\in\Gamma$ and
  \begin{align*}
    K_1 := \int_\Gamma\varphi_i^\sharp\eta\left(J_i^\sharp\sqrt{1+\varepsilon^2|\tau_\varepsilon^i|^2}-1\right)d\mathcal{H}^2, \quad K_2 := \int_\Gamma(\varphi_i^\sharp-M\varphi)\eta\,d\mathcal{H}^2.
  \end{align*}
  Then by the change of variables formula \eqref{E:CoV_Surf} we have
  \begin{align} \label{Pf_ABLB:Split}
    \int_{\Gamma_\varepsilon^i}\varphi\bar{\eta}\,d\mathcal{H}^2-\int_\Gamma(M\varphi)\eta\,d\mathcal{H}^2 = K_1+K_2.
  \end{align}
  By the mean value theorem for the function $\sqrt{1+s}$, $s\geq 0$ and \eqref{E:Tau_Bound},
  \begin{align*}
    0 \leq \sqrt{1+\varepsilon^2|\tau_\varepsilon^i|^2}-1 \leq \frac{\varepsilon^2}{2}|\tau_\varepsilon^i|^2 \leq c\varepsilon^2 \quad\text{on}\quad \Gamma.
  \end{align*}
  This inequality, \eqref{E:Tau_Bound}, and \eqref{E:Jac_Diff_03} imply that
  \begin{align*}
    \left|J_i^\sharp\sqrt{1+\varepsilon^2|\tau_\varepsilon^i|^2}-1\right| \leq |J_i^\sharp-1|\sqrt{1+\varepsilon^2|\tau_\varepsilon^i|^2}+\left(\sqrt{1+\varepsilon^2|\tau_\varepsilon^i|^2}-1\right) \leq c\varepsilon
  \end{align*}
  on $\Gamma$.
  From this inequality, \eqref{E:Lp_CoV_Surf}, and \eqref{E:Poin_Bo} we deduce that
  \begin{align*}
    |K_1| &\leq c\varepsilon\|\varphi_i^\sharp\|_{L^2(\Gamma)}\|\eta\|_{L^2(\Gamma)} \leq c\varepsilon\|\varphi\|_{L^2(\Gamma_\varepsilon^i)}\|\eta\|_{L^2(\Gamma)} \\
    &\leq c\varepsilon^{1/2}\|\varphi\|_{H^1(\Omega_\varepsilon)}\|\eta\|_{L^2(\Gamma)}.
  \end{align*}
  Also, noting that
  \begin{align} \label{Pf_ABLB:Pull_Bar}
    M\varphi(y) = \overline{M\varphi}(y+\varepsilon g_i(y)n(y)) = \Bigl(\overline{M\varphi}\Bigr)_i^\sharp(y), \quad y\in\Gamma,
  \end{align}
  we observe by \eqref{E:Lp_CoV_Surf} and \eqref{E:Ave_Diff_Bo} that
  \begin{align*}
    |K_2| &\leq \|\varphi_i^\sharp-M\varphi\|_{L^2(\Gamma)}\|\eta\|_{L^2(\Gamma)} \leq c\left\|\varphi-\overline{M\varphi}\right\|_{L^2(\Gamma_\varepsilon^i)}\|\eta\|_{L^2(\Gamma)} \\
    &\leq c\varepsilon^{1/2}\|\varphi\|_{H^1(\Omega_\varepsilon)}\|\eta\|_{L^2(\Gamma)}.
  \end{align*}
  Applying the above estimates for $K_1$ and $K_2$ to \eqref{Pf_ABLB:Split} we obtain \eqref{E:Ave_BiL2_Bo}.
\end{proof}

Next we consider bilinear forms including the strain rate tensor
\begin{align*}
  D(u) = (\nabla u)_S = \frac{\nabla u+(\nabla u)^T}{2}
\end{align*}
for a vector field $u$ on $\Omega_\varepsilon$.

\begin{lemma} \label{L:Ave_BiH1_TT}
  There exists a constant $c>0$ independent of $\varepsilon$ such that
  \begin{align} \label{E:Ave_BiH1_TT}
    \left|\int_{\Omega_\varepsilon}D(u):\overline{A}\,dx-\varepsilon\int_\Gamma gD_\Gamma(M_\tau u):A\,d\mathcal{H}^2\right| \leq c\varepsilon^{3/2}\|u\|_{H^1(\Omega_\varepsilon)}\|A\|_{L^2(\Gamma)}
  \end{align}
  for all $u\in H^1(\Omega_\varepsilon)^3$ satisfying \eqref{E:Bo_Imp} on $\Gamma_\varepsilon^0$ or on $\Gamma_\varepsilon^1$ and $A\in L^2(\Gamma)^{3\times3}$ satisfying
  \begin{align} \label{E:A_Cond}
     PA = AP = A \quad\text{on}\quad \Gamma.
   \end{align}
  Here $D_\Gamma(M_\tau u)$ is the surface strain rate tensor given by \eqref{E:Def_SSR}.
\end{lemma}

\begin{proof}
  By \eqref{E:A_Cond} and $P^T=P$ on $\Gamma$,
  \begin{align*}
    PM\bigl(D(u)\bigr)P:A = M\bigl(D(u)\bigr):A \quad\text{on}\quad \Gamma.
  \end{align*}
  From this equality, \eqref{E:Ave_BiL2_Dom}, and $\|D(u)\|_{L^2(\Omega_\varepsilon)}\leq c\|u\|_{H^1(\Omega_\varepsilon)}$ it follows that
  \begin{align} \label{Pf_ABTT:First}
    \left|\int_{\Omega_\varepsilon}D(u):\overline{A}\,dx-\varepsilon\int_\Gamma gPM\bigl(D(u)\bigr)P:A\,d\mathcal{H}^2\right| \leq c\varepsilon^{3/2}\|u\|_{H^1(\Omega_\varepsilon)}\|A\|_{L^2(\Gamma)}.
  \end{align}
  Next we observe by $P^T=P$ and $P\nabla_\Gamma Mu=\nabla_\Gamma Mu$ on $\Gamma$ that
  \begin{align*}
    PM\bigl(D(u)\bigr)P-D_\Gamma(Mu) &= [PM(\nabla u)P-P(\nabla_\Gamma Mu)P]_S \\
    &= [\{PM(\nabla u)-\nabla_\Gamma Mu\}P]_S
  \end{align*}
  on $\Gamma$, where $B_S=(B+B^T)/2$ for a $3\times 3$ matrix $B$.
  Moreover,
  \begin{align*}
    |PM(\nabla u)-\nabla_\Gamma Mu| &\leq \left|M\Bigl(d\overline{W}\nabla u\Bigr)\right|+|M(\psi_\varepsilon\otimes \partial_nu)| \leq c\varepsilon M(|\nabla u|)
  \end{align*}
  on $\Gamma$ by \eqref{E:Ave_Der}--\eqref{E:ADA_Bound} and $|d|\leq c\varepsilon$ in $\Omega_\varepsilon$.
  Hence (note that $|P|=2$ on $\Gamma$)
  \begin{align*}
    \left|PM\bigl(D(u)\bigr)P-D_\Gamma(Mu)\right| \leq |\{PM(\nabla u)-\nabla_\Gamma Mu\}P| \leq c\varepsilon M(|\nabla u|)
  \end{align*}
  on $\Gamma$.
  This inequality, the boundedness of $g$ on $\Gamma$, and \eqref{E:Ave_Lp_Surf} imply that
  \begin{multline} \label{Pf_ABTT:Second}
    \left|\int_\Gamma gPM\bigl(D(u)\bigr)P:A\,d\mathcal{H}^2-\int_\Gamma gD_\Gamma(Mu):A\,d\mathcal{H}^2\right| \\
    \leq c\varepsilon\|M(|\nabla u|)\|_{L^2(\Gamma)}\|A\|_{L^2(\Gamma)} \leq c\varepsilon^{1/2}\|u\|_{H^1(\Omega_\varepsilon)}\|A\|_{L^2(\Gamma)}.
  \end{multline}
  Now we use $Mu=(Mu\cdot n)n+M_\tau u$ and $-\nabla_\Gamma n=W$ on $\Gamma$ to get
  \begin{align*}
    \nabla_\Gamma Mu = [\nabla_\Gamma(Mu\cdot n)]\otimes n-(Mu\cdot n)W+\nabla_\Gamma M_\tau u \quad\text{on}\quad \Gamma.
  \end{align*}
  By this equality,
  \begin{align*}
    (a\otimes n)P = a\otimes(P^Tn) = a\otimes(Pn) = 0 \quad\text{on}\quad \Gamma
  \end{align*}
  for $a\in\mathbb{R}^3$, and \eqref{E:Form_W} we see that
  \begin{align*}
    P(\nabla_\Gamma Mu)P-P(\nabla_\Gamma M_\tau u)P = -(Mu\cdot n)PWP = -(Mu\cdot n)W \quad\text{on}\quad \Gamma.
  \end{align*}
  Since $W$ is bounded on $\Gamma$, it follows from the above equality that
  \begin{align*}
    |D_\Gamma(Mu)-D_\Gamma(M_\tau u)| \leq |P(\nabla_\Gamma Mu)P-P(\nabla_\Gamma M_\tau u)P| \leq c|Mu\cdot n| \quad\text{on}\quad \Gamma.
  \end{align*}
  By this inequality and \eqref{E:Ave_N_Lp} (note that $u$ satisfies \eqref{E:Bo_Imp} on $\Gamma_\varepsilon^0$ or on $\Gamma_\varepsilon^1$) we get
  \begin{align*}
    \left|\int_\Gamma gD_\Gamma(Mu):A\,d\mathcal{H}^2-\int_\Gamma gD_\Gamma(M_\tau u):A\,d\mathcal{H}^2\right| &\leq c\|Mu\cdot n\|_{L^2(\Gamma)}\|A\|_{L^2(\Gamma)} \\
    &\leq c\varepsilon^{1/2}\|u\|_{H^1(\Omega_\varepsilon)}\|A\|_{L^2(\Gamma)}.
  \end{align*}
  Combining this inequality, \eqref{Pf_ABTT:First}, and \eqref{Pf_ABTT:Second} we obtain \eqref{E:Ave_BiH1_TT}.
\end{proof}

\begin{lemma} \label{L:Ave_BiH1_NN}
  There exists a constant $c>0$ independent of $\varepsilon$ such that
  \begin{align} \label{E:Ave_BiH1_NN}
    \left|\int_{\Omega_\varepsilon}\Bigl(D(u):\overline{Q}\Bigr)\bar{\eta}\,dx-\varepsilon\int_\Gamma(M_\tau u\cdot\nabla_\Gamma g)\eta\,d\mathcal{H}^2\right| \leq c\varepsilon^{3/2}\|u\|_{H^1(\Omega_\varepsilon)}\|\eta\|_{L^2(\Gamma)}
  \end{align}
  for all $u\in H^1(\Omega_\varepsilon)^3$ satisfying \eqref{E:Bo_Imp} and $\eta\in L^2(\Gamma)$.
\end{lemma}

\begin{proof}
  By the change of variables formula \eqref{E:CoV_Dom} we have
  \begin{multline*}
    \int_{\Omega_\varepsilon}\Bigl(D(u):\overline{Q}\Bigr)\bar{\eta}\,dx-\int_\Gamma\left(\int_{\varepsilon g_0}^{\varepsilon g_1}D(u)^\sharp:Q\,dr\right)\eta\,d\mathcal{H}^2 \\
    = \int_\Gamma\left(\int_{\varepsilon g_0}^{\varepsilon g_1}D(u)^\sharp:Q\,dr\right)\eta(J-1)\,d\mathcal{H}^2.
  \end{multline*}
  We apply \eqref{E:Jac_Diff_03}, $|Q|=1$ on $\Gamma$, and \eqref{E:CoV_Equiv} to the right-hand side and use H\"{o}lder's inequality and \eqref{E:Con_Lp} to get
  \begin{multline} \label{Pf_ABNN:First}
    \left|\int_{\Omega_\varepsilon}\Bigl(D(u):\overline{Q}\Bigr)\bar{\eta}\,dx-\int_\Gamma\left(\int_{\varepsilon g_0}^{\varepsilon g_1}D(u)^\sharp:Q\,dr\right)\eta\,d\mathcal{H}^2\right| \\
    \leq c \varepsilon\int_{\Omega_\varepsilon}|D(u)||\bar{\eta}|\,dx \leq c\varepsilon^{3/2}\|u\|_{H^1(\Omega_\varepsilon)}\|\eta\|_{L^2(\Gamma)}.
  \end{multline}
  Next we compute the integral of
  \begin{align*}
    D(u)^\sharp(y,r):Q(y) = D(u)(y+rn(y)):Q(y)
  \end{align*}
  with respect to $r$.
  Since $Q(y)=n(y)\otimes n(y)$ is symmetric,
  \begin{align*}
    D(u)(y+rn(y)): Q(y) &= \nabla u(y+rn(y)):n(y)\otimes n(y) \\
    &= [(n(y)\cdot\nabla)u](y+rn(y))\cdot n(y) \\
    &= \frac{\partial}{\partial r}\Bigl(u(y+rn(y))\Bigr)\cdot n(y)
  \end{align*}
  for all $y\in\Gamma$ and $r\in(\varepsilon g_0(y),\varepsilon g_1(y))$.
  Thus
  \begin{align*}
    \int_{\varepsilon g_0(y)}^{\varepsilon g_1(y)}D(u)^\sharp(y,r):Q(y)\,dr &= \left\{\int_{\varepsilon g_0(y)}^{\varepsilon g_1(y)}\frac{\partial}{\partial r}\Bigl(u(y+rn(y))\Bigr)\,dr\right\}\cdot n(y) \\
    &= \{u(y+\varepsilon g_1(y)n(y))-u(y+\varepsilon g_0(y)n(y))\}\cdot n(y)
  \end{align*}
  for all $y\in\Gamma$.
  We use the notation \eqref{E:Pull_Bo} to write the above equality as
  \begin{align*}
    \int_{\varepsilon g_0}^{\varepsilon g_1}D(u)^\sharp:Q\,dr = u_1^\sharp\cdot n-u_0^\sharp\cdot n = \sum_{i=0,1}(-1)^{i+1}u_i^\sharp\cdot n \quad\text{on}\quad \Gamma.
  \end{align*}
  Moreover, since $u$ satisfies \eqref{E:Bo_Imp}, it follows that $u_i^\sharp\cdot n_{\varepsilon,i}^\sharp=0$ on $\Gamma$ for $i=0,1$ and thus
  \begin{align*}
    (-1)^{i+1}u_i^\sharp\cdot n = u_i^\sharp\cdot\{(-1)^{i+1}(n-\varepsilon\nabla_\Gamma g_i)-n_{\varepsilon,i}^\sharp\}+\varepsilon(-1)^{i+1}u_i^\sharp\cdot\nabla_\Gamma g_i
  \end{align*}
  on $\Gamma$ for $i=0,1$.
  By the above two equalities and $g=g_1-g_0$ on $\Gamma$ we have
  \begin{multline*}
    \int_{\varepsilon g_0}^{\varepsilon g_1}D(u)^\sharp:Q\,dr = \sum_{i=0,1}u_i^\sharp\cdot\{(-1)^{i+1}(n-\varepsilon\nabla_\Gamma g_i)-n_{\varepsilon,i}^\sharp\} \\
    +\varepsilon\sum_{i=0,1}(-1)^{i+1}(u_i^\sharp-Mu)\cdot\nabla_\Gamma g_i+\varepsilon Mu\cdot\nabla_\Gamma g
  \end{multline*}
  on $\Gamma$.
  Hence we apply \eqref{E:Comp_N} to the first term on the right-hand side to get
  \begin{align*}
    \left|\int_{\varepsilon g_0}^{\varepsilon g_1}D(u)^\sharp:Q\,dr-\varepsilon Mu\cdot\nabla_\Gamma g\right| \leq c\varepsilon\sum_{i=0,1}(\varepsilon|u_i^\sharp|+|u_i^\sharp-Mu|) \quad\text{on}\quad \Gamma
  \end{align*}
  and it follows from this inequality and H\"{o}lder's inequality that
  \begin{multline*}
    \left|\int_\Gamma\left(\int_{\varepsilon g_0}^{\varepsilon g_1}D(u)^\sharp:Q\,dr\right)\eta\,d\mathcal{H}^2-\varepsilon\int_\Gamma(Mu\cdot\nabla_\Gamma g)\eta\,d\mathcal{H}^2\right| \\
    \leq c\varepsilon\sum_{i=0,1}\left(\varepsilon\|u_i^\sharp\|_{L^2(\Gamma)}+\|u_i^\sharp-Mu\|_{L^2(\Gamma)}\right)\|\eta\|_{L^2(\Gamma)}.
  \end{multline*}
  Moreover, noting that the relation \eqref{Pf_ABLB:Pull_Bar} holds, we have
  \begin{align*}
    \varepsilon\|u_i^\sharp\|_{L^2(\Gamma)}+\|u_i^\sharp-Mu\|_{L^2(\Gamma)} &\leq c\left(\varepsilon\|u\|_{L^2(\Gamma_\varepsilon^i)}+\left\|u-\overline{Mu}\right\|_{L^2(\Gamma_\varepsilon^i)}\right) \\
    &\leq c\varepsilon^{1/2}\|u\|_{H^1(\Omega_\varepsilon)}
  \end{align*}
  by \eqref{E:Lp_CoV_Surf}, \eqref{E:Poin_Bo}, and \eqref{E:Ave_Diff_Bo}.
  Thus we obtain
  \begin{multline} \label{Pf_ABNN:Second}
    \left|\int_\Gamma\left(\int_{\varepsilon g_0}^{\varepsilon g_1}D(u)^\sharp:Q\,dr\right)\eta\,d\mathcal{H}^2-\varepsilon\int_\Gamma(Mu\cdot\nabla_\Gamma g)\eta\,d\mathcal{H}^2\right| \\
    \leq c\varepsilon^{3/2}\|u\|_{H^1(\Omega_\varepsilon)}\|\eta\|_{L^2(\Gamma)}.
  \end{multline}
  Here $Mu\cdot\nabla_\Gamma g=M_\tau u\cdot\nabla_\Gamma g$ on $\Gamma$ since $\nabla_\Gamma g$ is tangential on $\Gamma$.
  Therefore, \eqref{E:Ave_BiH1_NN} follows from \eqref{Pf_ABNN:First} and \eqref{Pf_ABNN:Second}.
\end{proof}

\begin{lemma} \label{L:Ave_BiH1_TN}
  Let $u\in H^2(\Omega_\varepsilon)^3$ and $v\in L^2(\Gamma,T\Gamma)$.
  Suppose that the inequalities \eqref{E:Fric_Upper} are valid and $u$ satisfies \eqref{E:Bo_Slip} on $\Gamma_\varepsilon^0$ or on $\Gamma_\varepsilon^1$.
  Then
  \begin{align} \label{E:Ave_BiH1_TN}
    \left|\int_{\Omega_\varepsilon}D(u):\bar{v}\otimes\bar{n}\,dx\right| \leq c\varepsilon^{3/2}\|u\|_{H^2(\Omega_\varepsilon)}\|v\|_{L^2(\Gamma)},
  \end{align}
  where $c>0$ is a constant independent of $\varepsilon$, $u$, and $v$.
\end{lemma}

\begin{proof}
  Since $v$ is tangential on $\Gamma$,
  \begin{align*}
    D(u):\bar{v}\otimes\bar{n} = \mathrm{tr}[D(u)^T(\bar{v}\otimes\bar{n})] = D(u)^T\bar{v}\cdot\bar{n} = \bar{v}\cdot D(u)\bar{n} = \bar{v}\cdot\overline{P}D(u)\bar{n}
  \end{align*}
  in $\Omega_\varepsilon$.
  Hence by \eqref{E:Con_Lp} and \eqref{E:Poin_Str} we see that
  \begin{align*}
    \left|\int_{\Omega_\varepsilon}D(u):\bar{v}\otimes\bar{n}\,dx\right| &\leq c\left\|\overline{P}D(u)\bar{n}\right\|_{L^2(\Omega_\varepsilon)}\|\bar{v}\|_{L^2(\Omega_\varepsilon)} \\
    &\leq c\varepsilon^{3/2}\|u\|_{H^2(\Omega_\varepsilon)}\|v\|_{L^2(\Gamma)}.
  \end{align*}
  Here we used the inequalities \eqref{E:Fric_Upper} and the condition on $u$ to apply \eqref{E:Poin_Str}.
\end{proof}

Now let us derive estimates for trilinear forms.

\begin{lemma} \label{L:Ave_TrT}
  Let $u_1\in H^2(\Omega_\varepsilon)^3$, $u_2\in H^1(\Omega_\varepsilon)^3$, and $A\in L^2(\Gamma)^{3\times3}$.
  Suppose that the inequalities \eqref{E:Fric_Upper} are valid, $u_1$ satisfies $\mathrm{div}\,u_1=0$ in $\Omega_\varepsilon$ and \eqref{E:Bo_Slip}, and $A$ satisfies \eqref{E:A_Cond}.
  Then
  \begin{multline} \label{E:Ave_TrT}
    \left|\int_{\Omega_\varepsilon}u_1\otimes u_2:\overline{A}\,dx-\varepsilon\int_\Gamma g(M_\tau u_1)\otimes(M_\tau u_2):A\,d\mathcal{H}^2\right| \\
    \leq cR_\varepsilon(u_1,u_2)\|A\|_{L^2(\Gamma)},
  \end{multline}
  where $c>0$ is a constant independent of $\varepsilon$, $u_1$, $u_2$, and $A$ and
  \begin{multline} \label{E:Ave_TrT_Re}
    R_\varepsilon(u_1,u_2) := \varepsilon\|u_1\|_{H^1(\Omega_\varepsilon)}\|u_2\|_{H^1(\Omega_\varepsilon)} \\
    +\left(\varepsilon\|u_1\|_{H^2(\Omega_\varepsilon)}+\varepsilon^{1/2}\|u_1\|_{L^2(\Omega_\varepsilon)}^{1/2}\|u_1\|_{H^2(\Omega_\varepsilon)}^{1/2}\right)\|u_2\|_{L^2(\Omega_\varepsilon)}.
  \end{multline}
\end{lemma}

\begin{proof}
  In what follows, we write
  \begin{align*}
    u_{i,\tau}(x) := \overline{P}(x)u_i(x), \quad x\in\Omega_\varepsilon, \, i=1,2.
  \end{align*}
  By \eqref{E:A_Cond} and $P^T=P$ on $\Gamma$ we have
  \begin{align*}
    u_1\otimes u_2:\overline{A} = \overline{P}(u_1\otimes u_2)\overline{P}:\overline{A} = u_{1,\tau}\otimes u_{2,\tau}:\overline{A} \quad\text{in}\quad \Omega_\varepsilon.
  \end{align*}
  Using this equality we decompose the difference
  \begin{align} \label{Pf_TrT:Split}
    \int_{\Omega_\varepsilon}u_1\otimes u_2:\overline{A}\,dx-\varepsilon\int_\Gamma g(M_\tau u_1)\otimes(M_\tau u_2):A\,d\mathcal{H}^2 = J_1+J_2
  \end{align}
  into
  \begin{align*}
    J_1 &:= \int_{\Omega_\varepsilon}u_{1,\tau}\otimes u_{2,\tau}:\overline{A}\,dx-\int_{\Omega_\varepsilon}\Bigl(\overline{M_\tau u_1}\Bigr)\otimes u_{2,\tau}:\overline{A}\,dx, \\
    J_2 &:= \int_{\Omega_\varepsilon}\Bigl(\overline{M_\tau u_1}\Bigr)\otimes u_{2,\tau}:\overline{A}\,dx-\varepsilon\int_\Gamma g(M_\tau u_1)\otimes(M_\tau u_2):A\,d\mathcal{H}^2.
  \end{align*}
  Let $u_1^a$ be the average part of $u_1$ given by \eqref{E:Def_ExAve} and $u_1^r:=u_1-u_1^a$.
  Since
  \begin{align*}
    u_{1,\tau}-\overline{M_\tau u_1} = \overline{P}u_1-\overline{P}u_1^a = \overline{P}u_1^r, \quad u_{2,\tau} = \overline{P}u_2 \quad\text{in}\quad \Omega_\varepsilon
  \end{align*}
  and $|Pa|\leq|a|$ on $\Gamma$ for $a\in\mathbb{R}^3$,
  \begin{align*}
    |J_1| = \left|\int_{\Omega_\varepsilon}\Bigl(\overline{P}u_1^r\Bigr)\otimes u_{2,\tau}:\overline{A}\,dx\right| \leq c\|u_1^r\|_{L^\infty(\Omega_\varepsilon)}\|u_2\|_{L^2(\Omega_\varepsilon)}\left\|\overline{A}\right\|_{L^2(\Omega_\varepsilon)}.
  \end{align*}
  We apply \eqref{E:Con_Lp} and \eqref{E:Linf_Ur} to the right-hand side to obtain
  \begin{align} \label{Pf_TrT:I1}
    |J_1| \leq c\left(\varepsilon\|u_1\|_{H^2(\Omega_\varepsilon)}+\varepsilon^{1/2}\|u_1\|_{L^2(\Omega_\varepsilon)}^{1/2}\|u_1\|_{H^2(\Omega_\varepsilon)}^{1/2}\right)\|u_2\|_{L^2(\Omega_\varepsilon)}\|A\|_{L^2(\Gamma)}.
  \end{align}
  Here we used the inequalities \eqref{E:Fric_Upper} and the conditions on $u_1$ to apply \eqref{E:Linf_Ur}.

  Let us estimate $J_2$.
  By $M_\tau u_2=Mu_{2,\tau}$ on $\Gamma$, \eqref{E:CoV_Dom}, and \eqref{E:Def_Ave} we have
  \begin{align*}
    J_2 = \int_\Gamma (M_\tau u_1)\otimes\left(\int_{\varepsilon g_0}^{\varepsilon g_1}u_{2,\tau}^\sharp(J-1)\,dr\right):A\,d\mathcal{H}^2.
  \end{align*}
  To the right-hand side we apply \eqref{E:Jac_Diff_03}, \eqref{E:CoV_Equiv},
  \begin{align*}
    |M_\tau u_1| = |PMu_1| \leq |Mu_1| \quad\text{on}\quad \Gamma, \quad |u_{2,\tau}| = \left|\overline{P}u_2\right| \leq |u_2| \quad\text{in}\quad \Omega_\varepsilon,
  \end{align*}
  and H\"{o}lder's inequality to deduce that
  \begin{align*}
    |J_2| \leq c\varepsilon\int_{\Omega_\varepsilon}\left|\overline{M_\tau u_1}\right||u_{2,\tau}|\left|\overline{A}\right|\,dx \leq c\varepsilon\left\|\,\left|\overline{Mu_1}\right|\,|u_2|\,\right\|_{L^2(\Omega_\varepsilon)}\left\|\overline{A}\right\|_{L^2(\Omega_\varepsilon)}.
  \end{align*}
  Moreover, from \eqref{E:Ave_Lp_Surf}, \eqref{E:Ave_Wmp_Surf}, and \eqref{E:Prod_Surf} it follows that
  \begin{align*}
    \left\|\,\left|\overline{Mu_1}\right|\,|u_2|\,\right\|_{L^2(\Omega_\varepsilon)} &\leq c\|Mu_1\|_{L^2(\Gamma)}^{1/2}\|Mu_1\|_{H^1(\Gamma)}^{1/2}\|u_2\|_{L^2(\Omega_\varepsilon)}^{1/2}\|u_2\|_{H^1(\Omega_\varepsilon)}^{1/2} \\
    &\leq c\varepsilon^{-1/2}\|u_1\|_{L^2(\Omega_\varepsilon)}^{1/2}\|u_1\|_{H^1(\Omega_\varepsilon)}^{1/2}\|u_2\|_{L^2(\Omega_\varepsilon)}^{1/2}\|u_2\|_{H^1(\Omega_\varepsilon)}^{1/2}.
  \end{align*}
  We apply this inequality and \eqref{E:Con_Lp} to the above estimate for $J_2$ to get
  \begin{align} \label{Pf_TrT:I2}
    \begin{aligned}
      |J_2| &\leq c\varepsilon\|u_1\|_{L^2(\Omega_\varepsilon)}^{1/2}\|u_1\|_{H^1(\Omega_\varepsilon)}^{1/2}\|u_2\|_{L^2(\Omega_\varepsilon)}^{1/2}\|u_2\|_{H^1(\Omega_\varepsilon)}^{1/2}\|A\|_{L^2(\Gamma)} \\
      &\leq c\varepsilon\|u_1\|_{H^1(\Omega_\varepsilon)}\|u_2\|_{H^1(\Omega_\varepsilon)}\|A\|_{L^2(\Gamma)}.
    \end{aligned}
  \end{align}
  By \eqref{Pf_TrT:Split}--\eqref{Pf_TrT:I2} we obtain \eqref{E:Ave_TrT} with $R_\varepsilon(u_1,u_2)$ given by \eqref{E:Ave_TrT_Re}.
\end{proof}

\begin{lemma} \label{L:Ave_TrN}
  Let $u_1,u_2\in H^1(\Omega_\varepsilon)^3$ and $v\in H^1(\Gamma)^3$.
  Suppose that $u_2$ satisfies \eqref{E:Bo_Imp} on $\Gamma_\varepsilon^0$ or on $\Gamma_\varepsilon^1$.
  Then there exists a constant $c>0$ independent of $\varepsilon$, $u_1$, $u_2$, and $v$ such that
  \begin{align} \label{E:Ave_TrN}
    \left|\int_{\Omega_\varepsilon}u_1\otimes u_2:\bar{v}\otimes\bar{n}\,dx\right| \leq c\varepsilon\|u_1\|_{H^1(\Omega_\varepsilon)}\|u_2\|_{H^1(\Omega_\varepsilon)}\|v\|_{H^1(\Gamma)}.
  \end{align}
\end{lemma}

\begin{proof}
  By $u_1\otimes u_2:\bar{v}\otimes\bar{n}=(u_1\cdot\bar{v})(u_2\cdot\bar{n})$ in $\Omega_\varepsilon$, \eqref{E:Poin_Nor}, and \eqref{E:Prod_Surf} we have
  \begin{align*}
    \left|\int_{\Omega_\varepsilon}u_1\otimes u_2:\bar{v}\otimes\bar{n}\,dx\right| &\leq \|u_1\cdot\bar{v}\|_{L^2(\Omega_\varepsilon)}\|u_2\cdot\bar{n}\|_{L^2(\Omega_\varepsilon)} \\
    &\leq c\varepsilon\|u_1\|_{L^2(\Omega_\varepsilon)}^{1/2}\|u_1\|_{H^1(\Omega_\varepsilon)}^{1/2}\|v\|_{L^2(\Gamma)}^{1/2}\|v\|_{H^1(\Gamma)}^{1/2}\|u_2\|_{H^1(\Omega_\varepsilon)} \\
    &\leq c\varepsilon\|u_1\|_{H^1(\Omega_\varepsilon)}\|u_2\|_{H^1(\Omega_\varepsilon)}\|v\|_{H^1(\Gamma)}.
  \end{align*}
  Here we used the condition on $u_2$ to apply \eqref{E:Poin_Nor} to $\|u_2\cdot\bar{n}\|_{L^2(\Omega_\varepsilon)}$.
\end{proof}

\section{Weighted solenoidal spaces on a closed surface} \label{S:WSol}
The purpose of this section is to study weighted solenoidal spaces
\begin{align*}
  \mathcal{X}_{g\sigma}(\Gamma,T\Gamma) = \{v\in\mathcal{X}(\Gamma,T\Gamma) \mid \text{$\mathrm{div}_\Gamma(gv)=0$ on $\Gamma$}\}, \quad \mathcal{X} = L^2,H^1
\end{align*}
on a closed surface $\Gamma$ with a given function $g$ on $\Gamma$.
These function spaces play an important role in the study of a singular limit problem for \eqref{E:NS_CTD}.

Throughout this section, let $\Gamma$ be a two-dimensional closed, connected, and oriented surface in $\mathbb{R}^3$ of class $C^2$.
We use the notations given in Section \ref{SS:Pre_Surf}.

\subsection{Ne\v{c}as inequality on a closed surface} \label{SS:WS_Nec}
Let $q\in L^2(\Gamma)$.
We consider $q$ and its weak tangential gradient in $H^{-1}(\Gamma)$ and $H^{-1}(\Gamma,T\Gamma)$ by \eqref{E:L2_Hin} and \eqref{E:TGr_HinT}.
Then it immediately follows that
\begin{align*}
  |\langle q,\xi\rangle_\Gamma| \leq \|q\|_{L^2(\Gamma)}\|\xi\|_{H^1(\Gamma)}, \quad |[\nabla_\Gamma q,v]_{T\Gamma}| \leq c\|q\|_{L^2(\Gamma)}\|v\|_{H^1(\Gamma)}
\end{align*}
for all $\xi\in H^1(\Gamma)$ and $v\in H^1(\Gamma,T\Gamma)$, and thus
\begin{align} \label{E:HinL2_Bo}
  \|q\|_{H^{-1}(\Gamma)}+\|\nabla_\Gamma q\|_{H^{-1}(\Gamma,T\Gamma)} \leq c\|q\|_{L^2(\Gamma)}.
\end{align}
For bounded Lipschitz domains in $\mathbb{R}^m$, $m\in\mathbb{N}$ the inverse inequality is also valid and known as the Ne\v{c}as inequality (see \cite{BoFa13}*{Theorem IV.1.1} and \cite{Ne12}*{Chapter 3, Lemma 7.1}).
Let us show the Ne\v{c}as inequality on the closed surface $\Gamma$.

\begin{lemma} \label{L:Necas_Surf}
  There exists a constant $c>0$ such that
  \begin{align} \label{E:Necas_Surf}
    \|q\|_{L^2(\Gamma)} \leq c\left(\|q\|_{H^{-1}(\Gamma)}+\|\nabla_\Gamma q\|_{H^{-1}(\Gamma,T\Gamma)}\right)
  \end{align}
  for all $q\in L^2(\Gamma)$.
\end{lemma}

To prove Lemma \ref{L:Necas_Surf} we employ a localization argument.
In Appendix \ref{S:Ap_Aux} we give auxiliary lemmas for calculations under a local coordinate system of $\Gamma$.
We also apply the Ne\v{c}as inequality on $\mathbb{R}^2$.
Let $H^{-1}(\mathbb{R}^2)$ be the dual space of $H^1(\mathbb{R}^2)$ (via the $L^2(\mathbb{R}^2)$-inner product) and $\langle\cdot,\cdot\rangle_{\mathbb{R}^2}$ the duality product between $H^{-1}(\mathbb{R}^2)$ and $H^1(\mathbb{R}^2)$.
We consider $q^\flat\in L^2(\mathbb{R}^2)$ in $H^{-1}(\mathbb{R}^2)$ by
\begin{align*}
  \langle q^\flat,\xi\rangle_{\mathbb{R}^2} := (q^\flat,\xi)_{L^2(\mathbb{R}^2)}, \quad \xi\in H^1(\mathbb{R}^2)
\end{align*}
and define $\nabla_sq^\flat\in H^{-1}(\mathbb{R}^2)^2$ by
\begin{align*}
  \langle\nabla_sq^\flat,\varphi\rangle_{\mathbb{R}^2} := -(q^\flat,\mathrm{div}_s\varphi)_{L^2(\mathbb{R}^2)}, \quad \varphi \in H^1(\mathbb{R}^2)^2.
\end{align*}
Here $\mathrm{div}_s$ is the divergence operator with respect to $s\in\mathbb{R}^2$, i.e.
\begin{align*}
  \mathrm{div}_s\varphi := \partial_{s_1}\varphi_1+\partial_{s_2}\varphi_2 \quad\text{on}\quad \mathbb{R}^2, \quad \varphi =
  \begin{pmatrix}
    \varphi_1 \\
    \varphi_2
  \end{pmatrix}.
\end{align*}
Then we have the following Ne\v{c}as inequality on $\mathbb{R}^2$ as an easy consequence of the characterization of $H^{-1}(\mathbb{R}^2)$ by the Fourier transform.

\begin{lemma} \label{L:Necas_R2}
  There exists a constant $c>0$ such that
  \begin{align} \label{E:Necas_R2}
    \|q^\flat\|_{L^2(\mathbb{R}^2)} \leq c\left(\|q^\flat\|_{H^{-1}(\mathbb{R}^2)}+\|\nabla_sq^\flat\|_{H^{-1}(\mathbb{R}^2)}\right)
  \end{align}
  for all $q^\flat\in L^2(\mathbb{R}^2)$.
\end{lemma}

For the proof of Lemma \ref{L:Necas_R2} we refer to \cite{BoFa13}*{Proposition IV.1.2} (see also the proof of \cite{Ne12}*{Chapter 3, Lemma 7.1}).
Now let us prove Lemma \ref{L:Necas_Surf}.

\begin{proof}[Proof of Lemma \ref{L:Necas_Surf}]
  First we note that it suffices to show \eqref{E:Necas_Surf} when $q$ is compactly supported in a relatively open subset of $\Gamma$ on which we can take a local coordinate system.
  To see this, let $q\in L^2(\Gamma)$ and $\eta\in C^2(\Gamma)$.
  For $\xi\in H^1(\Gamma)$ we have
  \begin{align*}
    |\langle \eta q,\xi \rangle_\Gamma| = |\langle q,\eta\xi \rangle_\Gamma| &\leq \|q\|_{H^{-1}(\Gamma)}\|\eta\xi\|_{H^1(\Gamma)} \\
    &\leq c\|\eta\|_{W^{1,\infty}(\Gamma)}\|q\|_{H^{-1}(\Gamma)}\|\xi\|_{H^1(\Gamma)},
  \end{align*}
  where $c>0$ is a constant independent of $q$, $\eta$, and $\xi$.
  Also,
  \begin{align*}
    [\nabla_\Gamma(\eta q),v]_{T\Gamma} &= -(\eta q,\mathrm{div}_\Gamma v)_{L^2(\Gamma)} = -\bigl(q,\mathrm{div}_\Gamma(\eta v)\bigr)_{L^2(\Gamma)}+(q,\nabla_\Gamma\eta\cdot v)_{L^2(\Gamma)} \\
    &= [\nabla_\Gamma q,\eta v]_{T\Gamma}+\langle q,\nabla_\Gamma\eta\cdot v\rangle_\Gamma
  \end{align*}
  for all $v\in H^1(\Gamma,T\Gamma)$ by \eqref{E:TGr_HinT} (note that $\eta v\in H^1(\Gamma,T\Gamma)$) and thus
  \begin{align*}
    |[\nabla_\Gamma(\eta q),v]_{T\Gamma}| &\leq \|\nabla_\Gamma q\|_{H^{-1}(\Gamma,T\Gamma)}\|\eta v\|_{H^1(\Gamma)}+\|q\|_{H^{-1}(\Gamma)}\|\nabla_\Gamma\eta\cdot v\|_{H^1(\Gamma)} \\
    &\leq c\|\eta\|_{W^{2,\infty}(\Gamma)}\left(\|q\|_{H^{-1}(\Gamma)}+\|\nabla_\Gamma q\|_{H^{-1}(\Gamma,T\Gamma)}\right)\|v\|_{H^1(\Gamma)}.
  \end{align*}
  From the above inequalities it follows that
  \begin{align*}
    \|\eta q\|_{H^{-1}(\Gamma)} &\leq c\|\eta\|_{W^{1,\infty}(\Gamma)}\|q\|_{H^{-1}(\Gamma)}, \\
    \|\nabla_\Gamma(\eta q)\|_{H^{-1}(\Gamma,T\Gamma)} &\leq c\|\eta\|_{W^{2,\infty}(\Gamma)}\left(\|q\|_{H^{-1}(\Gamma)}+\|\nabla_\Gamma q\|_{H^{-1}(\Gamma,T\Gamma)}\right).
  \end{align*}
  Hence if we localize $q$ by a partition of unity on $\Gamma$ consisting of functions in $C^2(\Gamma)$ (such functions exist by the $C^2$-regularity of $\Gamma$) and prove \eqref{E:Necas_Surf} for each localized function, then we can get \eqref{E:Necas_Surf} for $q$ by the above inequalities.

  Now we assume that $q\in L^2(\Gamma)$ is supported in $\mu(\mathcal{K})$, where $\mu\colon U\to\mathbb{R}^2$ is a $C^2$ local parametrization of $\Gamma$ with an open set $U$ in $\mathbb{R}^2$ and $\mathcal{K}$ is a compact subset of $U$.
  Let $\theta=(\theta_{ij})_{i,j}$ be the Riemannian metric of $\Gamma$ given by \eqref{E:Def_Met} and $\theta^{-1}=(\theta^{ij})_{i,j}$ its inverse.
  We define $q^\flat:=q\circ\mu$ on $U$.
  Then $q^\flat\in L^2(U)$ by Lemma \ref{L:Lp_Loc} and we extend $q^\flat$ to $\mathbb{R}^2$ by zero outside $U$ to get $q^\flat\in L^2(\mathbb{R}^2)$.
  Hence it follows from \eqref{E:Necas_R2} and \eqref{E:Lp_Loc} that
  \begin{align} \label{Pf_Nec:Whole}
    \|q\|_{L^2(\Gamma)} \leq c\|q^\flat\|_{L^2(U)} = c\|q^\flat\|_{L^2(\mathbb{R}^2)} \leq c\left(\|q^\flat\|_{H^{-1}(\mathbb{R}^2)}+\|\nabla_sq^\flat\|_{H^{-1}(\mathbb{R}^2)}\right).
  \end{align}
  Let us estimate the right-hand side of \eqref{Pf_Nec:Whole} by that of \eqref{E:Necas_Surf}.
  We first prove
  \begin{align} \label{Pf_Nec:Q_Hin}
    \|q^\flat\|_{H^{-1}(\mathbb{R}^2)} \leq c\|q\|_{H^{-1}(\Gamma)}.
  \end{align}
  To this end, for all $\xi\in H^1(\mathbb{R}^2)$ we show that
  \begin{align} \label{Pf_Nec:Q_Dual}
    |\langle q^\flat,\xi\rangle_{\mathbb{R}^2}| = |(q^\flat,\xi)_{L^2(\mathbb{R}^2)}| \leq c\|q\|_{H^{-1}(\Gamma)}\|\xi\|_{H^1(\mathbb{R}^2)}.
  \end{align}
  Since $C_c^1(\mathbb{R}^2)$ is dense in $H^1(\mathbb{R}^2)$, it is sufficient to show \eqref{Pf_Nec:Q_Dual} for all $\xi\in C_c^1(\mathbb{R}^2)$.
  Moreover, we may assume that $\xi$ is supported in $\mathcal{K}$ since $q^\flat$ is so.
  Let $\eta$ be a function on $\mu(\mathcal{K})\subset\Gamma$ defined by
  \begin{align} \label{Pf_Nec:Def_Eta}
    \eta(\mu(s)) := \frac{\xi(s)}{\sqrt{\det\theta(s)}}, \quad s\in\mathcal{K}.
  \end{align}
  We extend $\eta$ to $\Gamma$ by zero outside $\mu(\mathcal{K})$.
  Then $\eta\in C^1(\Gamma)\subset H^1(\Gamma)$ by \eqref{E:Metric} and the fact that $\xi\in C_c^1(\mathbb{R}^2)$ is supported in $\mathcal{K}$.
  Moreover,
  \begin{align*}
    \langle q^\flat,\xi\rangle_{\mathbb{R}^2} &= (q^\flat,\xi)_{L^2(\mathbb{R}^2)} = \int_{\mathcal{K}} q^\flat\left(\frac{\xi}{\sqrt{\det\theta}}\right)\sqrt{\det\theta}\,ds \\
    &= \int_{\mu(\mathcal{K})}q\eta\,d\mathcal{H}^2 = (q,\eta)_{L^2(\Gamma)} = \langle q,\eta\rangle_\Gamma
  \end{align*}
  and thus
  \begin{align} \label{Pf_Nec:QD_In}
    |\langle q^\flat,\xi\rangle_{\mathbb{R}^2}| = |\langle q,\eta\rangle_\Gamma| \leq \|q\|_{H^{-1}(\Gamma)}\|\eta\|_{H^1(\Gamma)}.
  \end{align}
  Let us estimate the $H^1(\Gamma)$-norm of $\eta$.
  Since $\eta$ is supported in $\mu(\mathcal{K})$,
  \begin{align} \label{Pf_Nec:QD_Eta}
    \begin{aligned}
      \|\eta\|_{L^2(\Gamma)}^2 &= \int_{\mathcal{K}}|\eta\circ\mu|^2\sqrt{\det\theta}\,ds = \int_{\mathcal{K}}\frac{|\xi|^2}{\sqrt{\det\theta}}\,ds \\
      &\leq c\|\xi\|_{L^2(\mathcal{K})}^2 = c\|\xi\|_{L^2(\mathbb{R}^2)}^2
    \end{aligned}
  \end{align}
  by \eqref{E:Metric}.
  Also, we differentiate \eqref{Pf_Nec:Def_Eta} with respect to $s_i$ and use \eqref{E:Metric} to get
  \begin{align*}
    |\partial_{s_i}(\eta\circ\mu)(s)| \leq c(|\xi(s)|+|\partial_{s_i}\xi(s)|), \quad s\in\mathcal{K},\,i=1,2.
  \end{align*}
  We apply this inequality, \eqref{E:Mu_Bound}, and \eqref{E:Metric} to \eqref{E:TGr_DG}.
  Then we have
  \begin{align*}
    |\nabla_\Gamma\eta(\mu(s))| \leq c(|\xi(s)|+|\nabla_s\xi(s)|), \quad s\in\mathcal{K},
  \end{align*}
  where $\nabla_s\xi:=(\partial_{s_1}\xi,\partial_{s_2}\xi)^T$ is the gradient of $\xi$ in $s\in\mathbb{R}^2$.
  Noting that $\eta$ is supported in $\mu(\mathcal{K})$, we deduce from this inequality and \eqref{E:Metric} that
  \begin{align*}
    \|\nabla_\Gamma\eta\|_{L^2(\Gamma)}^2 = \int_{\mathcal{K}}|(\nabla_\Gamma\eta)\circ\mu|^2\sqrt{\det\theta}\,ds \leq c\|\xi\|_{H^1(\mathcal{K})}^2 = c\|\xi\|_{H^1(\mathbb{R}^2)}^2.
  \end{align*}
  Applying this inequality and \eqref{Pf_Nec:QD_Eta} to \eqref{Pf_Nec:QD_In} we get \eqref{Pf_Nec:Q_Dual}.
  Hence \eqref{Pf_Nec:Q_Hin} is valid.

  Next we derive the estimate
  \begin{align} \label{Pf_Nec:QGr_Hin}
    \|\nabla_sq^\flat\|_{H^{-1}(\mathbb{R}^2)} \leq c\left(\|q\|_{H^{-1}(\Gamma)}+\|\nabla_\Gamma q\|_{H^{-1}(\Gamma,T\Gamma)}\right).
  \end{align}
  For this purpose, we show that, for all $\varphi=(\varphi_1,\varphi_2)^T\in H^1(\mathbb{R}^2)^2$,
  \begin{align} \label{Pf_Nec:QGr_Dual}
    |\langle\nabla_sq^\flat,\varphi\rangle_{\mathbb{R}^2}| \leq c\left(\|q\|_{H^{-1}(\Gamma)}+\|\nabla_\Gamma q\|_{H^{-1}(\Gamma,T\Gamma)}\right)\|\varphi\|_{H^1(\mathbb{R}^2)}.
  \end{align}
  By a density argument we may assume $\varphi\in C_c^1(\mathbb{R}^2)^2$.
  Let $\xi\in C_c^2(\mathbb{R}^2)$ satisfy
  \begin{align*}
    \mathcal{K} \subset \mathcal{C} := \mathrm{supp}\,\xi \subset U, \quad \xi = 1 \quad\text{on}\quad \mathcal{K}.
  \end{align*}
  Then since $q^\flat$ is supported in $\mathcal{K}$,
  \begin{align*}
    \langle\nabla_sq^\flat,\varphi\rangle_{\mathbb{R}^2} &= -(q^\flat,\mathrm{div}_s\varphi)_{L^2(\mathbb{R}^2)} = -(q^\flat,\mathrm{div}_s\varphi)_{L^2(\mathcal{K})} = -(q^\flat,\xi\,\mathrm{div}_s\varphi)_{L^2(\mathcal{K})} \\
    &= -\bigl(q^\flat,\mathrm{div}_s(\xi\varphi)\bigr)_{L^2(\mathcal{K})}+(q^\flat,\nabla_s\xi\cdot\varphi)_{L^2(\mathcal{K})}.
  \end{align*}
  Moreover, since $\xi\in C_c^2(\mathbb{R}^2)$, we have $\nabla_s\xi\cdot\varphi\in H^1(\mathbb{R}^2)$ and
  \begin{align*}
    |(q^\flat,\nabla_s\xi\cdot\varphi)_{L^2(\mathcal{K})}| = |\langle q^\flat,\nabla_s\xi\cdot\varphi\rangle_{\mathbb{R}^2}| &\leq \|q^\flat\|_{H^{-1}(\mathbb{R}^2)}\|\nabla_s\xi\cdot\varphi\|_{H^1(\mathbb{R}^2)} \\
    &\leq c\|q\|_{H^{-1}(\Gamma)}\|\varphi\|_{H^1(\mathbb{R}^2)}
  \end{align*}
  by \eqref{Pf_Nec:Q_Hin}.
  Hence it follows that
  \begin{align} \label{Pf_Nec:QGrD_In}
    |\langle\nabla_sq^\flat,\varphi\rangle_{\mathbb{R}^2}| \leq \left|\bigl(q^\flat,\mathrm{div}_s(\xi\varphi)\bigr)_{L^2(\mathcal{K})}\right|+c\|q\|_{H^{-1}(\Gamma)}\|\varphi\|_{H^1(\mathbb{R}^2)}.
  \end{align}
  Let us estimate the first term on the right-hand side.
  Since $\mathcal{C}=\mathrm{supp}\,\xi$ is a compact subset of $U$, the inequalities \eqref{E:Mu_Bound} and \eqref{E:Metric} are also valid on $\mathcal{C}$ with a different constant $c>0$.
  Hereafter we use this fact without mention.
  We define a tangential vector field $X$ on $\mu(\mathcal{C})\subset\Gamma$ by
  \begin{align*}
    X(\mu(s)) := \frac{\xi(s)}{\sqrt{\det\theta(s)}}\sum_{i=1}^2\varphi_i(s)\partial_{s_i}\mu(s), \quad s\in\mathcal{C}
  \end{align*}
  and extend it to $\Gamma$ by zero outside $\mu(\mathcal{C})$.
  Then
  \begin{align*}
    X \in C^1(\Gamma,T\Gamma) \subset H^1(\Gamma,T\Gamma)
  \end{align*}
  by \eqref{E:Metric} and $\varphi\in C_c^1(\mathbb{R}^2)^2$, since $\xi\in C_c^2(\mathbb{R}^2)$ is supported in $\mathcal{C}$.
  Moreover,
  \begin{align*}
    \sum_{j=1}^2\theta^{ij}(s)\partial_{s_j}\mu(s)\cdot X(\mu(s)) &= \frac{\xi(s)}{\sqrt{\det\theta(s)}}\sum_{j,k=1}^2\theta^{ij}(s)\theta_{jk}(s)\varphi_k(s) \\
    &= \frac{\xi(s)\varphi_i(s)}{\sqrt{\det\theta(s)}}
  \end{align*}
  for $s\in\mathcal{C}$ and $i=1,2$ and thus
  \begin{align*}
    \mathrm{div}_\Gamma X(\mu(s)) = \frac{1}{\sqrt{\det\theta(s)}}\sum_{i=1}^2\partial_{s_i}\bigl(\xi(s)\varphi_i(s)\bigr) = \frac{\mathrm{div}_s(\xi\varphi)(s)}{\sqrt{\det\theta(s)}}, \quad s\in\mathcal{C}
  \end{align*}
  by \eqref{E:DivG_DG}.
  Noting that $\mathcal{K}\subset\mathcal{C}$, we see by this equality and \eqref{E:TGr_HinT} that
  \begin{align*}
    \bigl(q^\flat,\mathrm{div}_s(\xi\varphi)\bigr)_{L^2(\mathcal{K})} &= \int_{\mathcal{K}}q^\flat\{(\mathrm{div}_\Gamma X)\circ\mu\}\sqrt{\det\theta}\,ds = \int_{\mu(\mathcal{K})}q\,\mathrm{div}_\Gamma X\,d\mathcal{H}^2 \\
    &= (q,\mathrm{div}_\Gamma X)_{L^2(\Gamma)} = -[\nabla_\Gamma q,X]_{T\Gamma}.
  \end{align*}
  Therefore,
  \begin{align} \label{Pf_Nec:QGrD_qX}
    \left|\bigl(q^\flat,\mathrm{div}_s(\xi\varphi)\bigr)_{L^2(\mathcal{K})}\right| = |[\nabla_\Gamma q,X]_{T\Gamma}| \leq \|\nabla_\Gamma q\|_{H^{-1}(\Gamma,T\Gamma)}\|X\|_{H^1(\Gamma)}.
  \end{align}
  Let us estimate the $H^1(\Gamma)$-norm of $X$.
  By \eqref{E:Mu_Bound}, \eqref{E:Metric}, and $\xi\in C_c^2(\mathbb{R}^2)$,
  \begin{align*}
    |X(\mu(s))| \leq c|\varphi(s)|, \quad s\in\mathcal{C}.
  \end{align*}
  Since $X$ is supported in $\mu(\mathcal{C})$, the above inequality and \eqref{E:Metric} show that
  \begin{align} \label{Pf_Nec:QGrD_XL2}
    \|X\|_{L^2(\Gamma)}^2 = \int_{\mathcal{C}}|X\circ\mu|^2\sqrt{\det\theta}\,ds \leq c\|\varphi\|_{L^2(\mathcal{C})}^2 \leq c\|\varphi\|_{L^2(\mathbb{R}^2)}^2.
  \end{align}
  For $m=1,2,3$ the $m$-th component of $X$ in $\mathbb{R}^3$ is of the form
  \begin{align*}
    X_m(\mu(s)) = \frac{\xi(s)}{\sqrt{\det\theta(s)}}\sum_{i=1}^2\varphi_i(s)\partial_{s_i}\mu_m(s), \quad s\in\mathcal{C}.
  \end{align*}
  We differentiate both sides in $s_k$ and use \eqref{E:Mu_Bound}, \eqref{E:Metric}, and $\xi\in C_c^2(\mathbb{R}^2)$ to get
  \begin{align*}
    |\partial_{s_k}(X_m\circ\mu)(s)| \leq c(|\varphi(s)|+|\nabla_s\varphi(s)|), \quad s\in\mathcal{C}, \, k=1,2,
  \end{align*}
  where $\nabla_s\varphi:=(\partial_{s_i}\varphi_j)_{i,j}$ is the gradient matrix of $\varphi$.
  Applying this inequality, \eqref{E:Mu_Bound}, and \eqref{E:Metric} to \eqref{E:TGr_DG} with $\eta=X_m$ we have
  \begin{align*}
    |\nabla_\Gamma X_m(\mu(s))| \leq c(|\varphi(s)|+|\nabla_s\varphi(s)|), \quad s\in\mathcal{C}, \, m=1,2,3
  \end{align*}
  and thus, by \eqref{E:Metric} and the fact that $X$ is supported in $\mu(\mathcal{C})$,
  \begin{align*}
    \|\nabla_\Gamma X\|_{L^2(\Gamma)}^2 &= \sum_{m=1}^3\|\nabla_\Gamma X_m\|_{L^2(\Gamma)}^2 = \sum_{m=1}^3\int_{\mathcal{C}}|(\nabla_\Gamma X_m)\circ\mu|^2\sqrt{\det\theta}\,ds \\
    &\leq c\|\varphi\|_{H^1(\mathcal{C})}^2 \leq c\|\varphi\|_{H^1(\mathbb{R}^2)}^2.
  \end{align*}
  From this inequality, \eqref{Pf_Nec:QGrD_qX}, and \eqref{Pf_Nec:QGrD_XL2} we deduce that
  \begin{align*}
    \left|\bigl(q^\flat,\mathrm{div}_s(\xi\varphi)\bigr)_{L^2(\mathcal{K})}\right| \leq c\|\nabla_\Gamma q\|_{H^{-1}(\Gamma,T\Gamma)}\|\varphi\|_{H^1(\mathbb{R}^2)}
  \end{align*}
  and combining this inequality and \eqref{Pf_Nec:QGrD_In} we obtain \eqref{Pf_Nec:QGr_Dual}.
  Hence \eqref{Pf_Nec:QGr_Hin} holds and the inequality \eqref{E:Necas_Surf} follows from \eqref{Pf_Nec:Whole}, \eqref{Pf_Nec:Q_Hin}, and \eqref{Pf_Nec:QGr_Hin}.
\end{proof}

Based on \eqref{E:Necas_Surf} we prove Poincar\'{e}'s inequality for $q\in L^2(\Gamma)$.
First we show that $\nabla_\Gamma q$ vanishes in $H^{-1}(\Gamma,T\Gamma)$ if and only if $q$ is constant on $\Gamma$.

\begin{lemma} \label{L:TGr_HinT_Con}
  Let $q\in L^2(\Gamma)$.
  Then
  \begin{align*}
    \nabla_\Gamma q = 0 \quad\text{in}\quad H^{-1}(\Gamma,T\Gamma)
  \end{align*}
  if and only if $q$ is constant on $\Gamma$.
\end{lemma}

\begin{proof}
  Suppose first that $q$ is constant on $\Gamma$.
  Then
  \begin{align*}
    [\nabla_\Gamma q,v]_{T\Gamma} = -q\int_\Gamma\mathrm{div}_\Gamma v\,d\mathcal{H}^2 = 0
  \end{align*}
  for all $v\in H^1(\Gamma,T\Gamma)$ by \eqref{E:IbP_WDivG_T} and \eqref{E:TGr_HinT}.
  Hence $\nabla_\Gamma q=0$ in $H^{-1}(\Gamma,T\Gamma)$.

  Conversely, suppose that $\nabla_\Gamma q=0$ in $H^{-1}(\Gamma,T\Gamma)$.
  Let us show that
  \begin{align} \label{Pf_THC:Di_Van}
    q\in H^1(\Gamma), \quad \underline{D}_iq = 0 \quad\text{in}\quad L^2(\Gamma), \, i=1,2,3.
  \end{align}
  For $\eta\in C^1(\Gamma)$ and $i=1,2,3$ we define $v:=\eta Pe_i$ on $\Gamma$, where $\{e_1,e_2,e_3\}$ is the standard basis of $\mathbb{R}^3$.
  Then $v\in H^1(\Gamma,T\Gamma)$ since $P$ is of class $C^1$ on $\Gamma$.
  Moreover,
  \begin{align*}
    \mathrm{div}_\Gamma v = \nabla_\Gamma \eta\cdot Pe_i+\eta(\mathrm{div}_\Gamma P\cdot e_i) = \underline{D}_i\eta+\eta Hn_i \quad\text{on}\quad \Gamma
  \end{align*}
  by $P^T=P$ on $\Gamma$, \eqref{E:P_TGr}, and \eqref{E:Div_P}.
  From this equality we deduce that
  \begin{align*}
    0 = [\nabla_\Gamma q,v]_{T\Gamma} = -(q,\mathrm{div}_\Gamma v)_{L^2(\Gamma)} = -(q,\underline{D}_i\eta+\eta Hn_i)_{L^2(\Gamma)}
  \end{align*}
  for all $\eta\in C^1(\Gamma)$ and $i=1,2,3$.
  Hence we obtain \eqref{Pf_THC:Di_Van} by the definition of the weak tangential derivative in $L^2(\Gamma)$ (see \eqref{E:Def_WTD}).
  From this fact and Lemma \ref{L:WTGr_Van} it follows that $q$ is constant on $\Gamma$.
\end{proof}

Next we estimate $q\in L^2(\Gamma)$ in $H^{-1}(\Gamma)$ by its weak tangential gradient.

\begin{lemma} \label{L:Poin_HinT}
  There exists a constant $c>0$ such that
  \begin{align} \label{E:Poin_HinT}
    \|q\|_{H^{-1}(\Gamma)} \leq c\left(\left|\int_\Gamma q\,d\mathcal{H}^2\right|+\|\nabla_\Gamma q\|_{H^{-1}(\Gamma,T\Gamma)}\right)
  \end{align}
  for all $q\in L^2(\Gamma)$.
\end{lemma}

\begin{proof}
  Assume to the contrary that there exists $q_k\in L^2(\Gamma)$ such that
  \begin{align*}
    \|q_k\|_{H^{-1}(\Gamma)} > k\left(\left|\int_\Gamma q_k\,d\mathcal{H}^2\right|+\|\nabla_\Gamma q_k\|_{H^{-1}(\Gamma,T\Gamma)}\right)
  \end{align*}
  for each $k\in\mathbb{N}$.
  Replacing $q_k$ with $q_k/\|q_k\|_{H^{-1}(\Gamma)}$ we may assume that
  \begin{align} \label{Pf_PHin:Contra}
    \|q_k\|_{H^{-1}(\Gamma)} = 1, \quad \left|\int_\Gamma q_k\,d\mathcal{H}^2\right|+\|\nabla_\Gamma q_k\|_{H^{-1}(\Gamma,T\Gamma)} < \frac{1}{k}, \quad k\in\mathbb{N}.
  \end{align}
  From the second inequality it follows that
  \begin{align} \label{Pf_PHin:Conv}
    \lim_{k\to\infty}\int_\Gamma q_k\,d\mathcal{H}^2 = \lim_{k\to\infty}(q_k,1)_{L^2(\Gamma)} = 0, \quad \lim_{k\to\infty}\|\nabla_\Gamma q_k\|_{H^{-1}(\Gamma,T\Gamma)} = 0.
  \end{align}
  Also, $\{q_k\}_{k=1}^\infty$ is bounded in $L^2(\Gamma)$ by \eqref{E:Necas_Surf} and \eqref{Pf_PHin:Contra}.
  By this fact and the compact embedding $L^2(\Gamma)\hookrightarrow H^{-1}(\Gamma)$ we see that $\{q_k\}_{k=1}^\infty$ converges (up to a subsequence) to some $q\in L^2(\Gamma)$ weakly in $L^2(\Gamma)$ and strongly in $H^{-1}(\Gamma)$.
  Then
  \begin{align} \label{Pf_PHin:Lim_Hin}
    \|q\|_{H^{-1}(\Gamma)} = \lim_{k\to\infty}\|q_k\|_{H^{-1}(\Gamma)} = 1
  \end{align}
  by \eqref{Pf_PHin:Contra}.
  Moreover, $\{\nabla_\Gamma q_k\}_{k=1}^\infty$ converges to $\nabla_\Gamma q$ weakly in $H^{-1}(\Gamma,T\Gamma)$ by \eqref{E:TGr_HinT} and the weak convergence of $\{q_k\}_{k=1}^\infty$ to $q$ in $L^2(\Gamma)$.
  Thus, by \eqref{Pf_PHin:Conv},
  \begin{align*}
    \int_\Gamma q\,d\mathcal{H}^2 = (q,1)_{L^2(\Gamma)} = 0, \quad \|\nabla_\Gamma q\|_{H^{-1}(\Gamma,T\Gamma)} = 0.
  \end{align*}
  These equalities and Lemma \ref{L:TGr_HinT_Con} imply that $q=0$ on $\Gamma$.
  Hence $\|q\|_{H^{-1}(\Gamma)}=0$, which contradicts with \eqref{Pf_PHin:Lim_Hin}.
  Therefore, \eqref{E:Poin_HinT} is valid.
\end{proof}

Combining \eqref{E:Necas_Surf} and \eqref{E:Poin_HinT} we obtain Poincar\'{e}'s inequality for $q\in L^2(\Gamma)$.

\begin{lemma} \label{L:Poin_L2_HinT}
  There exists a constant $c>0$ such that
  \begin{align} \label{E:Poin_L2_HinT}
    \|q\|_{L^2(\Gamma)} \leq c\left(\left|\int_\Gamma q\,d\mathcal{H}^2\right|+\|\nabla_\Gamma q\|_{H^{-1}(\Gamma,T\Gamma)}\right)
  \end{align}
  for all $q\in L^2(\Gamma)$.
\end{lemma}

\subsection{Annihilator of a weighted solenoidal space} \label{SS:WS_An}
Let $g\in C^1(\Gamma)$ satisfy
\begin{align} \label{E:LBo_G}
  g \geq c \quad\text{on}\quad \Gamma
\end{align}
with some constant $c>0$.
We define a weighted solenoidal space on $\Gamma$ by
\begin{align*}
  H_{g\sigma}^1(\Gamma,T\Gamma) := \{v\in H^1(\Gamma,T\Gamma) \mid \text{$\mathrm{div}_\Gamma(gv)=0$ on $\Gamma$}\}.
\end{align*}
Clearly, $H_{g\sigma}^1(\Gamma,T\Gamma)$ is closed in $H^1(\Gamma,T\Gamma)$.
If $q\in L^2(\Gamma)$, then
\begin{align*}
  [g\nabla_\Gamma q, v]_{T\Gamma} = -\bigl(q,\mathrm{div}_\Gamma(gv)\bigr)_{L^2(\Gamma)} = 0 \quad\text{for all}\quad v\in H_{g\sigma}^1(\Gamma,T\Gamma)
\end{align*}
by \eqref{E:TGr_HinT} and \eqref{E:Def_Mul_HinT}.
Let us prove the converse of this statement for an element of $H^{-1}(\Gamma,T\Gamma)$, which is a weighted version of de Rham's theorem.

\begin{theorem} \label{T:DeRham_T}
  Suppose that $f\in H^{-1}(\Gamma,T\Gamma)$ satisfies
  \begin{align*}
    [f,v]_{T\Gamma} = 0 \quad\text{for all}\quad v \in H_{g\sigma}^1(\Gamma,T\Gamma).
  \end{align*}
  Then there exists a unique $q\in L^2(\Gamma)$ such that
  \begin{align*}
    f = g\nabla_\Gamma q \quad\text{in}\quad H^{-1}(\Gamma,T\Gamma), \quad \int_\Gamma q\,d\mathcal{H}^2 = 0.
  \end{align*}
  Moreover, there exists a constant $c>0$ independent of $f$ such that
  \begin{align} \label{E:DeRham_T_Ineq}
    \|q\|_{L^2(\Gamma)} \leq c\|f\|_{H^{-1}(\Gamma,T\Gamma)}.
  \end{align}
\end{theorem}

We give auxiliary lemmas for Theorem \ref{T:DeRham_T}.

\begin{lemma} \label{L:TGr_G_HinT}
  There exists a constant $c>0$ such that
  \begin{align} \label{E:TGr_G_HinT}
    c^{-1}\|\nabla_\Gamma q\|_{H^{-1}(\Gamma,T\Gamma)} \leq \|g\nabla_\Gamma q\|_{H^{-1}(\Gamma,T\Gamma)} \leq c\|\nabla_\Gamma q\|_{H^{-1}(\Gamma,T\Gamma)}
  \end{align}
  for all $q\in L^2(\Gamma)$.
\end{lemma}

\begin{proof}
  Since $g\in C^1(\Gamma)$ satisfies \eqref{E:LBo_G}, we have $g^{-1}\in C^1(\Gamma)$ and thus
  \begin{align*}
    \left|[\nabla_\Gamma q,v]_{T\Gamma}\right| = \left|[g\nabla_\Gamma q,g^{-1}v]_{T\Gamma}\right| &\leq \|g\nabla_\Gamma q\|_{H^{-1}(\Gamma,T\Gamma)}\|g^{-1}v\|_{H^1(\Gamma)} \\
    &\leq c\|g\nabla_\Gamma q\|_{H^{-1}(\Gamma,T\Gamma)}\|v\|_{H^1(\Gamma)}
  \end{align*}
  for all $v\in H^1(\Gamma,T\Gamma)$.
  Hence the left-hand inequality of \eqref{E:TGr_G_HinT} holds.
  Similarly, we can show the right-hand inequality of \eqref{E:TGr_G_HinT}.
\end{proof}

\begin{lemma} \label{L:TGr_Closed}
  The subspace
  \begin{align} \label{E:Sub_TGr}
    \mathcal{X} := \{g\nabla_\Gamma q\in H^{-1}(\Gamma,T\Gamma) \mid q\in L^2(\Gamma)\}
  \end{align}
  is closed in $H^{-1}(\Gamma,T\Gamma)$.
\end{lemma}

\begin{proof}
  Let $\{q_k\}_{k=1}^\infty$ be a sequence in $L^2(\Gamma)$ such that
  \begin{align} \label{Pf_TGrC:F}
    \lim_{k\to\infty}g\nabla_\Gamma q_k = f \quad\text{strongly in}\quad H^{-1}(\Gamma,T\Gamma)
  \end{align}
  with some $f\in H^{-1}(\Gamma,T\Gamma)$.
  Replacing $q_k$ with
  \begin{align*}
    q_k(y)-\frac{1}{|\Gamma|}\int_\Gamma q_k(z)\,d\mathcal{H}^2(z), \quad y\in\Gamma,
  \end{align*}
  where $|\Gamma|$ is the area of $\Gamma$, we may assume that
  \begin{align*}
    \int_\Gamma q_k\,d\mathcal{H}^2 = 0
  \end{align*}
  without changing $g\nabla_\Gamma q_k$ in $H^{-1}(\Gamma,T\Gamma)$ for each $k\in\mathbb{N}$ (see Lemma \ref{L:TGr_HinT_Con}).
  Then it follows from \eqref{E:Poin_L2_HinT} and \eqref{E:TGr_G_HinT} that
  \begin{align*}
    \|q_k-q_l\|_{L^2(\Gamma)} \leq c\|g\nabla_\Gamma q_k-g\nabla_\Gamma q_l\|_{H^{-1}(\Gamma,T\Gamma)} \to 0 \quad\text{as}\quad k,l\to\infty.
  \end{align*}
  Hence $\{q_k\}_{k=1}^\infty$ is a Cauchy sequence in $L^2(\Gamma)$ and converges to some $q$ strongly in $L^2(\Gamma)$.
  From this fact, \eqref{E:HinL2_Bo}, and \eqref{E:TGr_G_HinT} we deduce that
  \begin{align} \label{Pf_TGrC:G_Tgr}
    \|g\nabla_\Gamma q-g\nabla_\Gamma q_k\|_{H^{-1}(\Gamma,T\Gamma)} \leq c\|q-q_k\|_{L^2(\Gamma)} \to 0 \quad\text{as}\quad k\to\infty.
  \end{align}
  By \eqref{Pf_TGrC:F} and \eqref{Pf_TGrC:G_Tgr} we get $f=g\nabla_\Gamma q\in\mathcal{X}$.
  Thus $\mathcal{X}$ is closed in $H^{-1}(\Gamma,T\Gamma)$.
\end{proof}

To prove Theorem \ref{T:DeRham_T} we use basic results of functional analysis.
Let $\mathcal{B}$ and $\mathcal{B}'$ be a Banach space and its dual space, and ${}_{\mathcal{B}'}\langle\cdot,\cdot\rangle_{\mathcal{B}}$ the duality product between $\mathcal{B}'$ and $\mathcal{B}$.
For a subset $\mathcal{X}$ of $\mathcal{B}$ we define the annihilator of $\mathcal{X}$ by
\begin{align*}
    \mathcal{X}^\perp:=\{f\in\mathcal{B}' \mid \text{${}_{\mathcal{B}'}\langle f,v\rangle_{\mathcal{B}}=0$ for all $v\in\mathcal{X}$}\}.
\end{align*}

\begin{lemma} \label{L:Re_FA_1}
  For subsets $\mathcal{X}$ and $\mathcal{Y}$ of $\mathcal{B}$, if $\mathcal{X}\subset\mathcal{Y}$ in $\mathcal{B}$ then $\mathcal{Y}^\perp\subset\mathcal{X}^\perp$ in $\mathcal{B}'$.
\end{lemma}

\begin{lemma} \label{L:Re_FA_2}
  If $\mathcal{B}$ is reflexive and $\mathcal{X}$ is a closed subspace of $\mathcal{B}$, then $(\mathcal{X}^\perp)^\perp=\mathcal{X}$.
\end{lemma}

Lemma \ref{L:Re_FA_1} easily follows from the definition of the annihilator and Lemma \ref{L:Re_FA_2} is due to the Hahn--Banach theorem (see e.g. \cite{Ru91}*{Theorem 4.7}).

\begin{proof}[Proof of Theorem \ref{T:DeRham_T}]
  First note that $H^{-1}(\Gamma,T\Gamma)$ is reflexive, since it is the dual space of the Hilbert space $H^1(\Gamma,T\Gamma)$.
  Let $\mathcal{X}$ be the subspace of $H^{-1}(\Gamma,T\Gamma)$ given by \eqref{E:Sub_TGr} and $v\in\mathcal{X}^\perp\subset H^1(\Gamma,T\Gamma)$.
  Then
  \begin{align*}
    0 = [g\nabla_\Gamma q,v]_{T\Gamma} = -\bigl(q,\mathrm{div}_\Gamma(gv)\bigr)_{L^2(\Gamma)}
  \end{align*}
  for all $q\in L^2(\Gamma)$ by $g\nabla_\Gamma q\in\mathcal{X}$, \eqref{E:TGr_HinT}, and \eqref{E:Def_Mul_HinT}, and thus
  \begin{align*}
    \mathrm{div}_\Gamma(gv) = 0 \quad\text{on}\quad \Gamma, \quad\text{i.e.}\quad v \in H_{g\sigma}^1(\Gamma,T\Gamma).
  \end{align*}
  Hence $\mathcal{X}^\perp\subset H_{g\sigma}^1(\Gamma,T\Gamma)$ in $H^1(\Gamma,T\Gamma)$ and it follows from Lemma \ref{L:Re_FA_1} that
  \begin{align*}
    H_{g\sigma}^1(\Gamma,T\Gamma)^\perp \subset (\mathcal{X}^\perp)^\perp \quad\text{in}\quad H^{-1}(\Gamma,T\Gamma).
  \end{align*}
  Moreover, since $\mathcal{X}$ is closed in $H^{-1}(\Gamma,T\Gamma)$ by Lemma \ref{L:TGr_Closed}, we have
  \begin{align*}
    (\mathcal{X}^\perp)^\perp = \mathcal{X} \quad\text{in}\quad H^{-1}(\Gamma,T\Gamma)
  \end{align*}
  by Lemma \ref{L:Re_FA_2}.
  Thus
  \begin{align*}
    H_{g\sigma}^1(\Gamma,T\Gamma)^\perp &= \{f\in H^{-1}(\Gamma,T\Gamma) \mid \text{$[f,v]_{T\Gamma}=0$ for all $v\in H_{g\sigma}^1(\Gamma,T\Gamma)$}\} \\
    &\subset \mathcal{X} \quad\text{in}\quad H^{-1}(\Gamma,T\Gamma),
  \end{align*}
  i.e. for each $f\in H_{g\sigma}^1(\Gamma,T\Gamma)^\perp$ there exists $q\in L^2(\Gamma)$ such that
  \begin{align*}
    f = g\nabla_\Gamma q \quad\text{in}\quad H^{-1}(\Gamma,T\Gamma).
  \end{align*}
  Moreover, replacing $q$ with
  \begin{align*}
    q(y)-\frac{1}{|\Gamma|}\int_\Gamma q(z)\,d\mathcal{H}^2(z), \quad y\in\Gamma
  \end{align*}
  we may assume that
  \begin{align*}
    \int_\Gamma q\,d\mathcal{H}^2 = 0
  \end{align*}
  without changing $g\nabla_\Gamma q$ in $H^{-1}(\Gamma,T\Gamma)$ (see Lemma \ref{L:TGr_HinT_Con}).
  Hence the existence part of the theorem is valid.
  To prove the uniqueness, suppose that $q\in L^2(\Gamma)$ satisfy
  \begin{align*}
    g\nabla_\Gamma q = 0 \quad\text{in}\quad H^{-1}(\Gamma,T\Gamma), \quad \int_\Gamma q\,d\mathcal{H}^2 = 0.
  \end{align*}
  Then it follows from \eqref{E:Poin_L2_HinT} and \eqref{E:TGr_G_HinT} that
  \begin{align*}
    \|q\|_{L^2(\Gamma)} \leq c\left(\left|\int_\Gamma q\,d\mathcal{H}^2\right|+\|g\nabla_\Gamma q\|_{H^{-1}(\Gamma,T\Gamma)}\right) = 0.
  \end{align*}
  Thus $q=0$ on $\Gamma$ and the uniqueness follows.
  We also have the estimate \eqref{E:DeRham_T_Ineq} by \eqref{E:Poin_L2_HinT} with $\int_\Gamma q\,d\mathcal{H}^2=0$ and \eqref{E:TGr_G_HinT}.
\end{proof}

\subsection{Weighted Helmholtz--Leray decomposition on a closed surface} \label{SS:WS_HL}
The purpose of this subsection is to prove the weighted Helmholtz--Leray decomposition for a tangential vector field on $\Gamma$.
We also derive several estimates for the difference between a tangential vector field on $\Gamma$ and its weighted solenoidal part.

As in Section \ref{SS:WS_An}, let $g\in C^1(\Gamma)$ satisfy \eqref{E:LBo_G}.
For $v\in L^2(\Gamma)^3$ we consider the surface divergence of $gv$ in $H^{-1}(\Gamma)$ by \eqref{E:Sdiv_Hin}.
Then
\begin{align*}
  |\langle\mathrm{div}_\Gamma(gv),\eta\rangle_\Gamma| = |(gv,\nabla_\Gamma\eta+\eta Hn)_{L^2(\Gamma)}| \leq c\|v\|_{L^2(\Gamma)}\|\eta\|_{H^1(\Gamma)}
\end{align*}
for all $\eta\in H^1(\Gamma)$ since $n$, $H$, and $g$ are bounded on $\Gamma$.
Hence
\begin{align} \label{E:DivG_Hin_Bo}
  \|\mathrm{div}_\Gamma(gv)\|_{H^{-1}(\Gamma)} \leq c\|v\|_{L^2(\Gamma)}
\end{align}
for all $v\in L^2(\Gamma)^3$.
We define a subspace of $L^2(\Gamma,T\Gamma)$ by
\begin{align*}
  L_{g\sigma}^2(\Gamma,T\Gamma) := \{v\in L^2(\Gamma,T\Gamma) \mid \text{$\mathrm{div}_\Gamma(gv)=0$ in $H^{-1}(\Gamma)$}\}.
\end{align*}
Then $L_{g\sigma}^2(\Gamma,T\Gamma)$ is closed in $L^2(\Gamma,T\Gamma)$ by \eqref{E:DivG_Hin_Bo} and we have the orthogonal decomposition
\begin{align} \label{E:L2T_OD}
  L^2(\Gamma,T\Gamma) = L_{g\sigma}^2(\Gamma,T\Gamma)\oplus L_{g\sigma}^2(\Gamma,T\Gamma)^\perp.
\end{align}
Let us give the characterization of the orthogonal complement $L_{g\sigma}^2(\Gamma,T\Gamma)^\perp$.

\begin{lemma} \label{L:L2gs_Orth}
  The orthogonal complement of $L_{g\sigma}^2(\Gamma,T\Gamma)$ is of the form
  \begin{align} \label{E:L2gs_Orth}
    L_{g\sigma}^2(\Gamma,T\Gamma)^\perp = \{g\nabla_\Gamma q \in L^2(\Gamma,T\Gamma) \mid q\in H^1(\Gamma)\}.
  \end{align}
\end{lemma}

\begin{proof}
  We denote by $\mathcal{X}$ the right-hand side of \eqref{E:L2gs_Orth}.
  By \eqref{E:Sdiv_Hin} we have
  \begin{align*}
    (v,g\nabla_\Gamma q)_{L^2(\Gamma)} = -\langle\mathrm{div}_\Gamma(gv),q\rangle_\Gamma = 0 \quad\text{for all}\quad v\in L_{g\sigma}^2(\Gamma,T\Gamma),\, q\in H^1(\Gamma).
  \end{align*}
  Thus $\mathcal{X}\subset L_{g\sigma}^2(\Gamma,T\Gamma)^\perp$.
  Conversely, let $f\in L_{g\sigma}^2(\Gamma,T\Gamma)^\perp$.
  Since $f=Pf$ on $\Gamma$,
  \begin{align*}
    \langle f,v\rangle_\Gamma = (f,v)_{L^2(\Gamma)} = (Pf,v)_{L^2(\Gamma)} = \langle Pf,v\rangle_\Gamma \quad\text{for all}\quad v\in H^1(\Gamma)^3.
  \end{align*}
  Hence $f=Pf$ in $H^{-1}(\Gamma)^3$ and we can consider $f$ in $H^{-1}(\Gamma,T\Gamma)$ (see Section \ref{SS:Pre_Surf}).
  Moreover, since
  \begin{align*}
    f \in L_{g\sigma}^2(\Gamma,T\Gamma)^\perp, \quad H_{g\sigma}^1(\Gamma,T\Gamma)\subset L_{g\sigma}^2(\Gamma,T\Gamma),
  \end{align*}
  as an element of $H^{-1}(\Gamma,T\Gamma)$ the vector field $f$ satisfies
  \begin{align*}
    [f,v]_{T\Gamma} = (f,v)_{L^2(\Gamma)} = 0 \quad\text{for all}\quad v\in H_{g\sigma}^1(\Gamma,T\Gamma).
  \end{align*}
  Hence by Theorem \ref{T:DeRham_T} there exists a unique $q\in L^2(\Gamma)$ such that
  \begin{align} \label{Pf_LO:Chara}
    f = g\nabla_\Gamma q \quad\text{in}\quad H^{-1}(\Gamma,T\Gamma), \quad \int_\Gamma q\,d\mathcal{H}^2 = 0.
  \end{align}
  Let us show $q\in H^1(\Gamma)$.
  For $\eta\in C^1(\Gamma)$ and $i=1,2,3$ let $v:=g^{-1}\eta Pe_i$ on $\Gamma$, where $\{e_1,e_2,e_3\}$ is the standard basis of $\mathbb{R}^3$.
  Then $v\in H^1(\Gamma,T\Gamma)$ by the $C^1$-regularity of $P$ and $g$ on $\Gamma$ and \eqref{E:LBo_G}.
  Moreover, by $P^T=P$ on $\Gamma$, \eqref{E:P_TGr}, and \eqref{E:Div_P},
  \begin{align*}
    \mathrm{div}_\Gamma(gv) = \nabla_\Gamma\eta\cdot Pe_i+\eta(\mathrm{div}_\Gamma P\cdot e_i) = \underline{D}_i\eta+\eta Hn_i \quad\text{on}\quad \Gamma.
  \end{align*}
  By this equality, \eqref{E:TGr_HinT}, \eqref{E:Def_Mul_HinT}, \eqref{Pf_LO:Chara}, and $P^Tf=Pf=f$ on $\Gamma$ we obtain
  \begin{align*}
    -(q,\underline{D}_i\eta+\eta Hn_i)_{L^2(\Gamma)} &= -\bigl(q,\mathrm{div}_\Gamma(gv)\bigr)_{L^2(\Gamma)} = [g\nabla_\Gamma q,v]_{T\Gamma} \\
    &= [f,v]_{T\Gamma} = (f,v)_{L^2(\Gamma)} = (g^{-1}f_i,\eta)_{L^2(\Gamma)}
  \end{align*}
  for all $\eta\in C^1(\Gamma)$ and $i=1,2,3$, where $f_i$ is the $i$-th component of $f$.
  Thus
  \begin{align*}
    \underline{D}_iq = g^{-1}f_i \in L^2(\Gamma), \quad i=1,2,3
  \end{align*}
  by the definition of the weak tangential derivative in $L^2(\Gamma)$ (see \eqref{E:Def_WTD}) and
  \begin{align*}
    q \in H^1(\Gamma), \quad f = g\nabla_\Gamma q \in \mathcal{X}.
  \end{align*}
  Hence $L_{g\sigma}^2(\Gamma,T\Gamma)^\perp\subset\mathcal{X}$ and the relation \eqref{E:L2gs_Orth} holds.
\end{proof}

Now we obtain the weighted Helmholtz--Leray decomposition in $L^2(\Gamma,T\Gamma)$.

\begin{theorem} \label{T:HL_L2T}
  For each $v\in L^2(\Gamma,T\Gamma)$ we have the orthogonal decomposition
  \begin{align} \label{E:HL_L2T}
    \begin{gathered}
      v = v_g+g\nabla_\Gamma q \quad\text{in}\quad L^2(\Gamma,T\Gamma), \\
      v_g\in L_{g\sigma}^2(\Gamma,T\Gamma), \, g\nabla_\Gamma q \in L_{g\sigma}^2(\Gamma,T\Gamma)^\perp.
    \end{gathered}
  \end{align}
  Here $q\in H^1(\Gamma)$ is determined uniquely up to a constant.
\end{theorem}

\begin{proof}
  The decomposition \eqref{E:HL_L2T} is an immediate consequence of \eqref{E:L2T_OD} and \eqref{E:L2gs_Orth}.
  If $q_1,q_2\in H^1(\Gamma)$ satisfy $g\nabla_\Gamma q_1=g\nabla_\Gamma q_2$ in $L^2(\Gamma,T\Gamma)$, then
  \begin{align*}
    \nabla_\Gamma(q_1-q_2) = 0 \quad\text{in}\quad L^2(\Gamma)^3
  \end{align*}
  by \eqref{E:LBo_G} and thus $q_1-q_2$ is constant on $\Gamma$ by Lemma \ref{L:WTGr_Van}.
  Hence $q\in H^1(\Gamma)$ in \eqref{E:HL_L2T} is determined uniquely up to a constant.
\end{proof}

Note that here we proved the weighted Helmholtz--Leray decomposition without introducing the notion of differential forms.
The decomposition \eqref{E:HL_L2T} with $g\equiv1$ was also shown in the recent work \cite{Reus20} without calculus of differential forms, where the solenoidal part is further decomposed into the curl of some function and a harmonic field whose surface divergence and curl vanish.

Next we consider approximation of vector fields in $L_{g\sigma}^2(\Gamma,T\Gamma)$.
In general,
\begin{align*}
  \mathrm{div}_\Gamma\bigl(g(\eta v)\bigr) = \eta\,\mathrm{div}_\Gamma(gv)+g\nabla_\Gamma\eta\cdot v = g\nabla_\Gamma\eta\cdot v
\end{align*}
does not vanish in $H^{-1}(\Gamma)$ for $v\in L_{g\sigma}^2(\Gamma,T\Gamma)$ and $\eta\in C^1(\Gamma)$.
Hence we cannot apply standard localization and mollification argument with a partition of unity on $\Gamma$ to approximate a vector field in $L_{g\sigma}^2(\Gamma,T\Gamma)$ by smooth weighted solenoidal vector fields on $\Gamma$.
Instead, we use a solution to Poisson's equation on $\Gamma$.

\begin{lemma} \label{L:Pois_Surf}
  For each $\eta\in H^{-1}(\Gamma)$ satisfying $\langle\eta,1\rangle_\Gamma=0$ there exists a unique weak solution $q\in H^1(\Gamma)$ to Poisson's equation
  \begin{align} \label{E:Pois_Surf}
    -\Delta_\Gamma q = \eta \quad\text{on}\quad \Gamma, \quad \int_\Gamma q\,d\mathcal{H}^2 = 0
  \end{align}
  in the sense that
  \begin{align} \label{E:Weak_Pois}
    (\nabla_\Gamma q,\nabla_\Gamma\xi)_{L^2(\Gamma)} = \langle\eta,\xi\rangle_\Gamma \quad\text{for all}\quad \xi\in H^1(\Gamma).
  \end{align}
  Moreover, there exists a constant $c>0$ such that
  \begin{align} \label{E:H1_Pois}
    \|q\|_{H^1(\Gamma)} \leq c\|\eta\|_{H^{-1}(\Gamma)}.
  \end{align}
  If in addition $\eta\in L^2(\Gamma)$, then $q\in H^2(\Gamma)$ and
  \begin{align} \label{E:H2_Pois}
    \|q\|_{H^2(\Gamma)} \leq c\|\eta\|_{L^2(\Gamma)}.
  \end{align}
\end{lemma}

The existence and uniqueness of a weak solution to \eqref{E:Pois_Surf} and the estimate \eqref{E:H1_Pois} follow from Poincar\'{e}'s inequality \eqref{E:Poin_Surf_Lp} and the Lax--Milgram theorem.
Also, the $H^2$-regularity of a weak solution and \eqref{E:H2_Pois} are proved by a localization argument and the elliptic regularity theorem.
For details, see \cite{DzEl13}*{Theorems 3.1 and 3.3}.

\begin{lemma} \label{L:H1gs_Dense}
  The space $H_{g\sigma}^1(\Gamma,T\Gamma)$ is dense in $L_{g\sigma}^2(\Gamma,T\Gamma)$.
\end{lemma}

\begin{proof}
  Let $v\in L_{g\sigma}^2(\Gamma,T\Gamma)$.
  Since $\Gamma$ is of class $C^2$, we can take a sequence $\{\tilde{v}_k\}_{k=1}^\infty$ in $C^1(\Gamma,T\Gamma)$ that converges to $v$ strongly in $L^2(\Gamma,T\Gamma)$ by Lemma \ref{L:Wmp_Tan_Appr}.
  Then
  \begin{align} \label{Pf_HgDe:Hin_Conv}
    \|\mathrm{div}_\Gamma(g\tilde{v}_k)\|_{H^{-1}(\Gamma)} = \|\mathrm{div}_\Gamma[g(\tilde{v}_k-v)]\|_{H^{-1}(\Gamma)} \leq c\|\tilde{v}_k-v\|_{L^2(\Gamma)}
  \end{align}
  for each $k\in\mathbb{N}$ by $\mathrm{div}_\Gamma(gv)=0$ in $H^{-1}(\Gamma)$ and \eqref{E:DivG_Hin_Bo}.
  Let
  \begin{align*}
    \eta_k := -\mathrm{div}_\Gamma(g\tilde{v}_k) \in C(\Gamma) \subset L^2(\Gamma).
  \end{align*}
  Then by $g\tilde{v}_k\in C^1(\Gamma,T\Gamma)$ and \eqref{E:SD_Thm} we have
  \begin{align*}
    \langle\eta_k,1\rangle_\Gamma = -\int_\Gamma\mathrm{div}_\Gamma(g\tilde{v}_k)\,d\mathcal{H}^2 = 0
  \end{align*}
  and thus there exists a unique solution $q_k\in H^2(\Gamma)$ to \eqref{E:Pois_Surf} with source term $\eta_k$ by Lemma \ref{L:Pois_Surf}.
  Moreover, by \eqref{E:H1_Pois} and \eqref{Pf_HgDe:Hin_Conv},
  \begin{align*}
    \|q_k\|_{H^1(\Gamma)} \leq c\|\eta_k\|_{H^{-1}(\Gamma)} = c\|\mathrm{div}_\Gamma(g\tilde{v}_k)\|_{H^{-1}(\Gamma)} \leq c\|v-\tilde{v}_k\|_{L^2(\Gamma)}.
  \end{align*}
  Hence $v_k:=\tilde{v}_k-g^{-1}\nabla_\Gamma q_k\in H_{g\sigma}^1(\Gamma,T\Gamma)$ for each $k\in\mathbb{N}$ and
  \begin{align*}
    \|v-v_k\|_{L^2(\Gamma)} &\leq \|v-\tilde{v}_k\|_{L^2(\Gamma)}+c\|q_k\|_{H^1(\Gamma)} \\
    &\leq c\|v-\tilde{v}_k\|_{L^2(\Gamma)} \to 0 \quad\text{as}\quad k\to\infty
  \end{align*}
  by \eqref{E:LBo_G} and the strong convergence of $\{\tilde{v}_k\}_{k=1}^\infty$ to $v$ in $L^2(\Gamma,T\Gamma)$.
\end{proof}

Let $\mathbb{P}_g$ be the orthogonal projection from $L^2(\Gamma,T\Gamma)$ onto $L_{g\sigma}^2(\Gamma,T\Gamma)$.
We call it the weighted Helmholtz--Leray projection.
Let us estimate the difference $v-\mathbb{P}_gv$.

\begin{lemma} \label{L:HLT_Est_L2}
  There exists a constant $c>0$ such that
  \begin{align} \label{E:HLT_Est_L2}
    \|v-\mathbb{P}_gv\|_{L^2(\Gamma)} \leq c\|\mathrm{div}_\Gamma(gv)\|_{H^{-1}(\Gamma)}
  \end{align}
  for all $v\in L^2(\Gamma,T\Gamma)$.
  If in addition $v\in H^1(\Gamma,T\Gamma)$, then $\mathbb{P}_gv\in H_{g\sigma}^1(\Gamma,T\Gamma)$ and
  \begin{align} \label{E:HLT_Est_H1}
    \|v-\mathbb{P}_gv\|_{H^1(\Gamma)} \leq c\|\mathrm{div}_\Gamma(gv)\|_{L^2(\Gamma)}.
  \end{align}
\end{lemma}

\begin{proof}
  Let $v\in L^2(\Gamma,T\Gamma)$ and $\eta:=-\mathrm{div}_\Gamma(gv)\in H^{-1}(\Gamma)$.
  Since
  \begin{align*}
    \langle\eta,1\rangle_\Gamma = (gv,Hn)_{L^2(\Gamma)} = \int_\Gamma g(v\cdot n)H\,d\mathcal{H}^2 = 0
  \end{align*}
  by \eqref{E:Sdiv_Hin} and $v\cdot n=0$ on $\Gamma$, we observe by Lemma \ref{L:Pois_Surf} that there exists a unique weak solution $q\in H^1(\Gamma)$ to \eqref{E:Pois_Surf} with source term $\eta$.
  Then
  \begin{align*}
    \mathbb{P}_gv = v-\frac{1}{g}\nabla_\Gamma q\in L_{g\sigma}^2(\Gamma,T\Gamma)
  \end{align*}
  by the uniqueness of the decomposition \eqref{E:HL_L2T} and
  \begin{align*}
    \|v-\mathbb{P}_gv\|_{L^2(\Gamma)} \leq c\|q\|_{H^1(\Gamma)} \leq c\|\eta\|_{H^{-1}(\Gamma)} = c\|\mathrm{div}_\Gamma(gv)\|_{H^{-1}(\Gamma)}
  \end{align*}
  by \eqref{E:LBo_G} and \eqref{E:H1_Pois}.
  Thus \eqref{E:HLT_Est_L2} holds.

  If $v\in H^1(\Gamma,T\Gamma)$, then $\eta=-\mathrm{div}_\Gamma(gv)\in L^2(\Gamma)$ and Lemma \ref{L:Pois_Surf} implies
  \begin{align*}
    q \in H^2(\Gamma), \quad \mathbb{P}_gv = v-\frac{1}{g}\nabla_\Gamma q \in H_{g\sigma}^1(\Gamma,T\Gamma).
  \end{align*}
  Moreover, by $g\in C^1(\Gamma)$, \eqref{E:LBo_G}, and \eqref{E:H2_Pois} we have
  \begin{align*}
    \|v-\mathbb{P}_gv\|_{H^1(\Gamma)} \leq c\|q\|_{H^2(\Gamma)} \leq c\|\eta\|_{L^2(\Gamma)} = c\|\mathrm{div}_\Gamma(gv)\|_{L^2(\Gamma)}.
  \end{align*}
  Hence \eqref{E:HLT_Est_H1} is valid.
\end{proof}

\begin{lemma} \label{L:HLT_Bound}
  There exists a constant $c>0$ such that
  \begin{align} \label{E:HLT_Bound}
    \|\mathbb{P}_gv\|_{H^k(\Gamma)} \leq c\|v\|_{H^k(\Gamma)}
  \end{align}
  for all $v\in H^k(\Gamma,T\Gamma)$, $k=0,1$ (note that $H^0=L^2$).
\end{lemma}

\begin{proof}
  If $k=0$, then \eqref{E:HLT_Bound} holds with $c=1$ since $\mathbb{P}_g$ is the orthogonal projection from $L^2(\Gamma,T\Gamma)$ onto $L_{g\sigma}^2(\Gamma,T\Gamma)$.
  Also, since
  \begin{align*}
    \|\mathrm{div}_\Gamma(gv)\|_{L^2(\Gamma)} \leq c\|v\|_{H^1(\Gamma)}
  \end{align*}
  for $v\in H^1(\Gamma)^3$, the inequality \eqref{E:HLT_Bound} for $k=1$ follows from \eqref{E:HLT_Est_H1}.
\end{proof}

Next we derive an estimate for the time derivative of $v-\mathbb{P}_gv$.
To this end, we consider the time derivative of a weak solution to Poisson's equation \eqref{E:Pois_Surf}.

\begin{lemma} \label{L:Pois_Dt}
  Let $T>0$.
  Suppose that $\eta\in H^1(0,T;H^{-1}(\Gamma))$ satisfies
  \begin{align*}
    \langle\eta(t),1\rangle_\Gamma = 0 \quad\text{for all}\quad t\in[0,T].
  \end{align*}
  For each $t\in[0,T]$ let $q(t)\in H^1(\Gamma)$ be a unique weak solution to \eqref{E:Pois_Surf} with source term $\eta(t)$.
  Then $q\in H^1(0,T;H^1(\Gamma))$ and there exists a constant $c>0$ such that
  \begin{align} \label{E:Pois_Dt}
    \|\partial_tq\|_{L^2(0,T;H^1(\Gamma))} \leq c\|\partial_t\eta\|_{L^2(0,T;H^{-1}(\Gamma))}.
  \end{align}
  Moreover, for a.a. $t\in(0,T)$ the time derivative $\partial_tq(t)\in H^1(\Gamma)$ is a unique weak solution to \eqref{E:Pois_Surf} with source term $\partial_t\eta(t)$.
\end{lemma}

Note that, since the inclusion
\begin{align*}
  H^1(0,T;H^{-1}(\Gamma)) \subset C([0,T];H^{-1}(\Gamma))
\end{align*}
holds, $\eta(t)\in H^{-1}(\Gamma)$ is well-defined for all $t\in[0,T]$ if $\eta\in H^1(0,T;H^{-1}(\Gamma))$.

\begin{proof}
  First note that $q\in L^2(0,T;H^1(\Gamma))$ by $\eta\in L^2(0,T;H^{-1}(\Gamma))$ and \eqref{E:H1_Pois}.
  Let us show $\partial_tq\in L^2(0,T;H^1(\Gamma))$ by means of the difference quotient.
  Fix $\delta\in(0,T/2)$ and $h\in\mathbb{R}\setminus\{0\}$ with $|h|<\delta/2$.
  For $t\in(\delta,T-\delta)$ we define
  \begin{align*}
    D_hq(t) := \frac{q(t+h)-q(t)}{h} \in H^1(\Gamma), \quad D_h\eta(t) := \frac{\eta(t+h)-\eta(t)}{h} \in H^{-1}(\Gamma).
  \end{align*}
  Note that these definitions make sense since $t+h\in(\delta/2,T-\delta/2)$.
  Moreover,
  \begin{align*}
    \int_\Gamma D_hq(t)\,d\mathcal{H}^2 = 0, \quad \langle D_h\eta(t),1\rangle_\Gamma = 0, \quad t\in(\delta,T-\delta)
  \end{align*}
  since $q(t)$ and $\eta(t)$ satisfy the same equalities for all $t\in[0,T]$.
  For each $\xi\in H^1(\Gamma)$ we subtract \eqref{E:Weak_Pois} for $q(t)$ from that for $q(t+h)$ and divide both sides by $h$ to get
  \begin{align} \label{Pf_PoDt:Weak}
    (\nabla_\Gamma D_hq(t),\nabla_\Gamma\xi)_{L^2(\Gamma)} = \langle D_h\eta(t),\xi\rangle_\Gamma.
  \end{align}
  Hence $D_hq(t)$ is a unique weak solution to \eqref{E:Pois_Surf} with source term $D_h\eta(t)$ and
  \begin{align*}
    \|D_hq(t)\|_{H^1(\Gamma)} \leq c\|D_h\eta(t)\|_{H^{-1}(\Gamma)}, \quad t\in(\delta,T-\delta)
  \end{align*}
  by \eqref{E:H1_Pois}, where $c>0$ is a constant independent of $t$, $\delta$, and $h$.
  Thus
  \begin{align*}
    \|D_hq\|_{L^2(\delta,T-\delta;H^1(\Gamma))} \leq c\|D_h\eta\|_{L^2(\delta,T-\delta;H^{-1}(\Gamma))}.
  \end{align*}
  Moreover, since $\eta\in H^1(0,T;H^{-1}(\Gamma))$,
  \begin{align*}
    \|D_h\eta\|_{L^2(\delta,T-\delta;H^{-1}(\Gamma))} \leq c\|\partial_t\eta\|_{L^2(0,T;H^{-1}(\Gamma))}
  \end{align*}
  with a constant $c>0$ independent of $h$ and $\delta$ (see \cite{Ev10}*{Section 5.8, Theorem 3 (i)}).
  Combining the above two estimates we obtain
  \begin{align*}
    \|D_hq\|_{L^2(\delta,T-\delta;H^1(\Gamma))} \leq c\|\partial_t\eta\|_{L^2(0,T;H^{-1}(\Gamma))}
  \end{align*}
  for all $h\in\mathbb{R}\setminus\{0\}$ with $|h|<\delta/2$.
  Since the right-hand side of this inequality is independent of $h$, it follows that $\partial_tq\in L^2(\delta,T-\delta;H^1(\Gamma))$ and
  \begin{align*}
    \|\partial_tq\|_{L^2(\delta,T-\delta;H^1(\Gamma))} \leq c\|\partial_t\eta\|_{L^2(0,T;H^{-1}(\Gamma))}
  \end{align*}
  for all $\delta\in(0,T/2)$ (see \cite{Ev10}*{Section 5.8, Theorem 3 (ii)}).
  Thus $\partial_tq(t)\in H^1(\Gamma)$ for a.a. $t\in(0,T)$ and, since the right-hand side of the above inequality is independent of $\delta$, the monotone convergence theorem yields
  \begin{align*}
    \|\partial_tq\|_{L^2(0,T;H^1(\Gamma))} = \lim_{\delta\to0}\|\partial_tq\|_{L^2(\delta,T-\delta;H^1(\Gamma))} \leq c\|\partial_t\eta\|_{L^2(0,T;H^{-1}(\Gamma))}.
  \end{align*}
  Hence $\partial_tq\in L^2(0,T;H^1(\Gamma))$ and the inequality \eqref{E:Pois_Dt} is valid.

  Next we show that $\partial_tq(t)$ is a unique weak solution to \eqref{E:Pois_Surf} with source term $\partial_t\eta(t)$ for a.a. $t\in(0,T)$.
  Let $\xi\in H^1(\Gamma)$ and $\varphi\in C_c^\infty(0,T)$.
  We may assume that $\varphi$ is supposed in $(\delta,T-\delta)$ with some $\delta\in(0,T/2)$.
  Moreover, we extend $\varphi$ to $\mathbb{R}$ by zero outside $(0,T)$.
  For $h\in\mathbb{R}\setminus\{0\}$ with $|h|<\delta/2$ we multiply both sides of \eqref{Pf_PoDt:Weak} by $\varphi(t)$, integrate them over $(\delta,T-\delta)$, and make the change of a variable
  \begin{align*}
    \int_\delta^{T-\delta}\psi(t+h)\varphi(t)\,dt = \int_{\delta+h}^{T-\delta+h}\psi(s)\varphi(s-h)\,ds
  \end{align*}
  for $\psi(t)=(\nabla_\Gamma q(t),\nabla_\Gamma\xi)_{L^2(\Gamma)}$ and $\psi(t)=\langle\eta(t),\xi\rangle_\Gamma$ to get
  \begin{align} \label{Pf_PoDt:Diff_Test}
    -\int_0^T(\nabla_\Gamma q(t),\nabla_\Gamma\xi)_{L^2(\Gamma)}D_{-h}\varphi(t)\,dt = -\int_0^T\langle\eta(t),\xi\rangle_\Gamma D_{-h}\varphi(t)\,dt,
  \end{align}
  where (note that $\varphi$ is supported in $(\delta,T-\delta)$ and $|h|<\delta/2$)
  \begin{align*}
    D_{-h}\varphi(t) := \frac{\varphi(t-h)-\varphi(t)}{-h}, \quad t\in(0,T).
  \end{align*}
  Let $h\to0$ in \eqref{Pf_PoDt:Diff_Test}.
  Then since $D_{-h}\varphi$ converges to $\partial_t\varphi$ uniformly on $(0,T)$,
  \begin{align*}
    -\int_0^T(\nabla_\Gamma q(t),\nabla_\Gamma\xi)_{L^2(\Gamma)}\partial_t\varphi(t)\,dt = -\int_0^T\langle\eta(t),\xi\rangle_\Gamma\partial_t\varphi(t)\,dt
  \end{align*}
  for all $\varphi \in C_c^\infty(0,T)$.
  By this equality, $\eta\in H^1(0,T;H^{-1}(\Gamma))$, and
  \begin{align*}
    q\in H^1(0,T;H^1(\Gamma)), \quad \partial_t(\nabla_\Gamma q) = \nabla_\Gamma(\partial_tq) \quad\text{a.e. on}\quad \Gamma\times(0,T)
  \end{align*}
  we obtain
  \begin{align*}
    ([\nabla_\Gamma(\partial_tq)](t),\nabla_\Gamma\xi)_{L^2(\Gamma)} = \langle\partial_t\eta(t),\xi\rangle_\Gamma
  \end{align*}
  for all $\xi\in H^1(\Gamma)$ and a.a. $t\in(0,T)$.
  In the same way we can show
  \begin{align*}
    \int_\Gamma\partial_tq(t)\,d\mathcal{H}^2 = 0 \quad\text{for a.a.}\quad t\in(0,T)
  \end{align*}
  since $q(t)$ satisfies the same equality for all $t\in[0,T]$.
  Hence $\partial_tq(t)$ is a unique weak solution to \eqref{E:Pois_Surf} with source term $\partial_t\eta(t)$ for a.a. $t\in(0,T)$.
\end{proof}

\begin{lemma} \label{L:HLT_Est_Dt}
  For $T>0$ let $v\in H^1(0,T;L^2(\Gamma,T\Gamma))$.
  Then
  \begin{align*}
    \mathbb{P}_gv\in H^1(0,T;L_{g\sigma}^2(\Gamma,T\Gamma))
  \end{align*}
  and there exists a constant $c>0$ such that
  \begin{align} \label{E:HLT_Est_Dt}
    \|\partial_tv-\partial_t\mathbb{P}_gv\|_{L^2(0,T;L^2(\Gamma))} \leq c\|\mathrm{div}_\Gamma(g\partial_t v)\|_{L^2(0,T;H^{-1}(\Gamma))}.
  \end{align}
\end{lemma}

\begin{proof}
  Since $v\in H^1(0,T;L^2(\Gamma,T\Gamma))$ and $g$ and $P$ are independent of time,
  \begin{align*}
    \eta := -\mathrm{div}_\Gamma(gv) \in H^1(0,T;H^{-1}(\Gamma)), \quad \partial_t\eta = -\mathrm{div}_\Gamma(g\partial_tv)\in L^2(0,T;H^{-1}(\Gamma)).
  \end{align*}
  Note that $P$ appears in the definition of the tangential derivatives.
  Also,
  \begin{align*}
    \langle\eta(t),1\rangle_\Gamma = (gv(t),Hn)_{L^2(\Gamma)} = \int_\Gamma g\{v(t)\cdot n\}H\,d\mathcal{H}^2 = 0, \quad t\in[0,T]
  \end{align*}
  by \eqref{E:Sdiv_Hin} and $v(t)\cdot n=0$ on $\Gamma$.
  For each $t\in[0,T]$ let $q(t)\in H^1(\Gamma)$ be a unique weak solution to \eqref{E:Pois_Surf} with source term $\eta(t)=-\mathrm{div}_\Gamma[gv(t)]$.
  Then
  \begin{align*}
    q \in H^1(0,T;H^1(\Gamma)), \quad \partial_t(\nabla_\Gamma q) = \nabla_\Gamma(\partial_tq) \quad\text{a.e. on}\quad \Gamma\times(0,T)
  \end{align*}
  by Lemma \ref{L:Pois_Dt} and thus
  \begin{align} \label{Pf_HEDt:Pgv}
    \begin{gathered}
      \mathbb{P}_gv = v-\frac{1}{g}\nabla_\Gamma q \in L^2(0,T;L_{g\sigma}^2(\Gamma,T\Gamma))\cap H^1(0,T;L^2(\Gamma,T\Gamma)), \\
      \partial_t\mathbb{P}_gv = \partial_tv-\frac{1}{g}\nabla_\Gamma(\partial_tq) \quad\text{a.e. on}\quad \Gamma\times(0,T).
    \end{gathered}
  \end{align}
  Moreover, $\partial_tq(t)\in H^1(\Gamma)$ is a unique weak solution to \eqref{E:Pois_Surf} with source term $\partial_t\eta(t)=-\mathrm{div}_\Gamma[g\partial_tv(t)]$ for a.a. $t\in(0,T)$ by Lemma \ref{L:Pois_Dt}.
  Thus
  \begin{align*}
    \mathbb{P}_g(\partial_tv) = \partial_tv-\frac{1}{g}\nabla_\Gamma(\partial_tq) \quad\text{a.e. on}\quad \Gamma\times(0,T).
  \end{align*}
  By this equality and \eqref{Pf_HEDt:Pgv} we obtain
  \begin{align*}
    \partial_t\mathbb{P}_gv = \mathbb{P}_g(\partial_tv) \in L^2(0,T;L_{g\sigma}^2(\Gamma,T\Gamma)), \quad \mathbb{P}_gv \in H^1(0,T;L_{g\sigma}^2(\Gamma,T\Gamma)).
  \end{align*}
  To prove \eqref{E:HLT_Est_Dt} we observe by \eqref{E:Pois_Dt} that
  \begin{align*}
    \|\partial_tq\|_{L^2(0,T;H^1(\Gamma))} \leq c\|\partial_t\eta\|_{L^2(0,T;H^{-1}(\Gamma))} = c\|\mathrm{div}_\Gamma(g\partial_tv)\|_{L^2(0,T;H^{-1}(\Gamma))}.
  \end{align*}
  From this inequality, \eqref{E:LBo_G}, and \eqref{Pf_HEDt:Pgv} we deduce that
  \begin{align*}
    \|\partial_tv-\partial_t\mathbb{P}_gv\|_{L^2(0,T;L^2(\Gamma))} &\leq c\|\partial_tq\|_{L^2(0,T;H^1(\Gamma))} \\
    &\leq c\|\mathrm{div}_\Gamma(g\partial_tv)\|_{L^2(0,T;H^{-1}(\Gamma))}.
  \end{align*}
  Hence \eqref{E:HLT_Est_Dt} is valid.
\end{proof}

\subsection{Solenoidal spaces of general vector fields} \label{SS:WS_NTS}
In this subsection we briefly study solenoidal spaces of general (not necessarily tangential) vector fields on $\Gamma$.
Although the results of this subsection are not used in the sequel, we expect them to be useful for the future study of surface fluid equations including fluid equations on an evolving surface (see \cites{JaOlRe18,KoLiGi17,Miu18}).
For the sake of simplicity, we only consider the case $g\equiv1$ and give a remark on the case $g\not\equiv1$ at the end of this subsection.

Let $q\in L^2(\Gamma)$.
By \eqref{E:L2_Hin} and \eqref{E:TGr_Hin} we have
\begin{align} \label{E:TGrHn_Hin}
  \langle \nabla_\Gamma q+qHn,v\rangle_\Gamma = -(q,\mathrm{div}_\Gamma v)_{L^2(\Gamma)}
\end{align}
for all $v\in H^1(\Gamma)^3$ and thus
\begin{align} \label{E:TGrHn_Hin_Bo}
  \|\nabla_\Gamma q+qHn\|_{H^{-1}(\Gamma)} \leq c\|q\|_{L^2(\Gamma)}.
\end{align}
Moreover, $\langle \nabla_\Gamma q+qHn,v\rangle_\Gamma=0$ for all $v$ in the solenoidal space
\begin{align*}
  H_\sigma^1(\Gamma) := \{v\in H^1(\Gamma)^3 \mid \text{$\mathrm{div}_\Gamma v = 0$ on $\Gamma$}\}.
\end{align*}
Let us show that each element of the annihilator of $H_\sigma^1(\Gamma)$ is of the form
\begin{align*}
  \nabla_\Gamma q+qHn, \quad q\in L^2(\Gamma).
\end{align*}
To this end, we give a few properties of a functional of this form.

\begin{lemma} \label{L:TGrHn_Ineq}
  For all $q\in L^2(\Gamma)$ we have
  \begin{align} \label{E:TGrHn_Ineq}
    \|\nabla_\Gamma q\|_{H^{-1}(\Gamma,T\Gamma)} \leq \|\nabla_\Gamma q+qHn\|_{H^{-1}(\Gamma)}.
  \end{align}
\end{lemma}

\begin{proof}
  By \eqref{E:TGr_HinT} and \eqref{E:TGrHn_Hin} we see that
  \begin{align*}
    |[\nabla_\Gamma q,v]_{T\Gamma}| = |\langle\nabla_\Gamma q+qHn,v\rangle_\Gamma| \leq \|\nabla_\Gamma q+qHn\|_{H^{-1}(\Gamma)}\|v\|_{H^1(\Gamma)}
  \end{align*}
  for all $v\in H^1(\Gamma,T\Gamma)$.
  Thus \eqref{E:TGrHn_Ineq} is valid.
\end{proof}

\begin{lemma} \label{L:TGrHn_Hin_Con}
  Let $q\in L^2(\Gamma)$.
  Then
  \begin{align*}
    \nabla_\Gamma q+qHn = 0 \quad\text{in}\quad H^{-1}(\Gamma)^3
  \end{align*}
  if and only if $q=0$ on $\Gamma$.
\end{lemma}

\begin{proof}
  Suppose that $q\in L^2(\Gamma)$ satisfies $\nabla_\Gamma q+qHn=0$ in $H^{-1}(\Gamma)^3$.
  Then
  \begin{align*}
    \nabla_\Gamma q = 0 \quad\text{in}\quad H^{-1}(\Gamma,T\Gamma)
  \end{align*}
  by \eqref{E:TGrHn_Ineq} and $q$ is constant on $\Gamma$ by Lemma \ref{L:TGr_HinT_Con}.
  Moreover, for each $\xi\in H^1(\Gamma)$ we set $v:=\xi n\in H^1(\Gamma)^3$ in \eqref{E:TGrHn_Hin} (note that $n\in C^1(\Gamma)^3$) and use
  \begin{align*}
    \mathrm{div}_\Gamma(\xi n) = \nabla_\Gamma\xi\cdot n+\xi\,\mathrm{div}_\Gamma n = -\xi H \quad\text{on}\quad \Gamma
  \end{align*}
  by \eqref{E:P_TGr} and \eqref{E:Def_Wein} to get
  \begin{align*}
    0 = \langle \nabla_\Gamma q+qHn,\xi n\rangle_\Gamma = -\bigl(q,\mathrm{div}_\Gamma(\xi n)\bigr)_{L^2(\Gamma)} = q\int_\Gamma \xi H\,d\mathcal{H}^2.
  \end{align*}
  Since $H\in C(\Gamma) \subset L^2(\Gamma)$ and $H^1(\Gamma)$ is dense in $L^2(\Gamma)$ by Lemma \ref{L:Wmp_Appr}, we observe by the above equality and a density argument that
  \begin{align*}
    q\int_\Gamma H^2\,d\mathcal{H}^2 = 0.
  \end{align*}
  Moreover, since $\Gamma$ is a closed surface in $\mathbb{R}^3$, we have (see (16.32) in \cite{GiTr01})
  \begin{align*}
    \int_\Gamma \left(\frac{H}{2}\right)^2\,d\mathcal{H}^2 \geq 4\pi, \quad\text{i.e.}\quad \int_\Gamma H^2\,d\mathcal{H}^2 \geq 16\pi > 0
  \end{align*}
  and thus $q=0$ (note that the mean curvature of $\Gamma$ is defined as $H/2$ in \cite{GiTr01}).

  Conversely, if $q=0$ on $\Gamma$, then $\nabla_\Gamma q+qHn=0$ in $H^{-1}(\Gamma)^3$ by \eqref{E:TGrHn_Hin}.
\end{proof}

\begin{lemma} \label{L:Poin_L2_HinN}
  There exists a constant $c>0$ such that
  \begin{align} \label{E:Poin_L2_HinN}
    \|q\|_{L^2(\Gamma)} \leq c\|\nabla_\Gamma q+qHn\|_{H^{-1}(\Gamma)}
  \end{align}
  for all $q\in L^2(\Gamma)$.
\end{lemma}

\begin{proof}
  By the Ne\v{c}as inequality \eqref{E:Necas_Surf} and \eqref{E:TGrHn_Ineq} we have
  \begin{align} \label{Pf_PLHN:Nec}
    \|q\|_{L^2(\Gamma)} \leq c\left(\|q\|_{H^{-1}(\Gamma)}+\|\nabla_\Gamma q+qHn\|_{H^{-1}(\Gamma)}\right)
  \end{align}
  for all $q\in L^2(\Gamma)$.
  Thus it is sufficient to show that
  \begin{align} \label{Pf_PLHN:Goal}
    \|q\|_{H^{-1}(\Gamma)} \leq c\|\nabla_\Gamma q+qHn\|_{H^{-1}(\Gamma)}.
  \end{align}
  Assume to the contrary that there exists a sequence $\{q_k\}_{k=1}^\infty$ in $L^2(\Gamma)$ such that
  \begin{align*}
    \|q_k\|_{H^{-1}(\Gamma)} > k\|\nabla_\Gamma q_k+q_kHn\|_{H^{-1}(\Gamma)}
  \end{align*}
  for all $k\in\mathbb{N}$.
  Replacing $q_k$ with $q_k/\|q_k\|_{H^{-1}(\Gamma)}$ we may assume that
  \begin{align} \label{Pf_PLHN:Contra}
    \|q_k\|_{H^{-1}(\Gamma)} = 1, \quad \|\nabla_\Gamma q_k+q_kHn\|_{H^{-1}(\Gamma)} < \frac{1}{k}, \quad k\in\mathbb{N}.
  \end{align}
  Then since $\{q_k\}_{k=1}^\infty$ is bounded in $L^2(\Gamma)$ by \eqref{Pf_PLHN:Nec} and \eqref{Pf_PLHN:Contra} and the embedding $L^2(\Gamma)\hookrightarrow H^{-1}(\Gamma)$ is compact, we see that $\{q_k\}_{k=1}^\infty$ converges (up to a subsequence) to some $q\in L^2(\Gamma)$ weakly in $L^2(\Gamma)$ and strongly in $H^{-1}(\Gamma)$.
  Hence
  \begin{align} \label{Pf_PLHN:Q_Hin}
    \|q\|_{H^{-1}(\Gamma)} = \lim_{k\to\infty}\|q_k\|_{H^{-1}(\Gamma)} = 1
  \end{align}
  by the first equality of \eqref{Pf_PLHN:Contra}.
  Moreover, by the weak convergence of $\{q_k\}_{k=1}^\infty$ to $q$ in $L^2(\Gamma)$ and \eqref{E:TGrHn_Hin} we have
  \begin{align*}
    \lim_{k\to\infty}(\nabla_\Gamma q_k+q_kHn) = \nabla_\Gamma q+qHn \quad\text{weakly in}\quad H^{-1}(\Gamma)^3
  \end{align*}
  and thus it follows from the second inequality of \eqref{Pf_PLHN:Contra} that
  \begin{align*}
    \|\nabla_\Gamma q+qHn\|_{H^{-1}(\Gamma)} \leq \liminf_{k\to\infty}\|\nabla_\Gamma q_k+q_kHn\|_{H^{-1}(\Gamma)} = 0,
  \end{align*}
  i.e. $\nabla_\Gamma q+qHn=0$ in $H^{-1}(\Gamma)^3$.
  Hence $q=0$ on $\Gamma$ by Lemma \ref{L:TGrHn_Hin_Con} and we obtain $\|q\|_{H^{-1}(\Gamma)}=0$, which contradicts with \eqref{Pf_PLHN:Q_Hin}.
  Therefore, \eqref{Pf_PLHN:Goal} is valid.
\end{proof}

Now we obtain de Rham's theorem for the annihilator of $H_\sigma^1(\Gamma)$.

\begin{theorem} \label{T:DeRham_Ge}
  Suppose that $f\in H^{-1}(\Gamma)^3$ satisfies
  \begin{align*}
    \langle f,v\rangle_\Gamma = 0 \quad\text{for all}\quad  v\in H_\sigma^1(\Gamma).
  \end{align*}
  Then there exists a unique $q\in L^2(\Gamma)$ such that
  \begin{align*}
    f = \nabla_\Gamma q+qHn \quad\text{in}\quad H^{-1}(\Gamma)^3, \quad \|q\|_{L^2(\Gamma)} \leq c\|f\|_{H^{-1}(\Gamma)}
  \end{align*}
  with a constant $c>0$ independent of $f$.
\end{theorem}

\begin{proof}
  By \eqref{E:TGrHn_Hin_Bo} and \eqref{E:Poin_L2_HinN} we can show as in the proof of Lemma \ref{L:TGr_Closed} that
  \begin{align*}
    \mathcal{X} := \{\nabla_\Gamma q+qHn\in H^{-1}(\Gamma)^3 \mid q\in L^2(\Gamma)\}
  \end{align*}
  is closed in $H^{-1}(\Gamma)^3$.
  Moreover, $\mathcal{X}^\perp\subset H_\sigma^1(\Gamma)$ in $H^1(\Gamma)^3$ by \eqref{E:TGrHn_Hin}.
  Since the dual space $H^{-1}(\Gamma)^3$ of the Hilbert space $H^1(\Gamma)^3$ is reflexive, we have
  \begin{align*}
    H_\sigma^1(\Gamma)^\perp = \{f\in H^{-1}(\Gamma)^3 \mid \text{$\langle f,v\rangle_\Gamma=0$ for all $v\in H_\sigma^1(\Gamma)$}\} \subset (\mathcal{X}^\perp)^\perp = \mathcal{X}
  \end{align*}
  in $H^{-1}(\Gamma)^3$ by Lemmas \ref{L:Re_FA_1} and \ref{L:Re_FA_2}.
  Hence the existence part of the theorem is valid.
  Also, the uniqueness and the estimate follow from \eqref{E:Poin_L2_HinN}.
\end{proof}

Next we derive the Helmholtz--Leray decomposition in $L^2(\Gamma)^3$.
Let
\begin{align*}
  L_\sigma^2(\Gamma) := \{v\in L^2(\Gamma)^3 \mid \text{$\mathrm{div}_\Gamma v=0$ in $H^{-1}(\Gamma)$}\}.
\end{align*}
Then we have the orthogonal decomposition
\begin{align} \label{E:L2Ge_OD}
  L^2(\Gamma)^3 = L_\sigma^2(\Gamma)\oplus L_\sigma^2(\Gamma)^\perp
\end{align}
since $L_\sigma^2(\Gamma)$ is a closed subspace of $L^2(\Gamma)^3$ by \eqref{E:DivG_Hin_Bo}.

\begin{lemma} \label{L:L2sGe_Orth}
  The orthogonal complement of $L_\sigma^2(\Gamma)$ in $L^2(\Gamma)^3$ is of the form
  \begin{align} \label{E:L2sGe_Orth}
    L_\sigma^2(\Gamma)^\perp = \{\nabla_\Gamma q+qHn\in L^2(\Gamma)^3 \mid q\in H^1(\Gamma)\}.
  \end{align}
\end{lemma}

\begin{proof}
  The proof is similar to that of Lemma \ref{L:L2gs_Orth}.
  We denote by $\mathcal{X}$ the right-hand side of \eqref{E:L2sGe_Orth}.
  Then
  \begin{align*}
    (v,\nabla_\Gamma q+qHn)_{L^2(\Gamma)} = -\langle\mathrm{div}_\Gamma v,q\rangle_\Gamma = 0 \quad\text{for all}\quad v\in L_\sigma^2(\Gamma), \, q\in H^1(\Gamma)
  \end{align*}
  by \eqref{E:Sdiv_Hin} and thus $\mathcal{X}\subset L_\sigma^2(\Gamma)^\perp$.
  Conversely, let $f\in L_\sigma^2(\Gamma)^\perp$.
  Since
  \begin{align*}
    \langle f,v\rangle_\Gamma = (f,v)_{L^2(\Gamma)} = 0 \quad\text{for all}\quad v\in H_\sigma^1(\Gamma)\subset L_\sigma^2(\Gamma),
  \end{align*}
  we observe by Theorem \ref{T:DeRham_Ge} that there exists a unique $q\in L^2(\Gamma)$ such that
  \begin{align*}
    f = \nabla_\Gamma q+qHn \quad\text{in}\quad H^{-1}(\Gamma)^3.
  \end{align*}
  To prove $q\in H^1(\Gamma)$, let $\{e_1,e_2,e_3\}$ be the standard basis of $\mathbb{R}^3$ and $v:=\eta e_i$ on $\Gamma$ for $\eta\in C^1(\Gamma)$ and $i=1,2,3$.
  Then since $v\in H^1(\Gamma)^3$ and $\mathrm{div}_\Gamma v=\underline{D}_i\eta$ on $\Gamma$,
  \begin{align*}
    -(q,\underline{D}_i\eta)_{L^2(\Gamma)} &= -(q,\mathrm{div}_\Gamma v)_{L^2} = \langle\nabla_\Gamma q+qHn,v\rangle_\Gamma \\
    &= \langle f,v\rangle_\Gamma = (f,v)_{L^2(\Gamma)} = (f_i,\eta)_{L^2(\Gamma)},
  \end{align*}
  where $f_i$ is the $i$-th component of $f$.
  From this equality we deduce that
  \begin{align*}
    -(q,\underline{D}_i\eta+\eta Hn_i)_{L^2(\Gamma)} = (f_i-qHn_i,\eta)_{L^2(\Gamma)}
  \end{align*}
  for all $\eta\in C^1(\Gamma)$.
  Hence
  \begin{align*}
    \underline{D}_iq = f_i-qHn_i \in L^2(\Gamma), \quad i=1,2,3
  \end{align*}
  by the definition of the weak tangential derivative in $L^2(\Gamma)$ (see \eqref{E:Def_WTD}) and
  \begin{align*}
    q \in H^1(\Gamma), \quad f = \nabla_\Gamma q+qHn \in \mathcal{X}.
  \end{align*}
  Thus $L_\sigma^2(\Gamma)^\perp\subset\mathcal{X}$ and \eqref{E:L2sGe_Orth} holds.
\end{proof}

A similar result to Lemma \ref{L:L2sGe_Orth} was given in \cite{KoLiGi17}*{Lemma 2.7}, where the authors showed that, when $\Gamma$ is smooth, each element of the annihilator in $L^2(\Gamma)^3$ of
\begin{align*}
  \{\eta \in C^\infty(\Gamma)^3 \mid \text{$\mathrm{div}_\Gamma\eta = 0$ on $\Gamma$}\}
\end{align*}
is of the form $\nabla_\Gamma q+qHn$ with $q\in H^1(\Gamma)$ (see also \cite{KoLiGi18Er}*{Theorem 1.1}).
Note that here we only assume that $\Gamma$ is of class $C^2$ and the proof of Lemma \ref{L:L2sGe_Orth} is different from those of \cite{KoLiGi17}*{Lemma 2.7} and \cite{KoLiGi18Er}*{Theorem 1.1}.

\begin{theorem} \label{T:HL_L2Ge}
  For each $v\in L^2(\Gamma)^3$ we have the orthogonal decomposition
  \begin{align} \label{E:HL_L2Ge}
    \begin{gathered}
      v = v_\sigma+\nabla_\Gamma q+qHn \quad\text{in}\quad L^2(\Gamma)^3, \\
      v_\sigma\in L_\sigma^2(\Gamma),\,\nabla_\Gamma q+qHn\in L_\sigma^2(\Gamma)^\perp.
    \end{gathered}
  \end{align}
  Here $q\in H^1(\Gamma)$ is determined uniquely.
\end{theorem}

\begin{proof}
  The decomposition \eqref{E:HL_L2Ge} follows from \eqref{E:L2Ge_OD} and \eqref{E:L2sGe_Orth}.
  Also, if
  \begin{align*}
    \nabla_\Gamma q+qHn = 0 \quad\text{in}\quad L^2(\Gamma)^3
  \end{align*}
  for $q\in H^1(\Gamma)$, then the same equality holds in $H^{-1}(\Gamma)^3$ by \eqref{E:L2_Hin} and thus $q=0$ on $\Gamma$ by Lemma \ref{L:TGrHn_Hin_Con}.
  Hence $q\in H^1(\Gamma)$ in \eqref{E:HL_L2Ge} is determined uniquely.
\end{proof}

The Helmholtz--Leray decomposition \eqref{E:HL_L2Ge} in $L^2(\Gamma)^3$ was stated in \cite{KoLiGi18Er} without an explicit formulation (see a remark after \cite{KoLiGi18Er}*{Theorem 1.1}), but the uniqueness of $q$ in \eqref{E:HL_L2Ge} is established first in this paper.

\begin{remark} \label{R:HL_L2Ge}
  For a vector field in $L^2(\Gamma,T\Gamma)$, the decomposition \eqref{E:HL_L2Ge} in $L^2(\Gamma)^3$ is not necessarily the same as the decomposition \eqref{E:HL_L2T} in $L^2(\Gamma,T\Gamma)$ with $g\equiv1$.
  To see this, suppose that $\Gamma$ is strictly convex and thus the mean curvature $H$ of $\Gamma$ does not vanish on the whole surface, i.e. $H\neq0$ on $\Gamma$.
  Let $v\in L^2(\Gamma,T\Gamma)$ satisfy
  \begin{align*}
    \mathrm{div}_\Gamma v \neq 0 \quad\text{in}\quad H^{-1}(\Gamma).
  \end{align*}
  Then by \eqref{E:HL_L2Ge} there exist unique $v_\sigma\in L_\sigma^2(\Gamma)$ and $q\in H^1(\Gamma)$ such that
  \begin{align*}
    v = v_\sigma+\nabla_\Gamma q+qHn \quad\text{on}\quad \Gamma.
  \end{align*}
  Moreover, since $v\cdot n=0$ and $H\neq0$ on $\Gamma$ and $q\not\equiv0$ on $\Gamma$ by $\mathrm{div}_\Gamma v\neq 0$ in $H^{-1}(\Gamma)$,
  \begin{align*}
    0 = v_\sigma\cdot n+qH, \quad\text{i.e.}\quad v_\sigma\cdot n = -qH \not\equiv 0 \quad\text{on}\quad \Gamma.
  \end{align*}
  Thus the solenoidal part $v_\sigma$ of $v$ in \eqref{E:HL_L2Ge} is not tangential on $\Gamma$, while the solenoidal part $v_g$ of the same $v$ in \eqref{E:HL_L2T} with $g\equiv1$ must be tangential on $\Gamma$.
\end{remark}

\begin{remark} \label{R:TGrHn}
  The vector field $\nabla_\Gamma q+qHn$ appears in the Navier--Stokes equations
  \begin{align} \label{E:NS_EVS}
    \dot{v}-2\mu_s\mathrm{div}_\Gamma[D_\Gamma(v)]+\nabla_\Gamma q+qHn = 0, \quad \mathrm{div}_\Gamma v = 0
  \end{align}
  on an evolving surface $\Gamma(t)$ in $\mathbb{R}^3$ (see \cites{JaOlRe18,KoLiGi17}).
  Here $v$, $q$, and $\mu_s$ denote the velocity of a fluid on $\Gamma(t)$ including the normal velocity, the surface pressure, and the surface shear viscosity, respectively.
  Also, $\dot{v}:=\partial_tv+(v\cdot\nabla)v$ denotes the material derivative of $v$ along itself and $D_\Gamma(v)$ is the surface strain rate tensor given by \eqref{E:Def_SSR}.
  We expect that the Helmholtz--Leray decomposition \eqref{E:HL_L2Ge} will play a fundamental role in the future study of \eqref{E:NS_EVS}.
  Also, for a function $q$ on $\Gamma$ we have
  \begin{align*}
    \mathrm{div}_\Gamma(qP) = P^T\nabla_\Gamma q+q\,\mathrm{div}_\Gamma P = \nabla_\Gamma q+qHn \quad\text{on}\quad \Gamma
  \end{align*}
  by \eqref{E:P_TGr}, \eqref{E:Div_P}, and $P^T=P$ on $\Gamma$, and thus \eqref{E:NS_EVS} can be written as
  \begin{align*}
    \dot{v}-\mathrm{div}_\Gamma S_\Gamma = 0, \quad \mathrm{div}_\Gamma v = 0 \quad\text{on}\quad \Gamma(t).
  \end{align*}
  Here $S_\Gamma$ is the Boussinesq--Scriven surface stress tensor given by
  \begin{align*}
    S_\Gamma := -qP+(\lambda_s-\mu_s)(\mathrm{div}_\Gamma v)P+2\mu_s D_\Gamma(v)
  \end{align*}
  with the surface dilatational viscosity $\lambda_s$, which was introduced by Boussinesq \cite{Bo1913} to describe the motion of the interface of a two-phase flow and then generalized by Scriven \cite{Sc60} to an arbitrary curved moving interface (see also \cites{Ar89,SlSaOh07}).
  For the study of two-phase flows with Boussinesq--Scriven surface fluids we refer to \cites{BaGaNu15,BoPr10,NiVoWe12}.
\end{remark}

Finally, we give a remark on the case $g\not\equiv1$.
Let $g\in C^1(\Gamma)$ satisfy \eqref{E:LBo_G}.
Then
\begin{alignat*}{2}
  \langle g(\nabla_\Gamma q+qHn),v\rangle_\Gamma &= -\bigl(q,\mathrm{div}_\Gamma(gv)\bigr)_{L^2(\Gamma)}, &\quad &q\in L^2(\Gamma),\,v\in H^1(\Gamma)^3, \\
  \langle \mathrm{div}_\Gamma(gw),\eta\rangle_\Gamma &= -\bigl(w,g(\nabla_\Gamma \eta+\eta Hn)\bigr)_{L^2(\Gamma)}, &\quad &w\in L^2(\Gamma)^3,\,\eta\in H^1(\Gamma)
\end{alignat*}
by \eqref{E:Def_Mul_Hin}, \eqref{E:TGr_Hin}, and \eqref{E:Sdiv_Hin}.
Moreover, as in Lemma \ref{L:TGr_G_HinT} we see that
\begin{align*}
  c^{-1}\|\nabla_\Gamma q+qHn\|_{H^{-1}(\Gamma)} \leq \|g(\nabla_\Gamma q+qHn)\|_{H^{-1}(\Gamma)} \leq c\|\nabla_\Gamma q+qHn\|_{H^{-1}(\Gamma)}
\end{align*}
for $q\in L^2(\Gamma)$ by \eqref{E:LBo_G}.
Using these relations, \eqref{E:TGrHn_Hin_Bo}, and \eqref{E:Poin_L2_HinN}, we can prove the following weighted versions of Theorems \ref{T:DeRham_Ge} and \ref{T:HL_L2Ge} for the weighted solenoidal spaces
\begin{align*}
  L_{g\sigma}^2(\Gamma) &:= \{v\in L^2(\Gamma)^3 \mid \text{$\mathrm{div}_\Gamma(gv)=0$ in $H^{-1}(\Gamma)$}\}, \\
    H_{g\sigma}^1(\Gamma) &:= \{v\in H^1(\Gamma)^3 \mid \text{$\mathrm{div}_\Gamma(gv)=0$ on $\Gamma$}\}
\end{align*}
as in the proofs of Theorems \ref{T:DeRham_Ge} and \ref{T:HL_L2Ge}.

\begin{theorem} \label{T:DeRham_Ge_W}
  Suppose that $f\in H^{-1}(\Gamma)^3$ satisfies
  \begin{align*}
    \langle f,v\rangle_\Gamma = 0 \quad\text{for all}\quad v\in H_{g\sigma}^1(\Gamma).
  \end{align*}
  Then there exists a unique $q\in L^2(\Gamma)$ such that
  \begin{align*}
    f = g(\nabla_\Gamma q+qHn) \quad\text{in}\quad H^{-1}(\Gamma)^3, \quad \|q\|_{L^2(\Gamma)} \leq c\|f\|_{H^{-1}(\Gamma)}
  \end{align*}
  with a constant $c>0$ independent of $f$.
\end{theorem}

\begin{theorem} \label{T:HL_L2Ge_W}
  For each $v\in L^2(\Gamma)^3$ we have the orthogonal decomposition
  \begin{gather*}
    v = v_g+g(\nabla_\Gamma q+qHn) \quad\text{in}\quad L^2(\Gamma)^3, \\
    v_g\in L_{g\sigma}^2(\Gamma),\,g(\nabla_\Gamma q+qHn)\in L_{g\sigma}^2(\Gamma)^\perp.
  \end{gather*}
  Here $q\in H^1(\Gamma)$ is determined uniquely.
\end{theorem}

\section{Singular limit problem as the thickness tends to zero} \label{S:SL}
In this section we study a singular limit problem for the Navier--Stokes equations \eqref{E:NS_CTD} as $\varepsilon\to0$ and show Theorems \ref{T:SL_Weak} and \ref{T:SL_Strong} and related results.

Throughout this section we impose Assumptions \ref{Assump_1} and \ref{Assump_2} and fix the constant $\varepsilon_0$ given in Lemma \ref{L:Uni_aeps}.
For $\varepsilon\in(0,\varepsilon_0]$ let $\mathcal{H}_\varepsilon$ and $\mathcal{V}_\varepsilon$ be the function spaces defined by \eqref{E:Def_Heps}, $\mathbb{P}_\varepsilon$ the orthogonal projection from $L^2(\Omega_\varepsilon)^3$ onto $\mathcal{H}_\varepsilon$, and $A_\varepsilon$ the Stokes operator on $\mathcal{H}_\varepsilon$ given in Section \ref{SS:Fund_St}.
Also, let $M$ and $M_\tau$ be the average operators given in Definition \ref{D:Average}.
We further denote by
\begin{align*}
  \mathcal{H}_g := L_{g\sigma}^2(\Gamma,T\Gamma), \quad \mathcal{V}_g := H_{g\sigma}^1(\Gamma,T\Gamma)
\end{align*}
the weighted solenoidal spaces on $\Gamma$ given in Section \ref{S:WSol}, by $\bar{\eta}=\eta\circ\pi$ the constant extension of a function $\eta$ on $\Gamma$ in the normal direction of $\Gamma$, and by $\partial_n\varphi=(\bar{n}\cdot\nabla)\varphi$ the derivative of a function $\varphi$ on $\Omega_\varepsilon$ in the normal direction of $\Gamma$.

\subsection{Estimates for a strong solution to the bulk equations} \label{SS:SL_EstSt}
In order to prove Theorem \ref{T:SL_Weak} we study the behavior of the average of a strong solution to \eqref{E:NS_CTD} as $\varepsilon\to0$.
To this end, we first show several estimates for the strong solution.

\begin{lemma} \label{L:Est_Ueps}
  Let $c_1$ and $c_2$ be positive constants, $\alpha\in(0,1]$, $\beta=1$, and $\varepsilon_1\in(0,\varepsilon_0]$ the constant given in Theorem \ref{T:UE}.
  For $\varepsilon\in(0,\varepsilon_1]$ suppose that the given data
  \begin{align*}
    u_0^\varepsilon\in \mathcal{V}_\varepsilon, \quad f^\varepsilon\in L^\infty(0,\infty;L^2(\Omega_\varepsilon)^3)
  \end{align*}
  satisfy \eqref{E:UE_Data}.
  If the condition (A3) of Assumption \ref{Assump_2} is imposed, suppose further that $f^\varepsilon(t)\in\mathcal{R}_g^\perp$ for a.a. $t\in(0,\infty)$.
  Then there exists a global strong solution
  \begin{align*}
    u^\varepsilon \in C([0,\infty);\mathcal{V}_\varepsilon)\cap L_{loc}^2([0,\infty);D(A_\varepsilon))\cap H_{loc}^1([0,\infty);\mathcal{H}_\varepsilon)
  \end{align*}
  to \eqref{E:NS_CTD}.
  Moreover, there exists a constant $c>0$ independent of $\varepsilon$ and $u^\varepsilon$ such that
  \begin{align} \label{E:Est_Ueps}
    \begin{gathered}
      \|u^\varepsilon(t)\|_{L^2(\Omega_\varepsilon)}^2 \leq c\varepsilon, \quad \int_0^t\|u^\varepsilon(s)\|_{H^1(\Omega_\varepsilon)}^2\,ds \leq c\varepsilon(1+t), \\
      \|u^\varepsilon(t)\|_{H^1(\Omega_\varepsilon)}^2 \leq c\varepsilon^{-1+\alpha}, \quad \int_0^t\|u^\varepsilon(s)\|_{H^2(\Omega_\varepsilon)}^2\,ds \leq c\varepsilon^{-1+\alpha}(1+t)
    \end{gathered}
  \end{align}
  for all $t\geq0$ and
  \begin{align}
    \int_0^t\|[u^\varepsilon\otimes u^\varepsilon](s)\|_{L^2(\Omega_\varepsilon)}^2\,ds &\leq c\varepsilon(1+t), \label{E:Est_UeUe} \\
    \int_0^t\|\partial_tu^\varepsilon(s)\|_{L^2(\Omega_\varepsilon)}^2\,ds &\leq c\varepsilon^{-1+\alpha}(1+t) \label{E:Est_DtUe}
  \end{align}
  for all $t\geq0$.
\end{lemma}

\begin{proof}
  The existence of a global strong solution $u^\varepsilon$ to \eqref{E:NS_CTD} is due to Theorem \ref{T:UE}.
  Also, the estimates \eqref{E:Est_Ueps} follow from \eqref{E:UE_L2} and \eqref{E:UE_H1}, since $\beta=1$ and $\varepsilon\leq\varepsilon^\alpha\leq1$ by $\alpha\in(0,1]$ and $\varepsilon\in(0,1]$.

  Let us show \eqref{E:Est_UeUe}.
  We suppress the argument $s$ in time integrals.
  Since $u^\varepsilon(t)\in D(A_\varepsilon)$ satisfies the conditions of Lemma \ref{L:Est_UU} for a.a. $t\in(0,\infty)$,
  \begin{multline} \label{Pf_EsSt:UeUe}
    \int_0^t\|u^\varepsilon\otimes u^\varepsilon\|_{L^2(\Omega_\varepsilon)}^2\,ds \leq c\left(\varepsilon^{-1}\int_0^t\|u^\varepsilon\|_{L^2(\Omega_\varepsilon)}^2\|u\|_{H^1(\Omega_\varepsilon)}^2\,ds\right.\\
    \left.+\varepsilon\int_0^t\|u^\varepsilon\|_{L^2(\Omega_\varepsilon)}^2\|u^\varepsilon\|_{H^2(\Omega_\varepsilon)}^2\,ds+\int_0^t\|u^\varepsilon\|_{L^2(\Omega_\varepsilon)}^3\|u^\varepsilon\|_{H^2(\Omega_\varepsilon)}\,ds\right)
  \end{multline}
  for all $t\geq 0$ by \eqref{E:Est_UU}.
  Moreover, we deduce from \eqref{E:Est_Ueps} that
  \begin{align}
    \int_0^t\|u^\varepsilon\|_{L^2(\Omega_\varepsilon)}^2\|u^\varepsilon\|_{H^1(\Omega_\varepsilon)}^2\,ds &\leq c\varepsilon^2(1+t), \label{Pf_EsSt:L2H1} \\
    \int_0^t\|u^\varepsilon\|_{L^2(\Omega_\varepsilon)}^2\|u^\varepsilon\|_{H^2(\Omega_\varepsilon)}^2\,ds &\leq c\varepsilon^\alpha(1+t). \label{Pf_EsSt:L2H2}
  \end{align}
  Also, H\"{o}lder's inequality and \eqref{E:Est_Ueps} imply that
  \begin{align} \label{Pf_EsSt:L2H2_31}
    \begin{aligned}
      \int_0^t\|u^\varepsilon\|^3_{L^2(\Omega_\varepsilon)}\|u^\varepsilon\|_{H^2(\Omega_\varepsilon)}\,ds &\leq c\varepsilon^{3/2}t^{1/2}\left(\int_0^t\|u^\varepsilon\|_{H^2(\Omega_\varepsilon)}^2\,ds\right)^{1/2} \\
      &\leq c\varepsilon^{1+\alpha/2}(1+t).
    \end{aligned}
  \end{align}
  Applying \eqref{Pf_EsSt:L2H1}--\eqref{Pf_EsSt:L2H2_31} to \eqref{Pf_EsSt:UeUe} and noting that $\varepsilon^\alpha,\varepsilon^{\alpha/2}\leq 1$ we obtain \eqref{E:Est_UeUe}.

  To derive \eqref{E:Est_DtUe} we observe by \eqref{E:St_Inter} and \eqref{E:Est_UGU} that
  \begin{align*}
    \|(u\cdot\nabla)u\|_{L^2(\Omega_\varepsilon)} \leq c\left(\varepsilon^{-1/2}\|u\|_{L^2(\Omega_\varepsilon)}+\varepsilon^{1/2}\|u\|_{H^1(\Omega_\varepsilon)}\right)\|u\|_{H^2(\Omega_\varepsilon)}
  \end{align*}
  for all $u\in D(A_\varepsilon)$ and thus
  \begin{align*}
    \int_0^t\|(u^\varepsilon\cdot\nabla)u^\varepsilon\|_{L^2(\Omega_\varepsilon)}^2\,ds \leq c\int_0^t\left(\varepsilon^{-1}\|u^\varepsilon\|_{L^2(\Omega_\varepsilon)}^2+\varepsilon\|u^\varepsilon\|_{H^1(\Omega_\varepsilon)}^2\right)\|u^\varepsilon\|_{H^2(\Omega_\varepsilon)}^2\,ds
  \end{align*}
  for all $t\geq0$.
  Moreover, by \eqref{E:Est_Ueps} we have
  \begin{align*}
    \int_0^t\|u^\varepsilon\|_{H^1(\Omega_\varepsilon)}^2\|u^\varepsilon\|_{H^2(\Omega_\varepsilon)}^2\,ds \leq c\varepsilon^{-2+2\alpha}(1+t).
  \end{align*}
  From the above two estimates and \eqref{Pf_EsSt:L2H2} it follows that
  \begin{align} \label{Pf_EsSt:UeGUe}
    \int_0^t\|(u^\varepsilon\cdot\nabla)u^\varepsilon\|_{L^2(\Omega_\varepsilon)}^2\,ds \leq c\varepsilon^{-1+\alpha}(1+t).
  \end{align}
  Now we take the $L^2(\Omega_\varepsilon)$-inner product of $\partial_tu^\varepsilon$ with
  \begin{align} \label{Pf_EsSt:Ab_Eq}
    \partial_tu^\varepsilon+A_\varepsilon u^\varepsilon+\mathbb{P}_\varepsilon[(u^\varepsilon\cdot\nabla)u^\varepsilon] = \mathbb{P}_\varepsilon f^\varepsilon \quad\text{on}\quad (0,\infty)
  \end{align}
  and use \eqref{E:L2in_Ahalf} and Young's inequality to deduce that
  \begin{align*}
    \|\partial_tu^\varepsilon\|_{L^2(\Omega_\varepsilon)}^2+\frac{d}{dt}\|A_\varepsilon^{1/2}u^\varepsilon\|_{L^2(\Omega_\varepsilon)}^2 \leq c\left(\|(u^\varepsilon\cdot\nabla)u^\varepsilon\|_{L^2(\Omega_\varepsilon)}^2+\|\mathbb{P}_\varepsilon f^\varepsilon\|_{L^2(\Omega_\varepsilon)}^2\right)
  \end{align*}
  on $(0,\infty)$.
  Hence for all $t\geq0$ we have
  \begin{multline*}
    \int_0^t\|\partial_tu^\varepsilon\|_{L^2(\Omega_\varepsilon)}^2\,ds+\|A_\varepsilon^{1/2}u^\varepsilon(t)\|_{L^2(\Omega_\varepsilon)}^2 \\
    \leq \|A_\varepsilon^{1/2}u_0^\varepsilon\|_{L^2(\Omega_\varepsilon)}^2+c\int_0^t\left(\|(u^\varepsilon\cdot\nabla)u^\varepsilon\|_{L^2(\Omega_\varepsilon)}^2+\|\mathbb{P}_\varepsilon f^\varepsilon\|_{L^2(\Omega_\varepsilon)}^2\right)ds
  \end{multline*}
  and to the right-hand side we apply \eqref{E:UE_Data}, \eqref{E:Stokes_H1}, and \eqref{Pf_EsSt:UeGUe} to get \eqref{E:Est_DtUe}.
\end{proof}

\subsection{Average of a weak formulation of the bulk equations} \label{SS:SL_Ave}
The aim of this subsection is to derive a weak formulation for the averaged tangential component of a strong solution to \eqref{E:NS_CTD}.

Let $\varepsilon_1$ and $\varepsilon_\sigma$ be the constants given in Theorem \ref{T:UE} with $\alpha\in(0,1]$ and $\beta=1$ and in Lemma \ref{L:Uni_Poin_Dom}.
In what follows, we suppose that
\begin{align*}
  0 < \varepsilon \leq \min\{\varepsilon_1,\varepsilon_\sigma\}
\end{align*}
and the conditions of Lemma \ref{L:Est_Ueps} are satisfied, and denote by $u^\varepsilon$ the global strong solution to \eqref{E:NS_CTD} given in Lemma \ref{L:Est_Ueps}.

Our starting point is a weak formulation of \eqref{E:NS_CTD} satisfied by the strong solution $u^\varepsilon$.
Let $a_\varepsilon$ be the bilinear form given by \eqref{E:Def_Bi_Dom}, i.e.
\begin{align*}
  a_\varepsilon(u_1,u_2) = 2\nu\bigl(D(u_1),D(u_2)\bigr)_{L^2(\Omega_\varepsilon)}+\gamma_\varepsilon^0(u_1,u_2)_{L^2(\Gamma_\varepsilon^0)}+\gamma_\varepsilon^1(u_1,u_2)_{L^2(\Gamma_\varepsilon^1)}
\end{align*}
for $u_1,u_2\in H^1(\Omega_\varepsilon)^3$ and $b_\varepsilon$ is a trilinear form defined by
\begin{align} \label{E:Def_Tri_Dom}
  b_\varepsilon(u_1,u_2,u_3) := -(u_1\otimes u_2,\nabla u_3)_{L^2(\Omega_\varepsilon)}
\end{align}
for $u_1,u_2,u_3\in H^1(\Omega_\varepsilon)^3$.
Note that $u_1\otimes u_2\in L^2(\Omega_\varepsilon)^{3\times3}$ for $u_1,u_2\in H^1(\Omega_\varepsilon)^3$ by the Sobolev embedding $H^1(\Omega_\varepsilon)\hookrightarrow L^4(\Omega_\varepsilon)$ (see \cite{AdFo03}).

\begin{lemma} \label{L:NS_Weak}
  Let $u^\varepsilon$ be as in Lemma \ref{L:Est_Ueps}.
  Then
  \begin{align} \label{E:NS_Weak}
    \int_0^T\{(\partial_tu^\varepsilon,\varphi)_{L^2(\Omega_\varepsilon)}+a_\varepsilon(u^\varepsilon,\varphi)+b_\varepsilon(u^\varepsilon,u^\varepsilon,\varphi)\}\,dt = \int_0^T(\mathbb{P}_\varepsilon f^\varepsilon,\varphi)_{L^2(\Omega_\varepsilon)}\,dt
  \end{align}
  for all $T>0$ and $\varphi\in L^2(0,T;\mathcal{V}_\varepsilon)$.
\end{lemma}

\begin{proof}
  We take the $L^2(\Omega_\varepsilon)$-inner product of $\varphi$ with \eqref{Pf_EsSt:Ab_Eq}, apply \eqref{E:L2in_Ahalf} and
  \begin{align*}
    (\mathbb{P}_\varepsilon[(u^\varepsilon\cdot\nabla)u^\varepsilon],\varphi)_{L^2(\Omega_\varepsilon)} = \bigl((u^\varepsilon\cdot\nabla)u^\varepsilon,\varphi\bigr)_{L^2(\Omega_\varepsilon)} = b_\varepsilon(u^\varepsilon,u^\varepsilon,\varphi)
  \end{align*}
  by $\varphi\in\mathcal{H}_\varepsilon$, integration by parts, $\mathrm{div}\,u^\varepsilon=0$ in $\Omega_\varepsilon$, and $u^\varepsilon\cdot n_\varepsilon=0$ on $\Gamma_\varepsilon$, and then integrate the resulting equality over $(0,T)$ to obtain \eqref{E:NS_Weak}.
\end{proof}

We transform \eqref{E:NS_Weak} into a weak formulation for $M_\tau u^\varepsilon$ in which we take a test function from the weighted solenoidal space $\mathcal{V}_g$.
To this end, we substitute an element of $\mathcal{V}_g$ for \eqref{E:NS_Weak} and apply the results of Section \ref{S:Ave}.
Since an element of $\mathcal{V}_g$ is defined on $\Gamma$, we need to extend it to $\Omega_\varepsilon$ appropriately in order to substitute it for \eqref{E:NS_Weak}.
For this purpose, we use the impermeable extension operator $E_\varepsilon$ given by \eqref{E:Def_ExImp} and the Helmholtz--Leray projection $\mathbb{L}_\varepsilon$ from $L^2(\Omega_\varepsilon)^3$ onto $L_\sigma^2(\Omega_\varepsilon)$.

\begin{lemma} \label{L:Const_Test}
  Let $\eta\in \mathcal{V}_g$.
  Then
  \begin{align*}
    \eta_\varepsilon := \mathbb{L}_\varepsilon E_\varepsilon\eta \in L_\sigma^2(\Omega_\varepsilon)\cap H^1(\Omega_\varepsilon)^3
  \end{align*}
  and there exists a constant $c>0$ independent of $\varepsilon$ and $\eta$ such that
  \begin{gather}
    \|\eta_\varepsilon-\bar{\eta}\|_{L^2(\Omega_\varepsilon)}+\left\|\nabla\eta_\varepsilon-\overline{F(\eta)}\right\|_{L^2(\Omega_\varepsilon)} \leq c\varepsilon^{3/2}\|\eta\|_{H^1(\Gamma)}, \label{E:Test_Dom} \\
    \|\eta_\varepsilon-\bar{\eta}\|_{L^2(\Gamma_\varepsilon)} \leq c\varepsilon\|\eta\|_{H^1(\Gamma)}, \label{E:Test_Bo}
  \end{gather}
  where $F(\eta)$ is the $3\times3$ matrix-valued function on $\Gamma$ given by \eqref{E:Def_Fv}, i.e.
  \begin{align*}
    F(\eta) = \nabla_\Gamma\eta+\frac{1}{g}(\eta\cdot\nabla_\Gamma g)Q \quad\text{on}\quad \Gamma.
  \end{align*}
\end{lemma}

Note that under the condition (A3) of Assumption \ref{Assump_2} we may have
\begin{align*}
  \mathcal{V}_\varepsilon \neq L_\sigma^2(\Omega_\varepsilon)\cap H^1(\Omega_\varepsilon)^3 \quad (\mathcal{V}_\varepsilon \subset L_\sigma^2(\Omega_\varepsilon)\cap H^1(\Omega_\varepsilon)^3)
\end{align*}
by \eqref{E:Def_Heps}.
In this case we cannot take the above $\eta_\varepsilon$ as a test function in \eqref{E:NS_Weak}.
We deal with this problem at the first step of derivation of a weak formulation for $M_\tau u^\varepsilon$ (see Lemma \ref{L:Mu_Weak} below).

\begin{proof}
  By $\eta\in H^1(\Gamma,T\Gamma)$ and Lemma \ref{L:ExImp_Wmp} we have (see also Section \ref{SS:Fund_HL})
  \begin{align*}
    E_\varepsilon\eta \in H^1(\Omega_\varepsilon)^3, \quad \eta_\varepsilon = \mathbb{L}_\varepsilon E_\varepsilon\eta \in L_\sigma^2(\Omega_\varepsilon)\cap H^1(\Omega_\varepsilon)^3.
  \end{align*}
  Let us show \eqref{E:Test_Dom}.
  By \eqref{E:ExImp_DiLp} and \eqref{E:ExImp_DiGr} we have
  \begin{align} \label{Pf_CTe:Eeta}
    \|E_\varepsilon\eta-\bar{\eta}\|_{L^2(\Omega_\varepsilon)}+\left\|\nabla E_\varepsilon\eta-\overline{F(\eta)}\right\|_{L^2(\Omega_\varepsilon)} \leq c\varepsilon^{3/2}\|\eta\|_{H^1(\Gamma)}.
  \end{align}
  Since $E_\varepsilon\eta\in H^1(\Omega_\varepsilon)^3$ satisfies \eqref{E:Bo_Imp} by Lemma \ref{L:ExImp_Bo}, we have
  \begin{align*}
    \|\eta_\varepsilon-E_\varepsilon\eta\|_{H^1(\Omega_\varepsilon)} \leq c\|\mathrm{div}(E_\varepsilon\eta)\|_{L^2(\Omega_\varepsilon)}
  \end{align*}
  by \eqref{E:HP_Dom} with $u=E_\varepsilon\eta$ and $\mathbb{L}_\varepsilon u=\eta_\varepsilon$.
  Noting that $\eta\in\mathcal{V}_g$ satisfies $\mathrm{div}_\Gamma(g\eta)=0$ on $\Gamma$, we further apply \eqref{E:ExImp_LpSol} to the right-hand side to get
  \begin{align} \label{Pf_CTe:HL}
    \|\eta_\varepsilon-E_\varepsilon\eta\|_{H^1(\Omega_\varepsilon)} \leq c\varepsilon^{3/2}\|\eta\|_{H^1(\Gamma)}.
  \end{align}
  Combining \eqref{Pf_CTe:Eeta} and \eqref{Pf_CTe:HL} we obtain \eqref{E:Test_Dom}.
  To prove \eqref{E:Test_Bo} we see that
  \begin{align*}
    \|\eta_\varepsilon-\bar{\eta}\|_{L^2(\Gamma_\varepsilon)} &\leq c\left(\varepsilon^{-1/2}\|\eta_\varepsilon-\bar{\eta}\|_{L^2(\Omega_\varepsilon)}+\varepsilon^{1/2}\|\partial_n\eta_\varepsilon-\partial_n\bar{\eta}\|_{L^2(\Omega_\varepsilon)}\right) \\
    &\leq c\left(\varepsilon\|\eta\|_{H^1(\Gamma)}+\varepsilon^{1/2}\|\eta_\varepsilon\|_{H^1(\Omega_\varepsilon)}\right)
  \end{align*}
  by \eqref{E:Poin_Bo}, \eqref{E:Test_Dom}, and $\partial_n\bar{\eta}=0$ in $\Omega_\varepsilon$.
  Moreover, by \eqref{E:ExImp_Wmp} and \eqref{Pf_CTe:HL} we get
  \begin{align*}
    \|\eta_\varepsilon\|_{H^1(\Omega_\varepsilon)} \leq \|E_\varepsilon\eta\|_{H^1(\Omega_\varepsilon)}+\|\eta_\varepsilon-E_\varepsilon\eta\|_{H^1(\Omega_\varepsilon)} \leq c\varepsilon^{1/2}\|\eta\|_{H^1(\Gamma)}
  \end{align*}
  and thus the estimate \eqref{E:Test_Bo} follows from the above two inequalities.
\end{proof}

Next we consider approximation of $a_\varepsilon$ and $b_\varepsilon$ by bilinear and trilinear forms for tangential vector fields on $\Gamma$.
Let $\gamma^0$ and $\gamma^1$ be nonnegative constants.
We define
\begin{multline} \label{E:Def_Bi_Surf}
  a_g(v_1,v_2) \\
  := 2\nu\left\{\bigl(gD_\Gamma(v_1),D_\Gamma(v_2)\bigr)_{L^2(\Gamma)}+\left(\frac{1}{g}(v_1\cdot\nabla_\Gamma g),v_2\cdot\nabla_\Gamma g\right)_{L^2(\Gamma)}\right\} \\
  +(\gamma^0+\gamma^1)(v_1,v_2)_{L^2(\Gamma)}
\end{multline}
for $v_1,v_2\in H^1(\Gamma,T\Gamma)$, where $D_\Gamma(v_1)$ is given by \eqref{E:Def_SSR}, and
\begin{align} \label{E:Def_Tri_Surf}
  b_g(v_1,v_2,v_3) := -(g(v_1\otimes v_2),\nabla_\Gamma v_3)_{L^2(\Gamma)}
\end{align}
for $v_1,v_2,v_3\in H^1(\Gamma,T\Gamma)$.
Note that for $v_1,v_2\in H^1(\Gamma)^3$ we have
\begin{align*}
  v_1, v_2 \in L^4(\Gamma)^3, \quad v_1\otimes v_2 \in L^2(\Gamma)^{3\times 3}
\end{align*}
by Ladyzhenskaya's inequality \eqref{E:La_Surf}.
Let us give their basic properties.

\begin{lemma} \label{L:Bi_Surf}
  There exists a constant $c>0$ such that
  \begin{align} \label{E:BiS_Bdd}
    |a_g(v_1,v_2)| \leq c\|v_1\|_{H^1(\Gamma)}\|v_2\|_{H^1(\Gamma)}
  \end{align}
  for all $v_1,v_2\in H^1(\Gamma,T\Gamma)$ and
  \begin{align} \label{E:Bi_Surf}
    \|\nabla_\Gamma v\|_{L^2(\Gamma)}^2 \leq c\left\{a_g(v,v)+\|v\|_{L^2(\Gamma)}^2\right\}
  \end{align}
  for all $v\in H^1(\Gamma,T\Gamma)$.
\end{lemma}

\begin{proof}
  The inequality \eqref{E:BiS_Bdd} immediately follows from \eqref{E:G_Inf} and \eqref{E:Def_Bi_Surf}.
  Also,
  \begin{align*}
    \|\nabla_\Gamma v\|_{L^2(\Gamma)}^2 &\leq c\left(\|D_\Gamma(v)\|_{L^2(\Gamma)}^2+\|v\|_{L^2(\Gamma)}^2\right) \\
    &\leq c\left(2\nu\|g^{1/2}D_\Gamma(v)\|_{L^2(\Gamma)}^2+\|v\|_{L^2(\Gamma)}^2\right) \\
    &\leq c\left\{a_g(v,v)+\|v\|_{L^2(\Gamma)}^2\right\}
  \end{align*}
  by \eqref{E:G_Inf} and Korn's inequality \eqref{E:Korn_STG}.
  Thus \eqref{E:Bi_Surf} is valid.
\end{proof}

\begin{lemma} \label{L:Tri_Surf}
  There exists a constant $c>0$ such that
  \begin{align} \label{E:Tri_Surf}
    |b_g(v_1,v_2,v_3)| \leq c\|v_1\|_{L^2(\Gamma)}^{1/2}\|v_1\|_{H^1(\Gamma)}^{1/2}\|v_2\|_{L^2(\Gamma)}^{1/2}\|v_2\|_{H^1(\Gamma)}^{1/2}\|v_3\|_{H^1(\Gamma)}
  \end{align}
  for all $v_1,v_2,v_3\in H^1(\Gamma,T\Gamma)$.
  Moreover,
  \begin{align} \label{E:TriS_Vg}
    b_g(v_1,v_2,v_3) = -b_g(v_1,v_3,v_2), \quad b_g(v_1,v_2,v_2) = 0
  \end{align}
  for all $v_1\in \mathcal{V}_g$ and $v_2,v_3\in H^1(\Gamma,T\Gamma)$.
\end{lemma}

\begin{proof}
  By H\"{o}lder's inequality and Ladyzhenskaya's inequality \eqref{E:La_Surf} we have
  \begin{align*}
    |b_g(v_1,v_2,v_3)| &\leq c\|v_1\|_{L^4(\Gamma)}\|v_2\|_{L^4(\Gamma)}\|\nabla_\Gamma v_3\|_{L^2(\Gamma)} \\
    &\leq c\|v_1\|_{L^2(\Gamma)}^{1/2}\|v_1\|_{H^1(\Gamma)}^{1/2}\|v_2\|_{L^2(\Gamma)}^{1/2}\|v_2\|_{H^1(\Gamma)}^{1/2}\|v_3\|_{H^1(\Gamma)}.
  \end{align*}
  Thus \eqref{E:Tri_Surf} is valid and we further get
  \begin{align} \label{Pf_TrS:Bg_H1}
    |b_g(v_1,v_2,v_3)| \leq c\|v_1\|_{H^1(\Gamma)}\|v_2\|_{H^1(\Gamma)}\|v_3\|_{H^1(\Gamma)}.
  \end{align}
  Let us show \eqref{E:TriS_Vg} for $v_1\in\mathcal{V}_g$ and $v_2,v_3\in H^1(\Gamma,T\Gamma)$.
  By Lemma \ref{L:Wmp_Tan_Appr} we can take sequences $\{v_{2,k}\}_{k=1}^\infty$ and $\{v_{3,k}\}_{k=1}^\infty$ in $C^1(\Gamma,T\Gamma)$ such that
  \begin{align} \label{Pf_TrS:Conv}
    \lim_{k\to\infty}\|v_i-v_{i,k}\|_{H^1(\Gamma)} = 0, \quad i=2,3.
  \end{align}
  Then since
  \begin{multline*}
    |b_g(v_1,v_2,v_3)-b_g(v_1,v_{2,k},v_{3,k})| \\
    \leq c\|v_1\|_{H^1(\Gamma)}\left(\|v_2-v_{2,k}\|_{H^1(\Gamma)}\|v_3\|_{H^1(\Gamma)}+\|v_{2,k}\|_{H^1(\Gamma)}\|v_3-v_{3,k}\|_{H^1(\Gamma)}\right)
  \end{multline*}
  by \eqref{Pf_TrS:Bg_H1} and the right-hand side converges to zero as $k\to\infty$ by \eqref{Pf_TrS:Conv},
  \begin{align*}
    \lim_{k\to\infty}b_g(v_1,v_{2,k},v_{3,k}) = b_g(v_1,v_2,v_3).
  \end{align*}
  Thus it is sufficient to prove \eqref{E:TriS_Vg} for $v_2,v_3\in C^1(\Gamma,T\Gamma)$.
  For $a\in\mathbb{R}^3$ and $i=1,2,3$ we denote by $a^i$ the $i$-th component of $a$.
  Then
  \begin{align*}
    g(v_1\otimes v_2):\nabla_\Gamma v_3 &= \sum_{i,j=1}^3gv_1^iv_2^j\underline{D}_iv_3^j \\
    &= \sum_{i,j=1}^3\{\underline{D}_i(gv_1^iv_2^jv_3^j)-v_2^jv_3^j\underline{D}_i(gv_1^i)-gv_1^iv_3^j\underline{D}_iv_2^j\} \\
    &= \mathrm{div}_\Gamma[g(v_2\cdot v_3)v_1]-(v_2\cdot v_3)\mathrm{div}_\Gamma(gv_1)-g(v_1\otimes v_3):\nabla_\Gamma v_2
  \end{align*}
  on $\Gamma$ and $v_1\in\mathcal{V}_g$ satisfies $\mathrm{div}_\Gamma(gv_1)=0$ on $\Gamma$.
  Hence
  \begin{align*}
    b_g(v_1,v_2,v_3) = -\int_\Gamma\mathrm{div}_\Gamma[g(v_2\cdot v_3)v_1]\,d\mathcal{H}^2-b_g(v_1,v_3,v_2).
  \end{align*}
  Moreover, since $g(v_2\cdot v_3)v_1\in H^1(\Gamma,T\Gamma)$ by $v_1\in\mathcal{V}_g$ and the $C^1$-regularity of $g$, $v_2$, and $v_3$ on $\Gamma$ (in fact $g\in C^4(\Gamma)$), we observe by \eqref{E:IbP_WDivG_T} that
  \begin{align*}
    \int_\Gamma\mathrm{div}_\Gamma[g(v_2\cdot v_3)v_1]\,d\mathcal{H}^2 = 0.
  \end{align*}
  Thus the first equality of \eqref{E:TriS_Vg} is valid.
  Also, we get the second equality of \eqref{E:TriS_Vg} by setting $v_3=v_2$ in the first one.
\end{proof}

We approximate $a_\varepsilon$ and $b_\varepsilon$ by $a_g$ and $b_g$ by using the results of Section \ref{S:Ave}.
\begin{lemma} \label{L:Appr_Bili}
  Let $u\in H^2(\Omega_\varepsilon)^3$ satisfy the slip boundary conditions \eqref{E:Bo_Slip} and
  \begin{align*}
    \eta \in \mathcal{V}_g, \quad \eta_\varepsilon := \mathbb{L}_\varepsilon E_\varepsilon\eta \in L_\sigma^2(\Omega_\varepsilon)\cap H^1(\Omega_\varepsilon)^3.
  \end{align*}
  Then there exists a constant $c>0$ independent of $\varepsilon$, $u$, and $\eta$ such that
  \begin{align} \label{E:Appr_Bili}
    |a_\varepsilon(u,\eta_\varepsilon)-\varepsilon a_g(M_\tau u,\eta)| \leq cR_\varepsilon^a(u)\|\eta\|_{H^1(\Gamma)},
  \end{align}
  where
  \begin{align} \label{E:ApBi_Re}
    R_\varepsilon^a(u) := \varepsilon^{3/2}\|u\|_{H^2(\Omega_\varepsilon)}+\varepsilon^{1/2}\|u\|_{L^2(\Omega_\varepsilon)}\sum_{i=0,1}\left|\frac{\gamma_\varepsilon^i}{\varepsilon}-\gamma^i\right|.
  \end{align}
\end{lemma}

\begin{proof}
  Let $F(\eta)$ be the $3\times 3$ matrix-valued function on $\Gamma$ given by \eqref{E:Def_Fv},
  \begin{align*}
    J_1 := \bigl(D(u),D(\eta_\varepsilon)\bigr)_{L^2(\Omega_\varepsilon)}-\Bigl(D(u),\overline{F(\eta)}\Bigr)_{L^2(\Omega_\varepsilon)},
  \end{align*}
  and
  \begin{multline*}
    J_2 := \Bigl(D(u),\overline{F(\eta)}\Bigr)_{L^2(\Omega_\varepsilon)} \\
    -\varepsilon\left\{\bigl(gD_\Gamma(M_\tau u),D_\Gamma(\eta)\bigr)_{L^2(\Gamma)}+\left(M_\tau u\cdot\nabla_\Gamma g,\frac{1}{g}(\eta\cdot\nabla_\Gamma g)\right)_{L^2(\Gamma)}\right\}.
  \end{multline*}
  We also define
  \begin{align*}
    K_1 &:= \sum_{i=0,1}\gamma_\varepsilon^i\{(u,\eta_\varepsilon)_{L^2(\Gamma_\varepsilon^i)}-(u,\bar{\eta})_{L^2(\Gamma_\varepsilon^i)}\}, \\
    K_2 &:= \sum_{i=0,1}\gamma_\varepsilon^i\{(u,\bar{\eta})_{L^2(\Gamma_\varepsilon^i)}-(M_\tau u,\eta)_{L^2(\Gamma)}\}, \\
    K_3 &:= \sum_{i=0,1}(\gamma_\varepsilon^i-\varepsilon\gamma^i)(M_\tau u,\eta)_{L^2(\Gamma)}
  \end{align*}
  so that
  \begin{align*}
    a_\varepsilon(u,\eta_\varepsilon)-\varepsilon a_g(M_\tau u,\eta) = 2\nu(J_1+J_2)+K_1+K_2+K_3.
  \end{align*}
  Let us estimate each term on the right-hand side.
  Since $D(u)$ is symmetric,
  \begin{align*}
    D(u):D(\eta_\varepsilon) = D(u):\nabla\eta_\varepsilon \quad\text{in}\quad \Omega_\varepsilon.
  \end{align*}
  By this equality and \eqref{E:Test_Dom} we have
  \begin{align} \label{Pf_ApBi:I1}
    |J_1| \leq \|D(u)\|_{L^2(\Omega_\varepsilon)}\left\|\nabla\eta_\varepsilon-\overline{F(\eta)}\right\|_{L^2(\Omega_\varepsilon)} \leq c\varepsilon^{3/2}\|u\|_{H^1(\Omega_\varepsilon)}\|\eta\|_{H^1(\Gamma)}.
  \end{align}
  Next we consider $J_2$.
  Since $\eta\cdot n=0$ on $\Gamma$, we can apply \eqref{E:Grad_W} to $\nabla_\Gamma\eta$ to get
  \begin{align*}
    F(\eta) = A+v\otimes n+\xi Q, \quad A := P(\nabla_\Gamma\eta)P, \quad v := W\eta, \quad \xi := \frac{1}{g}(\eta\cdot\nabla_\Gamma g)
  \end{align*}
  on $\Gamma$.
  Using this decomposition we split $J_2=J_2^1+J_2^2+J_2^3$ into
  \begin{align*}
    J_2^1 &:= \Bigl(D(u),\overline{A}\Bigr)_{L^2(\Omega_\varepsilon)}-\varepsilon\bigl(gD_\Gamma(M_\tau u),D_\Gamma(\eta)\bigr)_{L^2(\Gamma)}, \\
    J_2^2 &:= \Bigl(D(u),\overline{\xi Q}\Bigr)_{L^2(\Omega_\varepsilon)}-\varepsilon(M_\tau u\cdot\nabla_\Gamma g,\xi)_{L^2(\Gamma)}, \\
    J_2^3 &:= (D(u),\bar{v}\otimes\bar{n})_{L^2(\Omega_\varepsilon)}.
  \end{align*}
  Since $P^T=P$ on $\Gamma$, we have $D_\Gamma(\eta)=(A+A^T)/2$ on $\Gamma$ and thus
  \begin{align*}
    D_\Gamma(M_\tau u):D_\Gamma(\eta) = \frac{1}{2}D_\Gamma(M_\tau u):(A+A^T) = D_\Gamma(M_\tau u):A \quad\text{on}\quad \Gamma
  \end{align*}
  by the symmetry of $D_\Gamma(M_\tau u)$.
  Hence
  \begin{align*}
    J_2^1 = \Bigl(D(u),\overline{A}\Bigr)_{L^2(\Omega_\varepsilon)}-\varepsilon(gD_\Gamma(M_\tau u),A)_{L^2(\Gamma)}
  \end{align*}
  and, since $u$ and $A$ satisfy the conditions of Lemma \ref{L:Ave_BiH1_TT}, we can use \eqref{E:Ave_BiH1_TT} to deduce that
  \begin{align*}
    |J_2^1| \leq c\varepsilon^{3/2}\|u\|_{H^1(\Omega_\varepsilon)}\|A\|_{L^2(\Gamma)} \leq c\varepsilon^{3/2}\|u\|_{H^1(\Omega_\varepsilon)}\|\eta\|_{H^1(\Gamma)}.
  \end{align*}
  Noting that $u$ satisfies \eqref{E:Bo_Imp}, we apply \eqref{E:Ave_BiH1_NN} to $J_2^2$ and use \eqref{E:G_Inf} to get
  \begin{align*}
    |J_2^2| \leq c\varepsilon^{3/2}\|u\|_{H^1(\Omega_\varepsilon)}\|\xi\|_{L^2(\Gamma)} \leq c\varepsilon^{3/2}\|u\|_{H^1(\Omega_\varepsilon)}\|\eta\|_{L^2(\Gamma)}.
  \end{align*}
  Also, since $u$ satisfies \eqref{E:Bo_Slip}, $v=W\eta$ belongs to $L^2(\Gamma,T\Gamma)$, and the inequalities \eqref{E:Fric_Upper} are valid by Assumption \ref{Assump_1}, we can apply \eqref{E:Ave_BiH1_TN} to obtain
  \begin{align*}
    |J_2^3| \leq c\varepsilon^{3/2}\|u\|_{H^2(\Omega_\varepsilon)}\|v\|_{L^2(\Gamma)} \leq c\varepsilon^{3/2}\|u\|_{H^2(\Omega_\varepsilon)}\|\eta\|_{L^2(\Gamma)}.
  \end{align*}
  From the above three estimates it follows that
  \begin{align} \label{Pf_ApBi:I2}
    |J_2| \leq |J_2^1|+|J_2^2|+|J_2^3| \leq c\varepsilon^{3/2}\|u\|_{H^2(\Omega_\varepsilon)}\|\eta\|_{H^1(\Gamma)}.
  \end{align}
  Let us estimate $K_1$, $K_2$, and $K_3$.
  To $K_1$ we apply \eqref{E:Fric_Upper}, \eqref{E:Poin_Bo}, and \eqref{E:Test_Bo} to get
  \begin{align} \label{Pf_ApBi:J1}
    |K_1| \leq c\varepsilon\|u\|_{L^2(\Gamma_\varepsilon)}\|\eta_\varepsilon-\bar{\eta}\|_{L^2(\Gamma_\varepsilon)} \leq c\varepsilon^{3/2}\|u\|_{H^1(\Omega_\varepsilon)}\|\eta\|_{H^1(\Gamma)}.
  \end{align}
  Also, since $\eta$ is tangential on $\Gamma$, we have $M_\tau u\cdot\eta=Mu\cdot\eta$ on $\Gamma$ and thus
  \begin{align} \label{Pf_ApBi:J2}
    |K_2| \leq c\varepsilon\sum_{i=0,1}\left|(u,\bar{\eta})_{L^2(\Gamma_\varepsilon^i)}-(Mu,\eta)_{L^2(\Gamma)}\right| \leq c\varepsilon^{3/2}\|u\|_{H^1(\Omega_\varepsilon)}\|\eta\|_{L^2(\Gamma)}
  \end{align}
  by \eqref{E:Fric_Upper} and \eqref{E:Ave_BiL2_Bo}.
  To $K_3$ we just use \eqref{E:AveT_Lp_Surf} to obtain
  \begin{align} \label{Pf_ApBi:J3}
    \begin{aligned}
      |K_3| &\leq c\varepsilon^{-1/2}\|u\|_{L^2(\Omega_\varepsilon)}\|\eta\|_{L^2(\Gamma)}\sum_{i=0,1}|\gamma_\varepsilon^i-\varepsilon\gamma^i| \\
      &= c\varepsilon^{1/2}\|u\|_{L^2(\Omega_\varepsilon)}\|\eta\|_{L^2(\Gamma)}\sum_{i=0,1}\left|\frac{\gamma_\varepsilon^i}{\varepsilon}-\gamma^i\right|.
    \end{aligned}
  \end{align}
  Finally, we deduce from \eqref{Pf_ApBi:I1}--\eqref{Pf_ApBi:J3} that
  \begin{align*}
    |a_\varepsilon(u,\eta_\varepsilon)-\varepsilon a_g(M_\tau u,\eta)| &\leq c(|J_1|+|J_2|+|K_1|+|K_2|+|K_3|) \\
    &\leq cR_\varepsilon^a(u)\|\eta\|_{H^1(\Gamma)},
  \end{align*}
  where $R_\varepsilon^a(u)$ is given by \eqref{E:ApBi_Re}.
  Hence \eqref{E:Appr_Bili} is valid.
\end{proof}

\begin{lemma} \label{L:Appr_Tri}
  Let $u_1\in H^2(\Omega_\varepsilon)^3$, $u_2\in H^1(\Omega_\varepsilon)^3$, and
  \begin{align*}
    \eta \in \mathcal{V}_g, \quad \eta_\varepsilon := \mathbb{L}_\varepsilon E_\varepsilon\eta \in L_\sigma^2(\Omega_\varepsilon)\cap H^1(\Omega_\varepsilon)^3.
  \end{align*}
  Suppose that $u_1$ satisfies $\mathrm{div}\,u_1=0$ in $\Omega_\varepsilon$ and \eqref{E:Bo_Slip} and $u_2$ satisfies \eqref{E:Bo_Imp} on $\Gamma_\varepsilon^0$ or on $\Gamma_\varepsilon^1$.
  Then there exists a constant $c>0$ independent of $\varepsilon$, $u_1$, $u_2$, and $\eta$ such that
  \begin{align} \label{E:Appr_Tri}
    |b_\varepsilon(u_1,u_2,\eta_\varepsilon)-\varepsilon b_g(M_\tau u_1,M_\tau u_2,\eta)| \leq cR_\varepsilon^b(u_1,u_2)\|\eta\|_{H^1(\Gamma)},
  \end{align}
  where
  \begin{multline} \label{E:ApTr_Re}
    R_\varepsilon^b(u_1,u_2) := \varepsilon^{3/2}\|u_1\otimes u_2\|_{L^2(\Omega_\varepsilon)}+\varepsilon\|u_1\|_{H^1(\Omega_\varepsilon)}\|u_2\|_{H^1(\Omega_\varepsilon)} \\
    +\left(\varepsilon\|u_1\|_{H^2(\Omega_\varepsilon)}+\varepsilon^{1/2}\|u_1\|_{L^2(\Omega_\varepsilon)}^{1/2}\|u_1\|_{H^2(\Omega_\varepsilon)}^{1/2}\right)\|u_2\|_{L^2(\Omega_\varepsilon)}.
  \end{multline}
\end{lemma}

\begin{proof}
  Let $F(\eta)$ be the $3\times3$ matrix-valued function on $\Gamma$ given by \eqref{E:Def_Fv}.
  Since
  \begin{align*}
    \left|b_\varepsilon(u_1,u_2,\eta_\varepsilon)+\Bigl(u_1\otimes u_2,\overline{F(\eta)}\Bigr)_{L^2(\Omega_\varepsilon)}\right| \leq \|u_1\otimes u_2\|_{L^2(\Omega_\varepsilon)}\left\|\nabla\eta_\varepsilon-\overline{F(\eta)}\right\|_{L^2(\Omega_\varepsilon)}
  \end{align*}
  by \eqref{E:Def_Tri_Dom}, we apply \eqref{E:Test_Dom} to the right-hand side to get
  \begin{align} \label{Pf_ApTr:TAux}
    \left|b_\varepsilon(u_1,u_2,\eta_\varepsilon)+\Bigl(u_1\otimes u_2,\overline{F(\eta)}\Bigr)_{L^2(\Omega_\varepsilon)}\right| \leq c\varepsilon^{3/2}\|u_1\otimes u_2\|_{L^2(\Omega_\varepsilon)}\|\eta\|_{H^1(\Gamma)}.
  \end{align}
  Noting that $\eta$ is tangential on $\Gamma$, we use \eqref{E:Grad_W} to decompose $F(\eta)$ into
  \begin{align*}
    F(\eta) = A+v\otimes n, \quad A := P(\nabla_\Gamma\eta)P, \quad v := W\eta+\frac{1}{g}(\eta\cdot\nabla_\Gamma g)n \quad\text{on}\quad \Gamma.
  \end{align*}
  Since $u_1$ and $A$ satisfy the conditions of Lemma \ref{L:Ave_TrT}, we see by \eqref{E:Ave_TrT} that
  \begin{multline} \label{Pf_ApTr:TT}
    \left|\Bigl(u_1\otimes u_2,\overline{A}\Bigr)_{L^2(\Omega_\varepsilon)}-\varepsilon(g(M_\tau u_1)\otimes(M_\tau u_2),A)_{L^2(\Gamma)}\right| \\
    \leq cR_\varepsilon(u_1,u_2)\|A\|_{L^2(\Gamma)} \leq cR_\varepsilon(u_1,u_2)\|\eta\|_{H^1(\Gamma)}.
  \end{multline}
  Here $R_\varepsilon(u_1,u_2)$ is given by \eqref{E:Ave_TrT_Re}.
  Also, since $u_2$ satisfies \eqref{E:Bo_Imp} on $\Gamma_\varepsilon^0$ or on $\Gamma_\varepsilon^1$ and $v\in H^1(\Gamma)^3$, we can use \eqref{E:Ave_TrN} to get
  \begin{align} \label{Pf_ApTr:TN}
    \begin{aligned}
      |(u_1\otimes u_2,\bar{v}\otimes\bar{n})_{L^2(\Omega_\varepsilon)}| &\leq c\varepsilon\|u_1\|_{H^1(\Omega_\varepsilon)}\|u_2\|_{H^1(\Omega_\varepsilon)}\|v\|_{H^1(\Gamma)} \\
      &\leq c\varepsilon\|u_1\|_{H^1(\Omega_\varepsilon)}\|u_2\|_{H^1(\Omega_\varepsilon)}\|\eta\|_{H^1(\Gamma)}.
    \end{aligned}
  \end{align}
  Noting that $F(\eta)=A+v\otimes n$ on $\Gamma$ we deduce from \eqref{Pf_ApTr:TAux}--\eqref{Pf_ApTr:TN} that
  \begin{align} \label{Pf_ApTr:Goal}
    \left|b_\varepsilon(u_1,u_2,\eta_\varepsilon)+\varepsilon(g(M_\tau u_1)\otimes(M_\tau u_2),A)_{L^2(\Gamma)}\right| \leq cR_\varepsilon^b(u_1,u_2)\|\eta\|_{H^1(\Gamma)}
  \end{align}
  with $R_\varepsilon^b(u_1,u_2)$ given by \eqref{E:ApTr_Re}.
  Now we observe that
  \begin{align*}
    (M_\tau u_1)\otimes(M_\tau u_2):A = (M_\tau u_1)\otimes(M_\tau u_2):\nabla_\Gamma\eta \quad\text{on}\quad \Gamma
  \end{align*}
  by $A=P(\nabla_\Gamma\eta)P$, $P^T=P$, and $PM_\tau u_i=M_\tau u_i$ on $\Gamma$ for $i=1,2$.
  Hence
  \begin{align*}
    (g(M_\tau u_1)\otimes(M_\tau u_2),A)_{L^2(\Gamma)} = -b_g(M_\tau u_1,M_\tau u_2,\eta)
  \end{align*}
  by \eqref{E:Def_Tri_Surf} and the inequality \eqref{E:Appr_Tri} follows from \eqref{Pf_ApTr:Goal}.
\end{proof}

Now we are ready to derive a weak formulation for $M_\tau u^\varepsilon$ from \eqref{E:NS_Weak}.

\begin{lemma} \label{L:Mu_Weak}
  Let $u^\varepsilon$ be as in Lemma \ref{L:Est_Ueps}.
  Then
  \begin{align*}
    M_\tau u^\varepsilon \in C([0,\infty);H^1(\Gamma,T\Gamma))\cap H_{loc}^1([0,\infty);L^2(\Gamma,T\Gamma))
  \end{align*}
  and for all $T>0$ and $\eta\in L^2(0,T;\mathcal{V}_g)$ we have
  \begin{multline} \label{E:Mu_Weak}
    \int_0^T\{(g\partial_tM_\tau u^\varepsilon,\eta)_{L^2(\Gamma)}+a_g(M_\tau u^\varepsilon,\eta)+b_g(M_\tau u^\varepsilon,M_\tau u^\varepsilon,\eta)\}\,dt \\
    = \int_0^T(gM_\tau\mathbb{P}_\varepsilon f^\varepsilon,\eta)_{L^2(\Gamma)}\,dt+R_\varepsilon^1(\eta).
  \end{multline}
  Here $R_\varepsilon^1(\eta)$ is a residual term satisfying
  \begin{align} \label{E:Mu_Weak_Re}
    |R_\varepsilon^1(\eta)| \leq c\left(\varepsilon^{\alpha/4}+\sum_{i=0,1}\left|\frac{\gamma_\varepsilon^i}{\varepsilon}-\gamma^i\right|\right)(1+T)^{1/2}\|\eta\|_{L^2(0,T;H^1(\Gamma))}
  \end{align}
  with a constant $c>0$ independent of $\varepsilon$, $u^\varepsilon$, $\eta$, and $T$.
\end{lemma}

\begin{proof}
  The space-time regularity of $M_\tau u^\varepsilon$ follows from that of $u^\varepsilon$ and Lemmas \ref{L:AveT_Lp}, \ref{L:AveT_Wmp}, and \ref{L:Ave_Dt}.
  Let us show \eqref{E:Mu_Weak}.
  For $\eta\in L^2(0,T;\mathcal{V}_g)$ let
  \begin{align} \label{Pf_MuW:eta_eps}
    \eta_\varepsilon := \mathbb{L}_\varepsilon E_\varepsilon\eta \in L^2(0,T;L_\sigma^2(\Omega_\varepsilon)\cap H^1(\Omega_\varepsilon)^3).
  \end{align}
  Hereafter we sometimes suppress the argument $t\in(0,T)$.
  We first show
  \begin{multline} \label{Pf_MuW:Weak_eps}
    \int_0^T\{(\partial_tu^\varepsilon,\eta_\varepsilon)_{L^2(\Omega_\varepsilon)}+a_\varepsilon(u^\varepsilon,\eta_\varepsilon)+b_\varepsilon(u^\varepsilon,u^\varepsilon,\eta_\varepsilon)\}\,dt \\
    = \int_0^T(\mathbb{P}_\varepsilon f^\varepsilon,\eta_\varepsilon)_{L^2(\Omega_\varepsilon)}\,dt.
  \end{multline}
  If the condition (A1) or (A2) of Assumption \ref{Assump_2} is imposed, then $\eta_\varepsilon\in L^2(0,T;\mathcal{V}_\varepsilon)$ by \eqref{E:Def_Heps} and we can take $\eta_\varepsilon$ as a test function in \eqref{E:NS_Weak} to get \eqref{Pf_MuW:Weak_eps}.
  Suppose that the condition (A3) is imposed.
  Let $\mathcal{R}_g$ be the function space given by \eqref{E:Def_Rg} and $\mathcal{R}_g^\perp$ the orthogonal complement of $\mathcal{R}_g$ in $L^2(\Omega_\varepsilon)^3$.
  If $\mathcal{R}_g=\{0\}$, then
  \begin{align*}
    \mathcal{H}_\varepsilon = L_\sigma^2(\Omega_\varepsilon)\cap\mathcal{R}_g^\perp = L_\sigma^2(\Omega_\varepsilon), \quad \mathcal{V}_\varepsilon = \mathcal{H}_\varepsilon\cap H^1(\Omega_\varepsilon)^3 = L_\sigma^2(\Omega_\varepsilon)\cap H^1(\Omega_\varepsilon)^3
  \end{align*}
  and thus we can still take $\eta_\varepsilon\in L^2(0,T;\mathcal{V}_\varepsilon)$ as a test function in \eqref{E:NS_Weak} to get \eqref{Pf_MuW:Weak_eps}.
  On the other hand, if $\mathcal{R}_g\neq\{0\}$, then
  \begin{align*}
    \mathcal{H}_\varepsilon = L_\sigma^2(\Omega_\varepsilon)\cap\mathcal{R}_g^\perp \neq L_\sigma^2(\Omega_\varepsilon), \quad \mathcal{V}_\varepsilon = \mathcal{H}_\varepsilon\cap H^1(\Omega_\varepsilon)^3 \neq L_\sigma^2(\Omega_\varepsilon)\cap H^1(\Omega_\varepsilon)^3
  \end{align*}
  and we cannot substitute $\eta_\varepsilon$ for \eqref{E:NS_Weak}.
  In this case, however, $\mathcal{R}_g$ is a finite dimensional subspace of $L_\sigma^2(\Omega_\varepsilon)$ by the assumption $\mathcal{R}_g=\mathcal{R}_0\cap\mathcal{R}_1$ (see \cite{Miu_NSCTD_01}*{Lemma E.8}).
  Thus we can take an orthonormal basis $\{w_1,\dots,w_{k_0}\}$ of $\mathcal{R}_g$ in $L_\sigma^2(\Omega_\varepsilon)$ such that
  \begin{align} \label{Pf_MuW:w_k}
    w_k(x) = a_k\times x+b_k, \quad x\in\mathbb{R}^3
  \end{align}
  with some $a_k,b_k\in\mathbb{R}^3$ for $k=1,\dots,k_0$.
  Then setting
  \begin{align*}
    w_\varepsilon(t) := \sum_{k=1}^{k_0}(\eta_\varepsilon(t),w_k)_{L^2(\Omega_\varepsilon)}w_k \in \mathcal{R}_g
  \end{align*}
  we have the orthogonal decomposition
  \begin{align*}
    \eta_\varepsilon(t) = \mathbb{P}_\varepsilon\eta_\varepsilon(t)+w_\varepsilon(t) \quad\text{in}\quad L_\sigma^2(\Omega_\varepsilon), \quad \mathbb{P}_\varepsilon\eta_\varepsilon(t) \in \mathcal{H}_\varepsilon, \, w_\varepsilon(t) \in \mathcal{R}_g
  \end{align*}
  for a.a. $t\in(0,T)$.
  Moreover, since $w_k$ is independent of time and
  \begin{align*}
    w_k \in H^1(\Omega_\varepsilon)^3, \quad (\eta_\varepsilon(\cdot),w_k)_{L^2(\Omega_\varepsilon)} \in L^2(0,T), \quad k=1,\dots,k_0
  \end{align*}
  by \eqref{Pf_MuW:eta_eps} and \eqref{Pf_MuW:w_k}, we have $w_\varepsilon\in L^2(0,T;H^1(\Omega_\varepsilon)^3)$ and thus
  \begin{align*}
    \mathbb{P}_\varepsilon\eta_\varepsilon = \eta_\varepsilon-w_\varepsilon \in L^2(0,T;\mathcal{H}_\varepsilon\cap H^1(\Omega_\varepsilon)^3) = L^2(0,T;\mathcal{V}_\varepsilon).
  \end{align*}
  Hence we can substitute $\mathbb{P}_\varepsilon\eta_\varepsilon$ for \eqref{E:NS_Weak}.
  Moreover,
  \begin{align*}
    (\partial_tu^\varepsilon,w_\varepsilon)_{L^2(\Omega_\varepsilon)} = (\mathbb{P}_\varepsilon f^\varepsilon,w_\varepsilon)_{L^2(\Omega_\varepsilon)} = 0
  \end{align*}
  by $\partial_tu^\varepsilon,\mathbb{P}_\varepsilon f^\varepsilon\in\mathcal{H}_\varepsilon$ and $w_\varepsilon\in\mathcal{R}_g\subset\mathcal{H}_\varepsilon^\perp$.
  Also, since
  \begin{align*}
    \gamma_\varepsilon^0 = \gamma_\varepsilon^1 = 0, \quad D(w_\varepsilon) = \sum_{k=1}^{k_0}(\eta_\varepsilon,w_k)_{L^2(\Omega_\varepsilon)}D(w_k) = 0 \quad\text{in}\quad \Omega_\varepsilon
  \end{align*}
  by the condition (A3) and \eqref{Pf_MuW:w_k}, we see that
  \begin{align*}
    a_\varepsilon(u^\varepsilon,w_\varepsilon) = 2\nu\bigl(D(u^\varepsilon),D(w_\varepsilon)\bigr)_{L^2(\Omega_\varepsilon)} = 0.
  \end{align*}
  We further observe by direct calculations and \eqref{Pf_MuW:w_k} that
  \begin{align*}
    u^\varepsilon\otimes u^\varepsilon:\nabla w_k = u^\varepsilon\cdot(u^\varepsilon\cdot\nabla)w_k = u^\varepsilon\cdot(a_k\times u^\varepsilon) = 0 \quad\text{in}\quad \Omega_\varepsilon
  \end{align*}
  for all $k=1,\dots,k_0$ and thus
  \begin{align*}
    b_\varepsilon(u^\varepsilon,u^\varepsilon,w_\varepsilon) = \sum_{k=1}^{k_0}(\eta_\varepsilon,w_k)_{L^2(\Omega_\varepsilon)}b_\varepsilon(u^\varepsilon,u^\varepsilon,w_k) = 0.
  \end{align*}
  Hence all terms including $w_\varepsilon$ vanish in \eqref{E:NS_Weak} with $\varphi=\mathbb{P}_\varepsilon\eta_\varepsilon=\eta_\varepsilon-w_\varepsilon$ and we get \eqref{Pf_MuW:Weak_eps} under the condition (A3) with $\mathcal{R}_g\neq\{0\}$.
  Therefore, \eqref{Pf_MuW:Weak_eps} holds under any condition of (A1), (A2), and (A3).

  Now we divide both sides of \eqref{Pf_MuW:Weak_eps} by $\varepsilon$ and replace each term of the resulting equality by the corresponding term of \eqref{E:Mu_Weak}.
  Then we get \eqref{E:Mu_Weak} with
  \begin{align*}
    R_\varepsilon^1(\eta) := \varepsilon^{-1}(J_1+J_2+J_3+J_4),
  \end{align*}
  where
  \begin{align*}
    J_1 &:= -\int_0^T(\partial_tu^\varepsilon,\eta_\varepsilon)_{L^2(\Omega_\varepsilon)}\,dt+\varepsilon\int_0^T(g\partial_tM_\tau u^\varepsilon,\eta)_{L^2(\Gamma)}\,dt, \\
    J_2 &:= -\int_0^Ta_\varepsilon(u^\varepsilon,\eta_\varepsilon)\,dt+\varepsilon\int_0^Ta_g(M_\tau u^\varepsilon,\eta)\,dt, \\
    J_3 &:= -\int_0^T b_\varepsilon(u^\varepsilon,u^\varepsilon,\eta_\varepsilon)\,dt+\varepsilon\int_0^Tb_g(M_\tau u^\varepsilon,M_\tau u^\varepsilon,\eta)\,dt, \\
    J_4 &:= \int_0^T(\mathbb{P}_\varepsilon f^\varepsilon,\eta_\varepsilon)_{L^2(\Omega_\varepsilon)}\,dt-\varepsilon\int_0^T(gM_\tau\mathbb{P}_\varepsilon f^\varepsilon,\eta)_{L^2(\Gamma)}\,dt.
  \end{align*}
  Let us estimate these differences.
  First note that
  \begin{align*}
    (g\partial_tM_\tau u^\varepsilon,\eta)_{L^2(\Gamma)} = (gM_\tau(\partial_tu^\varepsilon),\eta)_{L^2(\Gamma)} = (gM(\partial_tu^\varepsilon),\eta)_{L^2(\Gamma)}
  \end{align*}
  by Lemma \ref{L:Ave_Dt} and $\eta\cdot n=0$ on $\Gamma$.
  Thus, by \eqref{E:Ave_BiL2_Dom} and \eqref{E:Test_Dom},
  \begin{multline*}
    |(\partial_tu^\varepsilon,\eta_\varepsilon)_{L^2(\Omega_\varepsilon)}-\varepsilon(g\partial_tM_\tau u^\varepsilon,\eta)_{L^2(\Gamma)}| \\
    \begin{aligned}
      &\leq |(\partial_tu^\varepsilon,\bar{\eta})_{L^2(\Omega_\varepsilon)}-\varepsilon(gM(\partial_tu^\varepsilon),\eta)_{L^2(\Gamma)}|+\|\partial_tu^\varepsilon\|_{L^2(\Omega_\varepsilon)}\|\eta_\varepsilon-\bar{\eta}\|_{L^2(\Omega_\varepsilon)} \\
      &\leq c\varepsilon^{3/2}\|\partial_tu^\varepsilon\|_{L^2(\Omega_\varepsilon)}\|\eta\|_{L^2(\Gamma)}.
    \end{aligned}
  \end{multline*}
  From this inequality, H\"{o}lder's inequality, and \eqref{E:Est_DtUe} it follows that
  \begin{align} \label{Pf_MuW:Dt}
    \begin{aligned}
      |J_1| &\leq c\varepsilon^{3/2}\|\partial_tu^\varepsilon\|_{L^2(0,T;L^2(\Omega_\varepsilon))}\|\eta\|_{L^2(0,T;L^2(\Gamma))} \\
      &\leq c\varepsilon^{1+\alpha/2}(1+T)^{1/2}\|\eta\|_{L^2(0,T;L^2(\Gamma))}.
    \end{aligned}
  \end{align}
  In the same way, we apply \eqref{E:Ave_BiL2_Dom} and \eqref{E:Test_Dom} to $J_4$ and then use \eqref{E:UE_Data} to get
  \begin{align} \label{Pf_MuW:F}
    |J_4| \leq c\varepsilon^{1+\alpha/2}T^{1/2}\|\eta\|_{L^2(0,T;L^2(\Gamma))}.
  \end{align}
  Next we deal with $J_2$.
  By \eqref{E:Appr_Bili} we see that
  \begin{align*}
    |J_2| \leq c\left(\int_0^TR_\varepsilon^a(u^\varepsilon)^2\,dt\right)^{1/2}\|\eta\|_{L^2(0,T;H^1(\Gamma))},
  \end{align*}
  where $R_\varepsilon^a(u^\varepsilon)$ is given by \eqref{E:ApBi_Re}.
  Moreover, by \eqref{E:Est_Ueps} we have
  \begin{align*}
    \int_0^TR_\varepsilon^a(u^\varepsilon)^2\,dt &\leq c\left(\varepsilon^3\int_0^T\|u^\varepsilon\|_{H^2(\Omega_\varepsilon)}^2\,dt+\varepsilon\gamma(\varepsilon)^2\int_0^T\|u^\varepsilon\|_{L^2(\Omega_\varepsilon)}^2\,dt\right) \\
    &\leq c\varepsilon^2\{\varepsilon^\alpha+\gamma(\varepsilon)^2\}(1+T)
  \end{align*}
  with $\gamma(\varepsilon):=\sum_{i=0,1}|\varepsilon^{-1}\gamma_\varepsilon^i-\gamma^i|$.
  Therefore,
  \begin{align} \label{Pf_MuW:A}
    |J_2| \leq c\varepsilon\{\varepsilon^{\alpha/2}+\gamma(\varepsilon)\}(1+T)^{1/2}\|\eta\|_{L^2(0,T;H^1(\Gamma))}.
  \end{align}
  Let us estimate $J_3$.
  By \eqref{E:Appr_Tri} we have
  \begin{align*}
    |J_3| \leq c\left(\int_0^TR_\varepsilon^b(u^\varepsilon,u^\varepsilon)^2\,dt\right)^{1/2}\|\eta\|_{L^2(0,T;H^1(\Gamma))}
  \end{align*}
  with $R_\varepsilon^b(u^\varepsilon,u^\varepsilon)$ given by \eqref{E:ApTr_Re}.
  To estimate the right-hand side, we see that
  \begin{align*}
    \int_0^T\|u^\varepsilon\|_{H^1(\Omega_\varepsilon)}^4\,dt \leq \left(\max_{t\in[0,T]}\|u^\varepsilon(t)\|_{H^1(\Omega_\varepsilon)}^2\right)\int_0^T\|u^\varepsilon\|_{H^1(\Omega_\varepsilon)}^2\,dt \leq c\varepsilon^\alpha(1+T)
  \end{align*}
  by \eqref{E:Est_Ueps}.
  Using this inequality, \eqref{E:Est_UeUe}, \eqref{Pf_EsSt:L2H2}, and \eqref{Pf_EsSt:L2H2_31} we deduce that
  \begin{align*}
    &\int_0^TR_\varepsilon^b(u^\varepsilon,u^\varepsilon)^2\,dt \\
    &\qquad \leq c\left(\varepsilon^3\int_0^T\|u^\varepsilon\otimes u^\varepsilon\|_{L^2(\Omega_\varepsilon)}^2\,dt+\varepsilon^2\int_0^T\|u^\varepsilon\|_{H^1(\Omega_\varepsilon)}^4\,dt\right. \\
    &\qquad\qquad \left.+\varepsilon^2\int_0^T\|u^\varepsilon\|_{L^2(\Omega_\varepsilon)}^2\|u^\varepsilon\|_{H^2(\Omega_\varepsilon)}^2\,dt+\varepsilon\int_0^T\|u^\varepsilon\|_{L^2(\Omega_\varepsilon)}^3\|u^\varepsilon\|_{H^2(\Omega_\varepsilon)}\,dt\right)\\
    &\qquad \leq c\varepsilon^2(\varepsilon^2+\varepsilon^\alpha+\varepsilon^{\alpha/2})(1+T) \leq c\varepsilon^{2+\alpha/2}(1+T).
  \end{align*}
  Note that $\varepsilon^2,\varepsilon^\alpha\leq\varepsilon^{\alpha/2}$ by $\varepsilon,\alpha\in(0,1]$.
  Hence we obtain
  \begin{align} \label{Pf_MuW:B}
    |J_3| \leq c\varepsilon^{1+\alpha/4}(1+T)^{1/2}\|\eta\|_{L^2(0,T;H^1(\Gamma))}.
  \end{align}
  Finally, we observe by \eqref{Pf_MuW:Dt}--\eqref{Pf_MuW:B} that
  \begin{align*}
    |R_\varepsilon^1(\eta)| \leq \varepsilon^{-1}\sum_{k=1}^4|J_k| \leq c\{\varepsilon^{\alpha/4}+\varepsilon^{\alpha/2}+\gamma(\varepsilon)\}(1+T)^{1/2}\|\eta\|_{L^2(0,T;H^1(\Gamma))}
  \end{align*}
  and thus \eqref{E:Mu_Weak_Re} holds by $\gamma(\varepsilon)=\sum_{i=0,1}|\varepsilon^{-1}\gamma_\varepsilon^i-\gamma^i|$ and $\varepsilon^{\alpha/2}\leq\varepsilon^{\alpha/4}$.
\end{proof}

\subsection{Energy estimate for the average of the strong solution} \label{SS:SL_Ener}
Next we derive the energy estimate for $M_\tau u^\varepsilon$.
We would easily get the energy estimate if we could take $M_\tau u^\varepsilon$ itself as a test function in \eqref{E:Mu_Weak}.
However, we cannot do that since $M_\tau u^\varepsilon$ is not in $\mathcal{V}_g$, i.e. the surface divergence of $gM_\tau u^\varepsilon$ does not vanish on $\Gamma$ in general.
To overcome this difficulty we use the weighted Helmholtz--Leray projection
\begin{align*}
  \mathbb{P}_g\colon L^2(\Gamma,T\Gamma) \to \mathcal{H}_g = L_{g\sigma}^2(\Gamma,T\Gamma)
\end{align*}
given in Section \ref{SS:WS_HL} and replace $M_\tau u^\varepsilon$ in \eqref{E:Mu_Weak} by $\mathbb{P}_gM_\tau u^\varepsilon$.

\begin{lemma} \label{L:Rep_WHL_L2}
  Let $u\in L_\sigma^2(\Omega_\varepsilon)$.
  Then $\mathbb{P}_gM_\tau u\in\mathcal{H}_g$ and there exists a constant $c>0$ independent of $\varepsilon$ and $u$ such that
  \begin{align} \label{E:Rep_WHL_L2}
    \|M_\tau u-\mathbb{P}_gM_\tau u\|_{L^2(\Gamma)} \leq c\varepsilon^{1/2}\|u\|_{L^2(\Omega_\varepsilon)}.
  \end{align}
  Moreover, if $u\in L_\sigma^2(\Omega_\varepsilon)\cap H^1(\Omega_\varepsilon)^3$, then $\mathbb{P}_gM_\tau u\in \mathcal{V}_g$ and
  \begin{align} \label{E:Rep_WHL_H1}
    \|M_\tau u-\mathbb{P}_gM_\tau u\|_{H^1(\Gamma)} \leq c\varepsilon^{1/2}\|u\|_{H^1(\Omega_\varepsilon)}.
  \end{align}
\end{lemma}

\begin{proof}
  By $u\in L_\sigma^2(\Omega_\varepsilon)$ and Lemma \ref{L:AveT_Lp} we have
  \begin{align*}
    M_\tau u \in L^2(\Gamma,T\Gamma), \quad \mathbb{P}_gM_\tau u \in \mathcal{H}_g.
  \end{align*}
  Also, we apply \eqref{E:HLT_Est_L2} to $v=M_\tau u$ and use \eqref{E:ADiv_Tan_Hin} to get \eqref{E:Rep_WHL_L2}.

  If $u\in L_\sigma^2(\Omega_\varepsilon)\cap H^1(\Omega_\varepsilon)^3$, then Lemmas \ref{L:AveT_Wmp} and \ref{L:HLT_Est_L2} imply that
  \begin{align*}
    M_\tau u \in H^1(\Gamma,T\Gamma), \quad \mathbb{P}_gM_\tau u \in \mathcal{V}_g.
  \end{align*}
  We also have \eqref{E:Rep_WHL_H1} by applying \eqref{E:HLT_Est_H1} to $v=M_\tau u$ and using \eqref{E:ADiv_Tan_Lp}.
\end{proof}

\begin{lemma} \label{L:Rep_WHL_Dt}
  For $T>0$ let $u\in H^1(0,T;L_\sigma^2(\Omega_\varepsilon))$.
  Then
  \begin{align} \label{E:Rep_WHL_Reg}
    \mathbb{P}_gM_\tau u\in H^1(0,T;\mathcal{H}_g)
  \end{align}
  and there exists a constant $c>0$ independent of $\varepsilon$ and $u$ such that
  \begin{align} \label{E:Rep_WHL_Dt}
    \|\partial_tM_\tau u-\partial_t\mathbb{P}_gM_\tau u\|_{L^2(0,T;L^2(\Gamma))} \leq c\varepsilon^{1/2}\|\partial_tu\|_{L^2(0,T;L^2(\Omega_\varepsilon))}.
  \end{align}
\end{lemma}

\begin{proof}
  By $u\in H^1(0,T;L_\sigma^2(\Omega_\varepsilon))$ and Lemma \ref{L:Ave_Dt} we have
  \begin{align*}
    M_\tau u \in H^1(0,T;L^2(\Gamma,T\Gamma)), \quad \partial_tM_\tau u = M_\tau(\partial_tu) \quad\text{in}\quad L^2(0,T;L^2(\Gamma,T\Gamma)).
  \end{align*}
  Hence \eqref{E:Rep_WHL_Reg} holds by Lemma \ref{L:HLT_Est_Dt} and we observe by \eqref{E:HLT_Est_Dt} that
  \begin{align*}
    \|\partial_tM_\tau u-\partial_t\mathbb{P}_gM_\tau u\|_{L^2(0,T;L^2(\Gamma))} &\leq c\|\mathrm{div}_\Gamma(g\partial_tM_\tau u)\|_{L^2(0,T;H^{-1}(\Gamma))} \\
    &= c\|\mathrm{div}_\Gamma[gM_\tau(\partial_tu)]\|_{L^2(0,T;H^{-1}(\Gamma))}.
  \end{align*}
  Moreover, since $\partial_tu\in L^2(0,T;L_\sigma^2(\Omega_\varepsilon))$, we can use \eqref{E:ADiv_Tan_Hin} to get
  \begin{align*}
    \|\mathrm{div}_\Gamma[gM_\tau(\partial_tu)]\|_{L^2(0,T;H^{-1}(\Gamma))} \leq c\varepsilon^{1/2}\|\partial_tu\|_{L^2(0,T;L^2(\Omega_\varepsilon))}.
  \end{align*}
  Combining the above two inequalities we obtain \eqref{E:Rep_WHL_Dt}.
\end{proof}

Using Lemmas \ref{L:Rep_WHL_L2} and \ref{L:Rep_WHL_Dt} we replace $M_\tau u^\varepsilon$ in \eqref{E:Mu_Weak} by $\mathbb{P}_gM_\tau u^\varepsilon$.

\begin{lemma} \label{L:PMu_Weak}
  Let $u^\varepsilon$ be as in Lemma \ref{L:Est_Ueps}.
  Then
  \begin{align*}
    v^\varepsilon := \mathbb{P}_gM_\tau u^\varepsilon \in C([0,\infty);\mathcal{V}_g)\cap H_{loc}^1([0,\infty),\mathcal{H}_g)
  \end{align*}
  and there exists a constant $c>0$ independent of $\varepsilon$ and $u^\varepsilon$ such that
  \begin{align} \label{E:Diff_PMu}
    \begin{aligned}
      \|M_\tau u^\varepsilon(t)-v^\varepsilon(t)\|_{L^2(\Gamma)}^2 &\leq c\varepsilon^2, \\
      \int_0^t\|M_\tau u^\varepsilon(s)-v^\varepsilon(s)\|_{H^1(\Gamma)}^2\,ds &\leq c\varepsilon^2(1+t)
      \end{aligned}
  \end{align}
  for all $t\geq 0$.
  Moreover, for all $T>0$ and $\eta\in L^2(0,T;\mathcal{V}_g)$ we have
  \begin{multline} \label{E:PMu_Weak}
    \int_0^T\{(g\partial_tv^\varepsilon,\eta)_{L^2(\Gamma)}+a_g(v^\varepsilon,\eta)+b_g(v^\varepsilon,v^\varepsilon,\eta)\}\,dt \\
    = \int_0^T(gM_\tau\mathbb{P}_\varepsilon f^\varepsilon,\eta)_{L^2(\Gamma)}\,dt+R_\varepsilon^1(\eta)+R_\varepsilon^2(\eta),
  \end{multline}
  where $R_\varepsilon^1(\eta)$ is given in Lemma \ref{L:Mu_Weak} and $R_\varepsilon^2(\eta)$ satisfies
  \begin{align} \label{E:PMu_Weak_Re}
    |R_\varepsilon^2(\eta)| \leq c\varepsilon^{\alpha/2}(1+T)^{1/2}\|\eta\|_{L^2(0,T;H^1(\Gamma))}
  \end{align}
  with a constant $c>0$ independent of $\varepsilon$, $v^\varepsilon$, $\eta$, and $T$.
\end{lemma}

\begin{proof}
  The space-time regularity of $v^\varepsilon$ follows from that of $u^\varepsilon$ and Lemmas \ref{L:Rep_WHL_L2} and \ref{L:Rep_WHL_Dt}.
  We also have \eqref{E:Diff_PMu} by \eqref{E:Est_Ueps}, \eqref{E:Rep_WHL_L2}, and \eqref{E:Rep_WHL_H1}.
  For $\eta\in L^2(0,T;\mathcal{V}_g)$ let
  \begin{align*}
    J_1 &:= -\int_0^T(g\partial_t M_\tau u^\varepsilon,\eta)_{L^2(\Gamma)}\,dt+\int_0^T(g\partial_tv^\varepsilon,\eta)_{L^2(\Gamma)}\,dt, \\
    J_2 &:= -\int_0^Ta_g(M_\tau u^\varepsilon,\eta)\,dt+\int_0^Ta_g(v^\varepsilon,\eta)\,dt, \\
    J_3 &:= -\int_0^Tb_g(M_\tau u^\varepsilon,M_\tau u^\varepsilon,\eta)\,dt+\int_0^Tb_g(v^\varepsilon,v^\varepsilon,\eta)\,dt.
  \end{align*}
  Then by \eqref{E:Mu_Weak} we have \eqref{E:PMu_Weak} with
  \begin{align} \label{Pf_PMuW:Res}
    R_\varepsilon^2(\eta) := J_1+J_2+J_3.
  \end{align}
  Let us estimate $J_1$, $J_2$, and $J_3$.
  For $J_1$, we observe by \eqref{E:Est_DtUe} and \eqref{E:Rep_WHL_Dt} that
  \begin{align} \label{Pf_PMuW:I1}
    \begin{aligned}
      |J_1| &\leq c\|\partial_tM_\tau u^\varepsilon-\partial_tv^\varepsilon\|_{L^2(0,T;L^2(\Gamma))}\|\eta\|_{L^2(0,T;L^2(\Gamma))} \\
      &\leq c\varepsilon^{1/2}\|\partial_tu^\varepsilon\|_{L^2(0,T;L^2(\Omega_\varepsilon))}\|\eta\|_{L^2(0,T;L^2(\Gamma))} \\
      &\leq c\varepsilon^{\alpha/2}(1+T)^{1/2}\|\eta\|_{L^2(0,T;L^2(\Gamma))}.
    \end{aligned}
  \end{align}
  Also, we use \eqref{E:BiS_Bdd} and \eqref{E:Diff_PMu} to get
  \begin{align} \label{Pf_PMuW:I2}
    \begin{aligned}
      |J_2| &\leq c\|M_\tau u^\varepsilon-v^\varepsilon\|_{L^2(0,T;H^1(\Gamma))}\|\eta\|_{L^2(0,T;H^1(\Gamma))} \\
      &\leq c\varepsilon(1+T)^{1/2}\|\eta\|_{L^2(0,T;H^1(\Gamma))}.
    \end{aligned}
  \end{align}
  To estimate $J_3$ we observe by
  \begin{multline*}
    b_g(M_\tau u^\varepsilon,M_\tau u^\varepsilon,\eta)-b_g(v^\varepsilon,v^\varepsilon,\eta) \\
    = b_g(M_\tau u^\varepsilon,M_\tau u^\varepsilon-v^\varepsilon,\eta)+b_g(M_\tau u^\varepsilon-v^\varepsilon,v^\varepsilon,\eta)
  \end{multline*}
  and \eqref{E:Tri_Surf} that
  \begin{align*}
    |b_g(M_\tau u^\varepsilon,M_\tau u^\varepsilon,\eta)-b_g(v^\varepsilon,v^\varepsilon,\eta)| \leq cK_1K_2\|\eta\|_{H^1(\Gamma)},
  \end{align*}
  where
  \begin{align*}
    K_1 &:= \|M_\tau u^\varepsilon\|_{L^2(\Gamma)}^{1/2}\|M_\tau u^\varepsilon\|_{H^1(\Gamma)}^{1/2}+\|v^\varepsilon\|_{L^2(\Gamma)}^{1/2}\|v^\varepsilon\|_{H^1(\Gamma)}^{1/2}, \\
    K_2 &:= \|M_\tau u^\varepsilon-v^\varepsilon\|_{L^2(\Gamma)}^{1/2}\|M_\tau u^\varepsilon-v^\varepsilon\|_{H^1(\Gamma)}^{1/2}.
  \end{align*}
  Moreover, we apply \eqref{E:HLT_Bound} to $v^\varepsilon$ and then use \eqref{E:AveT_Lp_Surf} and \eqref{E:AveT_Wmp_Surf} to get
  \begin{align*}
    K_1 \leq c\|M_\tau u^\varepsilon\|_{L^2(\Gamma)}^{1/2}\|M_\tau u^\varepsilon\|_{H^1(\Gamma)}^{1/2} \leq c\varepsilon^{-1/2}\|u^\varepsilon\|_{L^2(\Omega_\varepsilon)}^{1/2}\|u^\varepsilon\|_{H^1(\Omega_\varepsilon)}^{1/2}.
  \end{align*}
  We also deduce from \eqref{E:Rep_WHL_L2} and \eqref{E:Rep_WHL_H1} that
  \begin{align*}
    K_2 \leq c\varepsilon^{1/2}\|u^\varepsilon\|_{L^2(\Omega_\varepsilon)}^{1/2}\|u^\varepsilon\|_{H^1(\Omega_\varepsilon)}^{1/2}.
  \end{align*}
  Hence it follows from the above inequalities that
  \begin{align*}
    |b_g(M_\tau u^\varepsilon,M_\tau u^\varepsilon,\eta)-b_g(v^\varepsilon,v^\varepsilon,\eta)| \leq c\|u^\varepsilon\|_{L^2(\Omega_\varepsilon)}\|u^\varepsilon\|_{H^1(\Omega_\varepsilon)}\|\eta\|_{H^1(\Gamma)}
  \end{align*}
  and we use this inequality and \eqref{Pf_EsSt:L2H1} to get
  \begin{align} \label{Pf_PMuW:I3}
    \begin{aligned}
      |J_3| &\leq c\left(\int_0^T\|u^\varepsilon\|_{L^2(\Omega_\varepsilon)}^2\|u^\varepsilon\|_{H^1(\Omega_\varepsilon)}^2\,dt\right)^{1/2}\|\eta\|_{L^2(0,T;H^1(\Gamma))} \\
      &\leq c\varepsilon(1+T)^{1/2}\|\eta\|_{L^2(0,T;H^1(\Gamma))}.
    \end{aligned}
  \end{align}
  Applying \eqref{Pf_PMuW:I1}--\eqref{Pf_PMuW:I3} and $\varepsilon\leq\varepsilon^{\alpha/2}$ to \eqref{Pf_PMuW:Res} we obtain \eqref{E:PMu_Weak_Re}.
\end{proof}

Let us derive the energy estimate for $v^\varepsilon$ from \eqref{E:PMu_Weak}.

\begin{lemma} \label{L:PMu_Energy}
  Let $u^\varepsilon$ be as in Lemma \ref{L:Est_Ueps} and $v^\varepsilon=\mathbb{P}_gM_\tau u^\varepsilon$.
  Then
  \begin{align} \label{E:PMu_Energy}
    \max_{t\in[0,T]}\|v^\varepsilon(t)\|_{L^2(\Gamma)}^2+\int_0^T\|\nabla_\Gamma v^\varepsilon(t)\|_{L^2(\Gamma)}^2\,dt \leq c_T
  \end{align}
  for all $T>0$, where $c_T>0$ is a constant depending on $T$ but independent of $\varepsilon$.
\end{lemma}

\begin{proof}
  For $t\in[0,T]$ let $1_{[0,t]}\colon\mathbb{R}\to\mathbb{R}$ be the characteristic function of $[0,t]\subset\mathbb{R}$.
  Since $v^\varepsilon$ belongs to $C([0,\infty);\mathcal{V}_g)$, we can substitute
  \begin{align*}
    \eta := 1_{[0,t]}v^\varepsilon \in L^2(0,T;\mathcal{V}_g)
  \end{align*}
  for \eqref{E:PMu_Weak}.
  Then using \eqref{E:TriS_Vg} we get
  \begin{multline} \label{Pf_PMuE:Weak}
    \int_0^t\{(g\partial_sv^\varepsilon,v^\varepsilon)_{L^2(\Gamma)}+a_g(v^\varepsilon,v^\varepsilon)\}\,ds \\
    = \int_0^t(gM_\tau\mathbb{P}_\varepsilon f^\varepsilon,v^\varepsilon)_{L^2(\Gamma)}\,ds+R_\varepsilon^1(v^\varepsilon)+R_\varepsilon^2(v^\varepsilon),
  \end{multline}
  where $R_\varepsilon^1(v^\varepsilon)$ and $R_\varepsilon^2(v^\varepsilon)$ satisfy \eqref{E:Mu_Weak_Re} and \eqref{E:PMu_Weak_Re}.
  We compute each term of \eqref{Pf_PMuE:Weak}.
  Since $g$ is independent of time and nonnegative by \eqref{E:G_Inf},
  \begin{align} \label{Pf_PMuE:Dt}
    \begin{aligned}
      \int_0^t(g\partial_sv^\varepsilon,v^\varepsilon)_{L^2(\Gamma)}\,ds &= \frac{1}{2}\int_0^t\frac{d}{ds}\|g^{1/2}v^\varepsilon\|_{L^2(\Gamma)}^2\,ds \\
      &= \frac{1}{2}\|g^{1/2}v^\varepsilon(t)\|_{L^2(\Gamma)}^2-\frac{1}{2}\|g^{1/2}v^\varepsilon(0)\|_{L^2(\Gamma)}^2.
    \end{aligned}
  \end{align}
  Also, we see by \eqref{E:Bi_Surf} that
  \begin{align} \label{Pf_PMuE:A}
    \int_0^t\|\nabla_\Gamma v^\varepsilon\|_{L^2(\Gamma)}^2\,ds \leq c\int_0^t\left\{a_g(v^\varepsilon,v^\varepsilon)+\|v^\varepsilon\|_{L^2(\Gamma)}^2\right\}\,ds.
  \end{align}
  We consider $M_\tau\mathbb{P}_\varepsilon f^\varepsilon=PM_\tau\mathbb{P}_\varepsilon f^\varepsilon$ in $H^{-1}(\Gamma,T\Gamma)$ (see Section \ref{SS:Pre_Surf}) to get
  \begin{align} \label{Pf_PMuE:F}
    \begin{aligned}
      \left|\int_0^t(gM_\tau\mathbb{P}_\varepsilon f^\varepsilon,v^\varepsilon)_{L^2(\Gamma)}\,ds\right| &= \left|\int_0^t[M_\tau\mathbb{P}_\varepsilon f^\varepsilon,gv^\varepsilon]_{T\Gamma}\,ds\right| \\
      &\leq \int_0^t\|M_\tau\mathbb{P}_\varepsilon f^\varepsilon\|_{H^{-1}(\Gamma,T\Gamma)}\|gv^\varepsilon\|_{H^1(\Gamma)}\,ds \\
      &\leq c\int_0^t\|M_\tau\mathbb{P}_\varepsilon f^\varepsilon\|_{H^{-1}(\Gamma,T\Gamma)}\|v^\varepsilon\|_{H^1(\Gamma)}\,ds.
    \end{aligned}
  \end{align}
  To estimate the residual terms, we see that $\varepsilon^{-1}\gamma_\varepsilon^0$ and $\varepsilon^{-1}\gamma_\varepsilon^1$ are bounded by \eqref{E:Fric_Upper}.
  Hence by \eqref{E:Mu_Weak_Re} and \eqref{E:PMu_Weak_Re} (with $T$ replaced by $t$) we obtain
  \begin{align} \label{Pf_PMuE:Re}
    |R_\varepsilon^1(v^\varepsilon)|+|R_\varepsilon^2(v^\varepsilon)| \leq c(1+t)^{1/2}\left(\int_0^t\|v^\varepsilon\|_{H^1(\Gamma)}^2\,ds\right)^{1/2}.
  \end{align}
  Now we deduce from \eqref{Pf_PMuE:Weak}--\eqref{Pf_PMuE:Re} that
  \begin{multline*}
    \begin{aligned}
      &\|g^{1/2}v^\varepsilon(t)\|_{L^2(\Gamma)}^2+\int_0^t\|\nabla_\Gamma v^\varepsilon\|_{L^2(\Gamma)}^2\,ds \\
      &\qquad \leq c\,\Biggl\{\|g^{1/2}v^\varepsilon(0)\|_{L^2(\Gamma)}^2+\int_0^t\|v^\varepsilon\|_{L^2(\Gamma)}^2\,ds\Biggr.
    \end{aligned} \\
    +\int_0^t\|M_\tau\mathbb{P}_\varepsilon f^\varepsilon\|_{H^{-1}(\Gamma,T\Gamma)}\|v^\varepsilon\|_{H^1(\Gamma)}\,ds \\
    \Biggl.+(1+t)^{1/2}\left(\int_0^t\|v^\varepsilon\|_{H^1(\Gamma)}^2\,ds\right)^{1/2}\Biggr\}.
  \end{multline*}
  Noting that
  \begin{align*}
    \|v^\varepsilon\|_{H^1(\Gamma)}^2=\|v^\varepsilon\|_{L^2(\Gamma)}^2+\|\nabla_\Gamma v^\varepsilon\|_{L^2(\Gamma)}^2,
  \end{align*}
  we apply Young's inequality to the last two terms of this inequality to get
  \begin{multline*}
    \|g^{1/2}v^\varepsilon(t)\|_{L^2(\Gamma)}^2+\int_0^t\|\nabla_\Gamma v^\varepsilon\|_{L^2(\Gamma)}^2\,ds \\
    \leq c\left\{\|g^{1/2}v^\varepsilon(0)\|_{L^2(\Gamma)}^2+\int_0^t\left(\|v^\varepsilon\|_{L^2(\Gamma)}^2+\|M_\tau\mathbb{P}_\varepsilon f^\varepsilon\|_{H^{-1}(\Gamma,T\Gamma)}^2\right)ds+1+t\right\}\\
    +\frac{1}{2}\int_0^t\|\nabla_\Gamma v^\varepsilon\|_{L^2(\Gamma)}^2\,ds.
  \end{multline*}
  Then we make the last term absorbed into the left-hand side and apply \eqref{E:G_Inf},
  \begin{align*}
    \|g^{1/2}v^\varepsilon(0)\|_{L^2(\Gamma)} \leq c\|v^\varepsilon(0)\|_{L^2(\Gamma)} \leq c\|M_\tau u^\varepsilon(0)\|_{L^2(\Gamma)} = c\|M_\tau u_0^\varepsilon\|_{L^2(\Gamma)}
  \end{align*}
  by \eqref{E:HLT_Bound} with $k=0$ and $u^\varepsilon(0)=u_0^\varepsilon$ in $\mathcal{V}_\varepsilon$, and \eqref{E:UE_Data} with $\beta=1$ to obtain
  \begin{align} \label{Pf_PMuE:Gron}
    \|v^\varepsilon(t)\|_{L^2(\Gamma)}^2+\int_0^t\|\nabla_\Gamma v^\varepsilon\|_{L^2(\Gamma)}^2\,ds \leq c\left(1+t+\int_0^t\|v^\varepsilon\|_{L^2(\Gamma)}^2\,ds\right)
  \end{align}
  for all $t\in[0,T]$.
  From this inequality we deduce that
  \begin{align*}
    \|v^\varepsilon(t)\|_{L^2(\Gamma)}^2+1 \leq c\left\{1+\int_0^t\left(\|v^\varepsilon\|_{L^2(\Gamma)}^2+1\right)ds\right\}, \quad t\in[0,T]
  \end{align*}
  and thus Gronwall's inequality yields
  \begin{align*}
    \|v^\varepsilon(t)\|_{L^2(\Gamma)}^2+1 \leq ce^{ct} \leq ce^{cT}, \quad t\in[0,T].
  \end{align*}
  Applying this inequality to \eqref{Pf_PMuE:Gron} with $t=T$ we also get
  \begin{align*}
    \int_0^T\|\nabla_\Gamma v^\varepsilon\|_{L^2(\Gamma)}^2\,dt \leq c\left\{1+\int_0^T\left(\|v^\varepsilon\|_{L^2(\Gamma)}^2+1\right)dt\right\} \leq c(1+e^{cT}).
  \end{align*}
  Hence we conclude that \eqref{E:PMu_Energy} holds with $c_T:=c(1+e^{cT})$, where $c>0$ is a constant independent of $\varepsilon$ and $T$.
\end{proof}

By \eqref{E:Diff_PMu} and \eqref{E:PMu_Energy} we obtain the following energy estimate for $M_\tau u^\varepsilon$.

\begin{lemma} \label{L:Mu_Energy}
  Let $u^\varepsilon$ be as in Lemma \ref{L:Est_Ueps}.
  Then
  \begin{align} \label{E:Mu_Energy}
    \max_{t\in[0,T]}\|M_\tau u^\varepsilon(t)\|_{L^2(\Gamma)}^2+\int_0^T\|\nabla_\Gamma M_\tau u^\varepsilon(t)\|_{L^2(\Gamma)}^2\,dt \leq c_T
  \end{align}
  for all $T>0$, where $c_T>0$ is a constant depending on $T$ but independent of $\varepsilon$.
\end{lemma}

\subsection{Estimate for the time derivative of the average} \label{SS:SL_EDt}
By \eqref{E:Mu_Energy} we see that $M_\tau u^\varepsilon$ converges weakly in appropriate function spaces on $\Gamma$ as $\varepsilon\to0$.
However, for the convergence of the trilinear term in \eqref{E:Mu_Weak} we need to apply the Aubin--Lions lemma to get the strong convergence of $M_\tau u^\varepsilon$.
To this end, we estimate the time derivative of $M_\tau u^\varepsilon$.
We first construct an appropriate test function.

\begin{lemma} \label{L:Mu_Dt_Test}
  For $w\in H^1(\Gamma,T\Gamma)$ there exist unique $\eta\in\mathcal{V}_g$, $q\in H^2(\Gamma)$ such that
  \begin{align} \label{E:MuDT_Eq}
    w = g\eta+g\nabla_\Gamma q \quad\text{on}\quad \Gamma, \quad \int_\Gamma q\,d\mathcal{H}^2 = 0.
  \end{align}
  Moreover, there exists a constant $c>0$ independent of $w$ such that
  \begin{align} \label{E:MuDT_Bd}
    \|\eta\|_{H^1(\Gamma)}+\|q\|_{H^2(\Gamma)} \leq c\|w\|_{H^1(\Gamma)}.
  \end{align}
\end{lemma}

\begin{proof}
  Let $w\in H^1(\Gamma,T\Gamma)$ and $\xi:=-\mathrm{div}_\Gamma w\in L^2(\Gamma)$.
  Then
  \begin{align*}
    (\xi,1)_{L^2(\Gamma)} = -\int_\Gamma\mathrm{div}_\Gamma w\,d\mathcal{H}^2 = 0
  \end{align*}
  by \eqref{E:IbP_WDivG_T}.
  Also, for all $q\in H^1(\Gamma)$ satisfying $\int_\Gamma q\,d\mathcal{H}^2=0$ we have
  \begin{align} \label{Pf_MuDT:Poin}
    \|q\|_{L^2(\Gamma)} \leq c\|\nabla_\Gamma q\|_{L^2(\Gamma)} \leq c\|g^{1/2}\nabla_\Gamma q\|_{L^2(\Gamma)}
  \end{align}
  by \eqref{E:G_Inf} and Poincar\'{e}'s inequality \eqref{E:Poin_Surf_Lp}.
  Hence by the Lax--Milgram theorem there exists a unique weak solution $q\in H^1(\Gamma)$ to the problem
  \begin{align*}
    -\mathrm{div}_\Gamma(g\nabla_\Gamma q) = \xi \quad\text{on}\quad \Gamma, \quad \int_\Gamma q\,d\mathcal{H}^2 = 0
  \end{align*}
  in the sense that
  \begin{align} \label{Pf_MuDT:Weak}
    (g\nabla_\Gamma q,\nabla_\Gamma\varphi)_{L^2(\Gamma)} = (\xi,\varphi)_{L^2(\Gamma)} \quad\text{for all}\quad \varphi\in H^1(\Gamma).
  \end{align}
  From this equality with $\varphi= q$ and \eqref{Pf_MuDT:Poin} we deduce that
  \begin{align} \label{Pf_MuDT:Q_H1}
    \|q\|_{H^1(\Gamma)} \leq c\|\xi\|_{L^2(\Gamma)} = c\|\mathrm{div}_\Gamma w\|_{L^2(\Gamma)} \leq c\|w\|_{H^1(\Gamma)}.
  \end{align}
  Moreover, replacing $\varphi$ by $g^{-1}\varphi$ in \eqref{Pf_MuDT:Weak} we get
  \begin{align*}
    (\nabla_\Gamma q,\nabla_\Gamma\varphi)_{L^2(\Gamma)} = \left(\frac{1}{g}(\xi+\nabla_\Gamma g\cdot\nabla_\Gamma q),\varphi\right)_{L^2(\Gamma)} \quad\text{for all}\quad \varphi\in H^1(\Gamma),
  \end{align*}
  which combined with \eqref{E:Poin_Surf_Lp} shows that $q$ is a unique weak solution to
  \begin{align*}
    -\Delta_\Gamma\psi = \zeta := \frac{1}{g}(\xi+\nabla_\Gamma g\cdot\nabla_\Gamma q) \in L^2(\Gamma), \quad \int_\Gamma\psi\,d\mathcal{H}^2 = 0.
  \end{align*}
  Note that $(\zeta,1)_{L^2(\Gamma)}=0$ by \eqref{Pf_MuDT:Weak} with $\varphi=g^{-1}$.
  Hence $q\in H^2(\Gamma)$ by Lemma \ref{L:Pois_Surf}.
  Moreover, it follows from \eqref{E:G_Inf}, \eqref{E:H2_Pois}, and \eqref{Pf_MuDT:Q_H1} that
  \begin{align} \label{Pf_MuDT:Q_H2}
    \|q\|_{H^2(\Gamma)} \leq c\|\zeta\|_{L^2(\Gamma)} \leq c\left(\|\xi\|_{L^2(\Gamma)}+\|\nabla_\Gamma q\|_{L^2(\Gamma)}\right) \leq c\|w\|_{H^1(\Gamma)}.
  \end{align}
  Now we observe by $q\in H^2(\Gamma)$ and $\mathrm{div}_\Gamma(g\nabla_\Gamma q)=-\xi=\mathrm{div}_\Gamma w$ on $\Gamma$ that
  \begin{align*}
    \eta := \frac{1}{g}w-\nabla_\Gamma q \in \mathcal{V}_g
  \end{align*}
  and that $\eta$ and $q$ satisfy \eqref{E:MuDT_Eq}.
  We also deduce from \eqref{E:G_Inf} and \eqref{Pf_MuDT:Q_H2} that
  \begin{align*}
    \|\eta\|_{H^1(\Gamma)} \leq c\left(\|w\|_{H^1(\Gamma)}+\|\nabla_\Gamma q\|_{H^1(\Gamma)}\right) \leq c\|w\|_{H^1(\Gamma)}.
  \end{align*}
  By this inequality and \eqref{Pf_MuDT:Q_H2} we obtain \eqref{E:MuDT_Bd}.

  To prove the uniqueness of $\eta$ and $q$ satisfying \eqref{E:MuDT_Eq}, suppose that
  \begin{align*}
    g\eta+g\nabla_\Gamma q = 0 \quad\text{on}\quad \Gamma, \quad \int_\Gamma q\,d\mathcal{H}^2 = 0
  \end{align*}
  for $\eta\in\mathcal{V}_g$ and $q\in H^2(\Gamma)$.
  Then
  \begin{align*}
    \|g^{1/2}\nabla_\Gamma q\|_{L^2(\Gamma)}^2 &= (g\nabla_\Gamma q,\nabla_\Gamma q)_{L^2(\Gamma)} = -(g\eta,\nabla_\Gamma q)_{L^2(\Gamma)} \\
    &= -(\eta,g\nabla_\Gamma q)_{L^2(\Gamma)} = 0
  \end{align*}
  by $\eta\in\mathcal{H}_g$ and $g\nabla_\Gamma q\in\mathcal{H}_g^\perp$ (see Lemma \ref{L:L2gs_Orth}) and thus $q=0$ on $\Gamma$ by \eqref{Pf_MuDT:Poin}.
  By this fact and \eqref{E:G_Inf} we also get $\eta=0$ on $\Gamma$.
  Thus the uniqueness holds.
\end{proof}

By the proof of Lemma \ref{L:Mu_Dt_Test} we observe that the mapping
\begin{align*}
  H^1(\Gamma,T\Gamma)\ni w \mapsto (\eta,q)\in \mathcal{V}_g\times H^2(\Gamma),
\end{align*}
where $\eta$ and $q$ are unique functions satisfying \eqref{E:MuDT_Eq}, is linear and bounded.
Hence we have the following time-dependent version of Lemma \ref{L:Mu_Dt_Test}.

\begin{lemma} \label{L:MuDT_Time}
  For $w\in\mathcal{X}(0,T;H^1(\Gamma,T\Gamma))$ with $\mathcal{X}=C_c,L^2$ there exist unique
  \begin{align*}
    \eta\in\mathcal{X}(0,T;\mathcal{V}_g), \quad q\in\mathcal{X}(0,T;H^2(\Gamma))
  \end{align*}
  such that, for all (or a.a.) $t\in(0,T)$,
  \begin{align*}
    w(t) = g\eta(t)+g\nabla_\Gamma q(t) \quad\text{on}\quad \Gamma, \quad \int_\Gamma q(t)\,d\mathcal{H}^2 = 0.
  \end{align*}
  Moreover, there exists a constant $c>0$ independent of $w$ such that
  \begin{align*}
    \|\eta(t)\|_{H^1(\Gamma)}+\|q(t)\|_{H^2(\Gamma)} \leq c\|w(t)\|_{H^1(\Gamma)} \quad\text{for all (or a.a.)}\quad t\in(0,T).
  \end{align*}
\end{lemma}

As in Section \ref{SS:SL_Ener}, we derive an estimate for the time derivative of $v^\varepsilon=\mathbb{P}_gM_\tau u^\varepsilon$ and then use it to estimate the time derivative of $M_\tau u^\varepsilon$.

\begin{lemma} \label{L:PMu_Dt}
  Let $u^\varepsilon$ be as in Lemma \ref{L:Est_Ueps} and $v^\varepsilon=\mathbb{P}_gM_\tau u^\varepsilon$.
  Then
  \begin{align} \label{E:PMu_Dt}
    \|\partial_tv^\varepsilon\|_{L^2(0,T;H^{-1}(\Gamma,T\Gamma))} \leq c_T
  \end{align}
  for all $T>0$, where $c_T>0$ is a constant depending on $T$ but independent of $\varepsilon$.
\end{lemma}

Note that we estimate $\partial_tv^\varepsilon$ in the dual space $H^{-1}(\Gamma,T\Gamma)$ of $H^1(\Gamma,T\Gamma)$, not in the dual space of the weighted solenoidal space $\mathcal{V}_g$ (see Remark \ref{R:Mu_Dt} below).

\begin{proof}
  Let $w\in L^2(0,T;H^1(\Gamma,T\Gamma))$.
  Then by Lemma \ref{L:MuDT_Time} there exist
  \begin{align*}
    \eta\in L^2(0,T;\mathcal{V}_g), \quad q\in L^2(0,T;H^2(\Gamma))
  \end{align*}
  such that $w(t)=g\eta(t)+g\nabla_\Gamma q(t)$ on $\Gamma$ for a.a. $t\in(0,T)$ and
  \begin{align} \label{Pf_PMuDt:Test}
    \|\eta\|_{L^2(0,T;H^1(\Gamma))} \leq c\|w\|_{L^2(0,T;H^1(\Gamma))}.
  \end{align}
  Since $\partial_tv^\varepsilon(t)\in\mathcal{H}_g$ and $g\nabla_\Gamma q(t)\in\mathcal{H}_g^\perp$ for a.a. $t\in(0,T)$ by Lemmas \ref{L:L2gs_Orth} and \ref{L:PMu_Weak},
  \begin{align*}
    \int_0^T(\partial_tv^\varepsilon,g\nabla_\Gamma q)_{L^2(\Gamma)}\,dt = 0.
  \end{align*}
  By this equality and $g\eta=w-g\nabla_\Gamma q$ we have
  \begin{align*}
    \int_0^T(g\partial_tv^\varepsilon,\eta)_{L^2(\Gamma)}\,dt = \int_0^T(\partial_tv^\varepsilon,g\eta)_{L^2(\Gamma)}\,dt = \int_0^T(\partial_tv^\varepsilon,w)_{L^2(\Gamma)}\,dt.
  \end{align*}
  We substitute $\eta=g^{-1}w-\nabla_\Gamma q$ for \eqref{E:PMu_Weak} and use the above equality.
  Then
  \begin{multline} \label{Pf_PMuDt:Eq}
    \int_0^T(\partial_tv^\varepsilon,w)_{L^2(\Gamma)}\,dt = -\int_0^Ta_g(v^\varepsilon,\eta)\,dt-\int_0^Tb_g(v^\varepsilon,v^\varepsilon,\eta)\,dt \\
    +\int_0^T(gM_\tau\mathbb{P}_\varepsilon f^\varepsilon,\eta)_{L^2(\Gamma)}\,dt+R_\varepsilon^1(\eta)+R_\varepsilon^2(\eta),
  \end{multline}
  where $R_\varepsilon^1(\eta)$ and $R_\varepsilon^2(\eta)$ are given in Lemmas \ref{L:Mu_Weak} and \ref{L:PMu_Weak}.
  To the first term on the right-hand side we apply \eqref{E:BiS_Bdd}, \eqref{E:PMu_Energy}, and \eqref{Pf_PMuDt:Test} to get
  \begin{align*}
    \left|\int_0^Ta_g(v^\varepsilon,\eta)\,dt\right| \leq c\|v^\varepsilon\|_{L^2(0,T;H^1(\Gamma))}\|\eta\|_{L^2(0,T;H^1(\Gamma))} \leq c_T\|w\|_{L^2(0,T;H^1(\Gamma))}.
  \end{align*}
  Here and in what follows we denote by $c_T$ a general positive constant depending on $T$ but independent of $\varepsilon$.
  Also, by \eqref{E:Tri_Surf}, \eqref{E:PMu_Energy}, and \eqref{Pf_PMuDt:Test},
  \begin{align*}
    \left|\int_0^Tb_g(v^\varepsilon,v^\varepsilon,\eta)\,dt\right| &\leq c\int_0^T\|v^\varepsilon\|_{L^2(\Gamma)}\|v^\varepsilon\|_{H^1(\Gamma)}\|\eta\|_{H^1(\Gamma)}\,dt \\
    &\leq c\|v^\varepsilon\|_{L^\infty(0,T;L^2(\Gamma))}\|v^\varepsilon\|_{L^2(0,T;H^1(\Gamma))}\|\eta\|_{L^2(0,T;H^1(\Gamma))} \\
    &\leq c_T\|w\|_{L^2(0,T;H^1(\Gamma))}.
  \end{align*}
  For the other terms we proceed as in the proof of Lemma \ref{L:PMu_Energy} (see \eqref{Pf_PMuE:F}--\eqref{Pf_PMuE:Re}) and use \eqref{E:UE_Data} with $\beta=1$ and \eqref{Pf_PMuDt:Test}.
  Then we get
  \begin{align*}
    \left|\int_0^T(gM_\tau\mathbb{P}_\varepsilon f^\varepsilon,\eta)_{L^2(\Gamma)}\,dt\right| &\leq c\int_0^T\|M_\tau\mathbb{P}_\varepsilon f^\varepsilon\|_{H^{-1}(\Gamma,T\Gamma)}\|\eta\|_{H^1(\Gamma)}\,dt \\
    &\leq cT^{1/2}\|\eta\|_{L^2(0,T;H^1(\Gamma))} \\
    &\leq cT^{1/2}\|w\|_{L^2(0,T;H^1(\Gamma))}
  \end{align*}
  and
  \begin{align*}
    |R_\varepsilon^1(\eta)|+|R_\varepsilon^2(\eta)| &\leq c(1+T)^{1/2}\|\eta\|_{L^2(0,T;H^1(\Gamma))} \\
    &\leq c(1+T)^{1/2}\|w\|_{L^2(0,T;H^1(\Gamma))}.
  \end{align*}
  Applying these inequalities to the right-hand side of \eqref{Pf_PMuDt:Eq} we obtain
  \begin{align*}
    \left|\int_0^T(\partial_tv^\varepsilon,w)_{L^2(\Gamma)}\,dt\right| \leq c_T\|w\|_{L^2(0,T;H^1(\Gamma))}
  \end{align*}
  for all $w\in L^2(0,T;H^1(\Gamma,T\Gamma))$.
  Hence \eqref{E:PMu_Dt} holds.
\end{proof}

\begin{lemma} \label{L:Mu_Dt}
  Let $u^\varepsilon$ be as in Lemma \ref{L:Est_Ueps}.
  Then
  \begin{align} \label{E:Mu_Dt}
    \|\partial_tM_\tau u^\varepsilon\|_{L^2(0,T;H^{-1}(\Gamma,T\Gamma))} \leq c_T
  \end{align}
  for all $T>0$, where $c_T>0$ is a constant depending on $T$ but independent of $\varepsilon$.
\end{lemma}

\begin{proof}
  Let $v^\varepsilon=\mathbb{P}_gM_\tau u^\varepsilon$.
  Noting that
  \begin{align*}
    \partial_tM_\tau u^\varepsilon-\partial_tv^\varepsilon \in L^2(0,T;L^2(\Gamma,T\Gamma)) \subset L^2(0,T;H^{-1}(\Gamma,T\Gamma))
  \end{align*}
  by Lemmas \ref{L:Mu_Weak} and \ref{L:PMu_Weak} (see also Section \ref{SS:Pre_Surf}), we have
  \begin{align*}
    \|\partial_tM_\tau u^\varepsilon-\partial_tv^\varepsilon\|_{L^2(0,T;H^{-1}(\Gamma,T\Gamma))} &\leq \|\partial_tM_\tau u^\varepsilon-\partial_tv^\varepsilon\|_{L^2(0,T;L^2(\Gamma))} \\
    &\leq c\varepsilon^{1/2}\|\partial_tu^\varepsilon\|_{L^2(0,T;L^2(\Omega_\varepsilon))} \\
    &\leq c\varepsilon^{\alpha/2}(1+T)^{1/2}
  \end{align*}
  by \eqref{E:Est_DtUe} and \eqref{E:Rep_WHL_Dt}.
  This inequality, \eqref{E:PMu_Dt}, and $\varepsilon^{\alpha/2}\leq 1$ imply \eqref{E:Mu_Dt}.
\end{proof}

\begin{remark} \label{R:Mu_Dt}
  In construction of a weak solution to the Navier--Stokes equations, we usually estimate the time derivative of an approximate solution in the dual space of a solenoidal space.
  However, in Lemma \ref{L:PMu_Dt} we estimate $\partial_tv^\varepsilon$ in $H^{-1}(\Gamma,T\Gamma)$, not in the dual space $\mathcal{V}'_g$ of $\mathcal{V}_g$.
  This is due to the fact that we multiply $\partial_tv^\varepsilon$ by the function $g$ in \eqref{E:PMu_Weak}.
  When $f\in\mathcal{V}'_g$ we cannot define a functional $gf$ on $\mathcal{V}_g$ by
  \begin{align*}
    {}_{\mathcal{V}'_g}\langle gf,v \rangle_{\mathcal{V}_g} := {}_{\mathcal{V}'_g}\langle f,gv \rangle_{\mathcal{V}_g}, \quad v\in\mathcal{V}_g
  \end{align*}
  since $gv\not\in\mathcal{V}_g$ in general (here ${}_{\mathcal{V}'_g}\langle \cdot,\cdot \rangle_{\mathcal{V}_g}$ is the duality product between $\mathcal{V}'_g$ and $\mathcal{V}_g$).
  To avoid this problem, we consider $\partial_tv^\varepsilon$ and $\partial_tM_\tau u^\varepsilon$ in $H^{-1}(\Gamma,T\Gamma)$ (see \eqref{E:Def_Mul_HinT}).
\end{remark}

\subsection{Weak convergence of the average and characterization of the limit} \label{SS:SL_WeCh}
The goal of this subsection is to establish Theorem \ref{T:SL_Weak}.
We proceed as in the case of a bounded domain in $\mathbb{R}^2$ (see e.g. \cites{BoFa13,CoFo88,So01,Te79}).
First we give the definition of a weak solution to the limit equations \eqref{E:NS_Limit} based on \eqref{E:Mu_Weak}.

\begin{definition} \label{D:Lim_W_T}
  For $T>0$ and given data
  \begin{align*}
    v_0 \in \mathcal{H}_g, \quad f\in L^2(0,T;H^{-1}(\Gamma,T\Gamma)),
  \end{align*}
  we say that a vector field
  \begin{align*}
    v \in L^\infty(0,T;\mathcal{H}_g)\cap L^2(0,T;\mathcal{V}_g) \quad\text{with}\quad \partial_t v\in L^1(0,T;H^{-1}(\Gamma,T\Gamma))
  \end{align*}
  is a weak solution to the equations \eqref{E:NS_Limit} on $[0,T)$ if it satisfies
  \begin{align} \label{E:Limit_Weak}
    \int_0^T\{[g\partial_tv,\eta]_{T\Gamma}+a_g(v,\eta)+b_g(v,v,\eta)\}\,dt = \int_0^T[gf,\eta]_{T\Gamma}\,dt
  \end{align}
  for all $\eta\in C_c(0,T;\mathcal{V}_g)$ and $v|_{t=0}=v_0$ in $H^{-1}(\Gamma,T\Gamma)$.
\end{definition}

\begin{definition} \label{D:Lim_W_Inf}
  For given data
  \begin{align*}
    v_0 \in \mathcal{H}_g, \quad f\in L_{loc}^2([0,\infty);H^{-1}(\Gamma,T\Gamma)),
  \end{align*}
  we say that $v$ is a weak solution to \eqref{E:NS_Limit} on $[0,\infty)$ if it is a weak solution to \eqref{E:NS_Limit} on $[0,T)$ for all $T>0$.
\end{definition}

A weak solution $v$ to \eqref{E:NS_Limit} on $[0,T)$ with $T>0$ satisfies
\begin{align*}
  v\in W^{1,1}(0,T;H^{-1}(\Gamma,T\Gamma)) \subset C([0,T];H^{-1}(\Gamma,T\Gamma))
\end{align*}
and thus the initial condition $v|_{t=0}=v_0$ in $H^{-1}(\Gamma,T\Gamma)$ makes sense.
Moreover, it is continuous on $[0,T]$ with values in $\mathcal{H}_g$.

\begin{lemma} \label{L:Lim_W_L2}
  For $T>0$ let $f\in L^2(0,T;H^{-1}(\Gamma,T\Gamma))$.
  Suppose that
  \begin{align*}
    v \in L^\infty(0,T;\mathcal{H}_g)\cap L^2(0,T;\mathcal{V}_g)\quad\text{with}\quad \partial_tv\in L^1(0,T;H^{-1}(\Gamma,T\Gamma))
  \end{align*}
  satisfies \eqref{E:Limit_Weak} for all $\eta\in C_c(0,T;\mathcal{V}_g)$.
  Then
  \begin{align*}
    v \in C([0,T];\mathcal{H}_g), \quad \partial_tv\in L^2(0,T;H^{-1}(\Gamma,T\Gamma)),
  \end{align*}
  and \eqref{E:Limit_Weak} is valid for all $\eta\in L^2(0,T;\mathcal{V}_g)$.
\end{lemma}

Note that here the initial condition $v|_{t=0}=v_0$ in $H^{-1}(\Gamma,T\Gamma)$ is not imposed.

\begin{proof}
  We estimate $\partial_tv$ as in the proof of Lemma \ref{L:PMu_Dt}, where we used $\partial_tv^\varepsilon(t)\in \mathcal{H}_g$ for a.a. $t\in(0,T)$.
  This is not valid for $\partial_tv$, but we have
  \begin{align} \label{Pf_LWL2:Anni}
    \int_0^T[\partial_tv,g\nabla_\Gamma q]_{T\Gamma}\,dt = 0 \quad\text{for all}\quad q\in C_c(0,T;H^2(\Gamma)).
  \end{align}
  To see this, we first take functions $q_k\in C_c^\infty(0,T;H^2(\Gamma))$, $k\in\mathbb{N}$ such that
  \begin{align*}
    \lim_{k\to\infty}\|q-q_k\|_{L^\infty(0,T;H^2(\Gamma))} = 0
  \end{align*}
  by mollifying $q\in C_c(0,T;H^2(\Gamma))$ with respect to time.
  Then since
  \begin{multline*}
    \left|\int_0^T[\partial_tv,g\nabla_\Gamma q]_{T\Gamma}\,dt-\int_0^T[\partial_tv,g\nabla_\Gamma q_k]_{T\Gamma}\,dt\right| \\
    \begin{aligned}
      &\leq \|\partial_tv\|_{L^1(0,T;H^{-1}(\Gamma,T\Gamma))}\|g\nabla_\Gamma q-g\nabla_\Gamma q_k\|_{L^\infty(0,T;H^1(\Gamma))} \\
      &\leq c\|\partial_tv\|_{L^1(0,T;H^{-1}(\Gamma,T\Gamma))}\|q-q_k\|_{L^\infty(0,T;H^2(\Gamma))}
    \end{aligned}
  \end{multline*}
  and the last term converges to zero as $k\to\infty$, it is sufficient to show \eqref{Pf_LWL2:Anni} for $q_k$ instead of $q$.
  We suppress the subscript $k$ of $q_k$.
  Then since
  \begin{align*}
    q \in C_c^\infty(0,T;H^2(\Gamma)), \quad \partial_t(\nabla_\Gamma q) = \nabla_\Gamma(\partial_tq) \quad\text{in}\quad C_c^\infty(0,T;H^1(\Gamma,T\Gamma))
  \end{align*}
  and $g$ is independent of time, it follows that
  \begin{align*}
    \int_0^T[\partial_tv,g\nabla_\Gamma q]_{T\Gamma}\,dt = -\int_0^T[v,g\partial_t(\nabla_\Gamma q)]_{T\Gamma}\,dt = -\int_0^T[v,g\nabla_\Gamma(\partial_tq)]_{T\Gamma}\,dt.
  \end{align*}
  Moreover, since $v\in L^\infty(0,T;\mathcal{H}_g)$ and $q\in C_c^\infty(0,T;H^2(\Gamma))$, we have
  \begin{align*}
    v(t) \in \mathcal{H}_g \quad\text{for a.a.}\quad t\in(0,T), \quad g[\nabla_\Gamma(\partial_tq)](t) \in \mathcal{H}_g^\perp \quad\text{for all}\quad t\in(0,T)
  \end{align*}
  by Lemma \ref{L:L2gs_Orth} and thus (see also Section \ref{SS:Pre_Surf})
  \begin{align*}
    \int_0^T[v,g\nabla_\Gamma(\partial_tq)]_{T\Gamma}\,dt = \int_0^T\bigl(v,g\nabla_\Gamma(\partial_tq)\bigr)_{L^2(\Gamma)}\,dt = 0.
  \end{align*}
  Hence we obtain \eqref{Pf_LWL2:Anni} by the above two equalities.

  Now let $w\in C_c(0,T;H^1(\Gamma,T\Gamma))$.
  By Lemma \ref{L:MuDT_Time} we can take
  \begin{align*}
    \eta\in C_c(0,T;\mathcal{V}_g), \quad q\in C_c(0,T;H^2(\Gamma))
  \end{align*}
  such that $w(t)=g\eta(t)+g\nabla_\Gamma q(t)$ on $\Gamma$ for all $t\in(0,T)$ and \eqref{Pf_PMuDt:Test} holds.
  Then
  \begin{align*}
    \int_0^T[\partial_tv,w]_{T\Gamma}\,dt = \int_0^T\bigl([\partial_tv,g\eta]_{T\Gamma}+[\partial_tv,g\nabla_\Gamma q]_{T\Gamma}\bigr)\,dt = \int_0^T[g\partial_tv,\eta]_{T\Gamma}\,dt
  \end{align*}
  by \eqref{Pf_LWL2:Anni}.
  We substitute $\eta$ for \eqref{E:Limit_Weak}.
  Then using the above equality,
  \begin{align} \label{Pf_LWL2:AB}
    \begin{aligned}
      \left|\int_0^Ta_g(v,\eta)\,dt\right| &\leq c\|v\|_{L^2(0,T;H^1(\Gamma))}\|\eta\|_{L^2(0,T;H^1(\Gamma))}, \\
      \left|\int_0^Tb_g(v,v,\eta)\,dt\right| &\leq c\|v\|_{L^\infty(0,T;L^2(\Gamma))}\|v\|_{L^2(0,T;H^1(\Gamma))}\|\eta\|_{L^2(0,T;H^1(\Gamma))}
    \end{aligned}
  \end{align}
  by \eqref{E:BiS_Bdd} and \eqref{E:Tri_Surf}, the assumption on $f$ (see also \eqref{E:Def_Mul_HinT}), and \eqref{Pf_PMuDt:Test} we calculate as in the proof of Lemma \ref{L:PMu_Dt} to get
  \begin{align*}
    \left|\int_0^T[\partial_tv,w]_{T\Gamma}\,dt\right| \leq c\|w\|_{L^2(0,T;H^1(\Gamma))} \quad\text{for all}\quad w\in C_c(0,T;H^1(\Gamma,T\Gamma)).
  \end{align*}
  Since $C_c(0,T;H^1(\Gamma,T\Gamma))$ is dense in $L^2(0,T;H^1(\Gamma,T\Gamma))$, this inequality implies
  \begin{align} \label{Pf_LWL2:Dt}
    \partial_tv\in L^2(0,T;H^{-1}(\Gamma,T\Gamma)).
  \end{align}
  By this property and
  \begin{align*}
    v\in L^2(0,T;\mathcal{V}_g) \subset L^2(0,T;H^1(\Gamma,T\Gamma))
  \end{align*}
  we can apply the interpolation result of Lions--Magenes \cite{LiMa72}*{Chapter 1, Theorem 3.1} (see also \cite{Te79}*{Chapter III, Lemma 1.2}) to $v$ to get
  \begin{align*}
    v\in C([0,T];L^2(\Gamma,T\Gamma)).
  \end{align*}
  Moreover, since $v\in L^\infty(0,T;\mathcal{H}_g)$, we have $v(t)\in\mathcal{H}_g$ for a.a. $t\in(0,T)$ and, in particular, for all $t$ in a dense subset of $[0,T]$.
  Hence, by the continuity of $v(t)$ on $[0,T]$ in $L^2(\Gamma,T\Gamma)$ and the fact that $\mathcal{H}_g$ is closed in $L^2(\Gamma,T\Gamma)$,
  \begin{align*}
    v(t) \in \mathcal{H}_g \quad\text{for all}\quad t\in[0,T], \quad\text{i.e.}\quad v\in C([0,T];\mathcal{H}_g).
  \end{align*}
  Finally, since $C_c(0,T;\mathcal{V}_g)$ is dense in $L^2(0,T;\mathcal{V}_g)$ and both sides of \eqref{E:Limit_Weak} are linear and continuous for $\eta\in L^2(0,T;\mathcal{V}_g)$ by \eqref{Pf_LWL2:AB}, \eqref{Pf_LWL2:Dt}, and the assumption on $f$, the equality \eqref{E:Limit_Weak} is also valid for all $\eta\in L^2(0,T;\mathcal{V}_g)$.
\end{proof}

By Lemma \ref{L:Lim_W_L2} the initial condition for a weak solution to \eqref{E:NS_Limit} makes sense in $\mathcal{H}_g$.
Let us prove the uniqueness of a weak solution.

\begin{lemma} \label{L:Limit_Uni}
  For $T>0$ and given data
  \begin{align*}
    v_0 \in \mathcal{H}_g, \quad f\in L^2(0,T;H^{-1}(\Gamma,T\Gamma)),
  \end{align*}
  there exists at most one weak solution to \eqref{E:NS_Limit} on $[0,T)$.
\end{lemma}

\begin{proof}
  Let $v_1$ and $v_2$ be weak solutions to \eqref{E:NS_Limit} and $w:=v_1-v_2$.
  Then
  \begin{align} \label{Pf_LU:Class}
    w\in C([0,T];\mathcal{H}_g)\cap L^2(0,T;\mathcal{V}_g), \quad \partial_tw \in L^2(0,T;H^{-1}(\Gamma,T\Gamma)),
  \end{align}
  and $w|_{t=0}=0$ in $\mathcal{H}_g$ by Lemma \ref{L:Lim_W_L2}.
  Moreover, by the same lemma we can subtract the equality \eqref{E:Limit_Weak} for $\eta\in L^2(0,T;\mathcal{V}_g)$ with $v=v_2$ from that with $v=v_1$ to get
  \begin{align} \label{Pf_LU:Weak}
    \int_0^T\{[g\partial_sw,\eta]_{T\Gamma}+a_g(w,\eta)+b_g(w,v_1,\eta)+b_g(v_2,w,\eta)\}\,ds = 0.
  \end{align}
  For $t\in[0,T]$ let $1_{[0,t]}$ be the characteristic function of $[0,t]\subset\mathbb{R}$ and
  \begin{align*}
    \eta := 1_{[0,t]}w\in L^2(0,T;\mathcal{V}_g).
  \end{align*}
  We substitute this $\eta$ for \eqref{Pf_LU:Weak} and use \eqref{E:Bi_Surf},
  \begin{align} \label{Pf_LU:Int_Dt}
    \begin{aligned}
      \int_0^t[g\partial_sw,w]_{T\Gamma}\,ds &= \frac{1}{2}\int_0^t\frac{d}{ds}\|g^{1/2}w\|_{L^2(\Gamma)}^2\,ds \\
      &= \frac{1}{2}\|g^{1/2}w(t)\|_{L^2(\Gamma)}^2-\frac{1}{2}\|g^{1/2}w(0)\|_{L^2(\Gamma)}^2 \\
      &\geq c\left(\|w(t)\|_{L^2(\Gamma)}^2-\|w(0)\|_{L^2(\Gamma)}^2\right)
    \end{aligned}
  \end{align}
  by \eqref{Pf_LU:Class} and the fact that $g$ is bounded on $\Gamma$ and satisfies \eqref{E:G_Inf}, and
  \begin{align*}
    |b_g(w,v_1,w)| = |b_g(w,w,v_1)| \leq c\|w\|_{L^2(\Gamma)}\|w\|_{H^1(\Gamma)}\|v_1\|_{H^1(\Gamma)}
  \end{align*}
  and $b_g(v_2,w,w)=0$ by $v_2,w\in\mathcal{V}_g$ and \eqref{E:Tri_Surf}--\eqref{E:TriS_Vg} to obtain
  \begin{multline*}
    \|w(t)\|_{L^2(\Gamma)}^2+\int_0^t\|\nabla_\Gamma w\|_{L^2(\Gamma)}^2\,ds \\
    \leq c\left\{\|w(0)\|_{L^2(\Gamma)}^2+\int_0^t\left(\|w\|_{L^2(\Gamma)}^2+\|w\|_{L^2(\Gamma)}\|w\|_{H^1(\Gamma)}\|v_1\|_{H^1(\Gamma)}\right)ds\right\}.
  \end{multline*}
  We further apply Young's inequality to the last term to deduce that
  \begin{multline*}
    \|w(t)\|_{L^2(\Gamma)}^2+\int_0^t\|\nabla_\Gamma w\|_{L^2(\Gamma)}^2\,ds \\
    \leq c\left\{\|w(0)\|_{L^2(\Gamma)}^2+\int_0^t\left(1+\|v_1\|_{H^1(\Gamma)}^2\right)\|w\|_{L^2(\Gamma)}^2\,ds\right\} \\
    +\frac{1}{2}\int_0^t\|\nabla_\Gamma w\|_{L^2(\Gamma)}^2\,ds.
  \end{multline*}
  Then we make the last term absorbed into the left-hand side and use $w|_{t=0}=0$ in $\mathcal{H}_g$ to get (we omit the time integral of $\|\nabla_\Gamma w\|_{L^2(\Gamma)}^2$ on the left-hand side)
  \begin{align*}
    \|w(t)\|_{L^2(\Gamma)}^2 \leq c\int_0^t\xi\|w\|_{L^2(\Gamma)}^2\,ds, \quad \xi(t) := 1+\|v_1(t)\|_{H^1(\Gamma)}^2 \quad\text{for all}\quad t\in[0,T].
  \end{align*}
  Since $\xi\in L^1(0,T)$, we can apply Gronwall's inequality to this inequality to obtain
  \begin{align*}
    \|w(t)\|_{L^2(\Gamma)}^2 = 0 \quad\text{for all}\quad t\in[0,T].
  \end{align*}
  Hence $v_1=v_2$ and the uniqueness of a weak solution to \eqref{E:NS_Limit} holds.
\end{proof}

Next we construct an associated pressure in \eqref{E:NS_Limit} from \eqref{E:Limit_Weak} after giving two auxiliary results.
Recall that we identity $H^{-1}(\Gamma,T\Gamma)$ with the quotient space
\begin{align*}
  \mathcal{Q} = \{[f]\mid f\in H^{-1}(\Gamma)^3\}, \quad [f] = \{\tilde{f}\in H^{-1}(\Gamma)^3 \mid \text{$Pf=P\tilde{f}$ in $H^{-1}(\Gamma)^3$}\}
\end{align*}
and take $Pf$ (or $f$ when $Pf=f$ in $H^{-1}(\Gamma)^3$) as a representative of the equivalence class $[f]$ to write $[Pf,v]_{T\Gamma}=\langle f,v\rangle_\Gamma$ for $v\in H^1(\Gamma,T\Gamma)$ (see Section \ref{SS:Pre_Surf}).

\begin{lemma} \label{L:LWPA_Bi}
  Let $A\in L^2(\Gamma)^{3\times3}$ satisfy $A^T=PA=AP=A$ on $\Gamma$.
  Then
  \begin{align} \label{E:LWPA_Bi}
    \bigl(A,D_\Gamma(\eta)\bigr)_{L^2(\Gamma)} = -[P\mathrm{div}_\Gamma A,\eta]_{T\Gamma}
  \end{align}
  for all $\eta\in H^1(\Gamma,T\Gamma)$.
\end{lemma}

\begin{proof}
  The assumption on $A$ yields
  \begin{align*}
    A:D_\Gamma(\eta) = A:\nabla_\Gamma\eta, \quad A^Tn = 0 \quad\text{on}\quad \Gamma.
  \end{align*}
  Using these equalities and \eqref{E:Def_TD_Hin} and noting that $\eta\in H^1(\Gamma,T\Gamma)$ we get
  \begin{align*}
    \bigl(A,D_\Gamma(\eta)\bigr)_{L^2(\Gamma)} &= (A,\nabla_\Gamma\eta)_{L^2(\Gamma)} = \sum_{i,j=1}^3(A_{ij},\underline{D}_i\eta_j)_{L^2(\Gamma)} \\
    &= -\sum_{i,j=1}^3\left\{\langle \underline{D}_iA_{ij}, \eta_j\rangle_\Gamma+(A_{ij}Hn_i,\eta_j)_{L^2(\Gamma)}\right\} \\
    &= -\left\{\langle \mathrm{div}_\Gamma A,\eta\rangle_\Gamma+(HA^Tn,\eta)_{L^2(\Gamma)}\right\} = -[P\mathrm{div}_\Gamma A,\eta]_{T\Gamma}.
  \end{align*}
  Hence \eqref{E:LWPA_Bi} holds.
\end{proof}

\begin{lemma} \label{L:LWPA_Tri}
  Let $v\in H^1(\Gamma,T\Gamma)$.
  Then
  \begin{align*}
    \overline{\nabla}_vv = P(v\cdot\nabla_\Gamma)v \in L^{4/3}(\Gamma,T\Gamma)
  \end{align*}
  and we can consider $\overline{\nabla}_vv$ as an element of $H^{-1}(\Gamma,T\Gamma)$ by
  \begin{align*}
    \Bigl[\overline{\nabla}_vv,\eta\Bigr]_{T\Gamma} := \int_\Gamma\overline{\nabla}_vv\cdot\eta\,d\mathcal{H}^2, \quad \eta\in H^1(\Gamma,T\Gamma).
  \end{align*}
  Moreover, if $v\in\mathcal{V}_g$, then
  \begin{align} \label{E:LWPA_Tri}
    b_g(v,v,\eta) = \int_\Gamma g\overline{\nabla}_vv\cdot\eta\,d\mathcal{H}^2 = \Bigl[g\overline{\nabla}_vv,\eta\Bigr]_{T\Gamma}
  \end{align}
  for all $\eta\in H^1(\Gamma,T\Gamma)$ and
  \begin{align} \label{E:LWPA_HinT}
    \left\|g\overline{\nabla}_vv\right\|_{H^{-1}(\Gamma,T\Gamma)} \leq c\|v\|_{L^2(\Gamma)}\|v\|_{H^1(\Gamma)}
  \end{align}
  with a constant $c>0$ independent of $v$.
\end{lemma}

\begin{proof}
  Let $v\in H^1(\Gamma,T\Gamma)$.
  By H\"{o}lder's inequality and \eqref{E:La_Surf},
  \begin{align*}
    \left\|\overline{\nabla}_vv\right\|_{L^{4/3}(\Gamma)} \leq \|v\|_{L^4(\Gamma)}\|\nabla_\Gamma v\|_{L^2(\Gamma)} \leq c\|v\|_{L^2(\Gamma)}^{1/2}\|v\|_{H^1(\Gamma)}^{3/2}.
  \end{align*}
  Hence $\overline{\nabla}_vv\in L^{4/3}(\Gamma,T\Gamma)$.
  We again use H\"{o}lder's inequality and \eqref{E:La_Surf} to get
  \begin{align*}
    \left|\int_\Gamma \overline{\nabla}_vv\cdot\eta\,dx\right| \leq \left\|\overline{\nabla}_vv\right\|_{L^{4/3}(\Gamma)}\|\eta\|_{L^4(\Gamma)} \leq c\left\|\overline{\nabla}_vv\right\|_{L^{4/3}(\Gamma)}\|\eta\|_{H^1(\Gamma)}
  \end{align*}
  for all $\eta\in H^1(\Gamma,T\Gamma)$ and thus $\overline{\nabla}_vv\in H^{-1}(\Gamma,T\Gamma)$.
  Now let $v\in\mathcal{V}_g$.
  Then
  \begin{align*}
    b_g(v,v,\eta) = -b_g(v,\eta,v) = \int_\Gamma g(v\otimes\eta):\nabla_\Gamma v\,d\mathcal{H}^2
  \end{align*}
  for $\eta\in H^1(\Gamma,T\Gamma)$ by \eqref{E:TriS_Vg}.
  To the last term we apply
  \begin{align*}
    v\otimes\eta:\nabla_\Gamma v = (v\cdot\nabla_\Gamma)v\cdot\eta = \{P(v\cdot\nabla_\Gamma)v\}\cdot\eta = \overline{\nabla}_vv\cdot\eta \quad\text{on}\quad \Gamma
  \end{align*}
  by $\eta\cdot n=0$ on $\Gamma$ to get \eqref{E:LWPA_Tri}.
  Moreover, by \eqref{E:Tri_Surf} and \eqref{E:LWPA_Tri} we see that
  \begin{align*}
    \left|\Bigl[g\overline{\nabla}_vv,\eta\Bigr]_{T\Gamma}\right| = |b_g(v,v,\eta)| \leq c\|v\|_{L^2(\Gamma)}\|v\|_{H^1(\Gamma)}\|\eta\|_{H^1(\Gamma)}
  \end{align*}
  for all $\eta\in H^1(\Gamma,T\Gamma)$.
  Thus \eqref{E:LWPA_HinT} follows.
\end{proof}

\begin{lemma} \label{L:LW_Pres}
  For $T>0$ and given data
  \begin{align*}
    v_0 \in \mathcal{H}_g, \quad f\in L^2(0,T;H^{-1}(\Gamma,T\Gamma)),
  \end{align*}
  let $v$ be a weak solution to \eqref{E:NS_Limit} on $[0,T)$.
  Then there exists a unique
  \begin{align*}
    \hat{q}\in C([0,T];L^2(\Gamma))
  \end{align*}
  such that $\int_\Gamma\hat{q}(t)\,d\mathcal{H}^2=0$ for all $t\in[0,T]$ and
  \begin{multline} \label{E:LW_Pres}
    g\Bigl(\partial_tv+\overline{\nabla}_vv\Bigr)-2\nu\left\{P\,\mathrm{div}_\Gamma[gD_\Gamma(v)]-\frac{1}{g}(v\cdot\nabla_\Gamma g)\nabla_\Gamma g\right\} \\
    +(\gamma^0+\gamma^1)v+g\nabla_\Gamma q = gf \quad\text{in}\quad \mathcal{D}'(0,T;H^{-1}(\Gamma,T\Gamma))
  \end{multline}
  with $q:=\partial_t\hat{q}\in\mathcal{D}'(0,T;L^2(\Gamma))$ (see Section \ref{SS:Pre_Surf}).
\end{lemma}

\begin{proof}
  Let $v$ be a weak solution to \eqref{E:NS_Limit} on $[0,T)$ and
  \begin{align} \label{Pf_LWP:AgBg}
    \begin{aligned}
      A_gv &:= -2\nu\left\{P\mathrm{div}_\Gamma[gD_\Gamma(v)]-\frac{1}{g}(v\cdot\nabla_\Gamma g)\nabla_\Gamma g\right\}+(\gamma^0+\gamma^1)v, \\
      B_g(v,v) &:= g\overline{\nabla}_vv.
    \end{aligned}
  \end{align}
  Then by $g\in C^4(\Gamma)$, \eqref{E:G_Inf}, \eqref{E:LWPA_HinT}, and $v\in L^\infty(0,T;\mathcal{H}_g)\cap L^2(0,T;\mathcal{V}_g)$ we get
  \begin{align*}
    A_gv, B_g(v,v) \in L^2(0,T;H^{-1}(\Gamma,T\Gamma)).
  \end{align*}
  Since the given data $f$ also belongs to the same space, the functions
  \begin{gather*}
    \widehat{A}_gv(t) := \int_0^tA_gv(s)\,ds, \quad \widehat{B}_g(v,v)(t) := \int_0^tB_g(v(s),v(s))\,ds, \\
    \hat{f}(t) := \int_0^tf(s)\,ds, \quad t\in[0,T]
  \end{gather*}
  are continuous on $[0,T]$ with values in $H^{-1}(\Gamma,T\Gamma)$.
  Moreover,
  \begin{align} \label{Pf_LWP:Reg_V}
    v \in W^{1,1}(0,T;H^{-1}(\Gamma,T\Gamma))\subset C([0,T];H^{-1}(\Gamma,T\Gamma))
  \end{align}
  by Definition \ref{D:Lim_W_T}.
  Hence
  \begin{align} \label{Pf_LWP:F}
    F := gv-gv_0+\widehat{A}_gv+\widehat{B}_g(v,v)-g\hat{f} \in C([0,T];H^{-1}(\Gamma,T\Gamma)).
  \end{align}
  We show that $F$ satisfies the condition of Theorem \ref{T:DeRham_T}.
  Let $\eta\in L^2(0,T;\mathcal{V}_g)$.
  Then since \eqref{E:Limit_Weak} holds for $\eta$ by Lemma \ref{L:Lim_W_L2}, we have
  \begin{align} \label{Pf_LWP:Weak}
    \int_0^T\bigl([g\partial_sv,\eta]_{T\Gamma}+[A_gv,\eta]_{T\Gamma}+[B_g(v,v),\eta]_{T\Gamma}\bigr)\,ds = \int_0^T[gf,\eta]_{T\Gamma}\,ds
  \end{align}
  by applying \eqref{E:LWPA_Bi} with $A=gD_\Gamma(v)$ and \eqref{E:LWPA_Tri} to \eqref{E:Limit_Weak}.
  For $t\in[0,T]$ and $\xi\in\mathcal{V}_g$ let $1_{[0,t]}$ be the characteristic function of $[0,t]\subset\mathbb{R}$ and
  \begin{align*}
    \eta(s) := 1_{[0,t]}(s)\xi, \quad s\in[0,T], \quad \eta\in L^2(0,T;\mathcal{V}_g).
  \end{align*}
  We substitute this $\eta$ for \eqref{Pf_LWP:Weak}.
  Then since
  \begin{align*}
    \int_0^t\partial_sv(s)\,ds = v(t)-v_0 \quad\text{in}\quad H^{-1}(\Gamma,T\Gamma)
  \end{align*}
  by \eqref{Pf_LWP:Reg_V} and $g$ and $\xi$ are independent of time, for each $t\in[0,T]$ we have
  \begin{align*}
    F(t)\in H^{-1}(\Gamma,T\Gamma), \quad [F(t),\xi]_{T\Gamma} = 0 \quad\text{for all}\quad \xi\in\mathcal{V}_g.
  \end{align*}
  Hence by Theorem \ref{T:DeRham_T} there exists a unique $\hat{q}(t)\in L^2(\Gamma)$ such that
  \begin{align*}
    F(t) = -g\nabla_\Gamma\hat{q}(t) \quad\text{in}\quad H^{-1}(\Gamma,T\Gamma), \quad \int_\Gamma\hat{q}(t)\,d\mathcal{H}^2 = 0.
  \end{align*}
  Moreover, by \eqref{E:DeRham_T_Ineq} and $F\in C([0,T];H^{-1}(\Gamma,T\Gamma))$ we see that
  \begin{align*}
    \hat{q} \in C([0,T];L^2(\Gamma)) \subset L^2(0,T;L^2(\Gamma))
  \end{align*}
  and thus $q:=\partial_t\hat{q}\in \mathcal{D}'(0,T;L^2(\Gamma))$ is well-defined.
  Now we have (see \eqref{E:Dt_Dist_L2})
  \begin{align*}
    [\partial_tF+g\partial_t(\nabla_\Gamma\hat{q})](\varphi) = -\int_0^T\partial_t\varphi(t)\{F(t)+g\nabla_\Gamma\hat{q}(t)\}\,dt = 0
  \end{align*}
  in $H^{-1}(\Gamma,T\Gamma)$ for all $\varphi\in C_c^\infty(0,T)$, which means that
  \begin{align*}
    \partial_tF+g\partial_t(\nabla_\Gamma\hat{q}) = 0 \quad\text{in}\quad \mathcal{D}'(0,T;H^{-1}(\Gamma,T\Gamma)).
  \end{align*}
  This implies \eqref{E:LW_Pres} since
  \begin{align*}
     \partial_t(\nabla_\Gamma\hat{q}) = \nabla_\Gamma(\partial_t\hat{q}) = \nabla_\Gamma q, \quad \partial_tF = g\partial_tv+A_gv+B_g(v,v)-gf
  \end{align*}
  in $\mathcal{D}'(0,T;H^{-1}(\Gamma,T\Gamma))$ by \eqref{E:TGrDt_DHin} and \eqref{Pf_LWP:F} (see also \eqref{Pf_LWP:AgBg}).
\end{proof}
Now let us prove Theorem \ref{T:SL_Weak}.
First we present an auxiliary result on the weak limit of the averaged tangential component of a vector field in $L_\sigma^2(\Omega_\varepsilon)$.

\begin{lemma} \label{L:WC_Sole}
  For $\varepsilon\in(0,1]$ let $u^\varepsilon\in L_\sigma^2(\Omega_\varepsilon)$.
  Also, let $v\in L^2(\Gamma,T\Gamma)$.
  If
  \begin{align} \label{E:WCS_Conv}
    \lim_{\varepsilon\to0}M_\tau u^\varepsilon = v \quad\text{weakly in}\quad L^2(\Gamma,T\Gamma)
  \end{align}
  and there exist $\varepsilon'\in(0,1]$, $c>0$, and $\alpha>0$ such that
  \begin{align} \label{E:WCS_Est}
    \|u^\varepsilon\|_{L^2(\Omega_\varepsilon)}^2 \leq c\varepsilon^{-1+\alpha}
  \end{align}
  for all $\varepsilon\in(0,\varepsilon']$, then $v\in\mathcal{H}_g$.
\end{lemma}

\begin{proof}
  By \eqref{E:Sdiv_Hin} and \eqref{E:WCS_Conv} we see that
  \begin{align*}
    \lim_{\varepsilon\to0}\mathrm{div}_\Gamma(gM_\tau u^\varepsilon) = \mathrm{div}_\Gamma(gv) \quad\text{weakly in}\quad H^{-1}(\Gamma).
  \end{align*}
  Moreover, we apply \eqref{E:ADiv_Tan_Hin} to $u^\varepsilon\in L_\sigma^2(\Omega_\varepsilon)$ and use \eqref{E:WCS_Est} to get
  \begin{align*}
    \|\mathrm{div}_\Gamma(gM_\tau u^\varepsilon)\|_{H^{-1}(\Gamma)} \leq c\varepsilon^{1/2}\|u^\varepsilon\|_{L^2(\Omega_\varepsilon)} \leq c\varepsilon^{\alpha/2}
  \end{align*}
  for all $\varepsilon\in(0,\varepsilon']$ with $\alpha>0$.
  Hence
  \begin{align*}
    \|\mathrm{div}_\Gamma(gv)\|_{H^{-1}(\Gamma)} \leq \liminf_{\varepsilon\to0}\|\mathrm{div}_\Gamma(gM_\tau u^\varepsilon)\|_{H^{-1}(\Gamma)} = 0
  \end{align*}
  and $\mathrm{div}_\Gamma(gv)=0$ in $H^{-1}(\Gamma)$, which means that $v\in\mathcal{H}_g$.
\end{proof}

\begin{proof}[Proof of Theorem \ref{T:SL_Weak}]
  Suppose that the assumptions of Theorem \ref{T:SL_Weak} are satisfied.
  Then by the conditions (a) and (b) there exist constants
  \begin{align*}
    c_1, c_2 > 0, \quad \tilde{\varepsilon}_2\in(0,\varepsilon_2], \quad \alpha\in(0,1]
  \end{align*}
  such that for each $\varepsilon\in(0,\tilde{\varepsilon}_2]$ the given data
  \begin{align*}
    u_0^\varepsilon\in \mathcal{V}_\varepsilon, \quad f^\varepsilon\in L^\infty(0,\infty;L^2(\Omega_\varepsilon)^3)
  \end{align*}
  satisfy \eqref{E:UE_Data} with $\beta=1$.
  Let $\varepsilon_1$ and $\varepsilon_\sigma$ be the constants given in Theorem \ref{T:UE} with the above constants $c_1,c_2>0$, $\alpha\in(0,1]$, and $\beta=1$ and Lemma \ref{L:Uni_Poin_Dom}, and let
  \begin{align*}
    \varepsilon_3 := \min\{\varepsilon_1,\varepsilon_\sigma,\tilde{\varepsilon}_2\} \in (0,\varepsilon_2].
  \end{align*}
  Then for each $\varepsilon\in(0,\varepsilon_3]$ there exists a global-in-time strong solution $u^\varepsilon$ to \eqref{E:NS_CTD} satisfying \eqref{E:Est_Ueps}--\eqref{E:Est_DtUe} by Lemma \ref{L:Est_Ueps} and all results in Sections \ref{SS:SL_EstSt}--\ref{SS:SL_EDt} apply to $u^\varepsilon$.
  Moreover, we see by \eqref{E:Ave_N_Lp} and \eqref{E:Est_Ueps} that
  \begin{align*}
    \sup_{t\in[0,\infty)}\|Mu^\varepsilon(t)\cdot n\|_{L^2(\Gamma)} \leq c\varepsilon^{1/2}\sup_{t\in[0,\infty)}\|u^\varepsilon(t)\|_{H^1(\Omega_\varepsilon)} \leq c\varepsilon^{\alpha/2} \to 0
  \end{align*}
  as $\varepsilon\to0$.
  Thus $\{Mu^\varepsilon\cdot n\}_\varepsilon$ converges to zero strongly in $C([0,\infty);L^2(\Gamma))$.

  Now let us consider the averaged tangential component $M_\tau u^\varepsilon$.
  First note that the weak limit $v_0$ of $\{M_\tau u_0^\varepsilon\}_\varepsilon$ in $L^2(\Gamma,T\Gamma)$ actually belongs to $\mathcal{H}_g$ by Lemma \ref{L:WC_Sole} since $u_0^\varepsilon$ satisfies \eqref{E:WCS_Est} by the condition (a).
  For a fixed $T>0$ we observe by \eqref{E:Mu_Energy} and \eqref{E:Mu_Dt} that
  \begin{itemize}
    \item $\{M_\tau u^\varepsilon\}_\varepsilon$ is bounded in $L^\infty(0,T;L^2(\Gamma,T\Gamma))\cap L^2(0,T;H^1(\Gamma,T\Gamma))$,
    \item $\{\partial_tM_\tau u^\varepsilon\}_\varepsilon$ is bounded in $L^2(0,T;H^{-1}(\Gamma,T\Gamma))$.
  \end{itemize}
  Hence there exist a sequence $\{\varepsilon_k\}_{k=1}^\infty$ in $(0,\varepsilon_3]$ and a vector field
  \begin{align*}
    v \in L^\infty(0,T;L^2(\Gamma,T\Gamma))\cap L^2(0,T;H^1(\Gamma,T\Gamma))
  \end{align*}
  with $\partial_tv \in L^2(0,T;H^{-1}(\Gamma,T\Gamma))$ such that $\lim_{k\to\infty}\varepsilon_k=0$ and
  \begin{align} \label{Pf_SLW:W_Conv}
    \begin{alignedat}{3}
      \lim_{k\to\infty}M_\tau u^{\varepsilon_k} &= v &\quad &\text{weakly-$\star$ in} &\quad &L^\infty(0,T;L^2(\Gamma,T\Gamma)), \\
      \lim_{k\to\infty}M_\tau u^{\varepsilon_k} &= v &\quad &\text{weakly in} &\quad &L^2(0,T;H^1(\Gamma,T\Gamma)), \\
      \lim_{k\to\infty}\partial_tM_\tau u^{\varepsilon_k} &= \partial_tv &\quad &\text{weakly in} &\quad &L^2(0,T;H^{-1}(\Gamma,T\Gamma)).
    \end{alignedat}
  \end{align}
  Moreover, by the Aubin--Lions lemma (see e.g. \cite{BoFa13}*{Theorem II.5.16}) there exists a subsequence of $\{M_\tau u^{\varepsilon_k}\}_{k=1}^\infty$, which we denote by $\{M_\tau u^{\varepsilon_k}\}_{k=1}^\infty$ again, such that
  \begin{align} \label{Pf_SLW:St_Conv}
    \lim_{k\to\infty}M_\tau u^{\varepsilon_k} = v \quad\text{strongly in}\quad L^2(0,T;L^2(\Gamma,T\Gamma)).
  \end{align}
  Then we can take again a subsequence of $\{M_\tau u^{\varepsilon_k}\}_{k=1}^\infty$, still denoted by $\{M_\tau u^{\varepsilon_k}\}_{k=1}^\infty$, such that
  \begin{align*}
    \lim_{k\to\infty}M_\tau u^{\varepsilon_k}(t) = v(t) \quad\text{strongly in}\quad L^2(\Gamma,T\Gamma) \quad\text{for a.a.}\quad t\in(0,T).
  \end{align*}
  Hence $v(t)\in \mathcal{H}_g$ for a.a. $t\in (0,T)$ by \eqref{E:Est_Ueps} and Lemma \ref{L:WC_Sole} and we get
  \begin{align*}
    v \in L^\infty(0,T;\mathcal{H}_g)\cap L^2(0,T;\mathcal{V}_g).
  \end{align*}
  Let us show that $v$ is a weak solution to \eqref{E:NS_Limit} on $[0,T)$.
  First we see that $v$ satisfies the weak formulation \eqref{E:Limit_Weak} for all $\eta\in C_c(0,T;\mathcal{V}_g)$.
  In what follows, we write $c$ for a general positive constant that may depend on $v$, $\eta$, and $T$ but is independent of $\varepsilon_k$ and $u^{\varepsilon_k}$.
  We consider the weak formulation \eqref{E:Mu_Weak} for $M_\tau u^{\varepsilon_k}$:
  \begin{multline} \label{Pf_SLW:Weak_Muek}
    \int_0^T\{[g\partial_tM_\tau u^{\varepsilon_k},\eta]_{T\Gamma}+a_g(M_\tau u^{\varepsilon_k},\eta)+b_g(M_\tau u^{\varepsilon_k},M_\tau u^{\varepsilon_k},\eta)\}\,dt \\
    = \int_0^T[gM_\tau\mathbb{P}_{\varepsilon_k}f^{\varepsilon_k},\eta]_{T\Gamma}\,dt+R_{\varepsilon_k}^1(\eta).
  \end{multline}
  Here $\partial_tM_\tau u^{\varepsilon_k}$ and $M_\tau\mathbb{P}_{\varepsilon_k}f^{\varepsilon_k}$ are considered as elements of $H^{-1}(\Gamma,T\Gamma)$ (see Section \ref{SS:Pre_Surf}).
  Let $k\to\infty$ in \eqref{Pf_SLW:Weak_Muek}.
  Then noting that
  \begin{align*}
    \eta \in C_c(0,T;\mathcal{V}_g) \subset L^p(0,T;H^1(\Gamma,T\Gamma)), \quad p\in[1,\infty],
  \end{align*}
  we observe by the condition (b), \eqref{E:Def_Mul_HinT}, and \eqref{Pf_SLW:W_Conv} that
  \begin{align} \label{Pf_SLW:Conv_Lin}
    \begin{aligned}
      \lim_{k\to\infty}\int_0^T[g\partial_tM_\tau u^{\varepsilon_k},\eta]_{T\Gamma}\,dt &= \int_0^T[g\partial_tv,\eta]_{T\Gamma}\,dt, \\
      \lim_{k\to\infty}\int_0^Ta_g(M_\tau u^{\varepsilon_k},\eta)\,dt &= \int_0^Ta_g(v,\eta)\,dt, \\
      \lim_{k\to\infty}\int_0^T[gM_\tau\mathbb{P}_{\varepsilon_k}f^{\varepsilon_k},\eta]_{T\Gamma}\,dt &= \int_0^T[gf,\eta]_{T\Gamma}\,dt.
    \end{aligned}
  \end{align}
  Also, it follows from \eqref{E:Mu_Weak_Re}, the condition (c), and $\alpha>0$ that
  \begin{align} \label{Pf_SLW:Conv_Re}
    |R_{\varepsilon_k}^1(\eta)| \leq c\left(\varepsilon_k^{\alpha/4}+\sum_{i=0,1}\left|\frac{\gamma_{\varepsilon_k}^i}{\varepsilon_k}-\gamma^i\right|\right)(1+T)^{1/2}\|\eta\|_{L^2(0,T;H^1(\Gamma))} \to 0
  \end{align}
  as $k\to\infty$.
  To show the convergence of the trilinear term, we set
  \begin{align*}
    J_1^k &:= \int_0^Tb_g(M_\tau u^{\varepsilon_k},M_\tau u^{\varepsilon_k},\eta)\,dt-\int_0^Tb_g(v,M_\tau u^{\varepsilon_k},\eta)\,dt, \\
    J_2^k &:= \int_0^Tb_g(v,M_\tau u^{\varepsilon_k},\eta)\,dt-\int_0^Tb_g(v,v,\eta)\,dt.
  \end{align*}
  Since \eqref{E:Tri_Surf} holds and $\|\eta(t)\|_{H^1(\Gamma)}$ is bounded on $[0,T]$ by $\eta\in C_c(0,T;\mathcal{V}_g)$,
  \begin{align*}
    |J_1^k| &\leq c\int_0^T\|M_\tau u^{\varepsilon_k}-v\|_{L^2(\Gamma)}^{1/2}\|M_\tau u^{\varepsilon_k}-v\|_{H^1(\Gamma)}^{1/2}\|M_\tau u^{\varepsilon_k}\|_{H^1(\Gamma)}\|\eta\|_{H^1(\Gamma)}\,dt \\
    &\leq c\|M_\tau u^{\varepsilon_k}-v\|_{L^2(0,T;L^2(\Gamma))}^{1/2}\|M_\tau u^{\varepsilon_k}-v\|_{L^2(0,T;H^1(\Gamma))}^{1/2}\|M_\tau u^{\varepsilon_k}\|_{L^2(0,T;H^1(\Gamma))}.
  \end{align*}
  Applying \eqref{E:Mu_Energy} and \eqref{Pf_SLW:St_Conv} to the last line we obtain
  \begin{align} \label{Pf_SLW:I1}
    |J_1^k| \leq c\|M_\tau u^{\varepsilon_k}-v\|_{L^2(0,T;L^2(\Gamma))}^{1/2} \to 0 \quad\text{as}\quad k\to\infty.
  \end{align}
  For $J_2^k$, we consider the linear functional
  \begin{align*}
    \Phi(\xi) := \int_0^Tb_g(v,\xi,\eta)\,dt, \quad \xi \in L^2(0,T;H^1(\Gamma,T\Gamma)).
  \end{align*}
  By \eqref{E:Tri_Surf} and the boundedness of $\|\eta(t)\|_{H^1(\Gamma)}$ on $[0,T]$ we get
  \begin{align*}
    |\Phi(\xi)| \leq c\|\eta\|_{L^\infty(0,T;H^1(\Gamma))}\|v\|_{L^2(0,T;H^1(\Gamma))}\|\xi\|_{L^2(0,T;H^1(\Gamma))}.
  \end{align*}
  Hence $\Phi$ is bounded on $L^2(0,T;H^1(\Gamma,T\Gamma))$ and
  \begin{align*}
    \lim_{k\to\infty}J_2^k = \lim_{k\to\infty}\{\Phi(M_\tau u^{\varepsilon_k})-\Phi(v)\} = 0
  \end{align*}
  by the second line of \eqref{Pf_SLW:W_Conv}.
  Combining this equality with \eqref{Pf_SLW:I1} we obtain
  \begin{align} \label{Pf_SLW:Conv_Tri}
    \lim_{k\to\infty}\int_0^Tb_g(M_\tau u^{\varepsilon_k},M_\tau u^{\varepsilon_k},\eta)\,dt = \int_0^Tb_g(v,v,\eta)\,dt.
  \end{align}
  Thus we send $k\to\infty$ in \eqref{Pf_SLW:Weak_Muek} and apply \eqref{Pf_SLW:Conv_Lin}, \eqref{Pf_SLW:Conv_Re}, and \eqref{Pf_SLW:Conv_Tri} to find that $v$ satisfies \eqref{E:Limit_Weak} for all $\eta\in C_c(0,T;\mathcal{V}_g)$.
  We also observe that
  \begin{align*}
    v\in C([0,T],\mathcal{H}_g)
  \end{align*}
  and \eqref{E:Limit_Weak} is valid for all $\eta\in L^2(0,T;\mathcal{V}_g)$ by Lemma \ref{L:Lim_W_L2}.

  Next we show that $v$ satisfies the initial condition $v|_{t=0}=v_0$ in $\mathcal{H}_g$.
  For $\xi\in \mathcal{V}_g$ and $\varphi\in C^\infty([0,T])$ satisfying $\varphi(0)=1$ and $\varphi(T)=0$ let
  \begin{align*}
    \eta(t) := \varphi(t)\xi, \quad t\in[0,T], \quad \eta\in C^\infty([0,T];\mathcal{V}_g) \subset L^2(0,T;\mathcal{V}_g).
  \end{align*}
  We substitute this $\eta$ for \eqref{E:Limit_Weak} and \eqref{Pf_SLW:Weak_Muek} and carry out integration by parts with respect to time for the terms including $\partial_tv$ and $\partial_tM_\tau u^{\varepsilon_k}$.
  Then we get
  \begin{align} \label{Pf_SLW:Ini}
    (gv(0),\xi)_{L^2(\Gamma)} = J_\infty, \quad (gM_\tau u_0^{\varepsilon_k},\xi)_{L^2(\Gamma)} = J_k
  \end{align}
  since $\varphi(0)=1$, $\varphi(T)=0$, and $g$ and $\xi$ are independent of time, where
  \begin{align*}
    J_\infty := -\int_0^T\partial_t\varphi(gv,\xi)_{L^2(\Gamma)}\,dt+\int_0^T\{a_g(v,\eta)+b_g(v,v,\eta)\}\,dt-\int_0^T[gf,\eta]_{T\Gamma}\,dt
  \end{align*}
  and
  \begin{multline*}
    J_k := -\int_0^T\partial_t\varphi(gM_\tau u^{\varepsilon_k},\xi)_{L^2(\Gamma)}\,dt \\
    +\int_0^T\{a_g(M_\tau u^{\varepsilon_k},\eta)+b_g(M_\tau u^{\varepsilon_k},M_\tau u^{\varepsilon_k},\eta)\}\,dt \\
    -\int_0^T[gM_\tau\mathbb{P}_{\varepsilon_k}f^{\varepsilon_k},\eta]_{T\Gamma}\,dt-R_{\varepsilon_k}^1(\eta).
  \end{multline*}
  We send $k\to\infty$ in the second equality of \eqref{Pf_SLW:Ini}.
  Then the left-hand side converges to $(gv_0,\xi)_{L^2(\Gamma)}$ by the condition (b).
  Also, since
  \begin{align*}
    \eta\in C^\infty([0,T];\mathcal{V}_g) \subset L^p(0,T;H^1(\Gamma,T\Gamma)), \quad p\in[1,\infty],
  \end{align*}
  we can apply \eqref{Pf_SLW:Conv_Lin}, \eqref{Pf_SLW:Conv_Re}, \eqref{Pf_SLW:Conv_Tri}, and
  \begin{align*}
    \lim_{k\to\infty}\int_0^T\partial_t\varphi(gM_\tau u^{\varepsilon_k},\xi)_{L^2(\Gamma)}\,dt = \int_0^T\partial_t\varphi(gv,\xi)_{L^2(\Gamma)}\,dt
  \end{align*}
  by \eqref{Pf_SLW:St_Conv} to find that $\lim_{k\to\infty}J_k=J_\infty$.
  Therefore,
  \begin{align*}
    (gv(0),\xi)_{L^2(\Gamma)} = J_\infty = (gv_0,\xi)_{L^2(\Gamma)} \quad\text{for all}\quad \xi\in \mathcal{V}_g.
  \end{align*}
  This equality is also valid for all $\xi\in\mathcal{H}_g$ since $\mathcal{V}_g$ is dense in $\mathcal{H}_g$ (see Lemma \ref{L:H1gs_Dense}).
  Thus we can set $\xi:=v(0)-v_0\in\mathcal{H}_g$ in the above equality to get
  \begin{align*}
    (g\{v(0)-v_0\},v(0)-v_0)_{L^2(\Gamma)} = \|g^{1/2}\{v(0)-v_0\}\|_{L^2(\Gamma)}^2 = 0,
  \end{align*}
  which gives $v|_{t=0}=v_0$ in $\mathcal{H}_g$ when combined with \eqref{E:G_Inf}.
  Hence $v$ is a unique weak solution to \eqref{E:NS_Limit} on $[0,T)$ (here the uniqueness follows from Lemma \ref{L:Limit_Uni}).

  Now let us prove the convergence of the full sequence
  \begin{align} \label{Pf_SLW:Full_Conv}
    \lim_{\varepsilon\to0}M_\tau u^\varepsilon = v \quad\text{weakly in}\quad L^2(0,T;H^1(\Gamma,T\Gamma)).
  \end{align}
  Let $\{\varepsilon_l\}_{l=1}^\infty$ be an arbitrary sequence in $(0,\varepsilon_3]$ convergent to zero.
  Then we see by the same arguments as above that $\{M_\tau u^{\varepsilon_l}\}_{l=1}^\infty$ has a subsequence $\{M_\tau u^{\varepsilon_k}\}_{k=1}^\infty$ that converges to the same limit $v$, which is the unique weak solution to \eqref{E:NS_Limit} on $[0,T)$, in the sense of \eqref{Pf_SLW:W_Conv} and \eqref{Pf_SLW:St_Conv}.
  Thus we obtain \eqref{Pf_SLW:Full_Conv}.

  Since the strong solution $u^\varepsilon$ to \eqref{E:NS_CTD} exists globally in time for all $\varepsilon\in(0,\varepsilon_3]$, by the above arguments we get a unique weak solution
  \begin{align*}
    v_T \in C([0,T];\mathcal{H}_g)\cap L^2(0,T;\mathcal{V}_g)\cap H^1(0,T;H^{-1}(\Gamma,T\Gamma))
  \end{align*}
  to \eqref{E:NS_Limit} on $[0,T)$ satisfying \eqref{Pf_SLW:Full_Conv} for all $T>0$.
  Moreover, if $T<T'$ then $v_T=v_{T'}$ on $[0,T]$ by the uniqueness of a weak solution.
  Hence we can define
  \begin{align*}
    v \in C([0,\infty);\mathcal{H}_g)\cap L_{loc}^2([0,\infty);\mathcal{V}_g)\cap H_{loc}^1([0,\infty);H^{-1}(\Gamma,T\Gamma))
  \end{align*}
  by $v:=v_T$ on $[0,T]$ for each $T>0$, which is a unique weak solution to \eqref{E:NS_Limit} on $[0,\infty)$ and satisfies \eqref{Pf_SLW:Full_Conv} for all $T>0$.
\end{proof}

As a consequence of Theorem \ref{T:SL_Weak} we get the existence of a weak solution to \eqref{E:NS_Limit} when the initial velocity $v_0$ and the external force $f$ are the weak and weak-$\star$ limits of $\{M_\tau u_0^\varepsilon\}_\varepsilon$ and $\{M_\tau\mathbb{P}_\varepsilon f^\varepsilon\}_\varepsilon$, respectively.
For general data
\begin{align*}
  v_0\in \mathcal{H}_g, \quad f\in L_{loc}^2([0,\infty);H^{-1}(\Gamma,T\Gamma))
\end{align*}
we can construct a weak solution to \eqref{E:NS_Limit} by the Galerkin method as in the case of the Navier--Stokes equations in a bounded domain in $\mathbb{R}^2$ (see e.g. \cites{BoFa13,CoFo88,Te79}).
We give the outline of construction of a weak solution to \eqref{E:NS_Limit} by the Galerkin method in Appendix \ref{S:Ap_CWL}.

\subsection{Strong convergence of the average and error estimates} \label{SS:SL_ErST}
In this subsection we prove Theorem \ref{T:SL_Strong} by showing an energy estimate for the difference between the averaged tangential component of a strong solution to \eqref{E:NS_CTD} and a weak solution to \eqref{E:NS_Limit}.
We also give estimates in $\Omega_\varepsilon$ for the difference between a strong solution to \eqref{E:NS_CTD} and the constant extension of a weak solution to \eqref{E:NS_Limit}.

\begin{theorem} \label{T:Diff_Mu_V}
  Let $c_1$ and $c_2$ be positive constants, $\alpha\in(0,1]$, $\beta=1$, and $\varepsilon_1$ and $\varepsilon_\sigma$ the constants given in Theorem \ref{T:UE} and Lemma \ref{L:Uni_Poin_Dom}.
  Under the condition
  \begin{align*}
    0 < \varepsilon \leq \min\{\varepsilon_1,\varepsilon_\sigma\}
  \end{align*}
  and the assumptions of Lemma \ref{L:Est_Ueps}, let $u^\varepsilon$ be the global-in-time strong solution to \eqref{E:NS_CTD} given in Lemma \ref{L:Est_Ueps}.
  Also, for given data
  \begin{align*}
    v_0\in\mathcal{H}_g, \quad f\in L^\infty(0,\infty;H^{-1}(\Gamma,T\Gamma))
  \end{align*}
  let $v$ be a weak solution to \eqref{E:NS_Limit} on $[0,\infty)$.
  Then for all $T>0$ we have
  \begin{multline} \label{E:Diff_Mu_V}
    \max_{t\in[0,T]}\|M_\tau u^\varepsilon(t)-v(t)\|_{L^2(\Gamma)}^2+\int_0^T\|\nabla_\Gamma M_\tau u^\varepsilon(t)-\nabla_\Gamma v(t)\|_{L^2(\Gamma)}^2\,dt \\
    \leq c_T\left\{\delta(\varepsilon)^2+\|M_\tau u_0^\varepsilon-v_0\|_{L^2(\Gamma)}^2+\|M_\tau\mathbb{P}_\varepsilon f^\varepsilon-f\|_{L^\infty(0,\infty;H^{-1}(\Gamma,T\Gamma))}^2\right\},
  \end{multline}
  where $c_T>0$ is a constant depending only on $T$ and
  \begin{align} \label{E:DMuV_Re}
    \delta(\varepsilon) := \varepsilon^{\alpha/4}+\sum_{i=0,1}\left|\frac{\gamma_\varepsilon^i}{\varepsilon}-\gamma^i\right|.
  \end{align}
\end{theorem}

As in Section \ref{SS:SL_Ener}, we first compare the auxiliary vector field $v^\varepsilon=\mathbb{P}_gM_\tau u^\varepsilon$ with $v$ and then derive \eqref{E:Diff_Mu_V} by using the estimates \eqref{E:Diff_PMu} for $M_\tau u^\varepsilon-v^\varepsilon$.

\begin{lemma} \label{L:Diff_Ve_V}
  Under the assumptions of Theorem \ref{T:Diff_Mu_V}, let $v^\varepsilon=\mathbb{P}_gM_\tau u^\varepsilon$ be the vector field given in Lemma \ref{L:PMu_Weak}.
  Then for all $T>0$ we have
  \begin{multline} \label{E:Diff_Ve_V}
    \max_{t\in[0,T]}\|v^\varepsilon(t)-v(t)\|_{L^2(\Gamma)}^2+\int_0^T\|\nabla_\Gamma v^\varepsilon(t)-\nabla_\Gamma v(t)\|_{L^2(\Gamma)}^2\,dt \\
    \leq c_T\left\{\delta(\varepsilon)^2+\|v^\varepsilon(0)-v_0\|_{L^2(\Gamma)}^2+\|M_\tau\mathbb{P}_\varepsilon f^\varepsilon-f\|_{L^\infty(0,\infty;H^{-1}(\Gamma,T\Gamma))}^2\right\},
  \end{multline}
   where $c_T>0$ is a constant depending only on $T$ and $\delta(\varepsilon)$ is given by \eqref{E:DMuV_Re}.
\end{lemma}

\begin{proof}
  For the sake of simplicity, we set
  \begin{align*}
    w^\varepsilon := v^\varepsilon-v, \quad w_0^\varepsilon := v^\varepsilon(0)-v_0 = \mathbb{P}_gM_\tau u_0^\varepsilon-v_0, \quad F^\varepsilon := M_\tau \mathbb{P}_\varepsilon f^\varepsilon-f.
  \end{align*}
  Let $T>0$.
  By Lemmas \ref{L:PMu_Weak} and \ref{L:Lim_W_L2} we see that
  \begin{align*}
    w^\varepsilon\in C([0,T];\mathcal{H}_g)\cap L^2(0,T;\mathcal{V}_g), \quad \partial_tw^\varepsilon\in L^2(0,T;H^{-1}(\Gamma,T\Gamma)),
  \end{align*}
  and $w^\varepsilon|_{t=0}=w_0^\varepsilon$ in $\mathcal{H}_g$, and \eqref{E:PMu_Weak} and \eqref{E:Limit_Weak} are valid for all $\eta\in L^2(0,T;\mathcal{V}_g)$.
  We subtract both sides of \eqref{E:Limit_Weak} from those of \eqref{E:PMu_Weak} to get
  \begin{multline} \label{Pf_DVeV:Weak}
    \int_0^T\{[g\partial_sw^\varepsilon,\eta]_{T\Gamma}+a_g(w^\varepsilon,\eta)+b_g(w^\varepsilon,v^\varepsilon,\eta)+b_g(v,w^\varepsilon,\eta)\}\,ds \\
    = \int_0^T[gF^\varepsilon,\eta]_{T\Gamma}\,ds+R_\varepsilon^1(\eta)+R_\varepsilon^2(\eta),
  \end{multline}
  where $R_\varepsilon^1(\eta)$ and $R_\varepsilon^2(\eta)$ are given in Lemmas \ref{L:Mu_Weak} and \ref{L:PMu_Weak}.
  For $t\in[0,T]$ let $1_{[0,t]}$ be the characteristic function of $[0,t]\subset\mathbb{R}$.
  We substitute
  \begin{align*}
    \eta := 1_{[0,t]}w^\varepsilon\in L^2(0,T;\mathcal{V}_g)
  \end{align*}
  for \eqref{Pf_DVeV:Weak} and calculate as in the proofs of Lemmas \ref{L:PMu_Energy} and \ref{L:Limit_Uni} by using \eqref{E:Bi_Surf}--\eqref{E:TriS_Vg}, \eqref{E:Mu_Weak_Re}, \eqref{E:PMu_Weak_Re}, \eqref{Pf_LU:Int_Dt}, and Young's inequality.
  Then we get
  \begin{multline} \label{Pf_DVeV:Gron}
    \|w^\varepsilon(t)\|_{L^2(\Gamma)}^2+\int_0^t\|\nabla_\Gamma w^\varepsilon\|_{L^2(\Gamma)}^2\,ds \\
    \leq c\left\{\|w_0^\varepsilon\|_{L^2(\Gamma)}^2+\int_0^t\left(1+\|v^\varepsilon\|_{H^1(\Gamma)}^2\right)\|w^\varepsilon\|_{L^2(\Gamma)}^2\,ds \right. \\
    \left. +\int_0^t\|F^\varepsilon\|_{H^{-1}(\Gamma,T\Gamma)}^2\,ds+\delta(\varepsilon)^2(1+t)\right\}
  \end{multline}
  for all $t\in[0,T]$.
  Here the constant $\delta(\varepsilon)$ of the form \eqref{E:DMuV_Re} comes from \eqref{E:Mu_Weak_Re} and \eqref{E:PMu_Weak_Re} (note that $\varepsilon^{\alpha/2}\leq\varepsilon^{\alpha/4}$).
  Thus setting
  \begin{align*}
    \xi(t) := \delta(\varepsilon)^2+\|w^\varepsilon(t)\|_{L^2(\Gamma)}^2, \quad \varphi(t) := 1+\|v^\varepsilon(t)\|_{H^1(\Gamma)}^2, \quad t\in[0,T]
  \end{align*}
  we observe by \eqref{Pf_DVeV:Gron} that
  \begin{align*}
    \xi(t) \leq c\left\{\xi(0)+\int_0^t\left(\varphi\xi+\|F^\varepsilon\|_{H^{-1}(\Gamma,T\Gamma)}^2\right)ds\right\} \quad\text{for all}\quad t\in[0,T]
  \end{align*}
  and we apply Gronwall's inequality to this inequality to get
  \begin{align*}
    \xi(t) \leq c\left(\xi(0)+\int_0^t\|F^\varepsilon\|_{H^{-1}(\Gamma,T\Gamma)}^2\,ds\right)\exp\left(c\int_0^t\varphi\,ds\right), \quad t\in[0,T].
  \end{align*}
  From this inequality and the estimate \eqref{E:PMu_Energy} for $\|v^\varepsilon\|_{H^1(\Gamma)}^2$ we deduce that
  \begin{align*}
    \|w^\varepsilon(t)\|_{L^2(\Gamma)}^2 \leq c_T\left\{\delta(\varepsilon)^2+\|w_0^\varepsilon\|_{L^2(\Gamma)}^2+\|F^\varepsilon\|_{L^\infty(0,\infty;H^{-1}(\Gamma,T\Gamma))}^2\right\}
  \end{align*}
  for all $t\in[0,T]$, where $c_T>0$ is a constant depending only on $T$.
  Applying this inequality and \eqref{E:PMu_Energy} to \eqref{Pf_DVeV:Gron} we also get
  \begin{align*}
    \int_0^t\|\nabla_\Gamma w^\varepsilon\|_{L^2(\Gamma)}^2\,ds \leq c_T\left\{\delta(\varepsilon)^2+\|w_0^\varepsilon\|_{L^2(\Gamma)}^2+\|F^\varepsilon\|_{L^\infty(0,\infty;H^{-1}(\Gamma,T\Gamma))}^2\right\}
  \end{align*}
  for all $t\in[0,T]$ with another constant $c_T>0$ depending only on $T$.
  Therefore, the inequality \eqref{E:Diff_Ve_V} is valid.
\end{proof}

\begin{proof}[Proof of Theorem \ref{T:Diff_Mu_V}]
  Let $v^\varepsilon=\mathbb{P}_gM_\tau u^\varepsilon$.
  Since $v_0\in\mathcal{H}_g$ and $\mathbb{P}_g$ is the orthogonal projection from $L^2(\Gamma,T\Gamma)$ onto $\mathcal{H}_g$, we have
  \begin{align*}
    \|v^\varepsilon(0)-v_0\|_{L^2(\Gamma)} = \|\mathbb{P}_g(M_\tau u_0^\varepsilon-v_0)\|_{L^2(\Gamma)} \leq \|M_\tau u_0^\varepsilon-v_0\|_{L^2(\Gamma)}.
  \end{align*}
  This inequality, \eqref{E:Diff_PMu}, and \eqref{E:Diff_Ve_V} imply \eqref{E:Diff_Mu_V} (note that $\varepsilon^2\leq\delta(\varepsilon)^2$).
\end{proof}

Now Theorem \ref{T:SL_Strong} is an immediate consequence of Theorem \ref{T:Diff_Mu_V}.

\begin{proof}[Proof of Theorem \ref{T:SL_Strong}]
  Since the condition (b') of Theorem \ref{T:SL_Strong} implies the condition (b) of Theorem \ref{T:SL_Weak}, the statements of Theorem \ref{T:SL_Weak} are valid under the assumptions of Theorem \ref{T:SL_Strong}.
  Moreover,
  \begin{align*}
    \lim_{\varepsilon\to0}\left\{\delta(\varepsilon)^2+\|M_\tau u_0^\varepsilon-v_0\|_{L^2(\Gamma)}^2+\|M_\tau\mathbb{P}_\varepsilon f^\varepsilon-f\|_{L^\infty(0,\infty;H^{-1}(\Gamma,T\Gamma))}^2\right\} = 0
  \end{align*}
  by \eqref{E:DMuV_Re}, $\alpha>0$, and the conditions (b') and (c).
  Hence
  \begin{align*}
    \lim_{\varepsilon\to0}M_\tau u^\varepsilon = v \quad\text{strongly in}\quad C([0,T];L^2(\Gamma,T\Gamma))\cap L^2(0,T;H^1(\Gamma,T\Gamma))
  \end{align*}
  for all $T>0$ by \eqref{E:Diff_Mu_V}.
\end{proof}

Next we consider the difference between $u^\varepsilon$ and the constant extension of $v$ in $\Omega_\varepsilon$.
Recall that we denote by $\bar{\eta}=\eta\circ\pi$ the constant extension of a function $\eta$ on $\Gamma$ in the normal direction of $\Gamma$.

\begin{theorem} \label{T:Diff_Ue_Cv}
  Under the assumptions of Theorem \ref{T:Diff_Mu_V}, we have
  \begin{multline} \label{E:Diff_Ue_Cv}
    \max_{t\in[0,T]}\|u^\varepsilon(t)-\bar{v}(t)\|_{L^2(\Omega_\varepsilon)}^2+\int_0^T\left\|\overline{P}\nabla u^\varepsilon(t)-\overline{\nabla_\Gamma v}(t)\right\|_{L^2(\Omega_\varepsilon)}^2\,dt \\
  \leq c_T\varepsilon\left\{\delta(\varepsilon)^2+\|M_\tau u_0^\varepsilon-v_0\|_{L^2(\Gamma)}^2+\|M_\tau\mathbb{P}_\varepsilon f^\varepsilon-f\|_{L^\infty(0,\infty;H^{-1}(\Gamma,T\Gamma))}^2\right\}
  \end{multline}
  for all $T>0$, where $c_T>0$ is a constant depending only on $T$ and $\delta(\varepsilon)$ is given by \eqref{E:DMuV_Re}.
  In particular,
  \begin{align*}
    \lim_{\varepsilon\to0}\varepsilon^{-1}\|u^\varepsilon-\bar{v}\|_{C([0,T];L^2(\Omega_\varepsilon))}^2 = \lim_{\varepsilon\to0}\varepsilon^{-1}\left\|\overline{P}\nabla u^\varepsilon-\overline{\nabla_\Gamma v}\right\|_{L^2(0,T;L^2(\Omega_\varepsilon))}^2 = 0
  \end{align*}
  for all $T>0$ provided that
  \begin{align} \label{E:StCo_Data}
    \begin{gathered}
      \lim_{\varepsilon\to0}\|M_\tau u_0^\varepsilon-v_0\|_{L^2(\Gamma)}^2 = \lim_{\varepsilon\to0}\|M_\tau\mathbb{P}_\varepsilon f^\varepsilon-f\|_{L^\infty(0,\infty;H^{-1}(\Gamma,T\Gamma))}^2 = 0, \\
      \lim_{\varepsilon\to0}\frac{\gamma_\varepsilon^i}{\varepsilon} = \gamma^i, \quad i=0,1.
    \end{gathered}
  \end{align}
\end{theorem}

\begin{proof}
  Let $T>0$.
  In what follows, we suppress the argument $t\in[0,T]$, denote by $c_T$ a general positive constant depending only on $T$, and set
  \begin{align*}
    R_\varepsilon := \delta(\varepsilon)^2+\|M_\tau u_0^\varepsilon-v_0\|_{L^2(\Gamma)}^2+\|M_\tau\mathbb{P}_\varepsilon f^\varepsilon-f\|_{L^\infty(0,\infty;H^{-1}(\Gamma,T\Gamma))}^2.
  \end{align*}
  Let us estimate $u^\varepsilon-\bar{v}$.
  Since $u^\varepsilon\in D(A_\varepsilon)\subset H^2(\Omega_\varepsilon)^3$ satisfies \eqref{E:Bo_Imp} and
  \begin{align*}
    u^\varepsilon-\bar{v} = \Bigl(u^\varepsilon-\overline{M_\tau u^\varepsilon}\Bigr)+\Bigl(\overline{M_\tau u^\varepsilon}-\bar{v}\Bigr) \quad\text{in}\quad \Omega_\varepsilon,
  \end{align*}
  we apply \eqref{E:Con_Lp} and \eqref{E:AveT_Diff_Dom} to the right-hand side to get
  \begin{align*}
    \|u^\varepsilon-\bar{v}\|_{L^2(\Omega_\varepsilon)}^2 \leq c\varepsilon\left(\varepsilon\|u^\varepsilon\|_{H^1(\Omega_\varepsilon)}^2+\|M_\tau u^\varepsilon-v\|_{L^2(\Gamma)}^2\right).
  \end{align*}
  Hence we see by \eqref{E:Est_Ueps}, \eqref{E:Diff_Mu_V}, and $\varepsilon^\alpha\leq\delta(\varepsilon)^2$ that
  \begin{align} \label{Pf_DUeCv:L2}
    \|u^\varepsilon(t)-\bar{v}(t)\|_{L^2(\Omega_\varepsilon)}^2 \leq c_T\varepsilon(\varepsilon^\alpha+R_\varepsilon) \leq c_T\varepsilon R_\varepsilon \quad\text{for all}\quad t\in[0,T].
  \end{align}
  Next we consider the second term on the left-hand side of \eqref{E:Diff_Ue_Cv}.
  Let
  \begin{align*}
    J_1 &:= \left\|\overline{P}\nabla u^\varepsilon-\overline{\nabla_\Gamma Mu^\varepsilon}\right\|_{L^2(\Omega_\varepsilon)}, \\
    J_2 &:= \left\|\overline{\nabla_\Gamma M_\tau u^\varepsilon}-\overline{\nabla_\Gamma v}\right\|_{L^2(\Omega_\varepsilon)}, \\
    J_3 &:= \left\|\overline{\nabla_\Gamma[(Mu^\varepsilon\cdot n)n]}\right\|_{L^2(\Omega_\varepsilon)}.
  \end{align*}
  Then we apply \eqref{E:ADD_Dom} to $J_1$ and \eqref{E:Con_Lp} to $J_2$ to get
  \begin{align*}
    J_1 \leq c\varepsilon\|u^\varepsilon\|_{H^2(\Omega_\varepsilon)}, \quad J_2 \leq c\varepsilon^{1/2}\|\nabla_\Gamma M_\tau u^\varepsilon-\nabla_\Gamma v\|_{L^2(\Gamma)}.
  \end{align*}
  Also, since $|n|=1$ and $W=-\nabla_\Gamma n$ is bounded on $\Gamma$,
  \begin{align*}
    |\nabla_\Gamma[(Mu^\varepsilon\cdot n)n]| &= |[\nabla_\Gamma(Mu^\varepsilon\cdot n)]\otimes n-(Mu^\varepsilon\cdot n)W| \\
    &\leq c\left(|Mu^\varepsilon\cdot n|+|\nabla_\Gamma(Mu^\varepsilon\cdot n)|\right)
  \end{align*}
  on $\Gamma$.
  Thus the inequalities \eqref{E:Con_Lp} and \eqref{E:Ave_N_W1p} imply that
  \begin{align*}
    J_3 \leq c\varepsilon^{1/2}\|Mu^\varepsilon\cdot n\|_{H^1(\Gamma)} \leq c\varepsilon\|u^\varepsilon\|_{H^2(\Omega_\varepsilon)}.
  \end{align*}
  By these inequalities and $Mu^\varepsilon=M_\tau u^\varepsilon+(Mu^\varepsilon\cdot n)n$ on $\Gamma$ we see that
  \begin{align*}
    \left\|\overline{P}\nabla u^\varepsilon-\overline{\nabla_\Gamma v}\right\|_{L^2(\Omega_\varepsilon)} &\leq J_1+J_2+J_3 \\
    &\leq c\varepsilon^{1/2}\left(\varepsilon^{1/2}\|u^\varepsilon\|_{H^2(\Omega_\varepsilon)}+\|\nabla_\Gamma M_\tau u^\varepsilon-\nabla_\Gamma v\|_{L^2(\Gamma)}\right).
  \end{align*}
  Hence it follows from \eqref{E:Est_Ueps}, \eqref{E:Diff_Mu_V}, and $\varepsilon^\alpha\leq\delta(\varepsilon)^2$ that
  \begin{align*}
    \int_0^T\left\|\overline{P}\nabla u^\varepsilon-\overline{\nabla_\Gamma v}\right\|_{L^2(\Omega_\varepsilon)}^2\,dt &\leq c\varepsilon\int_0^T\left(\varepsilon\|u^\varepsilon\|_{H^2(\Omega_\varepsilon)}^2+\|\nabla_\Gamma M_\tau u^\varepsilon-\nabla_\Gamma v\|_{L^2(\Gamma)}^2\right)dt \\
    &\leq c_T\varepsilon(\varepsilon^\alpha+R_\varepsilon) \leq c_T\varepsilon R_\varepsilon.
  \end{align*}
  Combining this inequality and \eqref{Pf_DUeCv:L2} we obtain \eqref{E:Diff_Ue_Cv}.
\end{proof}

We also compare the derivative of $u^\varepsilon$ in the normal direction of $\Gamma$ and the constant extension of $v$.
Recall that for a function $\varphi$ on $\Omega_\varepsilon$ and $x\in\Omega_\varepsilon$ we set
\begin{align*}
  \partial_n\varphi(x) = (\bar{n}(x)\cdot\nabla)\varphi(x) = \frac{d}{dr}\Bigl(\varphi(y+rn(y))\Bigr)\Big|_{r=d(x)} \quad (y = \pi(x)\in\Gamma)
\end{align*}
and call $\partial_n\varphi$ the derivative of $\varphi$ in the normal direction of $\Gamma$.

\begin{theorem} \label{T:Diff_DnUe_Cv}
  Under the assumptions of Theorem \ref{T:Diff_Mu_V}, let $\partial_nu^\varepsilon$ be the derivative of $u^\varepsilon$ in the normal direction of $\Gamma$.
  Then
  \begin{multline} \label{E:Diff_DnUe_Cv}
    \int_0^T\left(\left\|\overline{P}\partial_nu^\varepsilon(t)+\overline{Wv}(t)\right\|_{L^2(\Omega_\varepsilon)}^2+\left\|\partial_nu^\varepsilon(t)\cdot\bar{n}-\frac{1}{\bar{g}}\bar{v}(t)\cdot\overline{\nabla_\Gamma g}\right\|_{L^2(\Omega_\varepsilon)}^2\right)dt \\
    \leq c_T\varepsilon\left\{\delta(\varepsilon)^2+\|M_\tau u_0^\varepsilon-v_0\|_{L^2(\Gamma)}^2+\|M_\tau\mathbb{P}_\varepsilon f^\varepsilon-f\|_{L^\infty(0,\infty;H^{-1}(\Gamma,T\Gamma))}^2\right\}
  \end{multline}
  for all $T>0$, where $c_T>0$ is a constant depending only on $T$ and $\delta(\varepsilon)$ is given by \eqref{E:DMuV_Re}.
  Therefore, setting
  \begin{align*}
    V := -Wv+\frac{1}{g}(v\cdot\nabla_\Gamma g)n \quad\text{on}\quad \Gamma\times(0,\infty)
  \end{align*}
  we have (note that $Wv$ is tangential on $\Gamma$)
  \begin{align*}
    \lim_{\varepsilon\to0}\varepsilon^{-1}\left\|\partial_nu^\varepsilon-\overline{V}\right\|_{L^2(0,T;L^2(\Omega_\varepsilon))}^2 = 0
  \end{align*}
  for all $T>0$ provided \eqref{E:StCo_Data}.
\end{theorem}

\begin{proof}
  We fix $T>0$ and suppress the argument $t\in[0,T]$.
  Noting that
  \begin{align*}
    \overline{P}\partial_nu^\varepsilon+\overline{Wv} = \Bigl(\overline{P}\partial_nu^\varepsilon+\overline{W}u^\varepsilon\Bigr)-\Bigl(\overline{W}u^\varepsilon-\overline{WM_\tau u^\varepsilon}\Bigr)-\Bigl(\overline{WM_\tau u^\varepsilon}-\overline{Wv}\Bigr)
  \end{align*}
  in $\Omega_\varepsilon$, we define
  \begin{align*}
    J_1 &:= \left\|\overline{P}\partial_nu^\varepsilon+\overline{W}u^\varepsilon\right\|_{L^2(\Omega_\varepsilon)}, \\
    J_2 &:= \left\|\overline{W}u^\varepsilon-\overline{WM_\tau u^\varepsilon}\right\|_{L^2(\Omega_\varepsilon)}, \\
    J_3 &:= \left\|\overline{WM_\tau u^\varepsilon}-\overline{Wv}\right\|_{L^2(\Omega_\varepsilon)}.
  \end{align*}
  Since $u^\varepsilon\in D(A_\varepsilon)\subset H^2(\Omega_\varepsilon)^3$ satisfies \eqref{E:Bo_Slip}, we can use \eqref{E:PDnU_WU} to $J_1$ to have
  \begin{align*}
    J_1 \leq c\varepsilon\|u^\varepsilon\|_{H^2(\Omega_\varepsilon)}.
  \end{align*}
  Also, noting that $W$ is bounded on $\Gamma$, we apply \eqref{E:AveT_Diff_Dom} to $J_2$ and \eqref{E:Con_Lp} to $J_3$ to get
  \begin{align*}
    J_2 \leq c\left\|u^\varepsilon-\overline{M_\tau u^\varepsilon}\right\|_{L^2(\Omega_\varepsilon)} \leq c\varepsilon\|u^\varepsilon\|_{H^1(\Omega_\varepsilon)}, \quad J_3 \leq c\varepsilon^{1/2}\|M_\tau u^\varepsilon-v\|_{L^2(\Gamma)}.
  \end{align*}
  From these inequalities we deduce that
  \begin{align*}
    \left\|\overline{P}\partial_nu^\varepsilon+\overline{Wv}\right\|_{L^2(\Omega_\varepsilon)} &\leq J_1+J_2+J_3 \\
    &\leq c\varepsilon^{1/2}\left(\varepsilon^{1/2}\|u^\varepsilon\|_{H^2(\Omega_\varepsilon)}+\|M_\tau u^\varepsilon-v\|_{L^2(\Gamma)}\right).
  \end{align*}
  We also observe by \eqref{E:G_Inf}, \eqref{E:Con_Lp}, and \eqref{E:DnU_N_Ave} that
  \begin{align*}
    &\left\|\partial_nu^\varepsilon\cdot\bar{n}-\frac{1}{\bar{g}}\bar{v}\cdot\overline{\nabla_\Gamma g}\right\|_{L^2(\Omega_\varepsilon)} \\
    &\qquad \leq \left\|\partial_nu^\varepsilon\cdot\bar{n}-\frac{1}{\bar{g}}\overline{M_\tau u^\varepsilon}\cdot\overline{\nabla_\Gamma g}\right\|_{L^2(\Omega_\varepsilon)}+\left\|\frac{1}{\bar{g}}\Bigl(\overline{M_\tau u^\varepsilon}-\bar{v}\Bigr)\cdot\overline{\nabla_\Gamma g}\right\|_{L^2(\Omega_\varepsilon)} \\
    &\qquad \leq c\varepsilon^{1/2}\left(\varepsilon^{1/2}\|u^\varepsilon\|_{H^2(\Omega_\varepsilon)}+\|M_\tau u^\varepsilon-v\|_{L^2(\Gamma)}\right).
  \end{align*}
  Hence we get \eqref{E:Diff_DnUe_Cv} by integrating the square of the above inequalities over $(0,T)$ and using \eqref{E:Est_Ueps} and \eqref{E:Diff_Mu_V} as in the proof of Theorem \ref{T:Diff_Ue_Cv}.
\end{proof}

\begin{remark} \label{R:Diff_DnUe_Cv}
  In the estimate \eqref{E:Diff_DnUe_Cv} for the derivative $\partial_nu^\varepsilon$ of the velocity of the bulk fluid in the normal direction of $\Gamma$, the Weingarten map $W$ represents the curvatures of the limit surface $\Gamma$.
  On the other hand, the functions $g_0$ and $g_1$ with $g=g_1-g_0$ are used to define the inner and outer boundaries of the curved thin domain $\Omega_\varepsilon$.
  Therefore, roughly speaking, the tangential component (with respect to $\Gamma$) of $\partial_nu^\varepsilon$ depends only on the shape of $\Gamma$, while the geometry of the boundaries of $\Omega_\varepsilon$ affects only the normal component of $\partial_nu^\varepsilon$.
\end{remark}

\begin{appendix}
\section{Notations on vectors and matrices} \label{S:Ap_Vec}
In this appendix we fix notations on vectors and matrices.
For $m\in\mathbb{N}$ we consider a vector $a\in\mathbb{R}^m$ as a column vector
\begin{align*}
  a =
  \begin{pmatrix}
    a_1 \\ \vdots \\ a_m
  \end{pmatrix}
  = (a_1, \cdots, a_m)^T
\end{align*}
and denote the $i$-th component of $a$ by $a_i$ or sometimes by $a^i$ or $[a]_i$ for $i=1,\dots,m$.
A matrix $A\in\mathbb{R}^{l\times m}$ with $l,m\in\mathbb{N}$ is expressed as
\begin{align*}
  A = (A_{ij})_{i,j} =
  \begin{pmatrix}
    A_{11} & \cdots & A_{1m} \\
    \vdots & & \vdots \\
    A_{l1} & \cdots & A_{lm}
  \end{pmatrix}.
\end{align*}
For $i=1,\dots,l$ and $j=1,\dots,m$ the $(i,j)$-entry of $A$ is denoted by $A_{ij}$ or sometimes by $[A]_{ij}$.
We write $A^T$ for the transpose of $A$.
When $l=m$, we denote the symmetric part of $A$ by $A_S:=(A+A^T)/2$ and the $m\times m$ identity matrix by $I_m$.
We define the tensor product of $a\in\mathbb{R}^l$ and $b\in\mathbb{R}^m$ with $l,m\in\mathbb{N}$ by
\begin{align*}
  a\otimes b := (a_ib_j)_{i,j} =
  \begin{pmatrix}
    a_1b_1 & \cdots & a_1b_m \\
    \vdots & & \vdots \\
    a_lb_1 & \cdots & a_lb_m
  \end{pmatrix}, \quad
  a =
  \begin{pmatrix}
    a_1 \\ \vdots \\ a_l
  \end{pmatrix}, \quad
  b =
  \begin{pmatrix}
    b_1 \\ \vdots \\ b_m
  \end{pmatrix}.
\end{align*}
For three-dimensional vector fields $u=(u_1,u_2,u_3)^T$ and $\varphi$ on an open set in $\mathbb{R}^3$ let
\begin{gather*}
  \nabla u :=
  \begin{pmatrix}
    \partial_1u_1 & \partial_1u_2 & \partial_1u_3 \\
    \partial_2u_1 & \partial_2u_2 & \partial_2u_3 \\
    \partial_3u_1 & \partial_3u_2 & \partial_3u_3
  \end{pmatrix}, \quad
  |\nabla^2u|^2 := \sum_{i,j,k=1}^3|\partial_i\partial_ju_k|^2 \quad\left(\partial_i := \frac{\partial}{\partial x_i}\right), \\
  (\varphi\cdot\nabla)u :=
  \begin{pmatrix}
    \varphi\cdot\nabla u_1 \\
    \varphi\cdot\nabla u_2 \\
    \varphi\cdot\nabla u_3
  \end{pmatrix}
  = (\nabla u)^T\varphi.
\end{gather*}
We define the inner product of $A,B\in\mathbb{R}^{3\times3}$ and the norm of $A$ by
\begin{align*}
  A: B := \mathrm{tr}[A^TB] = \sum_{i=1}^3AE_i\cdot BE_i, \quad |A| := \sqrt{A:A},
\end{align*}
where $\{E_1,E_2,E_3\}$ is an orthonormal basis of $\mathbb{R}^3$.
Note that $A:B$ does not depend on a choice of $\{E_1,E_2,E_3\}$.
In particular, taking the standard basis of $\mathbb{R}^3$ we get
\begin{align*}
  A:B = \sum_{i,j=1}^3A_{ij}B_{ij} = B:A = A^T:B^T, \quad AB:C = A:CB^T = B:A^TC
\end{align*}
for $A,B,C\in\mathbb{R}^{3\times3}$.
Also, for $a,b\in\mathbb{R}^3$ we have $|a\otimes b|=|a||b|$.

\section{Auxiliary results on a closed surface} \label{S:Ap_Aux}
The purpose of this subsection is to provide auxiliary results used in the proof of Lemma \ref{L:Necas_Surf} and to establish Lemmas \ref{L:WTGr_Van}--\ref{L:Poin_Surf_Lp}.
We assume that $\Gamma$ is a $C^2$ closed, connected, and oriented surface in $\mathbb{R}^3$ and use the notations in Section \ref{SS:Pre_Surf}.

First we give auxiliary lemmas for calculations under a local coordinate system of $\Gamma$.
Lemmas \ref{L:Metric}--\ref{L:Lp_Loc} below are shown by direct calculations of differential geometry.
For the precise proofs of them we refer to our first paper \cite{Miu_NSCTD_01}.

\begin{lemma}[{\cite{Miu_NSCTD_01}*{Lemma B.1}}] \label{L:Metric}
  Let $U$ be an open set in $\mathbb{R}^2$, $\mu\colon U\to\Gamma$ a $C^2$ local parametrization of $\Gamma$, and $\mathcal{K}$ a compact subset of $U$.
  Then there exists a constant $c>0$ such that
  \begin{align} \label{E:Mu_Bound}
    |\partial_{s_i}\mu(s)| \leq c, \quad |\partial_{s_i}\partial_{s_j}\mu(s)| \leq c \quad\text{for all}\quad s\in\mathcal{K},\,i,j=1,2.
  \end{align}
  We define the Riemannian metric $\theta=(\theta_{ij})_{i,j}$ of $\Gamma$ by
  \begin{align} \label{E:Def_Met}
    \theta(s) := \nabla_s\mu(s)\{\nabla_s\mu(s)\}^T, \quad s\in U, \quad \nabla_s\mu :=
    \begin{pmatrix}
      \partial_{s_1}\mu_1 & \partial_{s_1}\mu_2 & \partial_{s_1}\mu_3 \\
      \partial_{s_2}\mu_1 & \partial_{s_2}\mu_2 & \partial_{s_2}\mu_3
    \end{pmatrix}
  \end{align}
  and denote by $\theta^{-1}=(\theta^{ij})_{i,j}$ the inverse matrix of $\theta$.
  Then
  \begin{align} \label{E:Metric}
      |\theta^k(s)| \leq c, \quad |\partial_{s_i}\theta^k(s)| \leq c, \quad c^{-1} \leq \det\theta(s) \leq c
  \end{align}
  for all $s\in\mathcal{K}$, $i=1,2$, and $k=\pm1$.
\end{lemma}

From now on, we always write $\theta=(\theta_{ij})_{i,j}$ and $\theta^{-1}=(\theta^{ij})_{i,j}$ for the Riemannian metric given by \eqref{E:Def_Met} and its inverse matrix.

\begin{lemma}[{\cite{Miu_NSCTD_01}*{Lemma B.2}}] \label{L:TGr_DG}
  Let $U$ be an open set in $\mathbb{R}^2$ and $\mu\colon U\to\Gamma$ a $C^2$ local parametrization of $\Gamma$.
  For $\eta\in C^1(\Gamma)$ let $\eta^\flat:=\eta\circ\mu$ on $U$.
  Then
  \begin{align} \label{E:TGr_DG}
    \nabla_\Gamma\eta(\mu(s)) = \sum_{i,j=1}^2\theta^{ij}(s)\partial_{s_i}\eta^\flat(s)\partial_{s_j}\mu(s), \quad s\in U,
  \end{align}
  where $\nabla_\Gamma\eta$ is the tangential gradient of $\eta$ defined by \eqref{E:Def_TGr}.
\end{lemma}

\begin{lemma}[{\cite{Miu_NSCTD_01}*{Lemma B.3}}] \label{L:DivG_DG}
  Let $U$ be an open set in $\mathbb{R}^2$ and $\mu\colon U\to\Gamma$ a $C^2$ local parametrization of $\Gamma$.
  For $X\in C^1(\Gamma,T\Gamma)$ let
  \begin{align*}
    X^i(s) := \sum_{j=1}^2\theta^{ij}(s)\partial_{s_j}\mu(s)\cdot X(\mu(s)), \quad s\in U, \, i=1,2.
  \end{align*}
  Then the surface divergence of $X$ defined by \eqref{E:Def_DivG} is locally of the form
  \begin{align} \label{E:DivG_DG}
    \mathrm{div}_\Gamma X(\mu(s)) = \frac{1}{\sqrt{\det\theta(s)}}\sum_{i=1}^2\partial_{s_i}\Bigl(X^i(s)\sqrt{\det\theta(s)}\Bigr), \quad s\in U.
  \end{align}
\end{lemma}

Note that Lemmas \ref{L:TGr_DG} and \ref{L:DivG_DG} are not trivial since the tangential gradient and the surface divergence on $\Gamma$ are defined under the fixed coordinate system of $\mathbb{R}^3$.

\begin{lemma}[{\cite{Miu_NSCTD_01}*{Lemma B.4}}] \label{L:Lp_Loc}
  Let $U$ be an open set in $\mathbb{R}^2$, $\mu\colon U\to\Gamma$ a $C^2$ local parametrization of $\Gamma$, and $\mathcal{K}$ a compact subset of $U$.
  For $p\in[1,\infty]$ if $\eta\in L^p(\Gamma)$ is supported in $\mu(\mathcal{K})$, then $\eta^\flat:=\eta\circ\mu\in L^p(U)$ and
  \begin{align} \label{E:Lp_Loc}
    c^{-1}\|\eta^\flat\|_{L^p(U)} \leq \|\eta\|_{L^p(\Gamma)} \leq c\|\eta^\flat\|_{L^p(U)}.
  \end{align}
  If in addition $\eta\in W^{1,p}(\Gamma)$, then $\eta^\flat\in W^{1,p}(U)$, \eqref{E:TGr_DG} holds in $L^p(U)^3$, and
  \begin{align} \label{E:W1p_Loc}
    c^{-1}\|\nabla_s\eta^\flat\|_{L^p(U)} \leq \|\nabla_\Gamma\eta\|_{L^p(\Gamma)} \leq c\|\nabla_s\eta^\flat\|_{L^p(U)}.
  \end{align}
  Here $\nabla_s\eta^\flat=(\partial_{s_1}\eta^\flat,\partial_{s_2}\eta^\flat)^T$ is the gradient of $\eta^\flat$ in $s\in\mathbb{R}^2$.
\end{lemma}

Next we present the proofs of Lemmas \ref{L:WTGr_Van}--\ref{L:Poin_Surf_Lp}.

\begin{proof}[Proof of Lemma \ref{L:WTGr_Van}]
  Let $p\in[1,\infty]$ and $\eta\in W^{1,p}(\Gamma)$.
  Suppose first that $\eta$ is constant on $\Gamma$.
  Then for all $v\in C^1(\Gamma)^3$ we see by \eqref{E:IbP_DivG} and \eqref{E:Def_WTD} that
  \begin{align*}
    \int_\Gamma\nabla_\Gamma\eta\cdot v\,d\mathcal{H}^2 = -\eta\int_\Gamma\{\mathrm{div}_\Gamma v+(v\cdot n)H\}\,d\mathcal{H}^2 = 0.
  \end{align*}
  Hence $\nabla_\Gamma\eta=0$ in $L^p(\Gamma)^3$.

  Next we prove that $\eta$ is constant on $\Gamma$ if $\nabla_\Gamma\eta=0$ in $L^p(\Gamma)^3$.
  Since $\Gamma$ is closed and connected, it is sufficient to show that $\eta$ is constant on some open neighborhood in $\Gamma$ of each $y\in\Gamma$.
  Let $U$ be an open set in $\mathbb{R}^2$ and $\mu\colon U\to\Gamma$ a local parametrization of $\Gamma$ such that $y\in\mu(U)$.
  Then we can take a bounded and connected open subset $U_1$ of $U$ and a function $\xi^\flat\in C_c^1(U)$ such that
  \begin{align*}
    y\in\mu(U_1), \quad U_1 \subset \mathcal{K} := \mathrm{supp}\,\xi^\flat \subset U, \quad \xi^\flat = 1 \quad\text{on}\quad U_1.
  \end{align*}
  Note that $\mathcal{K}$ is compact in $U$.
  We set $\xi(\mu(s)):=\xi^\flat(s)$ for $s\in U$ and extend $\xi$ to $\Gamma$ by zero outside $\mu(U)$.
  Then
  \begin{align*}
    \xi\in C^1(\Gamma), \quad \mathrm{supp}\,\xi = \mu(\mathcal{K}) \subset \mu(U), \quad \xi = 1 \quad\text{on}\quad \mu(U_1),
  \end{align*}
  and $\xi\eta\in W^{1,p}(\Gamma)$ and it is supported in $\mu(\mathcal{K})$.
  Hence Lemma \ref{L:Lp_Loc} implies
  \begin{align*}
    (\xi\eta)^\flat := (\xi\eta)\circ\mu\in W^{1,p}(U)
  \end{align*}
  and \eqref{E:TGr_DG} holds in $L^p(U)^3$ with $\eta$ and $\eta^\flat$ replaced by $\xi\eta$ and $(\xi\eta)^\flat$.
  Moreover, since
  \begin{gather*}
    \nabla_\Gamma(\xi\eta) = \eta\nabla_\Gamma\xi+\xi\nabla_\Gamma\eta = 0 \quad\text{on}\quad \mu(U_1), \\
    \text{i.e.}\quad \{\nabla_\Gamma(\xi\eta)\}\circ\mu = 0 \quad\text{on}\quad U_1
  \end{gather*}
  by $\xi=1$ on $\mu(U_1)$ and $\nabla_\Gamma\eta=0$ on $\Gamma$, it follows from \eqref{E:TGr_DG} that
  \begin{align*}
    0 = [\{\nabla_\Gamma(\xi\eta)\}\circ\mu]\cdot\partial_{s_k}\mu = \sum_{i,j=1}^2\theta^{ij}\theta_{jk}\partial_{s_i}(\xi\eta)^\flat = \partial_{s_k}(\xi\eta)^\flat \quad\text{on}\quad U_1
  \end{align*}
  for $k=1,2$.
  Hence $(\xi\eta)^\flat$ is constant on $U_1$, i.e. $\xi\eta$ is constant on $\mu(U_1)$.
  By this fact and $\xi=1$ on $\mu(U_1)$ we conclude that $\eta$ is constant on the open neighborhood $\mu(U_1)$ in $\Gamma$ of $y\in\Gamma$.
\end{proof}

\begin{proof}[Proof of Lemma \ref{L:RK_Surf}]
  Let $p\in[1,\infty]$ and $\{\eta_k\}_{k=1}^\infty$ be a bounded sequence in $W^{1,p}(\Gamma)$.
  Since $\Gamma$ is closed, we can take a finite number of bounded open sets in $\mathbb{R}^2$ and local parametrizations of $\Gamma$
  \begin{align*}
    U_j \subset \mathbb{R}^2, \quad \mu_j\colon U_j\to\Gamma, \quad j=1,\dots,j_0
  \end{align*}
  such that $\{\mu_j(U_j)\}_{j=1}^{j_0}$ is an open covering of $\Gamma$.
  Let $\{\xi_j\}_{j=1}^{j_0}$ be a partition of unity on $\Gamma$ subordinate to $\{\mu_j(U_j)\}_{j=1}^{j_0}$.
  Also, for $j=1,\dots,j_0$ let $\mathcal{K}_j$ be a compact subset of $U_j$ such that $\xi_j$ is supported in $\mu_j(\mathcal{K}_j)$.
  For $j=1,\dots,j_0$ and $k\in\mathbb{N}$ we set
  \begin{align*}
    \eta_{j,k}(y) := \xi_j(y)\eta_k(y), \quad y\in \Gamma, \quad \eta_{j,k}^\flat(s) := \eta_{j,k}(\mu_j(s)), \quad s\in U_j.
  \end{align*}
  Then $\{\eta_{j,k}\}_{k=1}^\infty$ is bounded in $W^{1,p}(\Gamma)$ for each $j=1,\dots,j_0$.
  Moreover, since
  \begin{align*}
    \mathrm{supp}\,\eta_{j,k} \subset \mu_j(\mathcal{K}_j), \quad\text{i.e.}\quad \mathrm{supp}\,\eta_{j,k}^\flat \subset \mathcal{K}_j \subset U_j, \quad k\in\mathbb{N},
  \end{align*}
  we see by Lemma \ref{L:Lp_Loc} that $\{\eta_{j,k}^\flat\}_{k=1}^\infty$ is bounded in $W_0^{1,p}(U_j)$ for each $j=1,\dots,j_0$.
  Thus, using the compact embeddings $W_0^{1,p}(U_j)\hookrightarrow L^p(U_j)$, $j=1,\dots,j_0$ repeatedly, we can take a strictly increasing sequence $\{k(l)\}_{l=1}^\infty$ in $\mathbb{N}$ and functions $\eta_j^\flat\in L^p(U_j)$, $j=1,\dots,j_0$ such that
  \begin{align} \label{Pf_RK:Conv}
    \lim_{l\to\infty}\|\eta_j^\flat-\eta_{j,k(l)}^\flat\|_{L^p(U_j)} = 0 \quad\text{for all}\quad j=1,\dots,j_0.
  \end{align}
  Moreover, since $\eta_{j,k}^\flat$ is supported in $\mathcal{K}_j$ for all $k\in\mathbb{N}$,
  \begin{align*}
    \|\eta_j^\flat\|_{L^p(U_j\setminus\mathcal{K}_j)} = \|\eta_j^\flat-\eta_{j,k(l)}^\flat\|_{L^p(U_j\setminus\mathcal{K}_j)} \to 0 \quad\text{as}\quad l\to\infty
  \end{align*}
  and thus $\eta_j^\flat=0$ on $U_j\setminus\mathcal{K}_j$, i.e. $\eta_j^\flat$ is supported in $\mathcal{K}_j$ for each $j=1,\dots,j_0$.
  Now we define a function $\eta_j$ on $\mu_j(U_j)\subset\Gamma$ by
  \begin{align*}
    \eta_j(\mu_j(s)) := \eta_j^\flat(s), \quad s\in U_j
  \end{align*}
  and extend it to $\Gamma$ by zero outside $\mu_j(U_j)$ for $j=1,\dots,j_0$.
  Then since
  \begin{align*}
    \|\eta_j\|_{L^p(\Gamma)}^p &= \int_{U_j}|\eta_j^\flat|^p\sqrt{\det\theta_j}\,ds = \int_{\mathcal{K}_j}|\eta_j^\flat|^p\sqrt{\det\theta_j}\,ds, \quad p\neq\infty, \\
    \|\eta_j\|_{L^\infty(\Gamma)} &= \|\eta_j^\flat\|_{L^\infty(U_j)},
  \end{align*}
  where $\theta_j:=\nabla_s\mu_j(\nabla_s\mu_j)^T$ on $U_j$, we have $\eta_j\in L^p(\Gamma)$ by \eqref{E:Metric} and $\eta_j^\flat\in L^p(U_j)$ for all $j=1,\dots,j_0$.
  Hence $\eta:=\sum_{j=1}^{j_0}\eta_j\in L^p(\Gamma)$ and, since
  \begin{align*}
    \eta-\eta_{k(l)} = \sum_{j=1}^{j_0}(\eta_j-\eta_{j,k(l)}) \quad\text{on}\quad \Gamma, \quad \mathrm{supp}\,\eta_j, \, \mathrm{supp}\,\eta_{j,k(l)} \subset \mu_j(\mathcal{K}_j),
  \end{align*}
  we observe by \eqref{E:Lp_Loc} and \eqref{Pf_RK:Conv} that
  \begin{align*}
    \|\eta-\eta_{k(l)}\|_{L^p(\Gamma)} &\leq \sum_{j=1}^{j_0}\|\eta_j-\eta_{j,k(l)}\|_{L^p(\Gamma)} \leq c\sum_{j=1}^{j_0}\|\eta_j^\flat-\eta_{j,k(l)}^\flat\|_{L^p(U_j)} \to 0
  \end{align*}
  as $l\to\infty$.
  Therefore, $\{\eta_k\}_{k=1}^\infty$ has a subsequence $\{\eta_{k(l)}\}_{l=1}^\infty$ that converges strongly in $L^p(\Gamma)$.
\end{proof}

\begin{proof}[Proof of Lemma \ref{L:Poin_Surf_Lp}]
  Let $p\in[1,\infty]$.
  We prove \eqref{E:Poin_Surf_Lp} by contradiction.
  Assume to the contrary that there exists a sequence $\{\eta_k\}_{k=1}^\infty$ in $W^{1,p}(\Gamma)$ such that
  \begin{align*}
    \|\eta_k\|_{L^p(\Gamma)} > k\|\nabla_\Gamma\eta_k\|_{L^p(\Gamma)}, \quad \int_\Gamma\eta_k\,d\mathcal{H}^2 = 0, \quad k\in\mathbb{N}.
  \end{align*}
  Replacing $\eta_k$ with $\eta_k/\|\eta_k\|_{L^p(\Gamma)}$ we may assume $\|\eta_k\|_{L^p(\Gamma)}=1$ and
  \begin{align} \label{Pf_PSp:Contra_Gr}
    \|\nabla_\Gamma\eta_k\|_{L^p(\Gamma)} < \frac{1}{k}, \quad \int_\Gamma\eta_k\,d\mathcal{H}^2 = 0, \quad k\in\mathbb{N},
  \end{align}
  and thus $\{\eta_k\}_{k=1}^\infty$ is bounded in $W^{1,p}(\Gamma)$.
  Then $\{\eta_k\}_{k=1}^\infty$ converges (up to a subsequence) to some $\eta\in L^p(\Gamma)$ strongly in $L^p(\Gamma)$ by Lemma \ref{L:RK_Surf} and thus
  \begin{align} \label{Pf_PSp:Conv_Lp}
    \|\eta\|_{L^p(\Gamma)} = \lim_{k\to\infty}\|\eta_k\|_{L^p(\Gamma)} = 1.
  \end{align}
  Let $p'\in[1,\infty]$ satisfy $1/p+1/p'=1$.
  Then, for $\xi\in C^1(\Gamma)$ and $i=1,2,3$,
  \begin{multline*}
    \left|\int_\Gamma\eta(\underline{D}_i\xi+\xi Hn_i)\,d\mathcal{H}^2-\int_\Gamma\eta_k(\underline{D}_i\xi+\xi Hn_i)\,d\mathcal{H}^2\right| \\
    \leq c\|\eta-\eta_k\|_{L^p(\Gamma)}\|\xi\|_{W^{1,p'}(\Gamma)} \to 0 \quad\text{as}\quad k\to\infty
  \end{multline*}
  by the strong convergence of $\{\eta_k\}_{k=1}^\infty$ to $\eta$ in $L^p(\Gamma)$ and
  \begin{align*}
    \left|\int_\Gamma\eta_k(\underline{D}_i\xi+\xi Hn_i)\,d\mathcal{H}^2\right| &= \left|\int_\Gamma(\underline{D}_i\eta_k)\xi\,d\mathcal{H}^2\right| \\
    &\leq \|\underline{D}_i\eta_k\|_{L^p(\Gamma)}\|\xi\|_{L^{p'}(\Gamma)} \\
    &\leq \|\nabla_\Gamma\eta_k\|_{L^p(\Gamma)}\|\xi\|_{L^{p'}(\Gamma)} \to 0 \quad\text{as}\quad k\to\infty
  \end{align*}
  by \eqref{E:Def_WTD} and the first inequality of \eqref{Pf_PSp:Contra_Gr}.
  Hence
  \begin{align*}
    \int_\Gamma\eta(\underline{D}_i\xi+\xi Hn_i)\,d\mathcal{H}^2 = 0
  \end{align*}
  for all $\xi\in C^1(\Gamma)$ and $i=1,2,3$.
  By this equality and the definition of the weak tangential derivative in $L^p(\Gamma)$ (see \eqref{E:Def_WTD}) we have $\eta\in W^{1,p}(\Gamma)$ and
  \begin{align*}
    \underline{D}_i\eta = 0 \quad\text{in}\quad L^p(\Gamma), \, i=1,2,3.
  \end{align*}
  Thus $\eta$ is constant on $\Gamma$ by Lemma \ref{L:WTGr_Van}.
  Moreover, since
  \begin{align*}
    \int_\Gamma\eta\,d\mathcal{H}^2 = \lim_{k\to\infty}\int_\Gamma\eta_k\,d\mathcal{H}^2 = 0
  \end{align*}
  by the strong convergence of $\{\eta_k\}_{k=1}^\infty$ to $\eta$ in $L^p(\Gamma)$ and the second equality of \eqref{Pf_PSp:Contra_Gr} (note that $\Gamma$ is compact), we obtain $\eta=0$ on $\Gamma$.
  This contradicts with \eqref{Pf_PSp:Conv_Lp} and thus \eqref{E:Poin_Surf_Lp} is valid.
\end{proof}

\section{Formulas related to the viscous term in the surface Navier--Stokes equations} \label{S:Ap_Visc}
We present formulas for differential operators on a closed surface in $\mathbb{R}^3$ related to the viscous term in the surface Navier--Stokes equations.
The formulas given in this appendix, except for \eqref{E:Vec_Sph} and the second equality of \eqref{E:Ric_Cur}, are also valid for a hypersurface in a higher dimensional Euclidean space with easy modifications of the proofs.
Some results in this appendix were also shown in \cite{DudMitMit06}, where the authors expressed the Lam\'{e} operator and related differential operators on a hypersurface globally in a fixed coordinate system of the ambient Euclidean space.
Here we try to give more simple proofs than those in \cite{DudMitMit06}.

Throughout this appendix we assume that $\Gamma$ is a closed, connected, and oriented surface in $\mathbb{R}^3$ of class $C^3$ and employ the notations in Section \ref{SS:Pre_Surf}.
In what follows, we sometimes abuse notations from differential geometry and use expressions which may be not standard in differential geometry in order to carry out calculations on $\Gamma$ in the fixed coordinate system of $\mathbb{R}^3$.

\subsection{Notations from differential geometry} \label{SS:ApV_Not}
We introduce some notations from differential geometry.
For details, we refer to \cites{Lee13,Lee18}.

Let $T_y\Gamma$ and $T_y^\ast\Gamma$ be the tangent and cotangent spaces of $\Gamma$ at $y\in\Gamma$ and
\begin{align*}
  T\Gamma = \coprod_{y\in\Gamma}T_y\Gamma, \quad T^\ast\Gamma = \coprod_{y\in\Gamma}T_y^\ast\Gamma
\end{align*}
the tangent and cotangent bundles of $\Gamma$.
Here we consider $T_y\Gamma$ as a two-dimensional subspace of $\mathbb{R}^3$ for each $y\in\Gamma$.
Also, let $T_y^\ast\Gamma\otimes T_y^\ast\Gamma$ be the tensor product of $T_y^\ast\Gamma$ with itself, which is identified with the space of all bilinear forms on $T_y\Gamma$, and
\begin{align*}
  T^\ast\Gamma\otimes T^\ast\Gamma = \coprod_{y\in\Gamma}T_y^\ast\Gamma\otimes T_y^\ast\Gamma
\end{align*}
the bundle of $(0,2)$-tensors on $\Gamma$.
As in Section \ref{SS:Pre_Surf}, for $m=0,1,2$ let
\begin{align*}
  C^m(\Gamma,T\Gamma) = \{X\in C^m(\Gamma)^3 \mid \text{$X(y)\in T_y\Gamma$ for all $y\in\Gamma$}\}
\end{align*}
be the space of all $C^m$ tangential vector fields on $\Gamma$.
We also say that
\begin{align*}
  \omega\in C^m(\Gamma,T^\ast\Gamma), \quad S\in C^m(\Gamma,T^\ast\Gamma\otimes T^\ast\Gamma)
\end{align*}
if $\omega$ and $S$ are mappings $\omega\colon\Gamma\to T^\ast\Gamma$ and $S\colon\Gamma\to T^\ast\Gamma\otimes T^\ast\Gamma$ satisfying
\begin{gather*}
  \omega(y)\in T_y^\ast\Gamma, \quad S(y)\in T_y^\ast\Gamma\otimes T_y^\ast\Gamma \quad\text{for all}\quad y\in\Gamma, \\
  \omega(\cdot)(X(\cdot)),S(\cdot)(X(\cdot),Y(\cdot))\in C^m(\Gamma) \quad\text{for all}\quad X,Y\in C^m(\Gamma,T\Gamma).
\end{gather*}
In other words, $\omega$ and $S$ are $C^m$ sections of $T^\ast\Gamma$ and $T^\ast\Gamma\otimes T^\ast\Gamma$.
We write $\omega_y$ and $S_y$ instead of $\omega(y)$ and $S(y)$ for the values of $\omega$ and $S$ at $y\in\Gamma$ in the sequel.

Let $\theta$ be the Riemannian metric of $\Gamma$ induced by the Euclidean metric of $\mathbb{R}^3$, i.e.
\begin{align*}
  \theta_y(X_y,Y_y) := X_y\cdot Y_y, \quad X_y,Y_y \in T_y\Gamma \subset \mathbb{R}^3
\end{align*}
for $y\in\Gamma$.
We use the same notation $\theta$ for the metrics on $T^\ast\Gamma$ and $T^\ast\Gamma\otimes T^\ast\Gamma$, i.e. we define inner products on $T_y^\ast\Gamma$ and $T_y^\ast\Gamma\otimes T_y^\ast\Gamma$ for $y\in\Gamma$ by
\begin{align} \label{E:Def_InTe}
  \begin{alignedat}{3}
    \theta_y(\omega_y,\zeta_y) &:= \sum_{i=1}^2\omega_y(\tau_i)\zeta_y(\tau_i), &\quad &\omega_y,\zeta_y\in T_y^\ast\Gamma, \\
    \theta_y(S_y,T_y) &:= \sum_{i,j=1}^2S_y(\tau_i,\tau_j)T_y(\tau_i,\tau_j), &\quad &S_y,T_y\in T_y^\ast\Gamma\otimes T_y^\ast\Gamma,
  \end{alignedat}
\end{align}
where $\{\tau_1,\tau_2\}$ is an orthonormal basis of $T_y\Gamma$.
Note that the above definitions are independent of a choice of $\{\tau_1,\tau_2\}$.

Let $O$ be a relatively open subset of $\Gamma$.
If $O$ is sufficiently small, then we can take $C^2$ vector fields $\tau_1$ and $\tau_2$ on $O$ such that $\{\tau_1(y),\tau_2(y)\}$ is an orthonormal basis of $T_y\Gamma$ for all $y\in O$ by the $C^3$-regularity of $\Gamma$.
We call the pair $\{\tau_1,\tau_2\}$ of such vector fields a local orthonormal frame for $T\Gamma$ on $O$.

For $y\in\Gamma$ and $X_y\in T_y\Gamma$ we define $\Theta_y(X_y)\in T_y^\ast\Gamma$ by
\begin{align*}
  \Theta_y(X_y)(Y_y) := \theta_y(X_y,Y_y) = X_y\cdot Y_y, \quad Y_y\in T_y\Gamma.
\end{align*}
Then $\Theta_y\colon T_y\Gamma\to T_y^\ast\Gamma$ is a linear operator and its inverse is given by
\begin{align*}
  \Theta_y^{-1}(\omega_y) := \sum_{i=1}^2\omega_y(\tau_i)\tau_i, \quad \omega_y\in T_y^\ast\Gamma,
\end{align*}
where $\{\tau_1,\tau_2\}$ is an orthonormal basis of $T_y\Gamma$.
We easily observe that
\begin{align} \label{E:The_Pr}
  \begin{alignedat}{2}
    \theta_y(X_y,Y_y) &= \theta_y(\Theta_y(X_y),\Theta_y(Y_y)), &\quad X_y,Y_y&\in T_y\Gamma, \\
    \theta_y(\omega_y,\zeta_y) &= \theta_y(\Theta_y^{-1}(\omega_y),\Theta_y^{-1}(\zeta_y)), &\quad \omega_y,\zeta_y&\in T_y^\ast\Gamma.
  \end{alignedat}
\end{align}
Also, for $X\in C(\Gamma,T\Gamma)$ and $\omega\in C(\Gamma,T^\ast\Gamma)$ we define
\begin{align*}
  [\Theta(X)]_y := \Theta_y(X(y)) \in T_y^\ast\Gamma, \quad [\Theta^{-1}(\omega)](y) := \Theta_y^{-1}(\omega_y) \in T_y\Gamma, \quad y\in\Gamma.
\end{align*}
Then for $m=0,1,2$ we can consider $\Theta$ and $\Theta^{-1}$ as operators
\begin{align*}
  \Theta\colon C^m(\Gamma,T\Gamma)\to C^m(\Gamma,T^\ast\Gamma), \quad \Theta^{-1}\colon C^m(\Gamma,T^\ast\Gamma)\to C^m(\Gamma,T\Gamma)
\end{align*}
by taking a local orthonormal frame for $T\Gamma$ on a relatively open subset of $\Gamma$.

Next we introduce notations related to differential forms.
Let
\begin{align*}
  \Lambda^2T_y^\ast\Gamma := \{S_y\in T_y^\ast\Gamma\otimes T_y^\ast\Gamma \mid \text{$S_y(X_y,Y_y)=-S_y(Y_y,X_y)$ for all $X_y,Y_y\in T_y\Gamma$}\}
\end{align*}
be the space of all skew-symmetric bilinear forms on $T_y\Gamma$ for $y\in\Gamma$ and
\begin{align*}
  \Lambda^2T^\ast\Gamma = \coprod_{y\in\Gamma}\Lambda^2T_y^\ast\Gamma \, (\subset T^\ast\Gamma\otimes T^\ast\Gamma)
\end{align*}
the bundle of skew-symmetric $(0,2)$-tensors on $\Gamma$.
For $m=0,1,2$ we define
\begin{gather*}
  C^m(\Gamma,\Lambda^0T^\ast\Gamma) := C^m(\Gamma), \quad C^m(\Gamma,\Lambda^1T^\ast\Gamma) := C^m(\Gamma,T^\ast\Gamma), \\
  C^m(\Gamma,\Lambda^2T^\ast\Gamma) := \{S\in C^m(\Gamma,T^\ast\Gamma\otimes T^\ast\Gamma) \mid \text{$S_y\in \Lambda^2T_y^\ast\Gamma$ for all $y\in\Gamma$}\}
\end{gather*}
and for $k=0,1,2$ we call an element of $C^m(\Gamma,\Lambda^kT^\ast\Gamma)$ a $k$-form (of class $C^m$).
Let $y\in\Gamma$ and $\{\tau_1,\tau_2\}$ be an orthonormal basis of $T_y\Gamma$.
Then since
\begin{align*}
  S_y(\tau_1,\tau_1) = S_y(\tau_2,\tau_2) = 0, \quad S_y(\tau_2,\tau_1) = -S_y(\tau_1,\tau_2), \quad S_y\in\Lambda^2T_y^\ast\Gamma,
\end{align*}
the restriction on $\Lambda^2T_y^\ast\Gamma$ of the inner product of $T_y^\ast\Gamma\otimes T_y^\ast\Gamma$ is
\begin{align*}
  \theta_y(S_y,T_y) = 2S_y(\tau_1,\tau_2)T_y(\tau_1,\tau_2), \quad S_y,T_y\in \Lambda^2T_y^\ast\Gamma.
\end{align*}
Based on this equality, we introduce a new metric $\hat{\theta}$ on $\Lambda^2T^\ast\Gamma$ by
\begin{align} \label{E:Def_InWe}
  \hat{\theta}_y(S_y,T_y) := S_y(\tau_1,\tau_2)T_y(\tau_1,\tau_2), \quad S_y,T_y\in \Lambda^2T_y^\ast\Gamma
\end{align}
for each $y\in\Gamma$.
Note that, if we define the wedge product $\omega_y\wedge\zeta_y\in\Lambda^2T_y^\ast\Gamma$ of $\omega_y,\zeta_y\in T_y^\ast\Gamma$ by
\begin{align*}
  (\omega_y\wedge\zeta_y)(X_y,Y_y) := \omega_y(X_y)\zeta_y(Y_y)-\zeta_y(X_y)\omega_y(Y_y), \quad X_y,Y_y\in T_y\Gamma,
\end{align*}
then we observe by direct calculations that
\begin{align*}
  \hat{\theta}_y(\omega_y^1\wedge\omega_y^2,\zeta_y^1\wedge\zeta_y^2) = \det\left[\bigl(\theta(\omega_y^i,\zeta_y^j)\bigr)_{i,j}\right], \quad \omega_y^1,\omega_y^2,\zeta_y^1,\zeta_y^2\in T_y^\ast\Gamma.
\end{align*}
Hence $\hat{\theta}$ agrees with the standard metric on $\Lambda^2T^\ast\Gamma$ used in differential geometry.
For $k=0,1,2$ we define an inner product of $k$-forms by
\begin{align} \label{E:Def_InDF}
  \begin{alignedat}{2}
    \langle \eta,\xi\rangle_0 &:= (\eta,\xi)_{L^2(\Gamma)}, &\quad \eta,\xi&\in C(\Gamma,\Lambda^0T^\ast\Gamma) = C(\Gamma), \\
    \langle \omega,\zeta\rangle_1 &:= \int_\Gamma\theta(\omega,\zeta)\,d\mathcal{H}^2, &\quad \omega,\zeta&\in C(\Gamma,\Lambda^1T^\ast\Gamma) = C(\Gamma,T^\ast\Gamma), \\
    \langle S,T\rangle_2 &:= \int_\Gamma\hat{\theta}(S,T)\,d\mathcal{H}^2, &\quad S,T&\in C(\Gamma,\Lambda^2T^\ast\Gamma).
  \end{alignedat}
\end{align}
Let $d_\Gamma$ be the exterior derivative on $\Gamma$.
We consider $d_\Gamma$ as a mapping
\begin{align*}
  d_\Gamma\colon C^1(\Gamma,\Lambda^kT^\ast\Gamma)\to C(\Gamma,\Lambda^{k+1}T^\ast\Gamma), \quad k=0,1
\end{align*}
locally defined as follows: let $U$ be an open set in $\mathbb{R}^2$ and $\mu\colon U\to\Gamma$ a $C^3$ local parametrization of $\Gamma$.
For $s=(s^1,s^2)\in U$ let $\{(d_\Gamma s^1)_{\mu(s)},(d_\Gamma s^2)_{\mu(s)}\}$ be the dual basis for the basis $\{\partial_{s^1}\mu(s),\partial_{s^2}\mu(s)\}$ of $T_{\mu(s)}\Gamma$, i.e.
\begin{align} \label{E:Dual_B}
  (d_\Gamma s^i)_{\mu(s)}\in T_{\mu(s)}^\ast\Gamma, \quad (d_\Gamma s^i)_{\mu(s)}(\partial_{s^j}\mu(s)) = \delta_{ij}, \quad i,j=1,2,
\end{align}
where $\delta_{ij}$ is the Kronecker delta.
Then, for $\eta\in C^1(\Gamma,\Lambda^0T^\ast\Gamma)=C^1(\Gamma)$,
\begin{align} \label{E:Def_ExtZ}
  (d_\Gamma\eta)_{\mu(s)} := \sum_{i=1,2}\frac{\partial(\eta\circ\mu)}{\partial s^i}(s)(d_\Gamma s^i)_{\mu(s)}, \quad s\in U.
\end{align}
Also, when $\omega\in C^1(\Gamma,\Lambda^1T^\ast\Gamma)=C^1(\Gamma,T^\ast\Gamma)$ is locally of the form
\begin{align*}
  \omega_{\mu(s)} = \sum_{i=1,2}\omega_i(s)(d_\Gamma s^i)_{\mu(s)}, \quad s\in U,
\end{align*}
where $\omega_i(s):=\omega_{\mu(s)}(\partial_{s^i}\mu(s))$ for $i=1,2$, then
\begin{align} \label{E:Def_ExtO}
  (d_\Gamma \omega)_{\mu(s)} := \left(\frac{\partial\omega_2}{\partial s^1}(s)-\frac{\partial\omega_1}{\partial s^2}(s)\right)(d_\Gamma s^1)_{\mu(s)}\wedge(d_\Gamma s^2)_{\mu(s)}, \quad s\in U.
\end{align}
We write $d_\Gamma^\ast\colon C^1(\Gamma,\Lambda^{k+1}T^\ast\Gamma)\to C(\Gamma,\Lambda^kT^\ast\Gamma)$ for the formal adjoint of $d_\Gamma$:
\begin{align} \label{E:Def_FoEx}
  \langle d_\Gamma^\ast\omega,\eta\rangle_k := \langle\omega,d_\Gamma\eta\rangle_{k+1}, \quad \omega\in C^1(\Gamma,\Lambda^{k+1}T^\ast\Gamma), \, \eta\in C^1(\Gamma,\Lambda^kT^\ast\Gamma)
\end{align}
for $k=0,1$.

\subsection{Curvatures} \label{SS:ApV_Cuv}
For $X\in C^1(\Gamma,T\Gamma)$ and $Y\in C(\Gamma,T\Gamma)$ let
\begin{align*}
  \overline{\nabla}_YX = P(Y\cdot\nabla_\Gamma)X\in C(\Gamma,T\Gamma)
\end{align*}
be the covariant derivative of $X$ along $Y$.
Then since
\begin{align*}
  (Y\cdot\nabla_\Gamma)X\cdot n = Y\cdot\nabla_\Gamma(X\cdot n)-X\cdot(Y\cdot\nabla_\Gamma)n = X\cdot W^TY = WX\cdot Y
\end{align*}
on $\Gamma$ by $X\cdot n=0$ and $W=-\nabla_\Gamma n$ on $\Gamma$, we have
\begin{align} \label{E:Gauss}
  (Y\cdot\nabla_\Gamma)X = \overline{\nabla}_YX+(WX\cdot Y)n \quad\text{on}\quad \Gamma,
\end{align}
which is called the Gauss formula (see e.g. \cites{Ch15,Lee18}).
Moreover, the mapping
\begin{align*}
  \overline{\nabla}\colon C^1(\Gamma,T\Gamma)\times C(\Gamma,T\Gamma)\to C(\Gamma,T\Gamma), \quad (X,Y) \mapsto \overline{\nabla}_YX
\end{align*}
is the Riemannian (or Levi-Civita) connection on $\Gamma$, i.e. the following formulas are valid on $\Gamma$ (see \cite{Miu_NSCTD_01}*{Lemma D.2} for the proofs):
\begin{itemize}
  \item For $X\in C^1(\Gamma,T\Gamma)$, $Y,Z\in C(\Gamma,T\Gamma)$, and $\eta,\xi\in C(\Gamma)$,
  \begin{align*}
    \overline{\nabla}_{\eta Y+\xi Z}X = \eta\overline{\nabla}_YX+\xi\overline{\nabla}_ZX.
  \end{align*}
  \item For $X\in C^1(\Gamma,T\Gamma)$, $Y\in C(\Gamma,T\Gamma)$, and $\eta\in C^1(\Gamma)$,
  \begin{align*}
    \overline{\nabla}_Y(\eta X) = (Y\cdot\nabla_\Gamma\eta)X+\eta\overline{\nabla}_YX.
  \end{align*}
  \item For $X,Y\in C^1(\Gamma,T\Gamma)$ and $Z\in C(\Gamma,T\Gamma)$,
  \begin{align*}
    Z\cdot\nabla_\Gamma(X\cdot Y) = \overline{\nabla}_ZX\cdot Y+X\cdot\overline{\nabla}_ZY.
  \end{align*}
  \item For $X,Y\in C^1(\Gamma,T\Gamma)$ and $\eta\in C^2(\Gamma)$,
  \begin{align*}
    X\cdot\nabla_\Gamma(Y\cdot\nabla_\Gamma\eta)-Y\cdot\nabla_\Gamma(X\cdot\nabla_\Gamma\eta) = \Bigl(\overline{\nabla}_XY-\overline{\nabla}_YX\Bigr)\cdot\nabla_\Gamma\eta.
  \end{align*}
\end{itemize}
Note that the last formula stands for the torsion-free condition
\begin{align*}
  [X,Y] := XY-YX = \overline{\nabla}_XY-\overline{\nabla}_YX,
\end{align*}
where $[X,Y]$ is the Lie bracket of $X$ and $Y$.

Using the Riemannian connection $\overline{\nabla}$ we define the curvature tensor $R$ by\begin{align*}
  R(X,Y)Z := \overline{\nabla}_X\overline{\nabla}_YZ-\overline{\nabla}_Y\overline{\nabla}_XZ-\overline{\nabla}_{\overline{\nabla}_XY-\overline{\nabla}_YX}Z \quad\text{on}\quad \Gamma
\end{align*}
for $X,Y\in C^1(\Gamma,T\Gamma)$ and $Z\in C^2(\Gamma,T\Gamma)$.
It is known that
\begin{align} \label{E:Sym_CuT}
  R(X_1,X_2)Y_1\cdot Y_2 = R(X_2,X_1)Y_2\cdot Y_1 = R(Y_1,Y_2)X_1\cdot X_2 \quad\text{on}\quad \Gamma
\end{align}
for $X_1,X_2,Y_1,Y_2\in C^2(\Gamma,T\Gamma)$ (see e.g. \cite{Lee18}*{Proposition 7.12}), which also follows from \eqref{E:Cur_Ten} given below.
The Ricci curvature $\mathrm{Ric}$ is defined by
\begin{align*}
  \mathrm{Ric}(X,Y) := \mathrm{tr}(Z\mapsto R(Z,X)Y) \quad\text{on}\quad \Gamma
\end{align*}
for $X\in C^1(\Gamma,T\Gamma)$ and $Y\in C^2(\Gamma,T\Gamma)$.
More precisely, for a relatively open subset $O$ of $\Gamma$ and a local orthonormal frame $\{\tau_1,\tau_2\}$ for $T\Gamma$ on $O$ we set
\begin{align*}
  \mathrm{Ric}(X,Y) = \sum_{i=1,2}R(\tau_i,X)Y\cdot\tau_i \quad\text{on}\quad O.
\end{align*}
Then for $X,Y\in C^2(\Gamma,T\Gamma)$ we observe by \eqref{E:Sym_CuT} that
\begin{align*}
  \mathrm{Ric}(X,Y) = \sum_{i=1,2}R(X,\tau_i)\tau_i\cdot Y = \sum_{i=1,2}R(\tau_i,Y)X\cdot\tau_i = \mathrm{Ric}(Y,X) \quad\text{on}\quad O.
\end{align*}
Based on the first equality, we also define $\mathrm{Ric}(X)\in C(\Gamma,T\Gamma)$ locally by
\begin{align*}
  \mathrm{Ric}(X) := \sum_{i=1,2}R(X,\tau_i)\tau_i \quad\text{on}\quad O
\end{align*}
for $X\in C^1(\Gamma,T\Gamma)$.
The vector fields $R(X,Y)Z$ and $\mathrm{Ric}(X)$ are defined intrinsically, but we can express them in the fixed coordinate system of $\mathbb{R}^3$.

\begin{lemma} \label{L:Cur_Ten}
  Let $X,Y\in C^1(\Gamma,T\Gamma)$ and $Z\in C^2(\Gamma,T\Gamma)$.
  Then
  \begin{align} \label{E:Cur_Ten}
    R(X,Y)Z = (WZ\cdot Y)WX-(WZ\cdot X)WY \quad\text{on}\quad \Gamma.
  \end{align}
  Also, for all $X\in C^1(\Gamma,T\Gamma)$ we have
  \begin{align} \label{E:Ric_Cur}
    \mathrm{Ric}(X) = (HW-W^2)X = KX \quad\text{on}\quad \Gamma.
  \end{align}
\end{lemma}

The formula \eqref{E:Ric_Cur} was proved in \cite{DudMitMit06}*{Theorem 6.2}, where the authors used an equality equivalent to \eqref{E:Cur_Ten} following from the Gauss equation for a hypersurface in $\mathbb{R}^k$, $k\geq2$ (see e.g. \cite{Lee18}) and the fact that $\mathbb{R}^k$ has zero curvature.
We show \eqref{E:Cur_Ten} by applying \eqref{E:TD_Exc} and \eqref{E:Gauss} and then use it to get \eqref{E:Ric_Cur} below.

\begin{proof}
  We define $\xi_Y:=WZ\cdot Y$ and carry out calculations on $\Gamma$.
  Since
  \begin{align*}
    \overline{\nabla}_YZ = (Y\cdot\nabla_\Gamma)Z-\xi_Yn = \sum_{l=1}^3Y_l\underline{D}_lZ-\xi_Yn
  \end{align*}
  by \eqref{E:Gauss}, we apply $X\cdot\nabla_\Gamma=\sum_{k=1}^3X_k\underline{D}_k$ to both sides to get
  \begin{multline*}
    (X\cdot\nabla_\Gamma)\overline{\nabla}_YZ = \sum_{k,l=1}^3\{X_k(\underline{D}_kY_l)(\underline{D}_lZ)+X_kY_l\underline{D}_k\underline{D}_lZ\} \\
    -(X\cdot\nabla_\Gamma\xi_Y)n-\xi_Y(X\cdot\nabla_\Gamma)n.
  \end{multline*}
  We further multiply both sides by $P$ and use
  \begin{gather*}
    Pn = 0, \quad P(X\cdot\nabla_\Gamma)n = -PW^TX = -WX, \\
    \sum_{k,l=1}^3X_k(\underline{D}_kY_l)(\underline{D}_lZ) = [\{(X\cdot\nabla_\Gamma)Y\}\cdot\nabla_\Gamma]Z = \Bigl(\overline{\nabla}_XY\cdot\nabla_\Gamma\Bigr)Z,
  \end{gather*}
  where the last equality follows from \eqref{E:TGrM_Surf}, to deduce that
  \begin{align*}
    \overline{\nabla}_X\overline{\nabla}_YZ &= P(X\cdot\nabla_\Gamma)\overline{\nabla}_YZ \\
    &= P\Bigl(\overline{\nabla}_XY\cdot\nabla_\Gamma\Bigr)Z+\xi_YWX+P\sum_{k,l=1}^3X_kY_l\underline{D}_k\underline{D}_lZ \\
    &= \overline{\nabla}_{\overline{\nabla}_XY}Z+\xi_YWX+P\sum_{k,l=1}^3X_kY_l\underline{D}_k\underline{D}_lZ.
  \end{align*}
  In the same way we get
  \begin{align*}
    \overline{\nabla}_Y\overline{\nabla}_XZ = \overline{\nabla}_{\overline{\nabla}_YX}Z+\xi_XWY+P\sum_{k,l=1}^3Y_kX_l\underline{D}_k\underline{D}_lZ,
  \end{align*}
  where $\xi_X:=WZ\cdot X$, and thus
  \begin{align} \label{Pf_CuTe:RXY}
    R(X,Y)Z = \xi_YWX-\xi_XWY+P\sum_{k,l=1}^3(X_kY_l-Y_kX_l)\underline{D}_k\underline{D}_lZ.
  \end{align}
  Moreover, for $m=1,2,3$ the $m$-th component of
  \begin{align*}
    \sum_{k,l=1}^3(X_kY_l-Y_kX_l)\underline{D}_k\underline{D}_lZ = \sum_{k,l=1}^3X_kY_l(\underline{D}_k\underline{D}_lZ-\underline{D}_l\underline{D}_kZ)
  \end{align*}
  is of the form
  \begin{align*}
    &\sum_{k,l=1}^3X_kY_l(\underline{D}_k\underline{D}_lZ_m-\underline{D}_l\underline{D}_kZ_m) \\
    &\qquad = \sum_{k,l=1}^3X_kY_l([W\nabla_\Gamma Z_m]_kn_l-[W\nabla_\Gamma Z_m]_ln_k) \\
    &\qquad = (X\cdot W\nabla_\Gamma Z_m)(Y\cdot n)-(X\cdot n)(Y\cdot W\nabla_\Gamma Z_m) = 0
  \end{align*}
  by \eqref{E:TD_Exc} and $X\cdot n=Y\cdot n=0$.
  Thus the last term of \eqref{Pf_CuTe:RXY} vanishes and we obtain \eqref{E:Cur_Ten} (note that $\xi_X=WZ\cdot X$ and $\xi_Y=WZ\cdot Y$).

  Next we derive \eqref{E:Ric_Cur}.
  Let $\{\tau_1,\tau_2\}$ be an orthonormal frame for $T\Gamma$ on a relatively open subset $O$ of $\Gamma$.
  Then
  \begin{align*}
    H = \mathrm{tr}[W] = \sum_{i=1,2}W\tau_i\cdot\tau_i+Wn\cdot n = \sum_{i=1,2}W\tau_i\cdot\tau_i \quad\text{on}\quad O
  \end{align*}
  since $\{\tau_1,\tau_2,n\}$ is an orthonormal basis of $\mathbb{R}^3$ and $Wn=0$ on $\Gamma$.
  Noting that $WX$ is tangential on $\Gamma$ and $\{\tau_1,\tau_2\}$ is an orthonormal frame for $T\Gamma$, we deduce from the above equality, \eqref{E:Cur_Ten}, and $W^T=W$ on $\Gamma$ that
  \begin{align*}
    \mathrm{Ric}(X) &= \sum_{i=1,2}R(X,\tau_i)\tau_i = \sum_{i=1,2}\{(W\tau_i\cdot\tau_i)WX-(W\tau_i\cdot X)W\tau_i\} \\
    &= \mathrm{tr}[W]WX-W\sum_{i=1,2}(WX\cdot\tau_i)\tau_i = HWX-W^2X
  \end{align*}
  on $O$.
  Thus the first equality of \eqref{E:Ric_Cur} is valid.
  To show the second equality we fix and suppress $y\in\Gamma$.
  Since the real symmetric matrix $W$ has the eigenvalues zero, $\kappa_1$, and $\kappa_2$ with $Wn=0$, there exists an orthonormal basis $\{\tau_1,\tau_2,n\}$ of $\mathbb{R}^3$ such that $W\tau_i=\kappa_i\tau_i$ for $i=1,2$.
  Then since $X$ is tangential on $\Gamma$, we have
  \begin{align*}
    X = \sum_{i=1,2}(X\cdot\tau_i)\tau_i, \quad W^kX = \sum_{i=1,2}\kappa_i^k(X\cdot\tau_i)\tau_i, \quad k=1,2.
  \end{align*}
  From these equalities, $H=\kappa_1+\kappa_2$, and $K=\kappa_1\kappa_2$ it follows that
  \begin{align*}
    (HW-W^2)X &= (\kappa_1+\kappa_2)\sum_{i=1,2}\kappa_i(X\cdot\tau_i)\tau_i-\sum_{i=1,2}\kappa_i^2(X\cdot\tau_i)\tau_i \\
    &= \kappa_1\kappa_2\sum_{i=1,2}(X\cdot\tau_i)\tau_i = KX.
  \end{align*}
  Hence the second equality of \eqref{E:Ric_Cur} holds.
\end{proof}

\subsection{Laplace operators} \label{SS:ApV_La}
In this subsection we give formulas for Laplace operators on $\Gamma$.
First we show that the restriction on $\Gamma$ of the Laplace operator on $\mathbb{R}^3$ agrees with the Laplace--Beltrami operator on $\Gamma$ in an appropriate sense.

\begin{lemma} \label{L:Lap_Rest}
  For $\eta\in C^2(\Gamma)$ we have
  \begin{align} \label{E:Lap_Rest}
    \Delta\bar{\eta} = \Delta_\Gamma\eta \quad\text{on}\quad \Gamma,
  \end{align}
  where $\bar{\eta}=\eta\circ\pi$ is the constant extension of $\eta$ in the normal direction of $\Gamma$.
\end{lemma}

\begin{proof}
  For $x\in N$ let $R(x)$ be the inverse matrix of $I_3-d(x)\overline{W}(x)$ so that
  \begin{align} \label{Pf_LaRe:Gr}
    \nabla\bar{\eta}(x) = R(x)\overline{\nabla_\Gamma\eta}(x), \quad x\in N
  \end{align}
  by \eqref{E:ConDer_Dom}.
  Note that $R\in C^1(N)^{3\times3}$ by the $C^3$-regularity of $\Gamma$.
  We differentiate
  \begin{align*}
    R(x)\left\{I_3-d(x)\overline{W}(x)\right\} = I_3, \quad x\in N
  \end{align*}
  with respect to $x_i$, $i=1,2,3$ and set $x=y\in\Gamma$.
  Then
  \begin{align*}
    \partial_iR(y)-n_i(y)W(y) = 0, \quad\text{i.e.}\quad \partial_iR(y) = n_i(y)W(y), \quad y\in\Gamma
  \end{align*}
  by $d(y)=0$, $\partial_id(y)=n_i(y)$, and $R(y)=I_3$.
  Also, we observe by \eqref{E:ConDer_Surf} that
  \begin{align*}
    \Bigl[\partial_i\Bigl(\overline{\nabla_\Gamma\eta}\Bigr)\Bigr](y) = [\underline{D}_i(\nabla_\Gamma\eta)](y), \quad y\in\Gamma.
  \end{align*}
  We differentiate both sides of \eqref{Pf_LaRe:Gr} with respect to $x_i$, $i=1,2,3$, set $x=y\in\Gamma$, and use the above two equalities and $R(y)=I_3$ to obtain
  \begin{align*}
    [\partial_i(\nabla\bar{\eta})](y) = n_i(y)W(y)\nabla_\Gamma\eta(y)+[\underline{D}_i(\nabla_\Gamma\eta)](y), \quad y\in\Gamma.
  \end{align*}
  From this equality and $W^Tn=Wn=0$ on $\Gamma$ it follows that
  \begin{align*}
    \Delta\bar{\eta}(y) &= \sum_{i=1}^3\partial_i^2\bar{\eta}(y) = \sum_{i,j=1}^3n_i(y)W_{ij}(y)\underline{D}_j\eta(y)+\sum_{i=1}^3\underline{D}_i^2\eta(y) \\
    &= [W^Tn](y)\cdot\nabla_\Gamma\eta(y)+\Delta_\Gamma\eta(y) = \Delta_\Gamma\eta(y)
  \end{align*}
  for all $y\in\Gamma$.
  Thus \eqref{E:Lap_Rest} is valid.
\end{proof}

Next we deal with Laplace operators acting on tangential vector fields on $\Gamma$.
Let
\begin{align*}
  \Delta_H\colon C^2(\Gamma,\Lambda^1T^\ast\Gamma)\to C(\Gamma,\Lambda^1T^\ast\Gamma), \quad \Delta_H := -(d_\Gamma d_\Gamma^\ast+d_\Gamma^\ast d_\Gamma)
\end{align*}
be the Hodge Laplacian on $\Gamma$ (here we take the minus sign).
Using
\begin{align*}
  \Theta\colon C^m(\Gamma,T\Gamma)\to C^m(\Gamma,T^\ast\Gamma) = C^m(\Gamma,\Lambda^1T^\ast\Gamma), \quad m=0,1,2
\end{align*}
and its inverse $\Theta^{-1}$ we define
\begin{align*}
  \widetilde{\Delta}_H\colon C^2(\Gamma,T\Gamma)\to C(\Gamma,T\Gamma), \quad \widetilde{\Delta}_H := \Theta^{-1}\Delta_H\Theta
\end{align*}
and identity $\Delta_H$ with $\widetilde{\Delta}_H$ to apply $\Delta_H$ to tangential vector fields on $\Gamma$.
We consider the Riemannian connection $\overline{\nabla}$ as an operator
\begin{align*}
  \overline{\nabla}\colon C^1(\Gamma,T\Gamma)\to C(\Gamma,T^\ast\Gamma\otimes T^\ast\Gamma)
\end{align*}
which maps $X\in C^1(\Gamma,T\Gamma)$ to $\overline{\nabla}X\in C(\Gamma,T^\ast\Gamma\otimes T^\ast\Gamma)$ given by
\begin{align} \label{E:RC_Ten}
  \Bigl(\overline{\nabla}X\Bigr)(Y,Z) := \overline{\nabla}_ZX\cdot Y \quad\text{on}\quad \Gamma, \quad Y,Z\in C(\Gamma,T\Gamma)
\end{align}
and write $\overline{\nabla}^\ast\colon C^1(\Gamma,T^\ast\Gamma\otimes T^\ast\Gamma)\to C(\Gamma,T\Gamma)$ for the formal adjoint of $\overline{\nabla}$:
\begin{align} \label{E:Def_FoRC}
  \int_\Gamma\theta\Bigl(\overline{\nabla}^\ast S,X\Bigr)\,d\mathcal{H}^2 := \int_\Gamma\theta\Bigl(S,\overline{\nabla}X\Bigr)\,d\mathcal{H}^2
\end{align}
for $S\in C^1(\Gamma,T^\ast\Gamma\otimes T^\ast\Gamma)$ and $X\in C^1(\Gamma,T\Gamma)$.
Then we define
\begin{align*}
  \Delta_B\colon C^2(\Gamma,T\Gamma)\to C(\Gamma,T\Gamma), \quad \Delta_B := -\overline{\nabla}^\ast\overline{\nabla}
\end{align*}
and call $\Delta_B$ the Bochner Laplacian on $\Gamma$.
Note that $\Delta_H$ and $\Delta_B$ are defined intrinsically and map tangential vector fields on $\Gamma$ to tangential ones.
We can also consider the componentwise Laplace--Beltrami operator
\begin{align*}
  \Delta_\Gamma\colon C^2(\Gamma)^3\to C(\Gamma)^3, \quad v =
  \begin{pmatrix}
    v_1 \\
    v_2 \\
    v_3
  \end{pmatrix}
  \mapsto \Delta_\Gamma v =
  \begin{pmatrix}
    \Delta_\Gamma v_1 \\
    \Delta_\Gamma v_2 \\
    \Delta_\Gamma v_3
  \end{pmatrix}.
\end{align*}
Here $\Delta_\Gamma X$ is not tangential on $\Gamma$ in general even if $X\in C^2(\Gamma,T\Gamma)$.
Indeed,
\begin{align*}
  \Delta_\Gamma X\cdot n = \Delta_\Gamma(X\cdot n)+\mathrm{div}_\Gamma(WX)+W:\nabla_\Gamma X \quad\text{on}\quad \Gamma
\end{align*}
and only the first term on the right-hand side vanishes by $X\cdot n=0$ on $\Gamma$.
Let us establish relations between $\Delta_H$, $\Delta_B$, and $\Delta_\Gamma$ after giving three auxiliary lemmas.

\begin{lemma} \label{L:Exter}
  For $\eta\in C^1(\Gamma)=C^1(\Gamma,\Lambda^0T^\ast\Gamma)$ we have
  \begin{align} \label{E:Ext_Zero}
    d_\Gamma\eta = \Theta(\nabla_\Gamma\eta) \quad\text{on}\quad \Gamma.
  \end{align}
  Let $X\in C^1(\Gamma,T\Gamma)$.
  Then $\Theta(X)\in C^1(\Gamma,T^\ast\Gamma)=C^1(\Gamma,\Lambda^1T^\ast\Gamma)$ and
  \begin{align} \label{E:Ext_One}
    [d_\Gamma\Theta(X)](Y,Z) = \overline{\nabla}_YX\cdot Z-Y\cdot\overline{\nabla}_ZX \quad\text{on}\quad \Gamma
  \end{align}
  for all $Y,Z\in C(\Gamma,T\Gamma)$.
  We also have
  \begin{align} \label{E:Cod_One}
    d_\Gamma^\ast\Theta(X) = -\mathrm{div}_\Gamma X \quad\text{on}\quad \Gamma.
  \end{align}
\end{lemma}

\begin{proof}
  It is sufficient to show \eqref{E:Ext_Zero} and \eqref{E:Ext_One} on $\mu(U)$, where $U$ is an open set in $\mathbb{R}^2$ and $\mu\colon U\to\Gamma$ is a $C^3$ local parametrization of $\Gamma$.

  First we prove \eqref{E:Ext_Zero} for $\eta\in C^1(\Gamma)$.
  Let $s=(s^1,s^2)\in U$ and $X_{\mu(s)}\in T_{\mu(s)}\Gamma$.
  Since $\{\partial_{s^1}\mu(s),\partial_{s^2}\mu(s)\}$ is a basis of $T_{\mu(s)}\Gamma$, we can write
  \begin{align*}
    X_{\mu(s)} = X^1\partial_{s^1}\mu(s)+X^2\partial_{s^2}\mu(s), \quad X^1,X^2\in\mathbb{R}.
  \end{align*}
  Hence it follows from \eqref{E:Dual_B} and \eqref{E:Def_ExtZ} that
  \begin{align*}
    (d_\Gamma\eta)_{\mu(s)}(X_{\mu(s)}) = \sum_{i,j=1}^2X^j\frac{\partial(\eta\circ\mu)}{\partial s^i}(s)(d_\Gamma s^i)_{\mu(s)}(\partial_{s^j}\mu(s)) = \sum_{i=1,2}X^i\frac{\partial(\eta\circ\mu)}{\partial s^i}(s).
  \end{align*}
  Moreover, denoting by $\bar{\eta}=\eta\circ\pi$ the constant extension of $\eta$ in the normal direction of $\Gamma$, we observe by \eqref{E:ConDer_Surf} with $y=\mu(s)\in\Gamma$ that
  \begin{align} \label{Pf_Ext:Dsi}
    \frac{\partial(\eta\circ\mu)}{\partial s^i}(s) &= \frac{\partial}{\partial s^i}\Bigl(\bar{\eta}(\mu(s))\Bigr) = \partial_{s^i}\mu(s)\cdot\nabla\bar{\eta}(\mu(s)) = \partial_{s^i}\mu(s)\cdot\nabla_\Gamma\eta(\mu(s))
  \end{align}
  for $i=1,2$.
  Hence
  \begin{align*}
    (d_\Gamma\eta)_{\mu(s)}(X_{\mu(s)}) &= \sum_{i=1,2}X^i\partial_{s^i}\mu(s)\cdot\nabla_\Gamma\eta(\mu(s)) = X_{\mu(s)}\cdot \nabla_\Gamma\eta(\mu(s)) \\
    &= [\Theta(\nabla_\Gamma\eta)]_{\mu(s)}(X_{\mu(s)})
  \end{align*}
  for all $s\in U$ and $X_{\mu(s)}\in T_{\mu(s)}\Gamma$, which yields \eqref{E:Ext_Zero} on $\mu(U)$.

  Next we prove \eqref{E:Ext_One}.
  For $X=(X_1,X_2,X_3)^T\in C^1(\Gamma,T\Gamma)$ and $s\in U$ let
  \begin{gather*}
    [(\partial_{s^i}\mu(s)\cdot\nabla_\Gamma)X](\mu(s)) :=
    \begin{pmatrix}
      \partial_{s^i}\mu(s)\cdot\nabla_\Gamma X_1(\mu(s)) \\
      \partial_{s^i}\mu(s)\cdot\nabla_\Gamma X_2(\mu(s)) \\
      \partial_{s^i}\mu(s)\cdot\nabla_\Gamma X_3(\mu(s))
    \end{pmatrix}, \quad i=1,2
  \end{gather*}
  and
  \begin{align*}
    \Phi(s) := [(\partial_{s^1}\mu(s)\cdot\nabla_\Gamma)X](\mu(s))\cdot\partial_{s^2}\mu(s)-[(\partial_{s^2}\mu(s)\cdot\nabla_\Gamma)X](\mu(s))\cdot\partial_{s^1}\mu(s).
  \end{align*}
  Also, let $Y,Z\in C(\Gamma,T\Gamma)$ be locally of the form
  \begin{align} \label{Pf_Ext:YZ}
    Y(\mu(s)) = \sum_{i=1,2}Y^i(s)\partial_{s^i}\mu(s), \quad Z(\mu(s)) = \sum_{i=1,2}Z^i(s)\partial_{s^i}\mu(s)
  \end{align}
  for $s\in U$ with $Y^i(s),Z^i(s)\in\mathbb{R}$, $i=1,2$.
  Then since $Y$ and $Z$ are tangential on $\Gamma$, we have
  \begin{align*}
    \Bigl[\overline{\nabla}_YX\cdot Z\Bigr](\mu(s)) &= [(Y\cdot\nabla_\Gamma)X\cdot Z](\mu(s)) \\
    &= \sum_{i,j=1}^2Y^i(s)Z^j(s)[(\partial_{s^i}\mu(s)\cdot\nabla_\Gamma)X](\mu(s))\cdot\partial_{s^j}\mu(s)
  \end{align*}
  and a similar equality for $Y\cdot\overline{\nabla}_ZX$.
  Hence
  \begin{align} \label{Pf_Ext:As_Cov}
    \Bigl[\overline{\nabla}_YX\cdot Z-Y\cdot\overline{\nabla}_ZX\Bigr](\mu(s)) = \Phi(s)\{Y^1(s)Z^2(s)-Y^2(s)Z^1(s)\}
  \end{align}
  for $s\in U$.
  On the other hand, let
  \begin{align*}
    \omega_i(s) := [\Theta(X)]_{\mu(s)}(\partial_{s^i}\mu(s)) = X(\mu(s))\cdot\partial_{s^i}\mu(s), \quad s\in U, \, i=1,2.
  \end{align*}
  Then since $[\Theta(X)]_{\mu(s)}=\sum_{i=1,2}\omega_i(s)(d_\Gamma s^i)_{\mu(s)}$, we have
  \begin{align*}
    [d_\Gamma\Theta(X)]_{\mu(s)} = \left(\frac{\partial\omega_2}{\partial s^1}(s)-\frac{\partial\omega_1}{\partial s^2}(s)\right)(d_\Gamma s^1)_{\mu(s)}\wedge(d_\Gamma s^2)_{\mu(s)}, \quad s\in U
  \end{align*}
  by \eqref{E:Def_ExtO}.
  Also, it follows from \eqref{Pf_Ext:Dsi} with $\eta$ replaced by $X_k$, $k=1,2,3$ that
  \begin{align*}
    \frac{\partial\omega_i}{\partial s^j}(s) = [(\partial_{s^j}\mu(s)\cdot\nabla_\Gamma)X](\mu(s))\cdot\partial_{s^i}\mu(s)+X(\mu(s))\cdot\partial_{s^j}\partial_{s^i}\mu(s)
  \end{align*}
  for $s\in U$ and $i,j=1,2$.
  By this equality and $\partial_{s^j}\partial_{s^i}\mu(s)=\partial_{s^i}\partial_{s^j}\mu(s)$,
  \begin{align*}
    \frac{\partial\omega_2}{\partial s^1}(s)-\frac{\partial\omega_1}{\partial s^2}(s) = \Phi(s), \quad [d_\Gamma\Theta(X)]_{\mu(s)} = \Phi(s)(d_\Gamma s^1)_{\mu(s)}\wedge(d_\Gamma s^2)_{\mu(s)}
  \end{align*}
  for $s\in U$.
  Since $(d_\Gamma s^1)_{\mu(s)}\wedge(d_\Gamma s^2)_{\mu(s)}$ is a bilinear form on $T_{\mu(s)}\Gamma$ and
  \begin{align*}
    [(d_\Gamma s^1)_{\mu(s)}\wedge(d_\Gamma s^2)_{\mu(s)}](\partial_{s^i}\mu(s),\partial_{s^i}\mu(s)) &= 0, \quad i=1,2, \\
    [(d_\Gamma s^1)_{\mu(s)}\wedge(d_\Gamma s^2)_{\mu(s)}](\partial_{s^1}\mu(s),\partial_{s^2}\mu(s)) &= 1, \\
    [(d_\Gamma s^1)_{\mu(s)}\wedge(d_\Gamma s^2)_{\mu(s)}](\partial_{s^2}\mu(s),\partial_{s^1}\mu(s)) &= -1
  \end{align*}
  by \eqref{E:Dual_B}, we further observe by \eqref{Pf_Ext:YZ} that
  \begin{align*}
    [(d_\Gamma s^1)_{\mu(s)}\wedge(d_\Gamma s^2)_{\mu(s)}](Y(\mu(s)),Z(\mu(s))) = Y^1(s)Z^2(s)-Y^2(s)Z^1(s)
  \end{align*}
  and thus
  \begin{align} \label{Pf_Ext:DX}
    [d_\Gamma\Theta(X)]_{\mu(s)}(Y(\mu(s)),Z(\mu(s))) = \Phi(s)\{Y^1(s)Z^2(s)-Y^2(s)Z^1(s)\}
  \end{align}
  for $s\in U$.
  Therefore, we obtain \eqref{E:Ext_One} on $\mu(U)$ by \eqref{Pf_Ext:As_Cov} and \eqref{Pf_Ext:DX}.

  Now let us show \eqref{E:Cod_One}.
  For $X\in C^1(\Gamma,T\Gamma)$ and $\eta\in C^1(\Gamma)$ we have
  \begin{align*}
    \langle d_\Gamma^\ast\Theta(X),\eta\rangle_0 &= \langle\Theta(X),d_\Gamma\eta\rangle_1 = \langle\Theta(X),\Theta(\nabla_\Gamma\eta)\rangle_1 \\
    &= \int_\Gamma\theta(\Theta(X),\Theta(\nabla_\Gamma\eta))\,d\mathcal{H}^2
  \end{align*}
  by \eqref{E:Def_FoEx} and \eqref{E:Ext_Zero}.
  We apply \eqref{E:The_Pr} to the last term and use \eqref{E:IbP_TD} to get
  \begin{align*}
    \langle d_\Gamma^\ast\Theta(X),\eta\rangle_0 &= \int_\Gamma\theta(X,\nabla_\Gamma\eta)\,d\mathcal{H}^2 = \int_\Gamma X\cdot\nabla_\Gamma\eta\,d\mathcal{H}^2 \\
    &= -\int_\Gamma\{\mathrm{div}_\Gamma X+(X\cdot n)H\}\eta\,d\mathcal{H}^2 = \langle-\mathrm{div}_\Gamma X,\eta\rangle_0,
  \end{align*}
  where the last equality follows from $X\cdot n=0$ on $\Gamma$.
  Hence \eqref{E:Cod_One} holds.
\end{proof}

\begin{lemma} \label{L:Inn_LOF}
  Let $O$ be a relatively open subset of $\Gamma$ and $\{\tau_1,\tau_2\}$ a local orthonormal frame for $T\Gamma$ on $O$.
  Then for $X,Y\in C^1(\Gamma,T\Gamma)$ we have
  \begin{align}
    \nabla_\Gamma X:\nabla_\Gamma Y &= \sum_{i=1,2}\overline{\nabla}_iX\cdot\overline{\nabla}_iY+WX\cdot WY, \label{E:Inn_LOF} \\
    (\nabla_\Gamma X)^T:\nabla_\Gamma Y &= \sum_{i,j=1}^2\Bigl(\overline{\nabla}_iX\cdot\tau_j\Bigr)\Bigl(\overline{\nabla}_jY\cdot\tau_i\Bigr)  \label{E:InTr_LOF}
  \end{align}
  on $O$, where $\overline{\nabla}_i:=\overline{\nabla}_{\tau_i}$ for $i=1,2$.
\end{lemma}

\begin{proof}
  Throughout the proof we carry out calculations on $O$.
  Since $Pn=0$,
  \begin{align} \label{Pf_InL:TGrXn}
    (\nabla_\Gamma X)^Tn = (n\cdot\nabla_\Gamma)X = [(Pn)\cdot\nabla_\Gamma]X = 0
  \end{align}
  by \eqref{E:TGrM_Surf}.
  Also, noting that $\tau_i$ is tangential, we deduce from \eqref{E:Gauss} that
  \begin{align} \label{Pf_InL:TGrXt}
    (\nabla_\Gamma X)^T\tau_i = (\tau_i\cdot\nabla_\Gamma)X = \overline{\nabla}_iX+(WX\cdot\tau_i)n, \quad \overline{\nabla}_iX\cdot n = 0
  \end{align}
  for $i=1,2$.
  Since $\{\tau_1,\tau_2,n\}$ is an orthonormal basis of $\mathbb{R}^3$, we see by \eqref{Pf_InL:TGrXn} and \eqref{Pf_InL:TGrXt} that
  \begin{align*}
    \nabla_\Gamma X:\nabla_\Gamma Y &= (\nabla_\Gamma X)^T:(\nabla_\Gamma Y)^T \\
    &= \sum_{i=1,2}(\nabla_\Gamma X)^T\tau_i\cdot(\nabla_\Gamma Y)^T\tau_i+(\nabla_\Gamma X)^Tn\cdot(\nabla_\Gamma Y)^Tn \\
    &= \sum_{i=1,2}\Bigl\{\overline{\nabla}_iX+(WX\cdot\tau_i)n\Bigr\}\cdot\Bigl\{\overline{\nabla}_iY+(WY\cdot\tau_i)n\Bigr\} \\
    &= \sum_{i=1,2}\overline{\nabla}_iX\cdot\overline{\nabla}_iY+\sum_{i=1,2}(WX\cdot\tau_i)(WY\cdot\tau_i).
  \end{align*}
  Here the last term is equal to $WX\cdot WY$ since $WX$ and $WY$ are tangential and $\{\tau_1,\tau_2\}$ is a local orthonormal frame for $T\Gamma$.
  Hence \eqref{E:Inn_LOF} follows.

  To prove \eqref{E:InTr_LOF} we observe by \eqref{Pf_InL:TGrXt} and $\tau_i\cdot n=0$ that
  \begin{align*}
    (\nabla_\Gamma Y)\tau_i\cdot\tau_j = \tau_i\cdot(\nabla_\Gamma Y)^T\tau_j = \tau_i\cdot\overline{\nabla}_jY, \quad i,j=1,2.
  \end{align*}
  Since $(\nabla_\Gamma Y)\tau_i=P(\nabla_\Gamma Y)\tau_i$ is tangential and $\{\tau_1,\tau_2\}$ is a local orthonormal frame for $T\Gamma$, it follows from the above equality that
  \begin{align*}
    (\nabla_\Gamma Y)\tau_i = \sum_{j=1,2}\{(\nabla_\Gamma Y)\tau_i\cdot\tau_j\}\tau_j = \sum_{j=1,2}\Bigl(\overline{\nabla}_jY\cdot\tau_i\Bigr)\tau_j, \quad i=1,2.
  \end{align*}
  We deduce from this equality, \eqref{Pf_InL:TGrXn}, \eqref{Pf_InL:TGrXt}, and $\tau_j\cdot n=0$ that
  \begin{align*}
    (\nabla_\Gamma X)^T:\nabla_\Gamma Y &= \sum_{i=1,2}(\nabla_\Gamma X)^T\tau_i\cdot(\nabla_\Gamma Y)\tau_i+(\nabla_\Gamma X)^Tn\cdot(\nabla_\Gamma Y)n \\
    &= \sum_{i,j=1}^2\Bigl\{\overline{\nabla}_iX+(WX\cdot \tau_i)n\Bigr\}\cdot\Bigl\{\Bigl(\overline{\nabla}_jY\cdot\tau_i\Bigr)\tau_j\Bigr\} \\
    &= \sum_{i,j=1}^2\Bigl(\overline{\nabla}_iX\cdot\tau_j\Bigr)\Bigl(\overline{\nabla}_jY\cdot\tau_i\Bigr).
  \end{align*}
  Hence \eqref{E:InTr_LOF} is valid.
\end{proof}

\begin{lemma} \label{L:Inn_IbP}
  Let $X\in C^2(\Gamma,T\Gamma)$ and $Y\in C^1(\Gamma,T\Gamma)$.
  Then
  \begin{align}
    \int_\Gamma\nabla_\Gamma X:\nabla_\Gamma Y\,d\mathcal{H}^2 &= -\int_\Gamma\Delta_\Gamma X\cdot Y\,d\mathcal{H}^2, \label{E:Inn_IbP} \\
    \int_\Gamma(\nabla_\Gamma X)^T:\nabla_\Gamma Y\,d\mathcal{H}^2 &= -\int_\Gamma\{\nabla_\Gamma(\mathrm{div}_\Gamma X)+(HW-W^2)X\}\cdot Y\,d\mathcal{H}^2. \label{E:InTr_IbP}
  \end{align}
\end{lemma}

\begin{proof}
  We observe by \eqref{E:IbP_TD} that
  \begin{align*}
    \int_\Gamma\nabla_\Gamma X:\nabla_\Gamma Y\,d\mathcal{H}^2 &= \sum_{i,j=1}^3\int_\Gamma(\underline{D}_iX_j)(\underline{D}_iY_j)\,d\mathcal{H}^2 \\
    &= -\sum_{i,j=1}^3\int_\Gamma\{\underline{D}_i^2X_j+(\underline{D}_iX_j)Hn_i\}Y_j\,d\mathcal{H}^2 \\
    &= -\int_\Gamma\Delta_\Gamma X\cdot Y\,d\mathcal{H}^2-\sum_{j=1}^3\int_\Gamma(\nabla_\Gamma X_j\cdot n)HY_j\,d\mathcal{H}^2.
  \end{align*}
  Since $\nabla_\Gamma X_j\cdot n=0$ on $\Gamma$ for $j=1,2,3$, we get \eqref{E:Inn_IbP} by the above equality.

  Let us prove \eqref{E:InTr_IbP}.
  We again use \eqref{E:IbP_TD} to deduce that
  \begin{align} \label{Pf_InIbP:Spl}
    \int_\Gamma(\nabla_\Gamma X)^T:\nabla_\Gamma Y\,d\mathcal{H}^2 &= \sum_{i,j=1}^3\int_\Gamma(\underline{D}_jX_i)(\underline{D}_iY_j)\,d\mathcal{H}^2 = J_1+J_2,
  \end{align}
  where
  \begin{align*}
    J_1 := -\sum_{i,j=1}^3\int_\Gamma(\underline{D}_i\underline{D}_jX_i)Y_j\,d\mathcal{H}^2, \quad J_2 := -\sum_{i,j=1}^3\int_\Gamma(\underline{D}_jX_i)Hn_iY_j\,d\mathcal{H}^2.
  \end{align*}
  To $J_1$ we apply \eqref{E:TD_Exc} and
  \begin{align*}
    \sum_{i,j=1}^3[W\nabla_\Gamma X_i]_in_jY_j &= (Y\cdot n)\mathrm{tr}[W\nabla_\Gamma X] = 0, \\
    \sum_{i,j=1}^3[W\nabla_\Gamma X_i]_jn_iY_j &= \{W(\nabla_\Gamma X)n\}\cdot Y = W^2X\cdot Y
  \end{align*}
  on $\Gamma$ by $Y\cdot n=0$ on $\Gamma$ and \eqref{E:Grad_W}.
  Then
  \begin{align*}
    J_1 &= -\sum_{i,j=1}^3\int_\Gamma(\underline{D}_j\underline{D}_iX_i+[W\nabla_\Gamma X_i]_in_j-[W\nabla_\Gamma X_i]_jn_i)Y_j\,d\mathcal{H}^2 \\
    &= -\int_\Gamma\{\nabla_\Gamma(\mathrm{div}_\Gamma X)-W^2X\}\cdot Y\,d\mathcal{H}^2.
  \end{align*}
  Moreover, we use \eqref{E:Grad_W} to $J_2$ to get
  \begin{align*}
    J_2 = -\int_\Gamma H\{(\nabla_\Gamma X)n\cdot Y\}\,d\mathcal{H}^2 = -\int_\Gamma H(WX\cdot Y)\,d\mathcal{H}^2.
  \end{align*}
  Applying the above two equalities to \eqref{Pf_InIbP:Spl} we obtain \eqref{E:InTr_IbP}.
\end{proof}

\begin{lemma} \label{L:Vec_Lap}
  For $X\in C^2(\Gamma,T\Gamma)$ we have
  \begin{alignat}{3}
    \Delta_HX &= P(\Delta_\Gamma X)+(2W^2-HW)X &\quad &\text{on} &\quad &\Gamma, \label{E:HLap} \\
    \Delta_BX &= P(\Delta_\Gamma X)+W^2X &\quad &\text{on} &\quad &\Gamma. \label{E:BLap}
  \end{alignat}
\end{lemma}

\begin{proof}
  Let $X\in C^2(\Gamma,T\Gamma)$.
  By a localization argument with a partition of unity on $\Gamma$ we may assume that $X$ is supported in a relatively open subset $O$ of $\Gamma$ on which we can take a local orthonormal frame $\{\tau_1,\tau_2\}$ for $T\Gamma$.

  First we prove \eqref{E:HLap}.
  Let $Y\in C^1(\Gamma,T\Gamma)$.
  Noting that we identify
  \begin{align*}
    \Delta_H = -(d_\Gamma d_\Gamma^\ast+d_\Gamma^\ast d_\Gamma)\colon C^2(\Gamma,\Lambda^1 T^\ast\Gamma) \to C(\Gamma,\Lambda^1T^\ast\Gamma)
  \end{align*}
  with $\widetilde{\Delta}_H=\Theta^{-1}\Delta_H\Theta\colon C^2(\Gamma,T\Gamma)\to C(\Gamma,T\Gamma)$, we observe by \eqref{E:The_Pr} that
  \begin{align} \label{Pf_VL:HL_Spl}
    \begin{aligned}
      \int_\Gamma\Delta_HX\cdot Y\,d\mathcal{H}^2 &= \int_\Gamma\theta(\Theta^{-1}(\Delta_H\Theta(X)),Y)\,d\mathcal{H}^2 \\
      &= \int_\Gamma\theta(\Delta_H\Theta(X),\Theta(Y))\,d\mathcal{H}^2 = J_1+J_2,
    \end{aligned}
  \end{align}
  where (note that $d_\Gamma^\ast$ is the formal adjoint of $d_\Gamma$ given by \eqref{E:Def_FoEx})
  \begin{align*}
    J_1 &:= -\langle d_\Gamma d_\Gamma^\ast\Theta(X),\Theta(Y)\rangle_1 = -\langle d_\Gamma^\ast\Theta(X),d_\Gamma^\ast\Theta(Y)\rangle_0, \\
    J_2 &:= -\langle d_\Gamma^\ast d_\Gamma\Theta(X),\Theta(Y)\rangle_1 = -\langle d_\Gamma\Theta(X),d_\Gamma\Theta(Y)\rangle_2.
  \end{align*}
  We apply \eqref{E:Cod_One} to $J_1$ and then use \eqref{E:IbP_TD} and $Y\cdot n=0$ on $\Gamma$ to get
  \begin{align} \label{Pf_VL:HL_J1}
    \begin{aligned}
      J_1 &= -\int_\Gamma(\mathrm{div}_\Gamma X)(\mathrm{div}_\Gamma Y)\,d\mathcal{H}^2 \\
      &= \int_\Gamma\{\nabla_\Gamma(\mathrm{div}_\Gamma X)+(\mathrm{div}_\Gamma X)Hn\}\cdot Y\,d\mathcal{H}^2 \\
      &= \int_\Gamma\nabla_\Gamma(\mathrm{div}_\Gamma X)\cdot Y\,d\mathcal{H}^2.
    \end{aligned}
  \end{align}
  To compute $J_2$ we see that $d_\Gamma\Theta(X)$ is supported in $O$ since $X$ is so.
  Hence
  \begin{align*}
    J_2 = -\int_O\hat{\theta}(d_\Gamma\Theta(X),d_\Gamma\Theta(Y))\,d\mathcal{H}^2 = -\int_O[d_\Gamma\Theta(X)](\tau_1,\tau_2)[d_\Gamma\Theta(Y)](\tau_1,\tau_2)\,d\mathcal{H}^2
  \end{align*}
  by \eqref{E:Def_InWe} and \eqref{E:Def_InDF}.
  Moreover, it follows from \eqref{E:Ext_One} that
  \begin{align*}
    [d_\Gamma\Theta(X)](\tau_1,\tau_2) = \overline{\nabla}_1X\cdot\tau_2-\tau_1\cdot\overline{\nabla}_2X \quad\text{on}\quad O,
  \end{align*}
  where $\overline{\nabla}_i:=\overline{\nabla}_{\tau_i}$ for $i=1,2$, and thus
  \begin{multline*}
    [d_\Gamma\Theta(X)](\tau_1,\tau_2)[d_\Gamma\Theta(Y)](\tau_1,\tau_2) \\
    \begin{aligned}
      &= \sum_{i,j=1}^2\Bigl(\overline{\nabla}_iX\cdot\tau_j\Bigr)\Bigl(\overline{\nabla}_iY\cdot\tau_j\Bigr)-\sum_{i,j=1}^2\Bigl(\overline{\nabla}_iX\cdot\tau_j\Bigr)\Bigl(\overline{\nabla}_jY\cdot\tau_i\Bigr) \\
      &= \sum_{i=1,2}\overline{\nabla}_iX\cdot\overline{\nabla}_iY-\sum_{i,j=1}^2\Bigl(\overline{\nabla}_iX\cdot\tau_j\Bigr)\Bigl(\overline{\nabla}_jY\cdot\tau_i\Bigr)
    \end{aligned}
  \end{multline*}
  on $O$ (note that $\overline{\nabla}_iX$ and $\overline{\nabla}_iY$ are tangential and $\{\tau_1,\tau_2\}$ is a local orthonormal frame for $T\Gamma$ on $O$).
  We further apply \eqref{E:Inn_LOF} and \eqref{E:InTr_LOF} to the last line to get
  \begin{align*}
    [d_\Gamma\Theta(X)](\tau_1,\tau_2)[d_\Gamma\Theta(Y)](\tau_1,\tau_2) = \nabla_\Gamma X:\nabla_\Gamma Y-WX\cdot WY-(\nabla_\Gamma X)^T:\nabla_\Gamma Y
  \end{align*}
  on $O$.
  Hence we obtain (note that $X$ is supported in $O$)
  \begin{align} \label{Pf_VL:HL_J2}
    \begin{aligned}
      J_2 &= -\int_\Gamma\nabla_\Gamma X:\nabla_\Gamma Y\,d\mathcal{H}^2+\int_\Gamma WX\cdot WY\,d\mathcal{H}^2+\int_\Gamma(\nabla_\Gamma X)^T:\nabla_\Gamma Y\,d\mathcal{H}^2 \\
      &= \int_\Gamma\{\Delta_\Gamma X-\nabla_\Gamma(\mathrm{div}_\Gamma X)+(2W^2-HW)X\}\cdot Y\,d\mathcal{H}^2
    \end{aligned}
  \end{align}
  by \eqref{E:Inn_IbP}, \eqref{E:InTr_IbP}, and $W^T=W$ on $\Gamma$.
  From \eqref{Pf_VL:HL_Spl}--\eqref{Pf_VL:HL_J2} we deduce that
  \begin{align*}
    \int_\Gamma\Delta_HX\cdot Y\,d\mathcal{H}^2 = \int_\Gamma\{\Delta_\Gamma X+(2W^2-HW)X\}\cdot Y\,d\mathcal{H}^2
  \end{align*}
  for all $Y\in C^1(\Gamma,T\Gamma)$.
  Setting $Y:=Pv$ in this equality we further get
  \begin{align*}
    \int_\Gamma\Delta_HX\cdot v\,d\mathcal{H}^2 = \int_\Gamma\{P(\Delta_\Gamma X)+(2W^2-HW)X\}\cdot v\,d\mathcal{H}^2
  \end{align*}
  for all $v\in C^1(\Gamma)^3$ since $\Delta_HX$ and $(2W^2-HW)X$ are tangential on $\Gamma$.
  Hence we obtain \eqref{E:HLap} by the fundamental lemma of the calculus of variations.

  Next let us show \eqref{E:BLap}.
  For $Y\in C^1(\Gamma,T\Gamma)$ we see by \eqref{E:Def_FoRC} that
  \begin{align*}
    \int_\Gamma\Delta_BX\cdot Y\,d\mathcal{H}^2 = -\int_\Gamma\theta\Bigl(\overline{\nabla}^\ast\overline{\nabla}X,Y\Bigr)\,d\mathcal{H}^2 = -\int_\Gamma\theta\Bigl(\overline{\nabla}X,\overline{\nabla}Y\Bigr)\,d\mathcal{H}^2.
  \end{align*}
  Moreover, since $X$ is supported in $O$ and
  \begin{align*}
    \theta\Bigl(\overline{\nabla}X,\overline{\nabla}Y\Bigr) &= \sum_{i,j=1}^2\Bigl(\overline{\nabla}X\Bigr)(\tau_i,\tau_j)\Bigl(\overline{\nabla}Y\Bigr)(\tau_i,\tau_j) = \sum_{i,j=1}^2\Bigl(\overline{\nabla}_jX\cdot\tau_i\Bigr)\Bigl(\overline{\nabla}_jY\cdot\tau_i\Bigr) \\
    &= \sum_{j=1,2}\overline{\nabla}_jX\cdot\overline{\nabla}_jY = \nabla_\Gamma X:\nabla_\Gamma Y-WX\cdot WY
  \end{align*}
  on $O$ by \eqref{E:Def_InTe}, \eqref{E:RC_Ten}, and \eqref{E:Inn_LOF}, it follows that
  \begin{align*}
    \int_\Gamma\Delta_BX\cdot Y\,d\mathcal{H}^2 = -\int_\Gamma\nabla_\Gamma X:\nabla_\Gamma Y\,d\mathcal{H}^2+\int_\Gamma WX\cdot WY\,d\mathcal{H}^2.
  \end{align*}
  To the right-hand side we apply \eqref{E:Inn_IbP} and $W^T=W$ on $\Gamma$ to obtain
  \begin{align*}
    \int_\Gamma\Delta_BX\cdot Y\,d\mathcal{H}^2 = \int_\Gamma(\Delta_\Gamma X+W^2X)\cdot Y\,d\mathcal{H}^2
  \end{align*}
  for all $Y\in C^1(\Gamma,T\Gamma)$ and, since $\Delta_BX$ and $W^2X$ are tangential on $\Gamma$,
  \begin{align*}
    \int_\Gamma\Delta_BX\cdot v\,d\mathcal{H}^2 = \int_\Gamma\{P(\Delta_\Gamma X)+W^2X\}\cdot v\,d\mathcal{H}^2
  \end{align*}
  for all $v\in C^1(\Gamma)^3$.
  Therefore, \eqref{E:BLap} is valid.
\end{proof}

\begin{lemma} \label{L:Weitzen}
  For $X\in C^2(\Gamma,T\Gamma)$ we have
  \begin{align} \label{E:Weitzen}
    \Delta_BX = \Delta_HX+\mathrm{Ric}(X) \quad\text{on}\quad \Gamma.
  \end{align}
\end{lemma}

\begin{proof}
  By \eqref{E:HLap} and \eqref{E:BLap} we have
  \begin{align*}
    \Delta_BX = \Delta_HX+(HW-W^2)X \quad\text{on}\quad \Gamma.
  \end{align*}
  Applying \eqref{E:Ric_Cur} to the last term we obtain \eqref{E:Weitzen}.
\end{proof}

Note that only intrinsic quantities appear in \eqref{E:Weitzen} unlike in \eqref{E:HLap} and \eqref{E:BLap}.
The formula \eqref{E:Weitzen} is called the Weitzenb\"{o}ck formula (see e.g. \cites{Jo11,Pe06}).
Lemma \ref{L:Vec_Lap} was also shown in \cite{DudMitMit06}*{Theorem 6.3} for a hypersurface in $\mathbb{R}^k$, $k\geq2$, where the authors first proved \eqref{E:BLap} by identifying $\overline{\nabla}X$ with a $k\times k$ matrix for a tangential vector field $X$ on $\Gamma$ and then combined \eqref{E:BLap} and \eqref{E:Weitzen} to obtain \eqref{E:HLap}.
Here we proved both of \eqref{E:HLap} and \eqref{E:BLap} directly without identifying $\overline{\nabla}X$ with a $3\times3$ matrix.

Finally, let us consider a vector Laplace operator on a sphere
\begin{align*}
  \Gamma = S_a^2 := \{x\in\mathbb{R}^3 \mid |x|=a\}, \quad a>0
\end{align*}
introduced in \cite{TeZi97} (see also \cites{LiTeWa92a,LiTeWa92b}).
Let
\begin{align*}
  \mu_a(\vartheta,\varphi) := (a\sin\vartheta\cos\varphi,a\sin\vartheta\sin\varphi,a\cos\vartheta), \quad (\vartheta,\varphi)\in[0,\pi]\times[0,2\pi]
\end{align*}
be a parametrization of $S_a^2$ in spherical coordinates and
\begin{align*}
  e_\vartheta = e_\vartheta(\vartheta,\varphi) :=
  \begin{pmatrix}
    \cos\vartheta\cos\varphi \\
    \cos\vartheta\sin\varphi \\
    -\sin\vartheta
  \end{pmatrix}, \quad
  e_\varphi = e_\varphi(\varphi) :=
  \begin{pmatrix}
    -\sin\varphi \\
    \cos\varphi \\
    0
  \end{pmatrix}.
\end{align*}
Note that $\{e_\vartheta,e_\varphi\}$ is an orthonormal basis of $T_yS_a^2$ with $y=\mu_a(\vartheta,\varphi)$.
Also, for a $C^2$ function $\xi=\xi(\vartheta,\varphi)$ let
\begin{align} \label{E:Def_LSSp}
  \Delta_\mu\xi(\vartheta,\varphi) := \frac{1}{a^2\sin\vartheta}\left\{\frac{\partial}{\partial\vartheta}\Bigl(\sin\vartheta\frac{\partial\xi}{\partial\vartheta}\Bigr)+\frac{1}{\sin\vartheta}\frac{\partial^2\xi}{\partial\varphi^2}\right\},
\end{align}
which is an expression of the scalar Laplace--Beltrami operator $\Delta_\Gamma$ on $S_a^2$ in spherical coordinates.
For $X\in C^2(S_a^2,TS_a^2)$ of the form
\begin{align} \label{E:TVec_S}
  X(\mu_a(\vartheta,\varphi)) = X_\vartheta(\vartheta,\varphi)e_\vartheta+X_\varphi(\vartheta,\varphi)e_\varphi
\end{align}
the tangential Laplacian of $X$ on $S_a^2$ introduced in \cite{TeZi97} is defined by
\begin{align} \label{E:Def_LVSp}
  \Delta_2X(\mu_a(\vartheta,\varphi)) := (\Delta_2X)_\vartheta(\vartheta,\varphi)e_\vartheta+(\Delta_2X)_\varphi(\vartheta,\varphi)e_\varphi
\end{align}
for $(\vartheta,\varphi)\in[0,\pi]\times[0,2\pi]$, where
\begin{align} \label{E:LVSp_Com}
  \begin{aligned}
    (\Delta_2X)_\vartheta &:= \Delta_\mu X_\vartheta-\frac{X_\vartheta}{a^2\sin^2\vartheta}-\frac{2\cos\vartheta}{a^2\sin^2\vartheta}\frac{\partial X_\varphi}{\partial\varphi}, \\
    (\Delta_2X)_\varphi &:= \Delta_\mu X_\varphi-\frac{X_\varphi}{a^2\sin^2\vartheta}+\frac{2\cos\vartheta}{a^2\sin^2\vartheta}\frac{\partial X_\vartheta}{\partial\varphi}
  \end{aligned}
\end{align}
on $[0,\pi]\times[0,2\pi]$ with $\Delta_\mu X_\vartheta$ and $\Delta_\mu X_\varphi$ given by \eqref{E:Def_LSSp}.

\begin{lemma} \label{L:Vec_Sph}
  For $a>0$ let $X\in C^2(S_a^2,TS_a^2)$.
  Then
  \begin{align} \label{E:Vec_Sph}
    \Delta_2 X = \Delta_HX = \Delta_BX-\frac{1}{a^2}X \quad\text{on}\quad S_a^2.
  \end{align}
\end{lemma}

The formula \eqref{E:Vec_Sph} can be shown by calculations of differential forms in spherical coordinates, but here we prove it by using Lemmas \ref{L:Lap_Rest} and \ref{L:Vec_Lap}.

\begin{proof}
  For a sufficiently small $\delta>0$ let
  \begin{align*}
    N_a := \{x\in\mathbb{R}^3 \mid a-\delta<|x|<a+\delta\}
  \end{align*}
  be a tubular neighborhood of $S_a^2$ and
  \begin{multline*}
    \Psi(r,\vartheta,\varphi) := (r\sin\vartheta\cos\varphi,r\sin\vartheta\sin\varphi,r\cos\vartheta), \\
    (r,\vartheta,\varphi)\in(a-\delta,a+\delta)\times[0,\pi]\times[0,2\pi]
  \end{multline*}
  be a parametrization of $N_a$ in spherical coordinates.
  Also, let
  \begin{align*}
    e_r = e_r(\vartheta,\varphi) :=
    \begin{pmatrix}
      \sin\vartheta\cos\varphi \\
      \sin\vartheta\sin\varphi \\
      \cos\vartheta
    \end{pmatrix}, \quad (\vartheta,\varphi)\in[0,\pi]\times[0,2\pi].
  \end{align*}
  Note that $\{e_r,e_\vartheta,e_\varphi\}$ is an orthonormal basis of $\mathbb{R}^3$ and
  \begin{align} \label{Pf_VSp:UNV}
    e_r(\vartheta,\varphi) = n(\mu_a(\vartheta,\varphi)), \quad (\vartheta,\varphi)\in[0,\pi]\times[0,2\pi],
  \end{align}
  where $n(y)=y/a$, $y\in S_a^2$ is the unit outward normal vector field of $S_a^2$.
  For a $C^2$ function $\zeta=\zeta(r,\vartheta,\varphi)$ we set
  \begin{align} \label{Pf_VSp:ScL}
    \Delta_\Psi\zeta(r,\vartheta,\varphi) := \frac{1}{r^2}\frac{\partial}{\partial r}\Bigl(r^2\frac{\partial\zeta}{\partial r}\Bigr)+\frac{1}{r^2\sin\vartheta}\left\{\frac{\partial}{\partial\vartheta}\Bigl(\sin\vartheta\frac{\partial\zeta}{\partial\vartheta}\Bigr)+\frac{1}{\sin\vartheta}\frac{\partial^2\zeta}{\partial\varphi^2}\right\},
  \end{align}
  which is an expression of the scalar Laplace operator on $\mathbb{R}^3$ in spherical coordinates.
  Then for $u=(u_1,u_2,u_3)^T\in C^2(N_a)^3$ of the form
  \begin{align*}
    u(\Psi(r,\vartheta,\varphi)) = u_r(r,\vartheta,\varphi)e_r+u_\vartheta(r,\vartheta,\varphi)e_\vartheta+u_\varphi(r,\vartheta,\varphi)e_\varphi
  \end{align*}
  the vector Laplacian $\Delta u=(\Delta u_1,\Delta u_2,\Delta u_3)^T$ is expressed as
  \begin{align} \label{Pf_VSp:VeL}
    \Delta u(\Psi(r,\vartheta,\varphi)) = (\Delta u)_r(r,\vartheta,\varphi)e_r+(\Delta u)_\vartheta(r,\vartheta,\varphi)e_\vartheta+(\Delta u)_\varphi(r,\vartheta,\varphi)e_\varphi
  \end{align}
  for $(r,\vartheta,\varphi)\in(a-\delta,a+\delta)\times[0,\pi]\times[0,2\pi]$, where
  \begin{align} \label{Pf_VSp:Com}
    \begin{aligned}
      (\Delta u)_r &:= \Delta_\Psi u_r-\frac{2u_r}{r^2}-\frac{2}{r^2\sin\vartheta}\frac{\partial(u_\vartheta\sin\vartheta)}{\partial\vartheta}-\frac{2}{r^2\sin\vartheta}\frac{\partial u_\varphi}{\partial\varphi}, \\
      (\Delta u)_\vartheta &:= \Delta_\Psi u_\vartheta-\frac{u_\vartheta}{r^2\sin^2\vartheta}+\frac{2}{r^2}\frac{\partial u_r}{\partial\vartheta}-\frac{2\cos\vartheta}{r^2\sin^2\vartheta}\frac{\partial u_\varphi}{\partial\varphi}, \\
      (\Delta u)_\varphi &:= \Delta_\Psi u_\varphi-\frac{u_\varphi}{r^2\sin^2\vartheta}+\frac{2}{r^2\sin\vartheta}\frac{\partial u_r}{\partial\varphi}+\frac{2\cos\vartheta}{r^2\sin^2\vartheta}\frac{\partial u_\vartheta}{\partial\varphi}
    \end{aligned}
  \end{align}
  on $(a-\delta,a+\delta)\times[0,\pi]\times[0,2\pi]$ with $\Delta_\Psi u_r$, $\Delta_\Psi u_\vartheta$, $\Delta_\Psi u_\varphi$ given by \eqref{Pf_VSp:ScL}.

  Let $X\in C^2(S_a^2,TS_a^2)$ be of the form \eqref{E:TVec_S} and its constant extension
  \begin{align*}
    \overline{X}(x) = X(\pi(x)) = X\left(\frac{ax}{|x|}\right), \quad x\in N_a
  \end{align*}
  be expressed in spherical coordinates as
  \begin{multline} \label{Pf_VSp:Const}
    \overline{X}(\Psi(r,\vartheta,\varphi)) = \overline{X}_r(r,\vartheta,\varphi)e_r+\overline{X}_\vartheta(r,\vartheta,\varphi)e_\vartheta+\overline{X}_\varphi(r,\vartheta,\varphi)e_\varphi, \\
    (r,\vartheta,\varphi)\in(a-\delta,a+\delta)\times[0,\pi]\times[0,2\pi].
  \end{multline}
  Then since $\overline{X}(\Psi(r,\vartheta,\varphi))=X(\mu_a(\vartheta,\varphi))$, we see by \eqref{E:TVec_S} and \eqref{Pf_VSp:Const} that
  \begin{align*}
    \overline{X}_r(r,\vartheta,\varphi) = 0, \quad \overline{X}_\vartheta(r,\vartheta,\varphi) = X_\vartheta(\vartheta,\varphi), \quad \overline{X}_\varphi(r,\vartheta,\varphi) = X_\varphi(\vartheta,\varphi).
  \end{align*}
  From these equalities, \eqref{E:Def_LSSp}, \eqref{E:LVSp_Com}, \eqref{Pf_VSp:ScL}, and \eqref{Pf_VSp:Com} it follows that
  \begin{align} \label{Pf_VSp:LaCom}
    \left\{
    \begin{aligned}
      \Delta_\Psi\overline{X}_\lambda(a,\vartheta,\varphi) &= \Delta_\mu X_\lambda(\vartheta,\varphi), \\
      \Bigl(\Delta\overline{X}\Bigr)_\lambda(a,\vartheta,\varphi) &= (\Delta_2X)_\lambda(\vartheta,\varphi),
    \end{aligned}
    \right.
    \quad \lambda = \vartheta,\varphi
  \end{align}
  for $(\vartheta,\varphi)\in[0,\pi]\times[0,2\pi]$.
  Moreover, by \eqref{Pf_VSp:UNV} we have
  \begin{align} \label{Pf_VSp:PONB}
    P(\mu_a(\vartheta,\varphi))e_r = 0, \quad P(\mu_a(\vartheta,\varphi))e_\vartheta = e_\vartheta, \quad P(\mu_a(\vartheta,\varphi))e_\varphi = e_\varphi,
  \end{align}
  since $\{e_r,e_\vartheta,e_\varphi\}$ is an orthonormal basis of $\mathbb{R}^3$.
  Noting that
  \begin{align*}
    \Psi(a,\vartheta,\varphi)=\mu_a(\vartheta,\varphi) \in S_a^2, \quad (\vartheta,\varphi)\in[0,\pi]\times[0,2\pi],
  \end{align*}
  we deduce from \eqref{E:Def_LVSp}, \eqref{Pf_VSp:VeL}, \eqref{Pf_VSp:LaCom}, and \eqref{Pf_VSp:PONB} that
  \begin{align*}
    P(\mu_a(\vartheta,\varphi))\Bigl(\Delta\overline{X}\Bigr)(\mu_a(\vartheta,\varphi)) &= P(\mu_a(\vartheta,\varphi))\Bigl(\Delta\overline{X}\Bigr)(\Psi(a,\vartheta,\varphi)) \\
    &= \Bigl(\Delta\overline{X}\Bigr)_\vartheta(a,\vartheta,\varphi)e_\vartheta+\Bigl(\Delta\overline{X}\Bigr)_\varphi(a,\vartheta,\varphi)e_\varphi \\
    &= (\Delta_2X)_\vartheta(\vartheta,\varphi)e_\vartheta+(\Delta_2X)_\varphi(\vartheta,\varphi)e_\varphi \\
    &= \Delta_2X(\mu_a(\vartheta,\varphi))
  \end{align*}
  for all $(\vartheta,\varphi)\in[0,\pi]\times[0,2\pi]$.
  By this fact and \eqref{E:Lap_Rest} we have
  \begin{align} \label{Pf_VSp:D2_LV}
    \Delta_2X = P\Bigl(\Delta\overline{X}\Bigr) = P(\Delta_\Gamma X) \quad\text{on}\quad S_a^2.
  \end{align}
  Now we consider the extension $\tilde{n}(x):=x/a$, $x\in N_a$ of the unit outward normal vector field $n$ of $S_a^2$ to get
  \begin{align*}
    W(y) = -\nabla_\Gamma n(y) = -P(y)\nabla\tilde{n}(y) = -\frac{1}{a}P(y), \quad y\in S_a^2.
  \end{align*}
  Thus the principal curvatures of $S_a^2$ are $\kappa_1=\kappa_2=-1/a$ and
  \begin{align*}
    H = \kappa_1+\kappa_2 = -\frac{2}{a}, \quad WX = -\frac{1}{a}X, \quad W^2X = \frac{1}{a^2}X, \quad (2W^2-HW)X = 0
  \end{align*}
  on $S_a^2$.
  By these equalities, \eqref{E:HLap}, \eqref{E:BLap}, and \eqref{Pf_VSp:D2_LV} we obtain \eqref{E:Vec_Sph}.
\end{proof}

\subsection{Deformation operator} \label{SS:ApV_Def}
For $X\in C^1(\Gamma,T\Gamma)$ and $Y,Z\in C(\Gamma,T\Gamma)$ let
\begin{align} \label{E:Def_Defo}
  (\mathrm{Def}\,X)(Y,Z) := \frac{1}{2}\Bigl(\overline{\nabla}_YX\cdot Z+Y\cdot\overline{\nabla}_ZX\Bigr) \quad\text{on}\quad \Gamma,
\end{align}
i.e. $\mathrm{Def}\,X$ is the symmetric part of $\overline{\nabla}X$.
We call
\begin{align*}
  \mathrm{Def}\colon C^1(\Gamma,T\Gamma)\to C(\Gamma,T^\ast\Gamma\otimes T^\ast\Gamma), \quad X\mapsto\mathrm{Def}\,X
\end{align*}
the deformation operator and $\mathrm{Def}\,X$ the deformation tensor for $X$.
Also, let
\begin{align*}
  \mathrm{Def}^\ast\colon C^1(\Gamma,T^\ast\Gamma\otimes T^\ast\Gamma)\to C(\Gamma,T\Gamma)
\end{align*}
be the formal adjoint of $\mathrm{Def}$ given by
\begin{align} \label{E:Def_FoDe}
  \int_\Gamma\theta(\mathrm{Def}^\ast S,X)\,d\mathcal{H}^2 := \int_\Gamma\theta(S,\mathrm{Def}\,X)\,d\mathcal{H}^2
\end{align}
for $S\in C^1(\Gamma,T^\ast\Gamma\otimes T^\ast\Gamma)$ and $X\in C^1(\Gamma,T\Gamma)$.
We show that the viscous term in the limit equations \eqref{E:NS_Limit} is written in terms of $\mathrm{Def}$ and $\mathrm{Def}^\ast$.

\begin{lemma} \label{L:Vis_Defo}
  Let $g\in C^1(\Gamma)$ and $X\in C^2(\Gamma,T\Gamma)$.
  Then
  \begin{align} \label{E:Vis_Defo}
    P\mathrm{div}_\Gamma[gD_\Gamma(X)] = -\mathrm{Def}^\ast(g\,\mathrm{Def}\,X) \quad\text{on}\quad \Gamma.
  \end{align}
\end{lemma}

\begin{proof}
  By a localization argument with a partition of unity on $\Gamma$ we may assume that $X$ is supported in a relatively open subset $O$ of $\Gamma$ on which we can take a local orthonormal frame $\{\tau_1,\tau_2\}$ for $T\Gamma$.
  For $Y\in C^1(\Gamma,T\Gamma)$ we have
  \begin{align} \label{Pf_VDe:Int_DD}
    \begin{aligned}
      \int_\Gamma\mathrm{Def}^\ast(g\,\mathrm{Def}\,X)\cdot Y\,d\mathcal{H}^2 &= \int_\Gamma\theta(\mathrm{Def}^\ast(g\,\mathrm{Def}\,X),Y)\,d\mathcal{H}^2 \\
      &= \int_\Gamma\theta(g\,\mathrm{Def}\,X,\mathrm{Def}\,Y)\,d\mathcal{H}^2
    \end{aligned}
  \end{align}
  by \eqref{E:Def_FoDe}.
  Moreover, letting $\overline{\nabla}_i:=\overline{\nabla}_{\tau_i}$ for $i=1,2$ we see by \eqref{E:Def_InTe} that
  \begin{align*}
    \theta(g\,\mathrm{Def}\,X,\mathrm{Def}\,Y) &= g\sum_{i,j=1}^2(\mathrm{Def}\,X)(\tau_i,\tau_j)(\mathrm{Def}\,Y)(\tau_i,\tau_j) \\
    &= \frac{g}{2}\sum_{i,j=1}^2\left\{\Bigl(\overline{\nabla}_iX\cdot\tau_j\Bigr)\Bigl(\overline{\nabla}_iY\cdot\tau_j\Bigr)+\Bigl(\overline{\nabla}_iX\cdot\tau_j\Bigr)\Bigl(\overline{\nabla}_jY\cdot\tau_i\Bigr)\right\} \\
    &= \frac{g}{2}\left\{\sum_{i=1,2}\overline{\nabla}_iX\cdot\overline{\nabla}_iY+\sum_{i,j=1}^2\Bigl(\overline{\nabla}_iX\cdot\tau_j\Bigr)\Bigl(\overline{\nabla}_jY\cdot\tau_i\Bigr)\right\}
  \end{align*}
  on $O$.
  Hence we apply \eqref{E:Inn_LOF} and \eqref{E:InTr_LOF} to the last line to get
  \begin{align} \label{Pf_VDe:gDe}
    \theta(g\,\mathrm{Def}\,X,\mathrm{Def}\,Y) = \frac{g}{2}\{\nabla_\Gamma X:\nabla_\Gamma Y+(\nabla_\Gamma X)^T:\nabla_\Gamma Y-WX\cdot WY\}
  \end{align}
  on $O$.
  We also observe by \eqref{E:Def_SSR}, \eqref{E:Grad_W}, and $P^T=P$ on $\Gamma$ that
  \begin{multline} \label{Pf_VDe:Sym}
    \frac{1}{2}\{\nabla_\Gamma X:\nabla_\Gamma Y+(\nabla_\Gamma X)^T:\nabla_\Gamma Y\} \\
    = D_\Gamma(X):\nabla_\Gamma Y+\frac{1}{2}\{(WX)\otimes n+n\otimes(WX)\}:\nabla_\Gamma Y \quad\text{on}\quad O.
  \end{multline}
  Moreover, since $\{\tau_1,\tau_2,n\}$ is an orthonormal basis of $\mathbb{R}^3$,
  \begin{multline*}
    \{(WX)\otimes n+n\otimes(WX)\}:\nabla_\Gamma Y \\
    = \sum_{i=1,2}\{(WX)\otimes n+n\otimes(WX)\}\tau_i\cdot(\nabla_\Gamma Y)\tau_i \\
    +\{(WX)\otimes n+n\otimes(WX)\}n\cdot(\nabla_\Gamma Y)n
  \end{multline*}
  on $O$.
  To the right-hand side we apply $(\nabla_\Gamma Y)n=WY$ by \eqref{E:Grad_W} and
  \begin{align*}
    \{(WX)\otimes n+n\otimes(WX)\}\tau_i &= (n\cdot\tau_i)WX+(WX\cdot\tau_i)n = (WX\cdot\tau_i)n, \\
    \{(WX)\otimes n+n\otimes(WX)\}n &= (n\cdot n)WX+(WX\cdot n)n = WX
  \end{align*}
  and use the fact that $(\nabla_\Gamma Y)\tau_i=P(\nabla_\Gamma Y)\tau_i$ is tangential on $O$ to get
  \begin{align*}
    \{(WX)\otimes n+n\otimes(WX)\}:\nabla_\Gamma Y = WX\cdot WY \quad\text{on}\quad O.
  \end{align*}
  Applying this equality and \eqref{Pf_VDe:Sym} to the right-hand side of \eqref{Pf_VDe:gDe} we have
  \begin{align*}
    \theta(g\,\mathrm{Def}\,X,\mathrm{Def}\,Y) = g\{D_\Gamma(X):\nabla_\Gamma Y\} \quad\text{on}\quad O.
  \end{align*}
  Noting that $X$ is supported in $O$, we use this equality and \eqref{Pf_VDe:Int_DD} to get
  \begin{align} \label{Pf_VDe:Int_SS}
    \int_\Gamma\mathrm{Def}^\ast(g\,\mathrm{Def}\,X)\cdot Y\,d\mathcal{H}^2 = \int_\Gamma g\{D_\Gamma(X):\nabla_\Gamma Y\}\,d\mathcal{H}^2.
  \end{align}
  Moreover, for $A=(A_{kl})_{k,l}\in C^1(\Gamma)^{3\times 3}$ and $v=(v_1,v_2,v_3)^T\in C^1(\Gamma)^3$,
  \begin{align*}
    \int_\Gamma A:\nabla_\Gamma v\,d\mathcal{H}^2 &= \sum_{k,l=1}^3\int_\Gamma A_{kl}(\underline{D}_kv_l)\,d\mathcal{H}^2 \\
    &= -\sum_{k,l=1}^3\int_\Gamma(\underline{D}_kA_{kl}+A_{kl}Hn_k)v_l\,d\mathcal{H}^2 \\
    &= -\int_\Gamma(\mathrm{div}_\Gamma A+HA^Tn)\cdot v\,d\mathcal{H}^2
  \end{align*}
  by \eqref{E:IbP_TD}.
  Applying this with $A=gD_\Gamma(X)$ and $v=Y$ to \eqref{Pf_VDe:Int_SS} and using
  \begin{align*}
    D_\Gamma(X)^Tn = \frac{1}{2}P\{\nabla_\Gamma X+(\nabla_\Gamma X)^T\}Pn = 0 \quad\text{on}\quad \Gamma
  \end{align*}
  by \eqref{E:Def_SSR} and $P^T=P$ on $\Gamma$, we obtain
  \begin{align*}
    \int_\Gamma\mathrm{Def}^\ast(g\,\mathrm{Def}\,X)\cdot Y\,d\mathcal{H}^2 = -\int_\Gamma\mathrm{div}_\Gamma[gD_\Gamma(X)]\cdot Y\,d\mathcal{H}^2
  \end{align*}
  for all $Y\in C^1(\Gamma,T\Gamma)$.
  Hence (note that $\mathrm{Def}^\ast(g\,\mathrm{Def}\,X)$ is tangential on $\Gamma$)
  \begin{align*}
    \int_\Gamma\mathrm{Def}^\ast(g\,\mathrm{Def}\,X)\cdot v\,d\mathcal{H}^2 = -\int_\Gamma P\mathrm{div}_\Gamma[gD_\Gamma(X)]\cdot v\,d\mathcal{H}^2
  \end{align*}
  for all $v\in C^1(\Gamma)^3$ and \eqref{E:Vis_Defo} follows.
\end{proof}

When $g\equiv1$ we have another form of the right-hand side of \eqref{E:Vis_Defo}.

\begin{lemma} \label{L:DeDe_BL}
  For $X\in C^1(\Gamma,T\Gamma)$ we have
  \begin{align} \label{E:DeDe_BL}
    2\mathrm{Def}^\ast\mathrm{Def}\,X = -\Delta_BX-\nabla_\Gamma(\mathrm{div}_\Gamma X)-\mathrm{Ric}(X) \quad\text{on}\quad \Gamma.
  \end{align}
\end{lemma}

\begin{proof}
  We may assume that $X$ is supported in the relatively open subset $O$ of $\Gamma$ given in the proof of Lemma \ref{L:Vis_Defo}.
  Then by \eqref{Pf_VDe:Int_DD} and \eqref{Pf_VDe:gDe} with $g\equiv1$ we have
  \begin{align*}
    \int_\Gamma2\mathrm{Def}^\ast\mathrm{Def}\,X\cdot Y\,d\mathcal{H}^2 = \int_\Gamma\{\nabla_\Gamma X:\nabla_\Gamma Y+(\nabla_\Gamma X)^T:\nabla_\Gamma Y-WX\cdot WY\}\,d\mathcal{H}^2.
  \end{align*}
  We apply \eqref{E:Inn_IbP}, \eqref{E:InTr_IbP}, $W^T=W$ on $\Gamma$ to the right-hand side to get
  \begin{align*}
    \int_\Gamma2\mathrm{Def}^\ast\mathrm{Def}\,X\cdot Y\,d\mathcal{H}^2 = -\int_\Gamma\{\Delta_\Gamma X+\nabla_\Gamma(\mathrm{div}_\Gamma X)+HWX\}\cdot Y\,d\mathcal{H}^2
  \end{align*}
  for all $Y\in C^1(\Gamma,T\Gamma)$ and thus
  \begin{align*}
    \int_\Gamma2\mathrm{Def}^\ast\mathrm{Def}\,X\cdot v\,d\mathcal{H}^2 = -\int_\Gamma\{P(\Delta_\Gamma X)+\nabla_\Gamma(\mathrm{div}_\Gamma X)+HWX\}\cdot v\,d\mathcal{H}^2
  \end{align*}
  for all $v\in C^1(\Gamma)^3$ since $\mathrm{Def}^\ast\mathrm{Def}\,X$, $\nabla_\Gamma(\mathrm{div}_\Gamma X)$, and $WX$ are tangential on $\Gamma$.
  By this equality, \eqref{E:Ric_Cur}, and \eqref{E:BLap} we find that
  \begin{align*}
    2\mathrm{Def}^\ast\mathrm{Def}\,X &= -P(\nabla_\Gamma X)-\nabla_\Gamma(\mathrm{div}_\Gamma X)-HWX \\
    &= -\Delta_BX-\nabla_\Gamma(\mathrm{div}_\Gamma X)-\mathrm{Ric}(X)
  \end{align*}
  on $\Gamma$.
  Thus \eqref{E:DeDe_BL} is valid.
\end{proof}

By \eqref{E:Vis_Defo} with $g\equiv 1$ and \eqref{E:DeDe_BL} we also have the following formula.

\begin{lemma} \label{L:DiSR_BL}
  For $X\in C^2(\Gamma,T\Gamma)$ we have
  \begin{align} \label{E:DiSR_BL}
    2P\mathrm{div}_\Gamma[D_\Gamma(X)] = \Delta_BX+\nabla_\Gamma(\mathrm{div}_\Gamma X)+\mathrm{Ric}(X) \quad\text{on}\quad \Gamma.
  \end{align}
\end{lemma}

Note that the right-hand side of \eqref{E:DiSR_BL} is intrinsic although the left-hand side is defined componentwisely in the fixed coordinate system of $\mathbb{R}^3$.
Also, if $\mathrm{div}_\Gamma X=0$ on $\Gamma$, then \eqref{E:DiSR_BL} reduces to
\begin{align*}
  2P\mathrm{div}_\Gamma[D_\Gamma(X)] = \Delta_BX+\mathrm{Ric}(X) \quad\text{on}\quad \Gamma
\end{align*}
and the right-hand side gives the viscous term in the Navier--Stokes equations on a Riemannian manifold (see \cite{Ta92}).

\section{Construction of a weak solution to the limit equations} \label{S:Ap_CWL}
In this appendix we explain the outline of construction of a weak solution to the limit equations \eqref{E:NS_Limit} for given data
\begin{align*}
  v_0\in\mathcal{H}_g = L_{g\sigma}^2(\Gamma,T\Gamma), \quad f\in L_{loc}^2([0,\infty);H^{-1}(\Gamma,T\Gamma))
\end{align*}
by the Galerkin method (see also the case of the usual Navier--Stokes equations in a bounded domain in $\mathbb{R}^2$ explained in \cites{BoFa13,CoFo88,Te79}).
We assume that $\Gamma$ is a $C^5$ closed, connected, and oriented surface in $\Gamma$ but do not impose Assumptions \ref{Assump_1} and \ref{Assump_2}.
Also, we use the notations in Sections \ref{S:Pre}, \ref{S:WSol}, and \ref{S:SL}.

\subsection{Countable basis of the weighted solenoidal spaces} \label{SS:CWL_CoB}
We define a bilinear form on $H^1(\Gamma,T\Gamma)$ by
\begin{align*}
  \tilde{a}_g(v_1,v_2) := a_g(v_1,v_2)+(v_1,v_2)_{L^2(\Gamma)}, \quad v_1,v_2\in H^1(\Gamma,T\Gamma).
\end{align*}
Then $\tilde{a}_g$ is bounded, coercive, and symmetric on $\mathcal{V}_g=H_{g\sigma}^1(\Gamma,T\Gamma)$ by Lemma \ref{L:Bi_Surf}.
Since $\mathcal{V}_g$ is a closed subspace of $H^1(\Gamma,T\Gamma)$, the Lax--Milgram theorem implies that $\tilde{a}_g$ induces a bounded linear operator $\widetilde{A}_g$ from $\mathcal{V}_g$ into its dual space $\mathcal{V}'_g$.
We consider $\widetilde{A}_g$ as an unbounded operator on $\mathcal{H}_g$ equipped with inner product
\begin{align*}
  (v_1,v_2)_{L_g^2(\Gamma)} := (g^{1/2}v_1,g^{1/2}v_2)_{L^2(\Gamma)}, \quad v_1,v_2\in\mathcal{H}_g,
\end{align*}
which is equivalent to the canonical $L^2(\Gamma)$-inner product by \eqref{E:G_Inf}.
Then as in the case of the Stokes operator for a bounded domain (see e.g. \cite{BoFa13}*{Theorem IV.5.5}) we can show that there exists a sequence $\{w_k\}_{k=1}^\infty$ of eigenvectors of $\widetilde{A}_g$ that is an orthonormal basis of $\mathcal{H}_g$ equipped with inner product $(\cdot,\cdot)_{L_g^2(\Gamma)}$ and an orthogonal basis of $\mathcal{V}_g$ equipped with inner product $\tilde{a}_g(\cdot,\cdot)$.
In particular,
\begin{align} \label{E:Wei_Orth}
  (gw_i,w_j)_{L^2(\Gamma)} = (w_i,w_j)_{L_g^2(\Gamma)} = \delta_{ij}, \quad i,j\in\mathbb{N},
\end{align}
where $\delta_{ij}$ is the Kronecker delta.

\subsection{Approximate problem} \label{SS:CWL_Ap}
For a fixed $T>0$ let
\begin{align*}
  F_T(t) :=
  \begin{cases}
    f(t), &t\in(0,T), \\
    0, & t\in\mathbb{R}\setminus(0,T).
  \end{cases}
\end{align*}
Then $F_T\in L^2(\mathbb{R};H^{-1}(\Gamma,T\Gamma))$ by the assumption on $f$.
For $k\in\mathbb{N}$ let
\begin{align*}
  f_k(t) := \frac{1}{\varepsilon_k}\int_{-\infty}^\infty\rho\left(\frac{t-s}{\varepsilon_k}\right)F_T(s)\,ds, \quad t\in\mathbb{R} \quad \left(\varepsilon_k := \frac{1}{k}\right)
\end{align*}
be the regularization of $F_T$, where $\rho$ is a standard mollifier on $\mathbb{R}$.
Then
\begin{gather*}
  f_k \in C_c^\infty(\mathbb{R};H^{-1}(\Gamma,T\Gamma)), \\
  \|f_k\|_{L^2(0,T;H^{-1}(\Gamma,T\Gamma))} \leq \|F_T\|_{L^2(\mathbb{R};H^{-1}(\Gamma,T\Gamma))} = \|f\|_{L^2(0,T;H^{-1}(\Gamma,T\Gamma))}
\end{gather*}
for all $k\in\mathbb{N}$.
Moreover, since $F_T=f$ on $(0,T)$,
\begin{align*}
  \lim_{k\to\infty}f_k = f \quad\text{strongly in}\quad L^2(0,T;H^{-1}(\Gamma,T\Gamma)).
\end{align*}
For $k\in\mathbb{N}$ let $\mathcal{V}_g^k$ be the linear span of $\{w_i\}_{i=1}^k$.
We look for a vector field
\begin{align*}
  v_k \in C^1([0,T];\mathcal{V}_g^k), \quad v_k(t) = \sum_{i=1}^k\xi_i(t)w_i, \quad t\in[0,T]
\end{align*}
with $\xi_i\in C^1([0,T])$, $i=1,\dots,k$ satisfying the approximate problem
\begin{multline} \label{E:LimG_Ap}
  (g\partial_tv_k(t),\eta_k)_{L^2(\Gamma)}+a_g(v_k(t),\eta_k) +b_g(v_k(t),v_k(t),\eta_k) \\
  = [gf_k(t),\eta_k]_{T\Gamma}, \quad t\in(0,T)
\end{multline}
for all $\eta_k\in \mathcal{V}_g^k$ with initial condition
\begin{align} \label{E:LGAp_Ini}
  v_k(0) = \sum_{i=1}^k(v_0,w_i)_{L_g^2(\Gamma)}w_i.
\end{align}
This problem is equivalent to the system of ordinary differential equations
\begin{align*}
  \left\{
  \begin{aligned}
    \sum_{i=1}^k(gw_i,w_j)_{L^2(\Gamma)}\frac{d\xi_i}{dt}(t) &= \mathcal{P}_j(\xi(t))+[gf_k(t),w_j]_{T\Gamma}, \quad t\in(0,T), \\
    \xi_j(0) &= (v_0,w_j)_{L_g^2(\Gamma)}
  \end{aligned}
  \right.
\end{align*}
for $j=1,\dots,k$, where $\mathcal{P}_1,\dots,\mathcal{P}_k$ are polynomials of $\xi=(\xi_1,\dots,\xi_k)^T\in\mathbb{R}^k$.
Using \eqref{E:Wei_Orth} we see that this system reduces to
\begin{align*}
  \left\{
  \begin{aligned}
    \frac{d\xi_j}{dt}(t) &= \mathcal{P}_j(\xi(t))+[gf_k(t),w_j]_{T\Gamma}, \quad t\in(0,T), \\
    \xi_j(0) &= (v_0,w_j)_{L_g^2(\Gamma)}
  \end{aligned}
  \right.
\end{align*}
for $j=1,\dots,k$, which we can solve locally by the Cauchy--Lipschitz theorem.
Also, setting $\eta_k=v_k(t)\in\mathcal{V}_g^k$ in \eqref{E:LimG_Ap} and applying \eqref{E:Bi_Surf} and \eqref{E:TriS_Vg} we can derive the energy estimate for the approximate solution $v_k$ of the form
\begin{multline} \label{E:Ener_Ap}
  \max_{t\in[0,T_k]}\|v_k(t)\|_{L^2(\Gamma)}^2+\int_0^{T_k}\|\nabla_\Gamma v_k(t)\|_{L^2(\Gamma)}^2\,dt \\
  \leq c_T\left(\|v_0\|_{L^2(\Gamma)}^2+\|f\|_{L^2(0,T;H^{-1}(\Gamma,T\Gamma))}^2\right)
\end{multline}
as in the proof of Lemma \ref{L:PMu_Energy}, where $T_k\in(0,T]$ is the maximal existence time of $v_k$ and $c_T>0$ is a constant depending only on $T$.
By this estimate we further get $T_k=T$ since the right-hand side is independent of $T_k$.

\subsection{Time derivative of the approximate solution} \label{SS:CWL_Dt}
The energy estimate \eqref{E:Ener_Ap} implies the weak convergence of (a subsequence of) $\{v_k\}_{k=1}^\infty$ in appropriate function spaces on $\Gamma$.
To get the strong convergence of $\{v_k\}_{k=1}^\infty$ by the Aubin--Lions lemma we estimate $\partial_tv_k$ in $H^{-1}(\Gamma,T\Gamma)$ as in Section \ref{SS:SL_EDt} (see also Remark \ref{R:Mu_Dt}).

Let $w\in H^1(\Gamma,T\Gamma)$.
Then there exist $\eta\in\mathcal{V}_g$ and $q\in H^2(\Gamma)$ such that
\begin{align*}
  w = g\eta+g\nabla_\Gamma q \quad\text{on}\quad \Gamma, \quad \|\eta\|_{H^1(\Gamma)} \leq c\|w\|_{H^1(\Gamma)}
\end{align*}
by Lemma \ref{L:Mu_Dt_Test}.
Since $\{w_k\}_{k=1}^\infty$ is an orthogonal basis of $\mathcal{V}_g$ equipped with inner product $\tilde{a}_g(\cdot,\cdot)$,
\begin{align*}
  \eta = \sum_{i=1}^\infty \tilde{a}_g(\eta,\tilde{w_i})\tilde{w}_i \quad\text{in}\quad \mathcal{V}_g, \quad \tilde{w}_i := \frac{w_i}{a_g(w_i,w_i)^{1/2}}.
\end{align*}
Then for $\eta_k:=\sum_{i=1}^k\tilde{a}_g(\eta,\tilde{w}_i)\tilde{w}_i\in\mathcal{V}_g^k$, $k\in\mathbb{N}$ we have
\begin{align*}
  \|\eta_k\|_{H^1(\Gamma)} \leq c\|\eta\|_{H^1(\Gamma)} \leq c\|w\|_{H^1(\Gamma)}
\end{align*}
since $\tilde{a}_g(\cdot,\cdot)$ is equivalent to the canonical $H^1(\Gamma)$-inner product on $\mathcal{V}_g$.
Also,
\begin{align*}
  (g\partial_tv_k,\eta_k)_{L^2(\Gamma)} &= (g\partial_tv_k,\eta)_{L^2(\Gamma)} = (\partial_tv_k,g\eta)_{L^2(\Gamma)} \\
  &= (\partial_tv_k,w-g\nabla_\Gamma q)_{L^2(\Gamma)} = (\partial_tv_k,w)_{L^2(\Gamma)},
\end{align*}
where the first equality is due to \eqref{E:Wei_Orth} and the last equality follows from
\begin{align*}
  \partial_tv_k\in\mathcal{V}_g^k\subset\mathcal{H}_g, \quad g\nabla_\Gamma q\in\mathcal{H}_g^\perp
\end{align*}
by Lemma \ref{L:L2gs_Orth} (note that here we take the canonical $L^2(\Gamma)$-inner product).
Thus, substituting $\eta_k$ for \eqref{E:LimG_Ap} and using \eqref{E:Bi_Surf}--\eqref{E:TriS_Vg} and the above relations, we can show as in the proof of Lemma \ref{L:PMu_Dt} that
\begin{align*}
  \|\partial_tv_k(t)\|_{H^{-1}(\Gamma,T\Gamma)} \leq c\left\{\left(1+\|v_k(t)\|_{L^2(\Gamma)}\right)\|v_k(t)\|_{H^1(\Gamma)}+\|f_k(t)\|_{H^{-1}(\Gamma,T\Gamma)}\right\}
\end{align*}
for all $t\in(0,T)$.
By this inequality and \eqref{E:Ener_Ap} we get
\begin{align} \label{E:EDt_Ap}
  \|\partial_tv_k\|_{L^2(0,T;H^{-1}(\Gamma,T\Gamma))} \leq c = c(T,\|v_0\|_{L^2(\Gamma)},\|f\|_{L^2(0,T;H^{-1}(\Gamma,T\Gamma))}),
\end{align}
where the constant $c>0$ on the right-hand side is independent of $k\in\mathbb{N}$.

Now we observe by \eqref{E:Ener_Ap} and \eqref{E:EDt_Ap} that
\begin{itemize}
  \item $\{v_k\}_{k=1}^\infty$ is bounded in $L^\infty(0,T;\mathcal{H}_g)\cap L^2(0,T;\mathcal{V}_g)$,
  \item $\{\partial_tv_k\}_{k=1}^\infty$ is bounded in $L^2(0,T;H^{-1}(\Gamma,T\Gamma))$.
\end{itemize}
Moreover, $\{v_k(0)\}_{k=1}^\infty$ converges to $v_0$ strongly in $\mathcal{H}_g$ since $v_k(0)$ is given by \eqref{E:LGAp_Ini} and $\{w_k\}_{k=1}^\infty$ is an orthonormal basis of $\mathcal{H}_g$ equipped with inner product $(\cdot,\cdot)_{L_g^2(\Gamma)}$.
Hence we can show that $\{v_k\}_{k=1}^\infty$ converges to a unique weak solution $v_T$ to \eqref{E:NS_Limit} on $[0,T)$ as in the proof of Theorem \ref{T:SL_Weak}.
Setting $v:=v_T$ on $[0,T)$ for each $T>0$ we obtain a unique weak solution $v$ to \eqref{E:NS_Limit} on $[0,\infty)$.

\end{appendix}

\section*{Acknowledgments}
This work is an expanded version of a part of the doctoral thesis of the author \cite{Miu_DT} completed under the supervision of Professor Yoshikazu Giga at the University of Tokyo.
The author is grateful to him for his valuable comments on this work.
The author also would like to thank Mr. Yuuki Shimizu for fruitful discussions on Killing vector fields on surfaces and anonymous referees for valuable remarks.

The work of the author was supported by Grant-in-Aid for JSPS Fellows No. 16J02664 and No. 19J00693, and by the Program for Leading Graduate Schools, MEXT, Japan.


\end{document}